\documentclass[twoside,11pt]{article}

\RequirePackage{epsfig}
\RequirePackage{amssymb}
\RequirePackage{hyperref}
\RequirePackage{natbib}
\RequirePackage{graphicx}

\bibpunct{(}{)}{;}{a}{,}{,}
 
\newenvironment{keywords}
{\bgroup\leftskip 20pt\rightskip 20pt \small\noindent{\bf Keywords:} }%
{\par\egroup\vskip 0.25ex}

\newcommand{\BlackBox}{\rule{1.5ex}{1.5ex}}  
\newcommand{\qed}{\hfill\BlackBox\\[2mm]}
\newenvironment{proof}{\par\noindent{\bf Proof\ }}{\hfill\BlackBox\\[2mm]}
\newtheorem{example}{Example} 
\newtheorem{theorem}{Theorem}
\newtheorem{lemma}[theorem]{Lemma} 
\newtheorem{proposition}[theorem]{Proposition} 
\newtheorem{remark}[theorem]{Remark}

\newtheorem{proc}{Algorithm}

\usepackage{amsmath}

\usepackage{dsfont}
\newcommand{\un}{\mathds{1}}

\newcommand{\egaldef}{:=} % egalite definissant la quantite de gauche
\newcommand{\flens}{\rightarrow} % fleche d'application X->Y (ensembles)
\newcommand{\telque}{\ \text{s.t.} \ } % tel que dans une definition d'ensemble

\newcommand{\R}{\mathbb{R}} %corps des reels
\newcommand{\Z}{\mathbb{Z}} %anneau des entiers
\newcommand{\N}{\mathbb{N}} %entiers naturels

\newcommand{\B}{\mathcal{B}}
\newcommand{\Dcal}{\mathcal{D}}
\newcommand{\Ncal}{\mathcal{N}}
\newcommand{\X}{\mathcal{X}}

\newcommand{\dd}{\,\mathrm{d}}
\newcommand{\e}{\mathrm{e}}

\DeclareMathOperator{\tr}{tr} % trace
\DeclareMathOperator{\card}{Card} % cardinal
\DeclareMathOperator{\sign}{sign} % signe
\DeclareMathOperator{\argmintmp}{argmin} %argmin
\newcommand{\argmin}{\mathop{\argmintmp}}

\newcommand{\parenj}[1]{\mathopen{}\left( #1  \right) \mathclose{}} 
\newcommand{\parens}[1]{( #1 )} 
 
\newcommand{\parenb}[1]{\bigl( #1 \bigr)}
\newcommand{\parenB}[1]{\Bigl( #1  \Bigr)}
\newcommand{\parenbb}[1]{\biggl( #1 \biggr)}
\newcommand{\parenBb}[1]{\Biggl( #1  \Biggr)}
\newcommand{\parensq}[1]{\mathopen{}\left({#1}\right)^2\mathclose{}}

\newcommand{\parensqB}[1]{\Bigl( #1 \Bigr)^2}
\newcommand{\parensqbb}[1]{\biggl( #1 \biggr)^2}
\newcommand{\parensqBb}[1]{\Biggl( #1 \Biggr)^2}

\newcommand{\floor}[1]{\mathopen{}\left\lfloor #1 \right\rfloor\mathclose{}}
\newcommand{\floorb}[1]{\bigl\lfloor #1 \bigr\rfloor}

\newcommand{\crochj}[1]{\mathopen{}\left[ #1 \right] \mathclose{}} 
 
\newcommand{\crochs}[1]{[ #1 ]} 
\newcommand{\crochb}[1]{\bigl[ #1 \bigr]}
\newcommand{\crochB}[1]{\Bigl[ #1 \Bigr]}
\newcommand{\crochbb}[1]{\biggl[ #1 \biggr]}
\newcommand{\crochBb}[1]{\Biggl[ #1 \Biggr]}
\newcommand{\crochsq}[1]{\mathopen{}\left[ #1 \right]^2 \mathclose{}} 
\newcommand{\crochsqb}[1]{\bigl[ #1 \bigr]^2}
\newcommand{\crochsqB}[1]{\Bigl[ #1 \Bigr]^2}

\newcommand{\setj}[1]{\mathopen{}\left\{ #1 \right\} \mathclose{}}
\newcommand{\sets}[1]{\{ #1 \}}
\newcommand{\sset}[1]{\{ #1 \}}
\newcommand{\setb}[1]{\bigl\{ #1 \bigr\}}
\newcommand{\setB}[1]{\Bigl\{ #1  \Bigr\}}

\newcommand{\absj}[1]{\mathopen{} \left\lvert #1 \right\rvert \mathclose{}}
\newcommand{\abss}[1]{\lvert #1 \rvert}
\newcommand{\absb}[1]{\bigl\lvert #1 \bigr\rvert}
\newcommand{\absB}[1]{\Bigl\lvert #1 \Bigr\rvert}

\newcommand{\absBb}[1]{\Biggl\lvert #1 \Biggr\rvert}

\newcommand{\normsq}[1]{\mathopen{}\left \lVert #1 \right\rVert^2 \mathclose{}}
\newcommand{\normsqb}[1]{\bigl \lVert #1 \bigr\rVert^2}

\newcommand{\norms}[1]{\lVert #1 \rVert}
\newcommand{\normb}[1]{\bigl \lVert #1 \bigr\rVert}

\newcommand{\perte}[1]{\ensuremath{\mathopen{}\left \lVert #1 - \bayes \right\rVert^2 \mathclose{}}}%perte relative
\newcommand{\sperte}[1]{\ensuremath{\norms{ #1 - \bayes }^2}}%perte relative
\newcommand{\pertes}[1]{\ensuremath{\norms{ #1 - \bayes }^2}}%perte relative
\newcommand{\prodscal}[2]{\mathopen{}\left\langle #1 , \, #2 \right\rangle\mathclose{}} %produit scalaire de 2 vecteurs

\renewcommand{\P}{\mathbb{P}}
\newcommand{\Prob}{\mathbb{P}} %probabilite
\newcommand{\E}{\mathbb{E}} %esperance
\DeclareMathOperator{\var}{Var} %variance
\DeclareMathOperator{\cov}{Cov} %covariance
\newcommand{\sachant}{\, | \,}

\newcommand{\sachantb}{\, \big| \,}

\newcommand{\bayes}{s}%estimateur Bayesien
\newcommand{\ERM}{\widehat{s}}%ERM

\newcommand{\M}{\mathcal{M}}
\newcommand{\mM}{m \in \M}
\newcommand{\mo}{m^{\star}} 
\newcommand{\mh}{\widehat{m}}
\newcommand{\mHO}{\widehat{m}_{\HO}}

\DeclareMathOperator{\pen}{pen}%penalite
\DeclareMathOperator{\crit}{crit}%critere empirique
\newcommand{\pendim}{\pen_{\mathrm{dim}}}% penalite prop a d_m (type Cp/AIC)
\newcommand{\critVF}{\crit_{\mathrm{VFCV}}} % Critere V-fold classique
\newcommand{\critCV}{\crit_{\mathrm{CV}}}% Critere Validation croisee general
\newcommand{\critHO}{\crit_{\HO}}% Critere hold-out
\newcommand{\critcorrVF}{\crit_{\mathrm{corr},\mathrm{VFCV}}} % Critere V-fold corrige [Burman]
\newcommand{\VF}{\ensuremath{\mathrm{VF}}}
\newcommand{\HO}{\ensuremath{\mathrm{HO}}}

\newcommand{\penVF}{\pen_{\VF}} % penalite V-fold elle-meme
\newcommand{\penLOO}{\pen_{\mathrm{LOO}}} % penalite leave-one-out
\newcommand{\penLPO}{\pen_{\mathrm{LPO}}} % penalite leave-p-out
\newcommand{\critLPO}{\crit_{\mathrm{LPO}}} % Critere leave-p-out
\newcommand{\penHO}{\pen_{\HO}} % penalite hold-out
\newcommand{\penid}{\pen_{\mathrm{id}}} % penalite ideale
\newcommand{\critid}{\crit_{\mathrm{id}}} % penalite ideale
\newcommand{\ERMW}{\widehat{s}^{\,W}}

\newcommand{\lamm}{\lambda\in\Lambda_m} % ensemble des indices
\newcommand{\lL}{\lambda\in\Lambda}
\newcommand{\lpL}{\lambda^{\prime}\in\Lambda}

\newcommand{\lp}{\lambda^{\prime}}

\newcommand{\psil}{\psi_{\lambda}} % vecteurs de base de S_m
\newcommand{\psilp}{\psi_{\lambda^{\prime}}}

\newcommand{\Boule}{\mathbb{B}}

\newcommand{\ItHO}{T}%% training set du hold-out

\newcommand{\termeBvar}[1]{\beta\parenj{#1}}
\newcommand{\termeBvaracr}[1]{\mathbf{B}\parenj{#1}}

\newcommand{\cteMCCVa}{C_1^{\mathrm{MC}}}% MCCV
\newcommand{\cteMCCVb}{C_2^{\mathrm{MC}}}% MCCV
\newcommand{\cteMCCVi}{C_i^{\mathrm{MC}}}% MCCV
\newcommand{\ctepenVFa}{C_1^{\mathrm{penVF}}}% penVF
\newcommand{\ctepenVFb}{C_2^{\mathrm{penVF}}}% penVF
\newcommand{\ctepenVFi}{C_i^{\mathrm{penVF}}}% penVF
\newcommand{\cteMCVFa}{C_1^{\mathrm{MCVF}}}% MCVF
\newcommand{\cteMCVFb}{C_2^{\mathrm{MCVF}}}% MCVF
\newcommand{\cteVFCVa}{C_1^{\mathrm{VF}}}% VFCV
\newcommand{\cteVFCVb}{C_2^{\mathrm{VF}}}% VFCV
\newcommand{\cteGENi}{C_i}% cte generique indice i

\newcommand{\resteA}{\rho_1}
\newcommand{\resteB}{\rho_2}
\newcommand{\resteMCCV}{\rho_{3}}
\newcommand{\resteCVgal}{\rho_{4}}
\newcommand{\RbVFgal}{R^{\ref{le.conc.TermeBiaisCritVFGen}}}

\newcommand{\nred}{\widetilde{n}}
\newcommand{\Cor}{\ensuremath{C_{\mathrm{or}}}}

\newcommand{\CV}{\mathcal{C}}% lettre utilisee pour designer une methode de CV generique
\newcommand{\CVho}{\mathcal{C}^{\HO}}% methode hold-out generique
\newcommand{\critEpenid}{\CV_{\rm{id}}}
\newcommand{\Ratio}{\mathfrak{R}}
\newcommand{\CVFCV}{\ensuremath{\CV^\mathrm{VF}}}
\newcommand{\Newdiff}{\ensuremath{\Delta}}

\newcommand{\BS}{h}
\DeclareMathOperator{\SR}{SNR}
\newcommand{\petito}{\mathrm{o}}
\newcommand{\CVMCCV}{\CV^{\mathrm{MCCV}}}

\newcommand{\CteVFCV}{\ensuremath{C^{\mathrm{VF}}}}

\newcommand{\oA}{\overline{A}}
\newcommand{\oB}{\overline{B}}
\newcommand{\oC}{\overline{C}}
\newcommand{\oD}{\overline{D}}

\newcommand{\ctevarA}{a}
\newcommand{\ctevarB}{b}

\usepackage{stmaryrd}% pour la macro \inter
\newcommand{\inter}[1]{\left\llbracket #1 \right\rrbracket}

\renewcommand{\leq}{\leqslant}
\renewcommand{\le}{\leqslant}
\renewcommand{\geq}{\geqslant}
\renewcommand{\ge}{\geqslant}

\newcommand{\fractrain}{\tau}%%% Lettre utilisee pour noter la proportion de donnees de train dans l'ensemble de l'echantillon (= \fractrain_n)
\newcommand{\cteformCVb}{\sigma}%%% Lettre utilisee pour les poids apparaissant dans la deuxieme partie de la formule close des criteres validation croisee
\newcommand{\olcteformCVb}{\ensuremath{\overline{\cteformCVb}}}
\newcommand{\cteformHOb}{\cteformCVb^{\HO}}
\newcommand{\cteformVFCVb}{\cteformCVb^{\VF}}
\newcommand{\cteformLPOb}{\cteformCVb^{\mathrm{LPO}}}
\newcommand{\olomega}{\ensuremath{\overline{\omega}}}
\newcommand{\omegaHO}{\omega^{\HO}}
\newcommand{\omegaVFCV}{\omega^{\VF}}
\newcommand{\omegaLPO}{\omega^{\mathrm{LPO}}}
\newcommand{\tabespvert}{\noalign{\vspace*{0.075cm}}}

\newenvironment{proofof}[1]{\par\noindent{\bf Proof of #1\ }}{\hfill\BlackBox\\[2mm]}
\newcommand{\fauxparagraph}[1]{\paragraph{\textbf{\textit{#1.}}}}

\newlength{\largfiguniq}
\allowdisplaybreaks

\hypersetup{
    colorlinks,%
    citecolor=black,%
    filecolor=black,%
    linkcolor=black,%
    urlcolor=black
}

\begin{document}

\title{Choice of $V$ for $V$-Fold Cross-Validation in Least-Squares Density Estimation}

\author{Sylvain Arlot  \\ \texttt{sylvain.arlot@math.u-psud.fr} \\
       Laboratoire de Math\'ematiques d'Orsay \\
Univ. Paris-Sud, CNRS, Universit\'e Paris-Saclay \\
91405 Orsay, France
       \and\ 
       Matthieu Lerasle \\ \texttt{mlerasle@unice.fr} \\
       CNRS\\
       Univ. Nice Sophia Antipolis LJAD CNRS UMR 7351\\
		06100 Nice France}

\maketitle

\begin{abstract} 
This paper studies $V$-fold cross-validation for model selection in least-squares density estimation. 
The goal is to provide theoretical grounds for choosing $V$ in order to minimize the least-squares loss of the selected estimator. 
We first prove a non-asymptotic oracle inequality for $V$-fold cross-validation and its bias-corrected version ($V$-fold penalization). 
In particular, this result implies that $V$-fold penalization is asymptotically optimal in the nonparametric case. 
Then, we compute the variance of $V$-fold cross-validation and related criteria, as well as the variance of key quantities for model selection performance. 
We show that these variances depend on $V$ like $1+4/(V-1)$, at least in some particular cases, suggesting that the performance increases much from $V=2$ to $V=5$ or $10$, and then is almost constant. 
Overall, this can explain the common advice to take $V=5\,$---at least in our setting and when the computational power is limited---, as supported by some simulation experiments. 
An oracle inequality and exact formulas for the variance are also proved for Monte-Carlo cross-validation, also known as repeated cross-validation, where the parameter $V$ is replaced by the number $B$ of random splits of the data.
\end{abstract}

\begin{keywords}
$V$-fold cross-validation, Monte-Carlo cross-validation, leave-one-out, leave-$p$-out, resampling penalties, density estimation, model selection, penalization
\end{keywords}

\section{Introduction} \label{sec.intro}
Cross-validation methods are widely used in machine learning and statistics, 
for estimating the risk of a given statistical estimator \citep{Sto:1974,All:1974,Gei:1975} and for selecting among a family of estimators. 
For instance, cross-validation can be used for model selection, where a collection of linear spaces is given (the models) and the problem is to choose the best least-squares estimator over one of these models. 
Cross-validation is also often used for choosing hyperparameters of a given learning algorithm. 
We refer to \citet{Arl_Cel:2010:surveyCV} for more references about cross-validation for model selection. 

Model selection can target two different goals: 
(i) \emph{estimation}, that is, minimizing the risk of the final estimator, which is the goal of AIC and related methods, or 
(ii) \emph{identification}, that is, identifying the smallest true model in the family considered, assuming it exists and it is unique, which is the goal of BIC for instance;  
see the survey by \citet{Arl_Cel:2010:surveyCV} for more details about this distinction. 
These two goals cannot be attained simultaneously in general \citep{Yan:2005a}. 

We assume throughout the paper that the goal of model selection is estimation. 
We refer to \citet{Yan:2006,Yan:2007b} and \citet{Cel:2008} for some results and references on cross-validation methods with an identification goal. 

\medbreak

Then, a natural question arises: which cross-validation method should be used for minimizing the risk of the final estimator? 
For instance, a popular family of cross-validation methods is $V$-fold cross-validation \citep[often called $k$-fold cross-validation]{Gei:1975}, which depends on an integer parameter $V$, and enjoys a smaller computational cost than other classical cross-validation methods. 
The question becomes (1) which $V$ is optimal, and (2) can we do almost as well as the optimal $V$ with a small computational cost, that is, a small~$V$? 
Answering the second question is particularly useful for practical applications where the computational power is limited. 

Surprisingly, few theoretical results exist for answering these two questions, especially with a non-asymptotic point of view \citep{Arl_Cel:2010:surveyCV}. 
In short, it is proved in least-squares regression that at first order, $V$-fold cross-validation is suboptimal for model selection (with an estimation goal) if $V$ stays bounded, because $V$-fold cross-validation is biased \citep{Arl:2008a}. 
When correcting for the bias \citep{Bur:1989,Arl:2008a}, we recover asymptotic optimality whatever $V$, but without any theoretical result distinguishing among values of $V$ in second order terms in the risk bounds \citep{Arl:2008a}. 

Intuitively, if there is no bias, increasing $V$ should reduce the variance of the $V$-fold cross-validation estimator of the risk, hence reduce the risk of the final estimator, as supported by some simulation experiments \citep[for instance]{Arl:2008a}. 
But variance computations for unbiased $V$-fold methods have only been made in the asymptotic framework for a fixed estimator, and they focus on risk estimation instead of model selection \citep{Bur:1989}. 

\medbreak

This paper aims at providing theoretical grounds for the choice of $V$ by two means: a non-asymptotic oracle inequality valid for any $V$ (Section~\ref{sec.oracle}) and exact variance computations shedding light on the influence of $V$ on the variance (Section~\ref{sec.variance}). 
In particular, we would like to understand why the common advice in the literature is to take $V=5$ or $10$, based on simulation experiments \citep[for instance]{Bre_Spe:1992,Has_Tib_Fri:2001v2009}.

The results of the paper are proved in the least-squares density estimation framework, because we can then benefit from explicit closed-form formulas and simplifications for the $V$-fold criteria. 
In particular, we show that $V$-fold cross-validation and all leave-$p$-out methods are particular cases of $V$-fold penalties in least-squares density estimation (Lemma~\ref{le.penVF-VFCV}). 

The first main contribution of the paper (Theorem~\ref{thm.oracle-penVF.cas_reel}) is an oracle inequality with leading constant $1+\varepsilon_n$, 
with $\varepsilon_n \to 0$ as $n\to \infty$ for unbiased $V$-fold methods, 
which holds for any value of $V$. 
To the best of our knowledge, Theorem~\ref{thm.oracle-penVF.cas_reel} is the first non-asymptotic oracle inequality for $V$-fold methods enjoying such properties: the leading constant $1+\varepsilon_n$ is new in density estimation, and the fact that it holds whatever the value of $V$ had never been obtained in any framework. 
Theorem~\ref{thm.oracle-penVF.cas_reel} relies on a new concentration inequality for the $V$-fold penalty (Proposition~\ref{prop:concpenvf}).
Note that Theorem~\ref{thm.oracle-penVF.cas_reel} implicitly assumes that the oracle loss is of order $n^{-\alpha}$ for some $\alpha \in (0,1)$, that is, the setting is nonparametric; otherwise, Theorem~\ref{thm.oracle-penVF.cas_reel} may not imply the asymptotic optimality of $V$-fold penalization. 
Let us also emphasize that the leading constant is $1+\varepsilon_n$ whatever $V$ for unbiased $V$-fold methods, 
with $\varepsilon_n$ independent from $V$ in Theorem~\ref{thm.oracle-penVF.cas_reel}. 
So, second-order terms must be taken into account for understanding how the model selection performance depends on $V$. 
Section~\ref{sec.oracle.key-quant} proposes a heuristic for comparing these second order terms thanks to variance comparisons. 
This motivates our next result. 

The second main contribution of the paper (Theorem~\ref{theo.variance.penVF}) is the first non-asymptotic variance computation for $V$-fold criteria that allows to understand precisely how the \emph{model selection performance} of $V$-fold cross-validation or penalization depends on $V$. 
Previous results only focused on the variance of the $V$-fold criterion \citep{Bur:1989,Ben_Gra:2005,Cel:2008:phd,Cel:2008,Cel_Rob:2006}, which is not sufficient for our purpose, as explained in Section~\ref{sec.oracle.key-quant}. 
In our setting, we can explain, partly from theoretical results, partly from a heuristic argument, why taking, say, $V>10$ is not necessary for getting a performance close to the optimum,  
as supported by experiments on synthetic data in Section~\ref{sec.simus}. 

An oracle inequality and exact formulas for the variance are also proved for 
other cross-validation methods: 
Monte-Carlo cross-validation, also known as repeated cross-validation, where the parameter $V$ is replaced by the number $B$ of random splits of the data (Section~\ref{sec.discussion.MCCV}), 
and hold-out penalization (Section~\ref{sec.discussion.hold-out}). 

\fauxparagraph{Notation} 
%\hfill \\

For any integer $k\geq 1$, $\inter{k}$ denotes $\sets{1,\ldots,k}$. 
%\\

For any vector $\xi_{\inter{n}} \egaldef (\xi_1,\ldots,\xi_n)$ and any $B\subset \inter{n}$, $\xi_{B}$ denotes $(\xi_{i})_{i\in B}$, $|B|$ denotes the cardinality of $B$ and $B^{c}=\inter{n}\setminus B$. 
%\\

For any real numbers $t,u$, we define $t \vee u \egaldef \max\sset{t,u}$, $u_+ \egaldef u \vee 0$ and $u_-  \egaldef (-u) \vee 0$.
%\\

All asymptotic results and notation $\petito(\cdot)$ or $\mathcal{O}(\cdot)$ are for the regime when the number $n$ of observations tends to infinity. 
%%%All references to sections of the form B.$x$ refer to the Online Appendix. 

\section{Least-Squares Density Estimation and Definition of $V$-Fold Procedures} \label{sec.cadre}
This section introduces the framework of the paper, the main procedures studied, and some useful notation. 

\subsection{General Statistical Framework} \label{sec.cadre.general}
Let $\xi,\xi_1,...,\xi_n$ be independent random variables taking value in a Polish space $\X$, with common distribution $P$ and density $\bayes$ with respect to some known measure $\mu$. Suppose that $ \bayes \in L^{\infty}(\mu)$, 
which implies that $\bayes\in L^2(\mu)$. 
The goal is to estimate $\bayes$ from  $\xi_{\inter{n}}=(\xi_1,\ldots,\xi_n)$, that is, to build an estimator $\ERM = \ERM(\xi_{\inter{n}}) \in L^2(\mu)$ such that its loss $\perte{\ERM}$ is as small as possible, where for any $t \in L^2(\mu)$, $\norms{t}^2\egaldef \int_{\X} t^2 \dd\mu$.

Projection estimators are among the most classical estimators in this framework (see, for example, \citealp{Dev_Lor:1993}  and \citealp{Mas:2003:St-Flour}). 
Given a separable linear subspace $S_m$ of $L^2(\mu)$ (called a model), the projection estimator of $\bayes$ onto $S_m$ is defined by 
\begin{equation}
\label{def.ERM}
\ERM_m \egaldef \argmin_{t\in S_m} \setj{ \norms{t}^2 - 2 P_n (t) } \enspace ,
\end{equation}
where $P_n$ is the empirical measure; for any $t \in L^2(\mu)$, $P_n(t) = \int t dP_n = \frac1n\sum_{i=1}^n t\parenj{\xi_i} $. 
The quantity minimized in the definition of $\ERM_m$ is often called the empirical risk, and can be denoted by 
\[ P_n \gamma(t) =  \norms{t}^2 - 2 P_n (t) \qquad \text{where} \quad \forall x \in \X , \, \forall t \in L^2(\mu) , \quad \gamma(t;x)=\norms{t}^2 - 2 t(x) \enspace . \]
The function $\gamma$ is called the least-squares contrast. 
Note that $S_m \subset L^1(P)$ since $\bayes \in L^2(\mu)$.

\subsection{Model Selection} \label{sec.cadre.modsel}
When a finite collection of models $(S_m)_{\mM_n}$ is given, following \citet{Mas:2003:St-Flour}, we want to choose from data one among the corresponding projection estimators $(\ERM_m)_{\mM_n}$. 
The goal is to design a model selection procedure $\mh: \X^n \mapsto \M_n$ so that the final estimator $\widetilde{s}\egaldef\ERM_{\mh}$ has a quadratic loss as small as possible, that is, comparable to the oracle loss $\inf_{\mM_n} \pertes{\ERM_m}$.  %% \widetilde 
This goal is what is called the estimation goal in the Introduction.
More precisely, we aim at proving that an oracle inequality of the form 
\[ 
\perte{\ERM_{\mh}} \leq C_n \inf_{\mM_n} \setb{ \perte{\ERM_m} } + R_n 
\]
holds with a large probability. 
The procedure $\mh$ is called asymptotically optimal when $R_n$ is much smaller than the oracle loss and $C_n\to 1$, as $n \to +\infty$. In order to avoid trivial cases, we will always assume that $|\M_n|\ge 2$.

\medbreak

In this paper, we focus on model selection procedures of the form 
\[
\mh \egaldef \argmin_{\mM_n}\setb{\crit(m)} \enspace,
\]
where $\crit:\M_n \mapsto \R$ is some data-driven criterion. 
Since our goal is to satisfy an oracle inequality, an ideal criterion is 
\[
\critid(m) = \perte{\ERM_m}-\norms{\bayes}^2 = -2P(\ERM_m)+\norms{\ERM_m}^2 = P \gamma\parenj{\ERM_m} \enspace .
\]

Penalization is a popular way of designing a model selection criterion \citep{Bar_Bir_Mas:1999,Mas:2003:St-Flour} 
\[ \crit(m) = P_n \gamma\parens{\ERM_m} + \pen(m)  \]
for some penalty function $\pen: \M_n \flens \R$, possibly data-driven. 
From the ideal criterion $\critid$, we get the ideal penalty 
\begin{align} 
\label{def.penid}
\penid(m) &\egaldef \critid(m) - P_n\gamma\parenj{\ERM_m} 
= (P-P_n) \gamma\parenj{\ERM_m} 
= 2(P_n-P)(\ERM_m) 
\\
\notag 
&= 2(P_n-P)(\ERM_m-\bayes_m)+2(P_n-P)(\bayes_m) 
= 2 \normsq{ \ERM_m-\bayes_m } +2(P_n-P)(\bayes_m) 
 \enspace , 
\\
\text{where} 
\quad 
\bayes_m &\egaldef 
\argmin_{t \in S_m} \setb{ P\gamma(t)} 
= \argmin_{t\in S_m} \setb{\pertes{t}} 
\notag 
\end{align}
is the orthogonal projection of $\bayes$ onto $S_m$ in $L^2(\mu)$. 
Let us finally recall some useful and classical reformulations of the main term in the ideal penalty \eqref{def.penid}, that proves in particular the last equality in Eq.~\eqref{def.penid}: 
If $\Boule_m = \sets{t\in S_m \telque \norms{t}\leq 1}$ and $(\psil)_{\lamm}$ denotes an orthonormal basis of $S_m$ in $L^2(\mu)$, then 
\begin{equation}
\label{eq.p1.4formules}
\begin{split}
(P_n-P)(\ERM_m-\bayes_m) 
&= \sum_{\lamm} \crochsqb{(P_n-P)(\psil)}
\\
&= \normsq{ \ERM_m-\bayes_m }
= \sup_{t\in \Boule_{m}}\crochsqb{(P_{n}-P)(t)}
\enspace ,  
\end{split}
\end{equation}
where the last equality follows from Eq.~\eqref{eq.sup-CS} in Appendix~\ref{sec.proofs}. 

\subsection{$V$-Fold Cross-Validation} \label{sec.cadre.VFCV}
A standard approach for model selection is cross-validation. 
We refer the reader to \citet{Arl_Cel:2010:surveyCV} for references and a complete survey on cross-validation for model selection. 
This section only provides the minimal definitions and notation necessary for the remainder of the paper. 

For any subset $A\subset \inter{n}$, let
\begin{gather*}
P_{n}^{(A)} \egaldef \frac{1}{|A|}\sum_{i\in A}\delta_{\xi_{i}} 
\quad \text{and} \quad 
\ERM_{m}^{(A)} \egaldef \argmin_{t\in S_m} \setB{ \norms{t}^2 - 2 P^{(A)}_n (t) } 
\enspace .
\end{gather*}
The main idea of cross-validation is data splitting: some $T\subset \inter{n} $ is chosen, one first trains $\ERM_m(\cdot)$ with $\xi_{ T}$, then test the trained estimator on the remaining data $\xi_{ T^c}$. 
The hold-out criterion is the estimator of $\critid(m)$ obtained with this principle, that is, 
\begin{equation}\label{def:crit.HO}
\critHO (m, T) \egaldef P_{n}^{(T^{c})} \gamma\parenj{\ERM_m^{(T)}} = -2P_{n}^{(T^{c})} \parenj{ \ERM_{m}^{(T)}} +\normb{\ERM^{(T)}_{m}}^{2}\enspace , 
\end{equation}
and all cross-validation criteria are defined as averages of hold-out criteria with various subsets $T$. 

Let $V \in \{2, \ldots, n\}$ be a positive integer and let $\B=\B_{\inter{V}}=(\B_1,\ldots,\B_V)$ be some partition of $\inter{n}$. 
The $V$-fold cross-validation criterion is defined by 
\[ \critVF(m,\B) \egaldef \frac{1}{V} \sum_{K=1}^{V}\critHO(m, \B_K^c) \enspace . \]
Compared to the hold-out, one expects cross-validation to be less variable thanks to the averaging over $V$ splits of the sample into $\xi_{ \B_K}$ and $\xi_{ \B_K^c}$.

Since $\critVF(m,\B)$ is known to be a biased estimator of $\E\crochj{\critid(m)}$, \citet{Bur:1989} proposed the bias-corrected $V$-fold cross-validation criterion
\begin{align*} 
\critcorrVF(m, \B) 
&\egaldef \critVF(m,\B)+ P_n \gamma\parenj{ \ERM_m } - \frac{1}{V} \sum_{K=1}^{V} { P_n\gamma\parenj{ \ERM_{m}^{(\B_K^{c})} }  } 
%\\
%&= \critVF(m,\B)+\norms{\ERM_{m}}^{2}-2P_{n} \parenj{ \ERM_m } - \frac{1}{V} \sum_{K=1}^{V} \crochj{ \norms{\ERM_{m}^{(-\B_K)}}^{2} - 2 P_n \parenj{ \ERM_m^{(-\B_K)}} } 
\enspace .
\end{align*}
In the particular case where  $V=n$, this criterion is studied by \citet[Section 7.2.1, p.~204--205]{Mas:2003:St-Flour} under the name cross-validation estimator. 

\subsection{Resampling-Based and $V$-Fold Penalties} \label{sec.cadre.penVF}
Another approach for building general data-driven model selection criteria is penalization with a resampling-based estimator of the expectation of the ideal penalty, as proposed by \citet{Efr:1983} with the bootstrap and later generalized to all resampling schemes \citep{Arl:2009:RP}.   
Let $W \sim \mathcal{W}$ be some random vector of $\R^n$ independent from $\xi_{\inter{n}}$ with 
\[ \frac{1}{n} \sum_{i=1}^n W_i = 1 \enspace , \]
and denote by $P_n^W =  n^{-1}\sum_{i=1}^n W_i\delta_{\xi_i}$ the weighted empirical distribution of the sample. 
Then, the resampling-based penalty associated with $\mathcal{W}$ is defined as 
\begin{equation}\label{def.pen.Res}
\pen_{\mathcal{W}}(m) \egaldef C_{\mathcal{W}} \E_W \crochB{ \parenb{ P_n - P_n^W } \gamma\parenb{\ERMW_m} }\enspace, 
\end{equation}
where $\ERMW_m \in \argmin_{t \in S_m} \sets{ P_n^W \gamma\parens{t}}$, $\E_W\crochs{\cdot}$ denotes the expectation with respect to $W$ only (that is, conditionally to the sample $\xi_{\inter{n}}$), and $C_{\mathcal{W}}$ is some positive constant. 
Resampling-based penalties have been studied recently in the least-squares density estimation framework \citep{Le09}, assuming that $W$ is exchangeable, that is, its distribution is invariant by any permutation of its coordinates. 

Since computing exactly $\pen_{\mathcal{W}}(m)$ has a large computational cost in general for exchangeable $W$, some non-exchangeable resampling schemes were introduced by \citet{Arl:2008a}, inspired by $V$-fold cross-validation: 
given some partition $\B=\B_{\inter{V}}$ of $\inter{n}$, the weight vector $W$ is defined by $W_i=(1-\card(\B_J)/n)^{-1}\un_{i\notin \B_J}$ for some random variable $J$ with uniform distribution over $\inter{V}$. 
Then, $P_n^W = P_n^{(\B_J^{c})}$ so that the associated resampling penalty, called \emph{$V$-fold penalty}, is defined by 
\begin{align} \label{eq.def.penVF}
\penVF(m , \B , x) &\egaldef \frac{x}{V} \sum_{K=1}^V  \crochbb{ \parenB{P_n - P_n^{(\B_K^{c})} } \gamma\parenB{ \ERM_m^{(\B_K^{c})}} }\nonumber\\ 
&= \frac{2 x}{V} \sum_{K=1}^V  { \parenB{P_n^{(\B_K^{c})} - P_n } \parenB{ \ERM_m^{(\B_K^{c})}} }
\end{align}
where $x>0$ is left free for flexibility, which is quite useful according to Lemma~\ref{le.penVF-VFCV} below. 

\subsection{Links Between $V$-Fold Penalties, Resampling Penalties and (Corrected) $V$-Fold Cross-Validation} \label{sec.cadre.links}
In this paper, we focus our study on $V$-fold penalties because Lemma~\ref{le.penVF-VFCV} below shows that formula~\eqref{eq.def.penVF} covers all $V$-fold and resampling-based procedures mentioned in Sections \ref{sec.cadre.VFCV} and \ref{sec.cadre.penVF}.

First, when $V=n$, the only possible partition is $\B_{{\rm LOO}}=\sets{\sets{1}, \ldots, \sets{n}}$, and the $V$-fold penalty is called the leave-one-out penalty $\penLOO(m,x)\egaldef \penVF(m,\B_{{\rm LOO}},x)$. 
The associated weight vector $W$ is exchangeable, hence Eq.~\eqref{eq.def.penVF} leads to all exchangeable resampling penalties since they are all equal up to a deterministic multiplicative factor in the least-squares density estimation framework when $\sum_{i=1}^n W_i = n$, as proved by \citet{Le09}.

For $V$-fold methods, let us assume $\B$ is a regular partition of $\inter{n}$, that is, 
\begin{equation}
\label{hyp.part-reg.exact} 
\tag{\ensuremath{\mathbf{Reg}}}
V = |\B|\ge 2 \text{ divides } n\quad \text{and}\quad \forall K \in \inter{V} , \; \absj{\B_K} = \frac{n}{V}
\enspace . 
\end{equation}
Then, we get the following connection between $V$-fold penalization and cross-validation methods. 
\begin{lemma} \label{le.penVF-VFCV}
For least-squares density estimation with projection estimators, 
under assumption \eqref{hyp.part-reg.exact}, 
\begin{align} 
\label{eq.le.penVF-corrVFCV}  
\critcorrVF(m,\B) &=  P_n\gamma\parenj{\ERM_m} + \penVF\parenj{ m,\B,V-1 } 
\\
\label{eq.le.penVF-VFCV}
\critVF(m,\B) &=  P_n\gamma\parenj{\ERM_m} + \penVF\parenj{ m,\B,V-\frac{1}{2} } 
\\ 
\label{eq.le.penLPO-LPO}
\critLPO(m,p) &=  P_n\gamma\parenj{\ERM_m} +  \penLPO\parenj{ m,p,\frac{n}{p}-\frac{1}{2} }
\\ 
\label{eq.le.penLOO-LPO}
&= P_n\gamma\parenj{\ERM_m} +    \penLOO\parenj{ m,(n-1)\frac{n/p - 1/2}{ n/p - 1} }\\
\notag
&= P_n\gamma\parenj{\ERM_m} +    \penVF\parenj{ m,\B_{{\rm LOO}},(n-1)\frac{n/p - 1/2}{ n/p - 1} }
\end{align}
where for any $p \in \inter{n-1}$, the leave-$p$-out cross-validation criterion is defined by 
\begin{equation*}
\critLPO(m,p) \egaldef \frac{1}{|\mathcal{E}_p|} \sum_{A \in \mathcal{E}_p} P_n^{(A)} \gamma\parenB{\ERM_m^{(A^{c})}}
\qquad \text{with} \qquad 
\mathcal{E}_p \egaldef \setb{A \subset \inter{n} \telque |A|=p} 
\end{equation*}
and the leave-$p$-out penalty is defined by 
\[ \forall x >0 , \quad \penLPO(m,p,x) \egaldef \frac{x}{|\mathcal{E}_p|} \sum_{A \in \mathcal{E}_p} \parenB{ P_n - P_n^{(A^{c})} } \gamma\parenB{\ERM_m^{(A^{c})}} \enspace . \]
\end{lemma}
Lemma~\ref{le.penVF-VFCV} is proved in Section~\ref{sec.app.proof.lemmeLPO}. 
\begin{remark} \label{rk.le.penVF-VFCV.Alain}
Eq.~\eqref{eq.le.penVF-corrVFCV} was first proved by \citet{Arl:2008a} in a general framework that includes least-squares density estimation, assuming only \eqref{hyp.part-reg.exact}. 
Eq.~\eqref{eq.le.penLOO-LPO} follows from \citet[Lemma~A.11]{Le09} since $\penLPO$ belongs to the family of exchangeable resampling penalties, with weights $W_i\egaldef(1-p/n)^{-1}\un_{i\notin A}$ and $A$ is randomly chosen uniformly over $\mathcal{E}_p$; 
note that $\sum_{i=1}^n W_i =n$ for these weights. 
It can also be deduced from Proposition~3.1 by \citet{Cel:2008}, see Section~\ref{sec.app.proof.lemmeLPO}. 
\end{remark}
\begin{remark} \label{rk.le.penVF-VFCV.Pascal}
It is worth mentioning here the cross-validation estimators studied by \citet[Chapter~7]{Mas:2003:St-Flour}. 
First, the unbiased cross-validation criterion defined by \citet{Rud:1982} is exactly $\critcorrVF(m,\B_{{\rm LOO}})$ \citep[see also][Section 7.2.1]{Mas:2003:St-Flour}. 
Second, the penalized estimator of \citet[Theorem 7.6]{Mas:2003:St-Flour} is the estimator selected by the penalty 
 \[\penLOO \parenj{ m, \frac{(1+\epsilon)^6(n-1)^2}{2 \crochb{ n-(1+\epsilon)^6 } } } \]
for some $\epsilon>0$ such that $(1+\epsilon)^6<n$  \textup{(}see Section~\ref{sec.app.proof.lemmeLPO} for details\textup{)}. 
\end{remark}
So, in the least-squares density estimation framework and assuming only \eqref{hyp.part-reg.exact}, Lemma~\ref{le.penVF-VFCV} shows that it is sufficient to study $V$-fold penalization with a free multiplicative factor $x$ in front of the penalty for studying also $V$-fold cross-validation ($x=V-1/2$), corrected $V$-fold cross-validation ($x=V-1$), the leave-$p$-out ($V=n$ and $x=(n-1)(n/p-1/2)/(n/p-1)$) and all exchangeable resampling penalties. 
For any $C>0$ and $\B$ some partition of $\inter{n}$ into $V$ pieces, taking $x=C(V-1)$, the $V$-fold penalization criterion is denoted by 
\begin{equation}
\label{eq.crit-penVF-gal}
%\notag 
\CV_{(C,\B)}(m) \egaldef P_n\gamma\parenj{\ERM_m}+\penVF \parenb{ m,\B,C(V-1) }
\enspace . 
\end{equation} 
A key quantity in our results is the bias $\E\crochs{\CV_{(C,B)}(m)-\critid(m)}$. 
From Lemma~\ref{lem:exact.formula.penvf} in Section~\ref{sect:proofthmconcpenvf}, we have
\begin{align}\label{eq:ExpPenid}
 \E\crochb{\penVF(m,\B,V-1)}=\E\crochb{\penid(m)}
 = 2\E\crochj{\norms{\ERM_m-\bayes_m}^{2}}
 \enspace,
\end{align}
so that for any $C>0$, 
\begin{equation}
\label{eq.biais.critpenVF}
\E\crochj{ \CV_{(C,\B)}(m) - \critid(m) } 
= 2 (C-1) \E\crochj{\norms{\ERM_m-\bayes_m}^2} 
\enspace .  
\end{equation}
In Sections~\ref{sec.oracle}--\ref{sec.algo}, we focus our study on $V$-fold methods, that is, we study the performance of the $V$-fold penalized estimators $\ERM_{\mh}$, defined by 
\begin{equation}
\label{eq.mh-crit-penVF-gal}
\mh = \mh \parenb{ \CV_{(C,\B)} } =\argmin_{\mM_n}\setj{ \CV_{(C,\B)}(m) }
\enspace,
\end{equation} 
for all values of $V$ and $C>1/2$. 
Additional results on hold-out (penalization) are given in Section~\ref{sec.discussion.hold-out} to complete the picture. 

\section{Oracle Inequalities} \label{sec.oracle}
In this section, we state our first main result, that is, a non-asymptotic oracle inequality satisfied by $V$-fold procedures. 
This result holds for any divisor $V\ge 2$ of $n$, any constant $x=C (V-1)$ in front of the penalty with $C>1/2$, and provides an asymptotically optimal oracle inequality for the selected estimator when $C\to 1$ (assuming the setting is non parametric). 
In addition, as proved by Section~\ref{sec.cadre.links}, it implies oracle inequalities satisfied by leave-$p$-out procedures for all $p$. 

\subsection{Concentration of $V$-Fold Penalties}
Concentration is the key property to establish oracle inequalities. Let us start with some new concentration results for $V$-fold penalties.
\begin{proposition}\label{prop:concpenvf}
Let $\xi_{\inter{n}}$ be i.i.d. real-valued random variables with density $\bayes\in L^{\infty}(\mu)$, $\B$ some partition of $\inter{n}$ into $V$ pieces satisfying \eqref{hyp.part-reg.exact}, $S_m$ a separable linear space of measurable functions and $(\psil)_{\lL_m}$ an  orthonormal basis of $S_m$. 
Define 
\begin{gather*} 
\Boule_m=\setb{t\in S_m \telque \norms{t}\leq 1} 
\qquad 
\Psi_m=\sum_{\lL_m}\psil^{2}=\sup_{t\in\Boule_m}t^2
\qquad 
b_{m}\egaldef\norms{\sqrt{\Psi_m}}_\infty
%%\qquad 
\\
\Dcal_{m}\egaldef P(\Psi_m) -\norms{\bayes_{m}}^{2} 
=n\E\crochB{\norms{\bayes_m-\ERM_m}^2}
\enspace , 
\end{gather*}
where $\ERM_m$ is defined by Eq.~\eqref{def.ERM}, 
and for any $x,\epsilon>0$, 
\begin{equation*}
\resteA\parenj{m,\epsilon,\bayes,x,n}
\egaldef \frac{\norms{\bayes}_{\infty} x}{ \epsilon n}+\frac{ \parenb{ b_m^2+\norms{\bayes}^2 } x^2}{ \epsilon^3 n^2}
\enspace. 
\end{equation*}
Then, an absolute constant $\kappa$ exists such that for any $x \geq 0$, 
with probability at least $1-8\e^{-x}$, 
for any $\epsilon \in (0,1]$, 
the following two inequalities hold true 
\begin{align}
\label{eq:conc.penvf.exp}
\absj{ \penVF(m,\B,V-1) - \frac{2 \Dcal_{m}}n} &\leq \epsilon\frac{\Dcal_{m}}{n} + \kappa \resteA\parenj{m,\epsilon,\bayes,x,n} 
\\
\label{eq:concpenvf.penid}
\absj{ \penVF(m,\B,V-1) - 2 \norms{\bayes_m-\ERM_m}^2 } &\leq \epsilon\frac{\Dcal_{m}}{n} + \kappa \resteA\parenj{m,\epsilon,\bayes,x,n}
\enspace. 
\end{align}
\end{proposition}
Proposition~\ref{prop:concpenvf} is proved in Section~\ref{sect:proofthmconcpenvf}.
Eq.~\eqref{eq:conc.penvf.exp} gives the concentration of the $V$-fold penalty around its expectation $2 \Dcal_m / n = \E\crochs{\penid(m)}$, see Eq.~\eqref{eq:ExpPenid}. 
Eq.~\eqref{eq:concpenvf.penid} gives the concentration of the $V$-fold penalty around the ideal penalty, see Eq.~\eqref{def.penid}.
Optimizing over $\epsilon$, the first order of the deviations of $\penVF(m,\B,V-1)$ around $\penid(m)$ is driven by $\sqrt{\Dcal_m} / n$. 
The deviation term in Proposition~\ref{prop:concpenvf} does not depend on $V$ and cannot therefore help to discriminate between different values of this parameter. 

\subsection{Example: Histogram Models} \label{sec.ex.histos}
Histograms on $\R$ provide some classical examples of collections of models. 
Let $\X$ be a measurable subset of $\R$, $\mu$ denote the Lebesgue measure on $\X$ and $m$ be some countable partition of $\X$ such that $\mu(\lambda)>0$ for any $\lambda\in m$. 
The histogram space $S_{m}$ based on $m$ is the linear span of the functions $(\psil)_{\lambda\in \Lambda_m}$ where 
$\Lambda_m = m$ and for every $\lambda \in m$, 
$\psil = \mu(\lambda)^{-1/2} \un_{\lambda}$. 
More precisely, we illustrate our results with the following examples.
\begin{example}[Regular histograms on $\X=\R$]\label{ex:regular}
\[\M_{n}=\setb{m_{\BS},\BS\in\inter{n}}
\qquad \text{where} \qquad 
\forall \BS\in\inter{n},
\quad 
m_{\BS}=\setj{ \mathopen{} \left[\frac{\lambda}{\BS},\frac{\lambda+1}{\BS} \right) \mathclose{} ,\lambda\in\Z}
\enspace.\]
\end{example}
In Example~\ref{ex:regular}, defining $d_{m_{\BS}} = \BS$ for every $h \in \inter{n}$, for every $\mM_n$,  $\Dcal_m=d_m-\norms{\bayes_{m}}^{2}$ since $\Psi_m$ is constant and equal to $d_m$. 
Therefore, Proposition~\ref{prop:concpenvf} shows that $\penVF(m,\B,V-1)$ is asymptotically equivalent to $ \pendim(m) \egaldef 2 d_m/n$ when $d_m\to \infty$. 
Penalties of the form of $\pendim$ are classical and have been studied for instance by \citet{Bar_Bir_Mas:1999}. 
\begin{example}[$k$-rupture points on $\X=\crochs{0,1}$] \label{ex:krupt} 
\[\M_{n}=\setj{m_{\BS_{\inter{k+1}},x_{\inter{k}}} \telque x_{1}<\cdots<x_{k} \in\inter{n-1} \, \mathrm{and} \, \forall i \in \inter{k+1}, \BS_{i}\in \inter{x_{i}-x_{i-1}}}\enspace,\]
where $x_{0}=0$, $x_{k+1}=n$ and for any $x_{1},\ldots,x_{k}\in \inter{n-1}$ such that $x_{1}<\cdots<x_{k}$ and any $\BS_{\inter{k+1}}\in\N^{k+1}$, $m_{\BS_{\inter{k+1}},x_{\inter{k}}}$ is defined as the union
\[\bigcup_{i\in\inter{k}}\setj{\left[\frac{x_{i-1}}n+\frac{(x_{i}-x_{i-1})(\lambda-1)}{n \BS_{i}},\frac{x_{i-1}}n+\frac{(x_{i}-x_{i-1})\lambda}{n \BS_{i}}\right),\lambda\in\inter{\BS_{i}}}\enspace.\]
In other words, $m_{\BS_{\inter{k+1}},x_{\inter{k}}}$ splits $[0,1]$ into $k+1$ pieces \textup{(}at the $x_i$\textup{)}, and then splits the $i$-th piece into $h_i$ pieces of equal size. 
\end{example}
In Example~\ref{ex:krupt}, the function $\Psi_m$ is constant on each interval $[x_{i-1},x_{i})$, equal to $\BS_{i}$, therefore,  
\[\Dcal_m=\sum_{i=1}^{k+1} \BS_i\P\parenb{\xi\in[x_{i-1}, x_i)}-\norms{\bayes_{m}}^{2}\enspace.\] 

\subsection{Oracle Inequality for $V$-Fold Procedures} \label{sec.discussion.oracle}
In order to state the main result, we introduce the following hypotheses: 
\begin{itemize}
 \item \textit{A uniform bound on the $L^{\infty}$ norm of the $L^2$ ball of the models} 
 \begin{equation} 
\label{hyp.NormSupNorm2}
\tag{\ensuremath{\mathbf{H1}}} 
\forall \mM_n , \qquad 
 b_m \leq \sqrt{n}
\end{equation}
where we recall that $b_m \egaldef\sup_{t\in \Boule_m}\norms{t}_{\infty} $ and $\Boule_m\egaldef\setj{t\in S_m,\norms{t}\leq 1}$.

\item \textit{The family of the projections of $\bayes$ is uniformly bounded}. 
\begin{equation}
\label{hyp.UBbayes}\tag{\ensuremath{\mathbf{H2}}}
\exists a >0 , \quad 
 \forall \mM_n , \qquad \norms{\bayes_m}_{\infty}\leq a
 \enspace,
\end{equation}
\item \textit{The collection of models is nested}. 
\begin{equation}
\label{hyp.Nested}\tag{\ensuremath{\mathbf{H2'}}}
 \forall (m, m^{\prime}) \in \M_{n}^2 , \qquad S_{m}\cup S_{m^{\prime}}\in \left\{S_{m}, S_{m^{\prime}}\right\} 
\end{equation}
\end{itemize}
Hereafter, we define $A \egaldef a\vee \norms{\bayes}_{\infty}$ when \eqref{hyp.UBbayes} holds and $A \egaldef \norms{\bayes}_\infty$ when \eqref{hyp.Nested} holds. 
On histogram spaces, \eqref{hyp.NormSupNorm2}  holds if and only if $\inf_{m\in\M_n}\inf_{\lambda\in m}\mu(\lambda)\geq n^{-1}$, 
and \eqref{hyp.UBbayes} holds with $a=\norms{\bayes}_\infty$. 

\begin{theorem}
\label{thm.oracle-penVF.cas_reel}
Let $\xi_{\inter{n}}$ be i.i.d. real-valued random variables with common density $\bayes \in L^{\infty}(\mu)$, 
$\B$ some partition of $\inter{n}$ into $V$ pieces satisfying \eqref{hyp.part-reg.exact}  
and $(S_m)_{m\in\M_n}$ be a collection of separable linear spaces satisfying \eqref{hyp.NormSupNorm2}. 
Assume that either \eqref{hyp.UBbayes} or \eqref{hyp.Nested} holds true.
Let $C \in (1/2,2]$,  
\(
\delta\egaldef 2(C-1)\) and, for any $x,\epsilon>0$,
\begin{equation*}
\resteB\parenj{\epsilon,\bayes,x,n}
\egaldef 
\frac{A  x}{\epsilon n} + \parenj{ 1 + \frac{ \norms{\bayes}^2}{n} } \frac{x^2}{\epsilon^3 n}
\qquad \text{and} \qquad 
x_{n}=x+\log|\M_{n}|
 \enspace. 
\end{equation*}
For every $\mM_n$, let $\ERM_m$ be the estimator defined by Eq.~\eqref{def.ERM} 
and $\widetilde{s} = \ERM_{\mh}$ where 
\[ 
\mh = \mh \parenb{ \CV_{(C,\B)} }
\] 
is defined by Eq.~\eqref{eq.mh-crit-penVF-gal}.
Then, an absolute constant $\kappa$ exists such that, for any $x>0$, with probability at least $1-\e^{-x}$, for any $\epsilon \in (0,1]$, 
\begin{equation}
\label{eq.thm.oracle-penVF.cas_reel}
\frac{1-\delta_{-}-\epsilon}{1+\delta_{+}+\epsilon} \perte{\widetilde{s}}
\leq 
\inf_{m\in\M_{n}}\perte{\ERM_{m}}+\kappa\resteB(\epsilon,\bayes,x_n,n)
\enspace .
\end{equation}
\end{theorem}
Theorem~\ref{thm.oracle-penVF.cas_reel} is proved in Section~\ref{sec.app.proof.thm-oracle.histo}. 

Taking $\epsilon>0$ small enough in Eq.~\eqref{eq.thm.oracle-penVF.cas_reel}, Theorem~\ref{thm.oracle-penVF.cas_reel} proves that $V$-fold model selection procedures satisfy an oracle inequality with large probability. 
The remainder term can be bounded under the following classical assumption
\begin{equation} 
\label{hyp.PolColMod}\tag{\ensuremath{\mathbf{A3}}} 
\exists a'>0 , \, \forall n\in \N^{\star}, \qquad 
|\M_n|\leq n^{a'}\enspace .
\end{equation}
For instance, \eqref{hyp.PolColMod} holds in Example~\ref{ex:regular} with $a'=1$ and in Example~\ref{ex:krupt} with $a'=k$. 
Under \eqref{hyp.PolColMod}, the remainder term in Eq.~\eqref{eq.thm.oracle-penVF.cas_reel} is bounded by $L (\log n)^2 / (\epsilon^3 n)	$ for some $L>0$, which is much smaller than the oracle loss in the nonparametric case.

The leading constant in the oracle inequality~\eqref{eq.thm.oracle-penVF.cas_reel} is $(1+\delta_+)/(1-\delta_-) + \petito(1)$ by choosing $\epsilon = \petito(1)$, so the first-order behaviour of the upper bound on the loss is driven by $\delta$. 
An asymptotic optimality result can be derived from Eq.~\eqref{eq.thm.oracle-penVF.cas_reel} only if $\delta=\petito(1)$. 
The meaning of $\delta=2(C-1)$ is the amount of bias of the $V$-fold penalization criterion, as shown by Eq.~\eqref{eq.biais.critpenVF}. 
Given this interpretation of $\delta$, the model selection literature suggests that no asymptotic optimality result can be obtained in general when $\delta \neq \petito(1)$ in the nonparametric case  \citep[see, for instance,][]{Sha:1997}. 
Therefore, even if the leading constant $(1+\delta_+)/(1-\delta_-)$ is only an upper bound, we conjecture that it cannot be taken as small as $1+\petito(1)$ unless $\delta=\petito(1)$; 
such a result can be proved in our setting using similar arguments and assumptions as the ones of \citet{Arl:2008a} for instance. 

For bias-corrected $V$-fold cross-validation, that is, $C=1$ hence $\delta=0$, Theorem~\ref{thm.oracle-penVF.cas_reel} shows a first-order optimal non-asymptotic oracle inequality, since the leading constant $(1+\epsilon)/(1-\epsilon)$ can be taken equal to $1+\petito(1)$, and the remainder term is small enough in the nonparametric case, under assumption \eqref{hyp.PolColMod}, for instance. 
Such a result valid with no upper bound on $V$ had never been obtained before in any setting. 

$V$-fold cross-validation is also analyzed by Theorem~\ref{thm.oracle-penVF.cas_reel}, since by Lemma~\ref{le.penVF-VFCV} it corresponds to $C=1+1/(2(V-1))$, hence $\delta=1/(V-1)$. 
When $V$ is fixed, the oracle inequality is asymptotically sub-optimal, which is consistent with the result proved in regression by \citet{Arl:2008a}. 
On the contrary, if $\B=\B_n$ has $V_n$ blocs, with $V_n\to\infty$, Theorem~\ref{thm.oracle-penVF.cas_reel} implies under assumption \eqref{hyp.PolColMod} the asymptotic optimality of $V_n$-fold cross-validation in the nonparametric case. 

\medbreak

The bound obtained in Theorem~\ref{thm.oracle-penVF.cas_reel} can be integrated and we get 
\begin{equation*}
\frac{1-\delta_{-}-\epsilon}{1+\delta_{+}+\epsilon} \E\crochB{\pertes{\widetilde{s}}}
\leq  \E\crochB{\inf_{m\in\M_{n}}\perte{\ERM_{m}}}
+ \kappa^{\prime} \resteB \parenB{ \epsilon,\bayes,\log\parenb{ |\M_n| } }
\end{equation*}
for some absolute constant $\kappa^{\prime}>0$. 

Assuming $C>1/2$ is necessary, according to minimal penalty results proved by \citet{Le09}. 
Assuming $C \leq 2$ only simplifies the presentation; if $C>2$, the same proof shows that Theorem~\ref{thm.oracle-penVF.cas_reel} holds with $\kappa$ replaced by $C \kappa$.

An oracle inequality similar to Theorem~\ref{thm.oracle-penVF.cas_reel} holds in a more general setting, as proved in a previous version of this paper \citep[Theorem~1]{Arl_Ler:2012:penVF.v1}; we state a less general result here for simplifying the exposition, since it does not change the message of the paper. 
First, assumption \eqref{hyp.part-reg.exact} can be relaxed into assuming the partition $\B$ is close to regular, that is, 
\begin{equation}
\label{hyp.part-reg} \tag{\ensuremath{\mathbf{Reg'}}}
\B \text{ is a partition of } \inter{n} \text{ of size $V$ and } 
\sup_{k\in\inter{V}} \absj{ \card(\B_k) - \frac{n}{V} } \leq 1
\enspace , 
\end{equation}
which can hold for any $V \in \inter{n}$. 
Second, data $\xi_1, \ldots, \xi_n$ can belong to a general Polish space $\X$, at the price of some additional technical assumption.

\subsection{Comparison with Previous Works on $V$-Fold Procedures} \label{sec.oracle.previous-works}
Few non-asymptotic oracle inequalities have been proved for $V$-fold penalization or cross-validation procedures. 

Concerning cross-validation, previous oracle inequalities are listed in the survey by \citet{Arl_Cel:2010:surveyCV}. 
In the least-squares density estimation framework, oracle inequalities were proved by \citet{vdL_Dud_Kel:2004} in the $V$-fold case, but compared the risk of the selected estimator with the risk of an oracle trained with $n(V-1)/V$ data. 
In comparison, Theorem~\ref{thm.oracle-penVF.cas_reel} considers the strongest possible oracle, that is, trained with $n$ data. 
Optimal oracle inequalities were proved by \citet{Cel:2008} for leave-$p$-out estimators with $p \ll n$, a case also treated in Theorem~\ref{thm.oracle-penVF.cas_reel} by taking $V=n$ and $C=(n/p-1/2)/(n/p-1)$ as shown by Lemma~\ref{le.penVF-VFCV}. 
If $p \ll n$, $C \sim 1$, hence $\delta=\petito(1)$ and we recover the result of \citet{Cel:2008}.

Concerning $V$-fold penalization, previous results were either valid for $V=n$ only---by \citet[Theorem 7.6]{Mas:2003:St-Flour} and \citet{Le09} for least-squares density estimation, by \citet{Arl:2009:RP} for regressogram estimators---, or for $V$ bounded when $n$ tends to infinity---by \citet{Arl:2008a} for regressogram estimators. 
In comparison, Theorem~\ref{thm.oracle-penVF.cas_reel} provides a result valid for all $V$, except for the assumption that $V$ divides $n$, which can be removed \citep{Arl_Ler:2012:penVF.v1}. 
In particular, the loss bound by \citet{Arl:2008a} deteriorates when $V$ grows, while it remains stable in our result. 
Our result is therefore much closer to the typical behavior of the loss ratio $ \pertes{ \widetilde{s} } / \inf_{\mM_n} \pertes{ \ERM_m } $ of $V$-fold penalization, which usually decreases as a function of $V$ in simulation experiments, see Section~\ref{sec.simus} and the experiments by \citet{Arl:2008a}, for instance.

Theorem~\ref{thm.oracle-penVF.cas_reel} may not satisfactorily address the parametric setting, that is, when the collection 
$(S_m)_{\mM_n}$ contains some fixed true model. 
In such a case, the usual way to obtain asymptotic optimality is to use a model selection procedure targetting identification, that is, taking $C \to +\infty$ when $n \to +\infty$. 
For instance, \citet[Theorem~3.3]{Cel:2008} shows that $\log(n) \ll C \ll n$ is a sufficient condition for such a result. 

\section{How to Compare Theoretically the Performances of Model Selection Procedures for Estimation?}
\label{sec.oracle.key-quant}

The main goal of the paper is to compare the model selection performances of several ($V$-fold) cross-validation methods, when the goal is estimation, that is, minimizing the loss $\pertes{\ERM_{\mh}}$ of the final estimator. 
In this section, we discuss how such a comparison can be made on theoretical grounds, in a general setting. 

\medbreak

For some data-driven function $\CV: \M_n \to \R$, the goal is to understand how $\sperte{\ERM_{\mh(\CV)}}$ depends on $\CV$ when the selected model is 
\begin{equation} 
\label{eq.def.mhCVgal}
\mh(\CV) \in \argmin_{\mM_n} \setb{\CV(m)}
\enspace . 
\end{equation}
From now on, in this section, $\CV$ is assumed to be a cross-validation estimator of the risk, but the heuristic developed here applies to the general case.  

\fauxparagraph{Ideal comparison}
Ideally, for proving that $\CV_1$ is a better method than $\CV_2$ in some setting, we would like to prove that 
\begin{equation} 
\label{eq.ideal-comparison}
\perte{\ERM_{\mh(\CV_1)}} < (1-\varepsilon_n) \perte{\ERM_{\mh(\CV_2)}}
\end{equation}
with a large probability, for some $\varepsilon_n \geq 0$. 

\fauxparagraph{Previous works and their limits}
When the goal is estimation, the classical way to analyze the performance of a model selection procedure is to prove an oracle inequality, that is, to \emph{upper bound} (with a large probability or in expectation) 
\[ 
\perte{\ERM_{\mh(\CV)}} - \inf_{\mM_n} \setj{\perte{\ERM_m}} 
\qquad \text{or} \qquad 
\Ratio_n(\CV) \egaldef \frac{\perte{\ERM_{\mh(\CV)}}} { \inf_{\mM_n} \setb{\perte{\ERM_m}} } 
\enspace . 
\]
Alternatively, asymptotic results show that when $n$ tends to infinity, $\Ratio_n(\CV) \to 1$ (asymptotic optimality of $\CV$) or $\Ratio_n(\CV_1) \sim \Ratio_n(\CV_2)$ (asymptotic equivalence of $\CV_1$ and $\CV_2$); 
see \citet[Section~6]{Arl_Cel:2010:surveyCV} for a review of such results. 
Nevertheless, proving Eq.~\eqref{eq.ideal-comparison} requires a lower bound on $\Ratio_n(\CV)$ (asymptotic or not), which has been done only once for some cross-validation method, to the best of our knowledge. 
In some least-squares regression setting, $V$-fold cross-validation ($\CVFCV$) performs (asymptotically) worse than all asymptotically optimal model selection procedures since $\Ratio_n(\CVFCV) \geq \kappa(V) > 1$ with a large probability \citep{Arl:2008a}. 

The major limitation of all these previous results is that they can only compare $\CV_1$ to $\CV_2$ at first order, that is, according to $\lim_{n \to \infty} \Ratio_n(\CV_1)/\Ratio_n(\CV_2)$, which only depends on the bias of $\CV_i(m)$ ($i=1,2$) as an estimator of $\E\crochs{\sperte{\ERM_m}}$, hence, on the asymptotic ratio between the training set size and the sample size \citep[Section~6]{Arl_Cel:2010:surveyCV}. 
For instance, the leave-$p$-out and the hold-out with a training set of size $(n-p)$ cannot be distinguished at first order, while the leave-$p$-out performs much better in practice, certainly because  its ``variance'' is much smaller. 

\fauxparagraph{Beyond first-order}
So, we must go beyond the first-order of $\Ratio_n(\CV)$ and take into account the variance of $\CV(m)$. 
Nevertheless, proving a lower bound on $\Ratio_n(\CV)$ is already challenging at first order---probably the reason why only one has been proved up to now, in a specific setting only---so the challenge of computing a precise lower bound on the second order term of $\Ratio_n(\CV)$ seems too high for the present paper. 
We propose instead a heuristic showing that the variances of some quantities---depending on $(\CV_i)_{i=1,2}$ and on $\M_n$---can be used as a proxy to a proper comparison of $\Ratio_n(\CV_1)$ and $\Ratio_n(\CV_2)$ at second order. 
Since we focus on second-order terms, from now on, we assume that $\CV_1$ and $\CV_2$ have the same bias, that is, 
\begin{equation}
\label{hyp.heur-var.same-bias}
\tag{\ensuremath{\mathbf{SameBias}}}
\forall \mM_n ,  \quad \E\crochb{\CV_1(m)} = \E\crochb{\CV_2(m)} \enspace . 
\end{equation}
In least-squares density estimation, given Lemma~\ref{le.penVF-VFCV}, this means that for $i \in \sset{1,2}$, 
\[
\CV_i = \CV_{(C,\B_i)}
\]
as defined by Eq.~\eqref{eq.crit-penVF-gal}, with different partitions $\B_i$ satisfying \eqref{hyp.part-reg.exact} with different $V=V_i$, but the same constant $C>0$; $C=1$ corresponds to the unbiased case. 

\fauxparagraph{The variance of the cross-validation criteria is not the correct quantity to look at}
If we were only comparing cross-validation methods $\CV_1, \CV_2$ as estimators of $\E\crochb{\perte{\ERM_m}}$ for every single $\mM_n$, we could naturally compare them through their mean squared errors. 
Under assumption \eqref{hyp.heur-var.same-bias}, this would mean to compare their variances. 
This can be done from Eq.~\eqref{eq.pro.variance.critpenVF} below, but it is not sufficient to solve our problem, since it is known that the best cross-validation estimator of the risk does not necessarily yield the best model selection procedure \citep{Bre_Spe:1992}. 
More precisely, 
the selected model $\mh(\CV)$ defined by Eq.~\eqref{eq.def.mhCVgal} is unchanged when $\CV(m)$ is translated by any random quantity, but such a translation does change $\var\parens{\CV(m)}$ and can make it as large as desired. 
For model selection, what really matters is that 
\begin{equation}
\notag %%\label{eq.modsel-ranking}
\sign \parenb{ \CV(m_1) - \CV(m_2) } = \sign \parenj{ \perte{\ERM_{m_1}} - \perte{\ERM_{m_2}} }
\end{equation}
as often as possible for every $(m_1,m_2) \in \M_n^2$, 
and that most mistakes in the ranking of models occur when $\sperte{\ERM_{m_1}} - \sperte{\ERM_{m_2}}$ is small, so that 
$\sperte{\ERM_{\mh(\CV)}}$ cannot be much larger than $\inf_{\mM_n} \sset{\sperte{\ERM_m}}$. 

\fauxparagraph{Heuristic} 
The heuristic we propose goes as follows. 
For simplicity, we assume that $\mo = \argmin_{\mM_n} \E\crochs{\sperte{\ERM_m}}$ is uniquely defined. 
If the goal was identification, we could directly state that for any $\CV$, the smaller is $\Prob(m=\mh(\CV))$ for all $m \neq \mo$, the better should be the performance of $\mh(\CV)$. 
In this paper, our goal is estimation, but a similar claim can be conjectured by considering ``all $\mM_n$ sufficiently far from $\mo$ in terms of risk'', 
that is, all $\mM_n$ such that $\E\crochs{\sperte{\ERM_m}}$ is significantly worse than $\E\crochs{\sperte{\ERM_{\mo}}}$. 
Indeed, for any $m$ ``close to $\mo$'' in terms of risk, selecting $m$ instead of $\mo$ 
does not significantly change the performance of $\mh(\CV)$; 
on the contrary, for any $m$ ``far from $\mo$'' in terms of risk, selecting $m$ instead of $\mo$ does increase significantly the risk $\E\crochs{\sperte{\ERM_{\mh(\CV)}}}$. 

Then, our idea is to find a proxy for $\Prob(m=\mh(\CV))$, that is, a quantity that should behave similarly as a function of $\CV$ and its ``variance'' properties. 
For all $m,m^{\prime} \in \M_n$, 
let $\Delta_{\CV}(m,m^{\prime}) \egaldef \CV(m) - \CV(m^{\prime})$, 
$\Ncal$ some standard Gaussian random variable and, 
for all $t \in \R$, 
$\overline{\Phi}(t) = \Prob\parenj{\Ncal > t}$. 
Then, for every $\mM_n$ 
\begin{align}
\notag 
& \quad \, \, \Prob\parenb{\mh(\CV) = m} 
= \Prob\parenb{\forall m^{\prime} \neq m , \, \Delta_{\CV}(m,m^{\prime}) < 0 } 
\\
\label{eq.moresimus.heur.approx1}
&\asymp 
\min_{m^{\prime} \neq m} \Prob\parenb{\Delta_{\CV}(m,m^{\prime}) < 0 } 
\\
\label{eq.moresimus.heur.approx2}
&\approx \min_{m^{\prime} \neq m} \Prob\parenj{\E\crochj{ \Delta_{\CV}(m,m^{\prime}) } + \Ncal  \sqrt{\var\parenj{\Delta_{\CV}(m,m^{\prime})}} < 0 }
\\
\notag 
&=   \overline{\Phi} \parenb{ \SR_{\, \CV}(m) }
\quad \text{where} \quad 
\SR_{\, \CV}(m) \egaldef \max_{m^{\prime} \neq m} \frac{\E\crochb{ \Delta_{\CV}(m,m^{\prime}) }}{\sqrt{\var\parenb{\Delta_{\CV}(m,m^{\prime})}}} 
\enspace . 
\end{align}
So, if $\SR_{\, \CV_1}(m) > \SR_{\, \CV_2}(m)$  for all $m$ ``sufficiently far from $\mo$'', $\CV_1$ should be better than $\CV_2$. 
Assuming \eqref{hyp.heur-var.same-bias} holds true and that 
\begin{equation} \label{hyp.heur-var.same-min}
\tag{\ensuremath{\mathbf{SameMin}}}
\setj{\mo} =  \argmin_{\mM_n} \E\crochb{\CV_1(m)} = \argmin_{\mM_n} \E\crochb{\CV_2(m)}
\enspace , 
\end{equation} 
this leads to the following heuristic 
\begin{equation}
\label{eq.heuristic.variance-comp}
%\notag 
\forall m \neq m^{\prime}  , \qquad 
\var\parenj{\Newdiff_{\CV_1}(m,m^{\prime})} < \var\parenj{\Newdiff_{\CV_2}(m,m^{\prime})}
\Rightarrow 
\CV_1 \text{ better than } \CV_2 
\enspace . 
\end{equation}
Indeed, for every $m \neq m^{\prime}$, assumption~\eqref{hyp.heur-var.same-min} implies that $\SR_{\, \CV_i}(m)>0$ for $i=1,2$, hence we can restrict the max in the definition of $\SR_{\, \CV_i}$ to all $m'$ such that $\E\crochs{ \Delta_{\CV_i}(m,m^{\prime}) }$ 
is positive. 
By assumption~\eqref{hyp.heur-var.same-bias}, the numerator in the definition of $\SR_{\, \CV_i}$ does not depend on $i$, hence the ratio is maximal when the denominator is minimal, which leads to Eq.~\eqref{eq.heuristic.variance-comp}. 
Let us make some remarks. 
\begin{itemize}
\item The quantity $\Newdiff_{\CV}(m,m^{\prime})$ appears in relative bounds \citep[Section~1.4]{Cat:2007} which can be used as a tool for model selection \citep{Aud:2004:tech}.

\item Assumptions~\eqref{hyp.heur-var.same-bias} and~\eqref{hyp.heur-var.same-min} hold true in particular in the unbiased case, that is, when $\E\crochs{\CV_i(m)} = \E\crochs{\sperte{\ERM_m}}$ for all $\mM_n$ and $i\in \{1,2\}$. 

\item Assumption~\eqref{hyp.heur-var.same-min} is necessary: Figure~\ref{fig.surpen.LS-Dya2.n500} shows an example where a larger variance corresponds to better performance under assumption~\eqref{hyp.heur-var.same-bias} alone. 

\item As noticed above, the heuristic \eqref{eq.heuristic.variance-comp} should apply when the goal is estimation \emph{and} when the goal is identification, provided that \eqref{hyp.heur-var.same-bias} and~\eqref{hyp.heur-var.same-min} hold true. 
What should depend on the goal is the suitable amount of bias for $\CV_i(m)$ as an estimator of the risk $\E\crochs{\sperte{\ERM_m}}$. 

\item Approximation \eqref{eq.moresimus.heur.approx1} is the strongest one. 
Clearly, inequality $\leq$ holds true. 
The equality case occurs is for a very particular dependence setting, 
that is, when one among the events $ \parens{\sset{\Delta_{\CV}(m,m^{\prime}) < 0}}$, $m^{\prime} \in \M_n$, is included into all the others. 
In general, the left-hand side is significantly smaller than the right-hand side; we conjecture that they vary similarly as a function of $\CV$. 

\item The Gaussian approximation \eqref{eq.moresimus.heur.approx2} for $\Delta_{\CV}(m,m^{\prime})$ does not hold exactly, but it seems reasonable to make it, at first order at least. 

\item The validity of approximations \eqref{eq.moresimus.heur.approx1} and~\eqref{eq.moresimus.heur.approx2} is supported by the numerical experiments of Section~\ref{sec.simus}. 
\end{itemize}
In the heuristic \eqref{eq.heuristic.variance-comp}, all $(m,m^{\prime})$ do not matter equally for explaining a quantitative difference in the performances of $\CV$. 
First, we can fix $m^{\prime} = \mo$, since intuitively, the strongest candidate against any $m \neq \mo$ is $\mo$, which clearly holds in all our experiments, 
see Figures~\ref{fig.variance.SR-vs-Rmo.SiG5-Li01Regu.n100} and~\ref{fig.variance.SR-vs-Rmo.SiG5-Li01Regu.n500} 
in Section~\ref{sec.supmat.simus}. 
Second, as mentioned above, if $m$ and $\mo$ are very close, that is, $\sperte{\ERM_m}/\sperte{\ERM_{\mo}}$ is smaller than the minimal order of magnitude we can expect for $\Ratio_n(\CV)$ with a data-driven $\CV$, taking $m$ instead of $\mo$ does not decrease the performance significantly. 
Third, if $ \overline{\Phi} \parenj{ \SR_{\, \CV}(m) }$ is very small, increasing it even by an order of magnitude will not affect the performance of $\mh(\CV)$ significantly; hence, all $m$ such that, say,  $\SR_{\, \CV}(m) \gg (\log(n))^{\alpha}$ for all $\alpha>0$, can also be discarded.  
Overall, pairs $(m,m^{\prime})$ that  really matter in \eqref{eq.heuristic.variance-comp} are pairs $(m,\mo)$ that are at a ``moderate distance'', in terms of $\E\crochs{ \sperte{\ERM_m} - \sperte{\ERM_{\mo}} }$. 

\section{Dependence on $V$ of $V$-Fold Penalization and Cross-Validation} \label{sec.variance}
Let us now come back to the least-squares density estimation setting. 
Our goal is to compare the performance of cross-validation methods having the same bias, that is, according to Section~\ref{sec.cadre.links}, $\mh(\CV_{(C,\B)})$ with the same constant $C$ but different partitions $\B$, where $\mh(\CV_{(C,\B)})$ is defined by Eq.~\eqref{eq.mh-crit-penVF-gal}. 
\begin{theorem} \label{theo.variance.penVF}
Let $\xi_{\inter{n}}$ be i.i.d. random variables with common density $\bayes\in L^{\infty}(\mu)$, $\B$ some partition of $\inter{n}$ into $V$ pieces satisfying \eqref{hyp.part-reg.exact}, and $(\psil)_{\lL_{m_1}}$, $(\psil)_{\lL_{m_2}}$ two orthonormal families in $L^2(\mu)$. 
For any $m,m' \in \sets{m_1,m_2}$, we define $S_{m}$ 
the linear span of $\parens{\psil}_{\lL_{m}}$, 
$\bayes_m$ the orthogonal projection of $\bayes$ onto $S_m$ 
in $L^2(\mu)$, 

$\Boule_{m}=\{t \in S_{m} \telque \norms{t} \leq 1\}$, 
$\Psi_{m} \egaldef \sup_{t\in\Boule_{m}}t^2$, 
\begin{align*}
 \termeBvar{m,m'} 
 &\egaldef \sum_{\lL_m} \sum_{\lpL_{m'}} \parenbb{\E\crochB{\parenb{\psil(\xi_1)-P\psil}\parenb{\psilp(\xi_1)-P\psilp}}}^2
 \\
\text{and} \quad 
\termeBvaracr{m_1,m_2} 
&\egaldef \termeBvar{m_1 , m_1} + \termeBvar{m_2 , m_2} - 2\termeBvar{m_1 , m_2} 
\enspace .
\end{align*}
Then, for every $C>0$, 
\begin{align}
\label{eq.pro.variance.critpenVF}
 &\var \parenj{ \CV_{(C,\B)}(m_1)} 
 = \frac{2}{n^2}\parenj{1+\frac{4C^2}{V-1}-\frac{\parens{2C-1}^2}{n} } \termeBvar{m_1 , m_1} \\
\notag 
&\qquad +\frac{4}{n}\var \parenj{ \parenj{1+\frac{2C-1}{n}}\bayes_{m_1}(\xi_1)-\frac{2C-1}{2n}\Psi_{m_1}(\xi_1)}
\\ 
\text{and} \quad 
\label{eq.pro.variance.penVF-penid-incrReg}
&\var \parenj{  \CV_{(C,\B)}(m_1) - \CV_{(C,\B)}(m_2)}=\frac{2}{n^2}\parenj{1+\frac{4C^2}{V-1}-\frac{(2C-1)^2}{n}} \termeBvaracr{m_1 , {m_2} }
\\
\notag 
&\qquad +\frac{4}{n} \var \parenj{ \parenj{1+\frac{2C-1}{n}} (\bayes_{m_1}-\bayes_{m_2}) (\xi_1)-\frac{2C-1}{2n}(\Psi_{m_1}-\Psi_{m_2})(\xi_1)}
\end{align} 
where $\CV_{(C,\B)}$ is defined by Eq.~\eqref{eq.crit-penVF-gal}.
\end{theorem}
Theorem~\ref{theo.variance.penVF} is proved in Section~\ref{sec.proof.variance.main}.

\fauxparagraph{Unbiased case} When $C=1$, Theorem~\ref{theo.variance.penVF} shows that 
\[\var \parenj{ \CV_{(1,\B)}(m_1)-\CV_{(1,\B)}(m_2) } = \ctevarA + \parenj{1+\frac{4}{V-1}-\frac{1}{n}} \ctevarB \] 
for some $\ctevarA, \ctevarB \geq 0$ depending on $n,m_1,m_2$ but not on $V$. 
If we admit that the heuristic \eqref{eq.heuristic.variance-comp} holds true, this implies that the model selection performance of bias-corrected $V$-fold cross-validation improves when $V$ increases, but the improvement is at most in a second order term as soon as $V$ is large. 
In particular, even if $\ctevarA \ll \ctevarB$, the improvement from $V=2$ to $5$ or $10$ is much larger than from $V=10$ to $V=n$, which can justify the commonly used principle that taking $V=5$ or $V=10$ is large enough. 

Assuming in addition that $S_{m_1}$ and $S_{m_2}$ are regular histogram models (Example~\ref{ex:regular} in Section~\ref{sec.ex.histos}) with $d_{m_1}$ that divides $d_{m_2}$, then, by Lemma~\ref{lem:var2} in Section~\ref{sec.app.variance.eval-terms-histos},  
\begin{align*}
\ctevarA &=\frac{4}{n} \parensq{1+\frac{1}{n}} \var \parenb{ \bayes_{m_1}(\xi) - \bayes_{m_2}(\xi)} 
\approx \mathcal{O} \parenj{ \frac{1}{n} \norms{\bayes_{m_1}-\bayes_{m_2}}^2 }
\\
\text{and} \quad 
\ctevarB &= \frac{2}{n^2} \termeBvaracr{m_1,m_2} 
\asymp \norms{\bayes_{m_2}}^2 \frac{d_{m_2}}{n^{2}} 
\enspace . 
\end{align*}
When $d_{m_2} / n$ is at least as large as $\norms{\bayes_{m_1}-\bayes_{m_2}}^2$, we obtain that the first-order term in the variance is of the form $\alpha + \beta/(V-1)$ where $\alpha,\beta>0$ do not depend on $V$ and are of the same order of magnitude, as supported by the numerical experiments of Section~\ref{sec.simus}. 
Then, increasing $V$ from $2$ to $n$ does reduce significantly the variance, by a constant multiplicative factor. 

Let $\critEpenid(m) \egaldef P_n\gamma(\ERM_m)+\E\crochj{\penid(m)}$ be the criterion we could use if we knew the expectation of the ideal penalty. 
From Proposition~\ref{pro.variance} in Section~\ref{sec.app.proof.thm-variance}, 
\begin{multline*}
\var \parenb{ \critEpenid(m_1)-\critEpenid(m_2)}=\frac{2}{n^2}\parenj{1-\frac{1}{n}} \termeBvaracr{m_1 , m_2 }\\
+\frac{4}{n} \var \parenj{ \parenj{1-\frac{1}{n}}(\bayes_{m_1}-\bayes_{m_2})(\xi_1) + \frac{1}{2n}(\Psi_{m_1}-\Psi_{m_2}) (\xi_1)} 
\end{multline*}
which easily compares to formula~\eqref{eq.pro.variance.penVF-penid-incrReg} obtained for the $V$-fold criterion when $C=1$. 
Up to smaller order terms, the difference lies in the first term, where $(1+4/(V-1)-1/n)$ is replaced by $(1-1/n)$ when using the expectation of the ideal penalty instead of a $V$-fold penalty. 
In other words, the leave-one-out penalty---that is, taking $V=n$---behaves like the expectation of the ideal penalty. 

We can also compare Eq.~\eqref{eq.pro.variance.critpenVF} with the asymptotic results 
obtained by \citet{Bur:1989}, 
which imply that for any fixed model $m_1$ 
\[
\var \parenj{ \CV_{(1,\B)}(m_1) - P \gamma\parenj{\ERM_{m_1}} }
= \frac{\gamma_0}{n} + \parenj{ \frac{V}{V-1} \gamma_1 + \gamma_2} \frac{1}{n^2} 
+ \petito\parenj{\frac{1}{n^2}}
\]
with $\gamma_0,\gamma_1,\gamma_2$ that depend on $m_1$ and $\gamma_1>0$. 
Here, putting $C=1$ in Eq.~\eqref{eq.pro.variance.critpenVF} yields a result with a similar flavour, 
valid for all $n \geq 1$, 
even if Eq.~\eqref{eq.pro.variance.critpenVF} computes the variance of a slightly different quantity.

\fauxparagraph{Cross-validation criteria} 
$V$-fold cross-validation and the leave-$p$-out are also covered by Theorem~\ref{theo.variance.penVF}, according to Lemma~\ref{le.penVF-VFCV}, respectively with $C=1+1/(2(V-1))$ and with $V=n$ and $C=1 + 1/(2(n/p-1))$. 
As in the unbiased case, increasing $V$ decreases the variance, and 
if we admit that the heuristic \eqref{eq.heuristic.variance-comp} holds true,
$V$-fold cross-validation performs almost as well as the leave-$(n/V)$-out as soon as $V$ is larger than $5$ or $10$. 

Similarly, the variances of the $V$-fold cross-validation and leave-$p$-out criteria, for instance, can be derived from Eq.~\eqref{eq.pro.variance.critpenVF}. 
In the leave-$p$-out case, we recover formulas obtained by \citet{Cel:2008} and \citet{Cel_Rob:2006}, with a different grouping of the variance components; 
Eq.~\eqref{eq.pro.variance.critpenVF} clearly emphasizes the influence of the bias---through $(C-1)$---on the variance. 
For $V$-fold cross-validation, we believe that Eq.~\eqref{eq.pro.variance.critpenVF} shows in a simpler way how the variance depends on $V$, compared to the result of \citet{Cel_Rob:2006} which was focusing on the difference between $V$-fold cross-validation and the leave-$(n/V)$-out; here the difference can be written 
\[
\frac{8}{n^{2}} \parenj{\frac{1}{V-1}-\frac{1}{n-1}} \parensq{1+\frac{1}{2(V-1)}}  
\termeBvar{m_1 , m_1} 
\enspace.\]
A major novelty in Eq.~\eqref{eq.pro.variance.critpenVF} is also to cover a larger set of criteria, such as bias-corrected $V$-fold cross-validation. 
Note that $\var(\CV_{(C,\B)}(m_1))$ is generally much larger than 
\[ 
\var \parenb{ \CV_{(C,\B)}(m_1)-\CV_{(C,\B)}(m_2) }
\enspace , 
\] 
which illustrates again why computing the former quantity might not help for understanding the model selection properties of $\CV_{(C,\B)}$, as explained in Section~\ref{sec.oracle.key-quant}.  
For instance, comparing Eq.~\eqref{eq.pro.variance.critpenVF} and~\eqref{eq.pro.variance.penVF-penid-incrReg}, changing $\bayes_{m_1}$ into $\bayes_{m_1} - \bayes_{m_2}$ in the second term can reduce dramatically the variance when $\bayes_{m_1}$ and $\bayes_{m_2}$ are close, which happens for the pairs $(m_1,m_2)$ that matter for model selection according to Section~\ref{sec.oracle.key-quant}. 

The variance of other criteria and their increments are computed in subsequent sections of the paper: 
Monte-Carlo cross-validation (Theorem~\ref{thm.var-MCCV} in Section~\ref{sec.discussion.MCCV} 
and Theorem~\ref{thm.var-MCCV-gal} in Section~\ref{sec.supmat.MCCV.var}) 
and 
hold-out penalization (Proposition~\ref{pro.variance.penHO} in Section~\ref{sec.supmat.penHO.variance}).

\begin{remark} 
The term $\termeBvaracr{m_1,m_2}$ does not depend on the choice of particular bases of $S_{m_1}$ and $S_{m_2}$\textup{:} as proved by Proposition~\ref{pro.calcul.termes.variance} in Section~\ref{sec.app.proof.thm-variance} 
\begin{equation}
\notag %%\label{eq.calcvar.B.simplifie}
\termeBvaracr{m_1,m_2} = n \var\parenb{\parenj{\ERM_{m_1} - \ERM_{m_2}} (\xi)}- (n+1) \var\parenb{\parens{\bayes_{m_1} - \bayes_{m_2}} (\xi)} 
\enspace .
\end{equation}
\end{remark}

\section{Simulation Study} \label{sec.simus}
This section illustrates the main theoretical results of the paper with some experiments on synthetic data. 

\subsection{Setting} \label{sec.simus.setting}
In this section, we take $\X=[0,1]$ and $\mu$ is the Lebesgue measure on $\X$. 
Two examples are considered for the target density $\bayes$ and for the collection of models $(S_m)_{\mM_n}$. 

\emph{Two density functions} $\bayes$ are considered, see Figure~\ref{fig.cadres}: 
\begin{itemize}
\item Setting L: 
$\bayes(x) = \frac{10 x}{3} \un_{0 \leq x < 1/3} + \parens{ 1 + \frac{x}{3} } \un_{1 \geq x \geq 1/3}$. 
\item Setting S: 
$\bayes$ is the mixture of the piecewise linear density $x \mapsto (8 x - 4 ) \un_{1 \geq x \geq 1/2}  $ (with weight 0.8) and four truncated Gaussian densities with means $(k/10)_{k=1 , \ldots , 4}$ and standard deviation $1/60$ (each with weight 0.05). 
\end{itemize}
\begin{figure}
\begin{minipage}[b]{.48\linewidth}
\includegraphics[width=\textwidth]{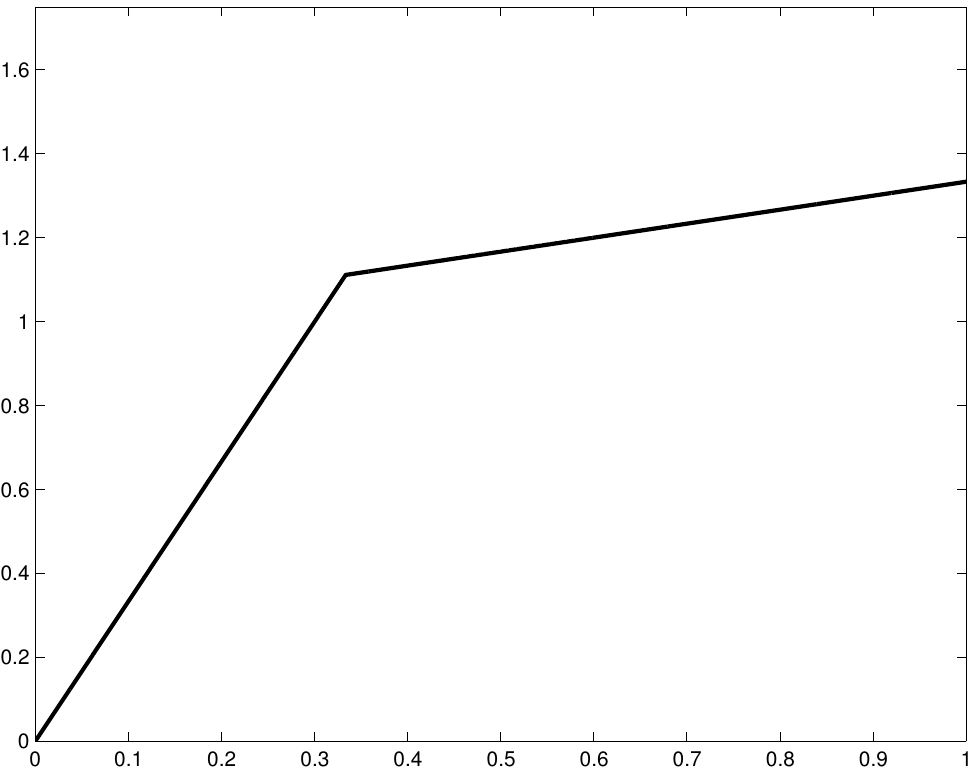}
\end{minipage}
\hspace{.025\linewidth}
\begin{minipage}[b]{.48\linewidth}
\includegraphics[width=\textwidth]{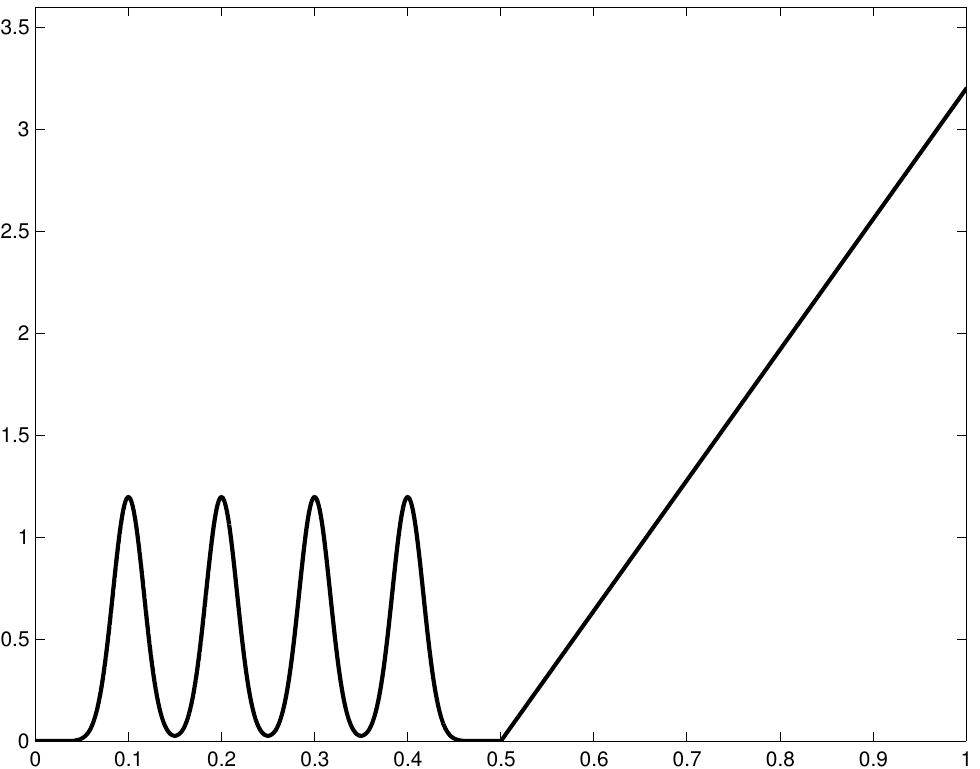}
\end{minipage}
\caption{The two densities considered. Left: setting L. Right: setting S.
\label{fig.cadres}
}
\end{figure}

\emph{Two collections of models} are considered, both leading to histogram estimators: for every $\mM_n$, $S_m$ is the set of piecewise constant functions on some partition $\Lambda_m$ of  $\X$. 
\begin{itemize}
\item ``Regu'' for regular histograms: $\M_n = \setj{1, \ldots, n}$ where for every $\mM_n$, $ \Lambda_m $ is the regular partition of $[0,1]$ into $m$ bins. 
\item ``Dya2'' for dyadic regular histograms with two bin sizes and a variable change-point: 
\[ 
\M_n = \bigcup_{ k \in \sets{1, \ldots, \nred } } \sets{ k } \times \setB{ 0 , \ldots, \floorb{\log_2(k)} }  \times \setB{ 0 , \ldots, \floorb{\log_2( \nred - k)} } 
\]
where $\nred=\floor{n/\log(n)}$ and for every $(k,i,j) \in \M_n$, $\Lambda_{(k,i,j)}$ is the union of the regular partition of $[0,k/\nred)$ into $2^i$ pieces and the regular partition of $[k/\nred,1]$ into $2^j$ pieces. 
\end{itemize}

The difference between ``Regu'' and ``Dya2'' can be visualized on Figure~\ref{fig.oracles.Regu-Dya2.SiG5}, on which the corresponding oracle estimators $\ERM_{\mh^{\star}}$ 
have been plotted for one sample in setting S, where \[ 
\mh^{\star} \in \argmin_{\mM_n} \, \pertes{\ERM_m} 
\enspace . \]
While ``Regu'' is one of the simplest and most classical collections for density estimation, the flexibility of ``Dya2'' allows to adapt to the variability of the smoothness of $\bayes$. Intuitively, in settings L and S, the optimal bin size is smaller on $[0,1/2]$ (where $\bayes$ is varying fastly) than on $[1/2,1]$ (where $\absj{\bayes^{\prime}}$ is much smaller). 
\begin{figure}
\begin{minipage}[b]{.48\linewidth}
\includegraphics[width=\textwidth]{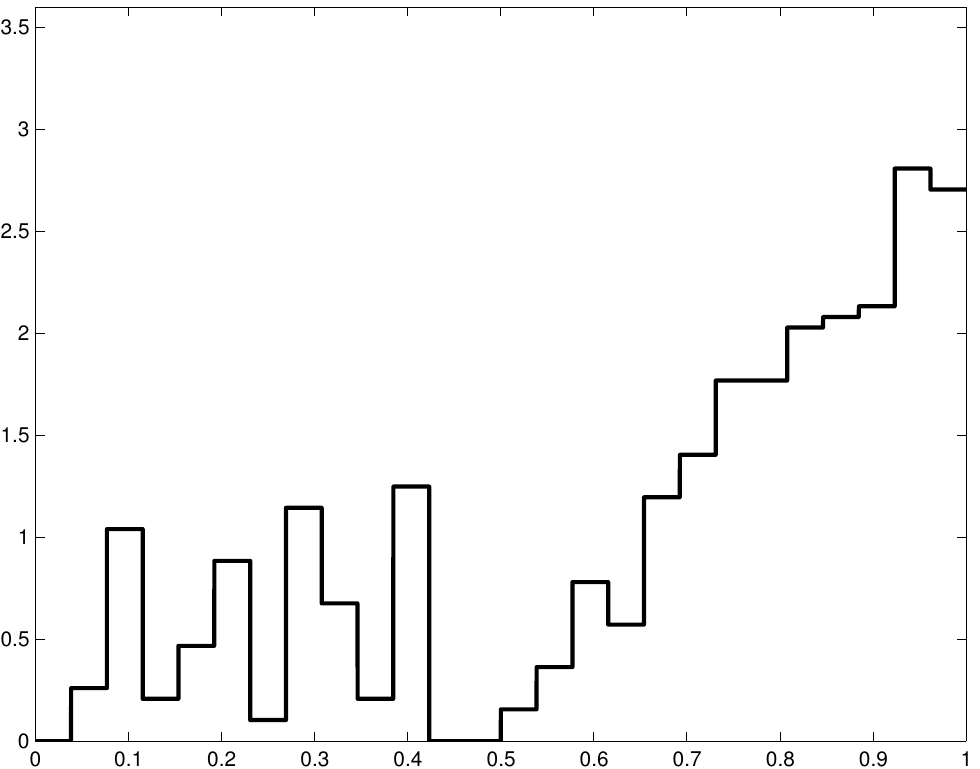}
\end{minipage}
\hspace{.025\linewidth}
\begin{minipage}[b]{.48\linewidth}
\includegraphics[width=\textwidth]{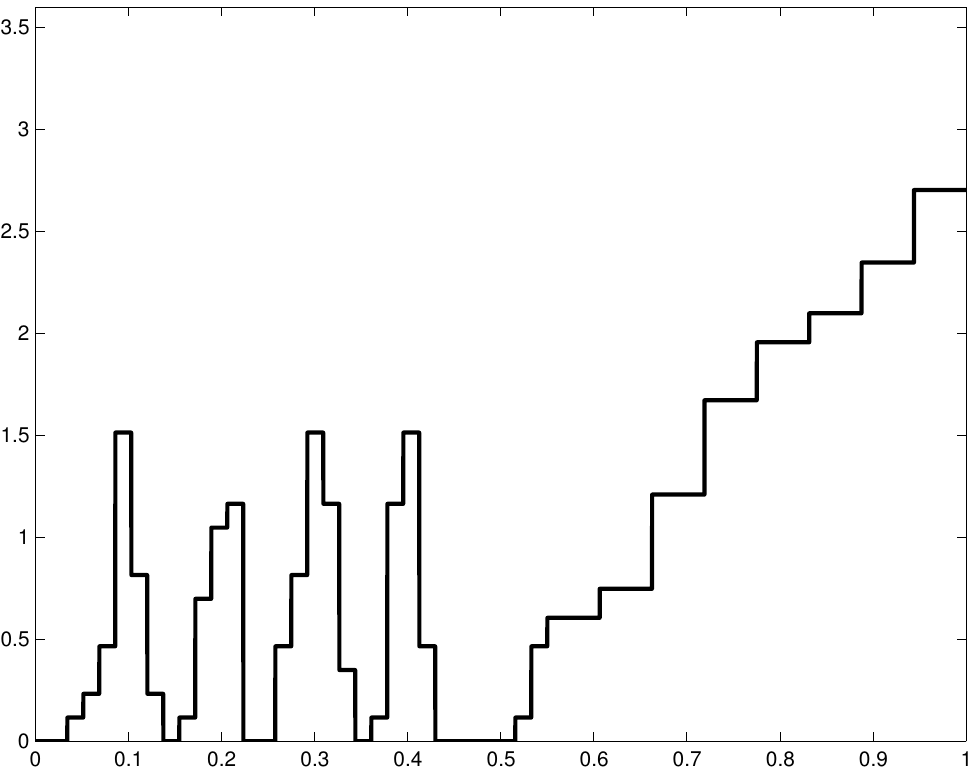}
\end{minipage}
\caption{Oracle estimator for one sample of size $n=500$, in setting S. Left: Regu. Right: Dya2. 
\label{fig.oracles.Regu-Dya2.SiG5}
}
\end{figure}

Another point of comparison of Regu and Dya2 is given by Table~\ref{tab.comp.Regu-Dya2}, that reports values of the quadratic risks obtained depending on the collection of models considered. 
Table~\ref{tab.comp.Regu-Dya2} shows that in settings L and S, the collection Dya2 helps reducing the quadratic risk by approximately 20\% (when comparing the best data-driven procedures of our experiment), and even more when comparing oracle estimators (30\% in setting S, 59\% in setting L). 
Therefore, in settings L and S, it is worth considering more complex collections of models (such as Dya2) than regular histograms. 
\newlength{\cellulewidth}
\setlength{\cellulewidth}{0.16\textwidth}
\begin{table} 
\begin{center}
\begin{tabular}{l@{\hspace{0.025\textwidth}} c @{\hspace{0.025\textwidth}} c @{\hspace{0.025\textwidth}} c @{\hspace{0.025\textwidth}} c}
\hline
\tabespvert
Setting      & \makebox[\cellulewidth][c]{Oracle(Regu)}   & \makebox[\cellulewidth][c]{Oracle(Dya2)}    & \makebox[\cellulewidth][c]{Best(Regu)}      & \makebox[\cellulewidth][c]{Best(Dya2)}    \\
\tabespvert
\hline
\tabespvert
L            & \makebox[\cellulewidth][c]{$13.4 \pm 0.1$} & \makebox[\cellulewidth][c]{$\phantom{4}5.46 \pm 0.02$} & \makebox[\cellulewidth][c]{$\phantom{1}25.8 \pm 0.1$}  & \makebox[\cellulewidth][c]{$19.4 \pm 0.1$} \\
S            & \makebox[\cellulewidth][c]{$62.4 \pm 0.1$} & \makebox[\cellulewidth][c]{$43.9\phantom{6} \pm 0.1\phantom{0}$}  & \makebox[\cellulewidth][c]{$100.9 \pm 0.2$} & \makebox[\cellulewidth][c]{$83.4 \pm 0.2$} \\
\tabespvert
\hline
\end{tabular} 
%%
%%\begin{tabular}
%%{p{0.16\textwidth}@{\hspace{0.025\textwidth}}p{0.18\textwidth}@{\hspace{0.025\textwidth}}p{0.16\textwidth}@{\hspace{0.025\textwidth}}p{0.17\textwidth}@{\hspace{0.025\textwidth}}p{0.18\textwidth}}
%%\hline
%%\tabespvert
%%Setting      & \makebox[\linewidth][c]{Oracle(Regu)}   & Oracle(Dya2)    & Best(Regu)      & Best(Dya2)    \\
%%\tabespvert
%%\hline
%%\tabespvert
%%L            & \makebox[\linewidth][c]{$13.4 \pm 0.1$} & $\phantom{4}5.46 \pm 0.02$ & $\phantom{1}25.8 \pm 0.1$  & $19.4 \pm 0.1$ \\
%%S            & \makebox[\linewidth][c]{$62.4 \pm 0.1$} & $43.9\phantom{6} \pm 0.1\phantom{0}$  & $100.9 \pm 0.2$ & $83.4 \pm 0.2$ \\
%%\tabespvert
%%\hline
%%\end{tabular} 
\end{center}
\caption{Comparison of Regu and Dya2: quadratic risks $\E\crochs{ \pertes{ \ERM_{\mh}} }$ of ``Oracle'' and ``Best'' estimators \textup{(}multiplied by $10^{3}$\textup{)} with the two collections of models. ``Best'' means that $\mh$ is the data-driven procedure minimizing $\E\crochs{ \pertes{ \ERM_{\mh} } }$ among all the data-driven procedures we considered in our experiments \textup{(}see Section~\ref{sec.simus.proc}\textup{)}. ``Oracle'' means that $\mh \in \argmin_{\mM_n} \perte{ \ERM_m}$ is the oracle model for each sample. \label{tab.comp.Regu-Dya2}}
\end{table}

Let us finally remark that Dya2 does not reduce the quadratic risk in all settings as significantly as in settings L and S. 
We performed similar experiments with a few other density functions, sometimes leading to less important differences between Regu and Dya2 in terms of risk (results not shown). 
The oracle model was always better with Dya2, but in two cases, the risk of the best data-driven procedure with Dya2 was larger than with Regu by 6 to 8\%.

\subsection{Procedures Compared} \label{sec.simus.proc}
In each setting, we consider the following model selection procedures:
\begin{itemize}
\item $\pendim$ \citep{Bar_Bir_Mas:1999}: penalization with $\pen(m) = 2 \card(\Lambda_m) / n$.
\item $V$-fold cross-validation with $V \in \sset{2,5,10,n}$, see Section~\ref{sec.cadre.VFCV}.
\item $V$-fold penalties (with leading constant $x=V-1$, that is, bias-corrected $V$-fold cross-validation), for $V \in \setj{2,5,10,n}$, see Section~\ref{sec.cadre.penVF}. 
\item for comparison, penalization with $\E\crochs{\penid(m)}$, that is, $\mh(\critEpenid)$. 
\end{itemize}

Since it is often suggested to multiply the usual penalties by some factor larger than one \citep{Arl:2008a}, we consider all penalties above multiplied by a factor $C \in [0,10]$. 
Complete results can be found in Section~\ref{sec.supmat.simus}.  

\subsection{Model Selection Performances} \label{sec.simus.modsel}
In each setting, all procedures are compared on $N=10\,000$ independent synthetic data sets of size $n=500$. 
For measuring their respective model selection performances, for each procedure $\mh(\CV)$ we estimate
\[ \Cor(\CV) \egaldef \E\crochb{ \Ratio_n(\CV) } = \E\crochj{ \frac{\perte{\ERM_{\mh(\CV)}}}{ \inf_{\mM_n} \perte{ \ERM_m } } } \]
by the corresponding average over the $N$ simulated data sets; 
$ \Cor(\CV)$ represents the constant that would appear in front of an oracle inequality. 
The uncertainty of estimation of $\Cor(\CV)$ is measured by the empirical standard deviation of $\Ratio_n(\CV)$ divided by $\sqrt{N}$. 
The results are reported in Table~\ref{tab.simus-results} for settings L and S, with the collection Dya2. 
\begin{table} 
\begin{center}
\begin{tabular}
{p{0.15\textwidth}@{\hspace{0.05\textwidth}}p{0.14\textwidth}@{\hspace{0.05\textwidth}}p{0.12\textwidth}}
\hline
\tabespvert
Procedure          & \makebox[\linewidth][c]{L--Dya2}             & \makebox[\linewidth][c]{S--Dya2}              \\
\tabespvert
\hline
\tabespvert
$\pendim$              & $\phantom{1}8.27 \pm 0.07$      & $3.21 \pm 0.01$       \\
\tabespvert
\hline
\tabespvert
pen2F         & $10.21 \pm 0.08$ & $2.39 \pm 0.01$ \\
pen5F         & $\phantom{1}7.47 \pm 0.06$                  & $2.16 \pm 0.01$ \\
pen10F        & $\phantom{1}6.89 \pm 0.06$                  & $2.11 \pm 0.01$ \\
penLOO        & $\phantom{1}6.35 \pm 0.05$                  & $2.06 \pm 0.01$ \\
\tabespvert
\hline
\tabespvert
2FCV               & $\phantom{1}6.41 \pm 0.05$       & $2.05 \pm 0.01$       \\
5FCV               & $\phantom{1}6.27 \pm 0.05$       & $2.05 \pm 0.01$       \\
10FCV              & $\phantom{1}6.24 \pm 0.05$       & $2.05 \pm 0.01$       \\
LOO                & $\phantom{1}6.34 \pm 0.05$       & $2.06 \pm 0.01$       \\
\tabespvert
\hline\hline
\tabespvert
$\E\crochs{\penid}$ & $\phantom{1}6.52 \pm 0.05$       & $2.07 \pm 0.01$       \\
\tabespvert
\hline
\end{tabular} 
\end{center}
\caption{Estimated model selection performances, see text. `LOO' is a shortcut for `leave-one-out', that is, $V$-fold with $V=n=500$. 
\label{tab.simus-results}}
\end{table}

Results for Regu are not reported here since dimensionality-based penalties are already known to work well with Regu \citep{Le09}, so $V$-fold methods cannot improve significantly their performance, with a larger computational cost. Complete results (including Regu, with $n=100$ and $n=500$) are given in Tables~\ref{tab.complet.Li01.SiG5} and~\ref{tab.complet.Li01.SiG5.n100} in Section~\ref{sec.supmat.simus}, showing that the performances of $\pendim$ and $V$-fold methods indeed are very close. 

\fauxparagraph{Performance as a function of $V$}
Let us first consider $V$-fold penalization. 
In both settings L and S, as suggested by our theoretical results, $\Cor$ decreases when $V$ increases. 
The improvement is large when $V$ goes from 2 to 5  ($27$\% for L, $10$\% for S) and small when $V$ goes from 5 to 10 and when $V$ goes from 10 to $n=500$ (each time, $8$\% for L, $2$\% for S). 
Since the main influence of $V$ is on the variance of the $V$-fold penalty, these experiments support our interpretation of Theorem~\ref{theo.variance.penVF} in Section~\ref{sec.variance}: increasing $V$ helps much more from 2 to 5 or 10 than from 10 to $n$.

The picture is less clear for $V$-fold cross-validation, for which almost no difference is observed among $V \in \setj{2,5,10,n}$---less than $2\%$---, and \Cor\ is minimized for $V\in \setj{5,10}$. 
Indeed, increasing $V$ simultaneously decreases the bias and the variance of the $V$-fold cross-validation criterion, leading to various possible behaviours of \Cor\ as a function of $V$, depending on the setting. 
The same phenomenon has been observed in regression \citep{Arl:2008a}. 

\fauxparagraph{Overpenalization}
In all settings considered in this paper, $V$-fold penalization 
performs much better when multiplying the penalty by $C>1$, as illustrated by Figure~\ref{fig.surpen.LS-Dya2.n500}. 
In particular, the best overpenalization factor for $\penLOO$ is $C^{\star}_n\approx 2.5$ for L-Dya2 and $C^{\star}_n\approx 1.4$ for S-Dya2, when $n=500$. 
Such a phenomenon, which can also be observed in regression \citep{Arl:2008a},  
is related to the fact that some nonparametric model selection problems are ``practically parametric'', using the terminology of \citet{Liu_Yan:2011}, that is, BIC beats AIC and the optimal $C$ is closer to $\log(n)/2$ than to~$1$. 
For instance, Figure~\ref{fig.surpen.LS-Dya2.n500} shows that L-Dya2 is practically parametric, while S-Dya2 is practically nonparametric since AIC beats BIC and the optimal $C$ is close to~1.

Given an overpenalization factor $C$ close to its optimal value $C^{\star}_n$, $V$-fold penalization performs significantly better than $V$-fold cross-validation in settings S-Dya2 and L-Dya2 (Figure~\ref{fig.surpen.LS-Dya2.n500}). 
Since $V$-fold cross-validation corresponds to taking 
\[ 
C=\CteVFCV(V) \egaldef 1 + \frac{1}{2(V-1)} 
\] 
according to Lemma~\ref{le.penVF-VFCV}, this mostly means that $\CteVFCV(V)$ is not close to $C^{\star}_n$ in these settings. 
In addition, when $C\approx C^{\star}_n$ is fixed, increasing $V$ always improves the performance of $V$-fold penalization, as predicted by the heuristic of Section~\ref{sec.oracle.key-quant} and the theoretical results of Section~\ref{sec.variance}. 
Let us emphasize that this fact does not depend on the parametricness of the setting: although the value of $C^{\star}_n$ is quite different for S-Dya2 and L-Dya2, in both cases, we observe qualitatively the same relationship between $V$ and the performance of the procedure. 

\begin{figure}
\begin{center}
\begin{minipage}[b]{.48\linewidth}
\includegraphics[width=\textwidth]{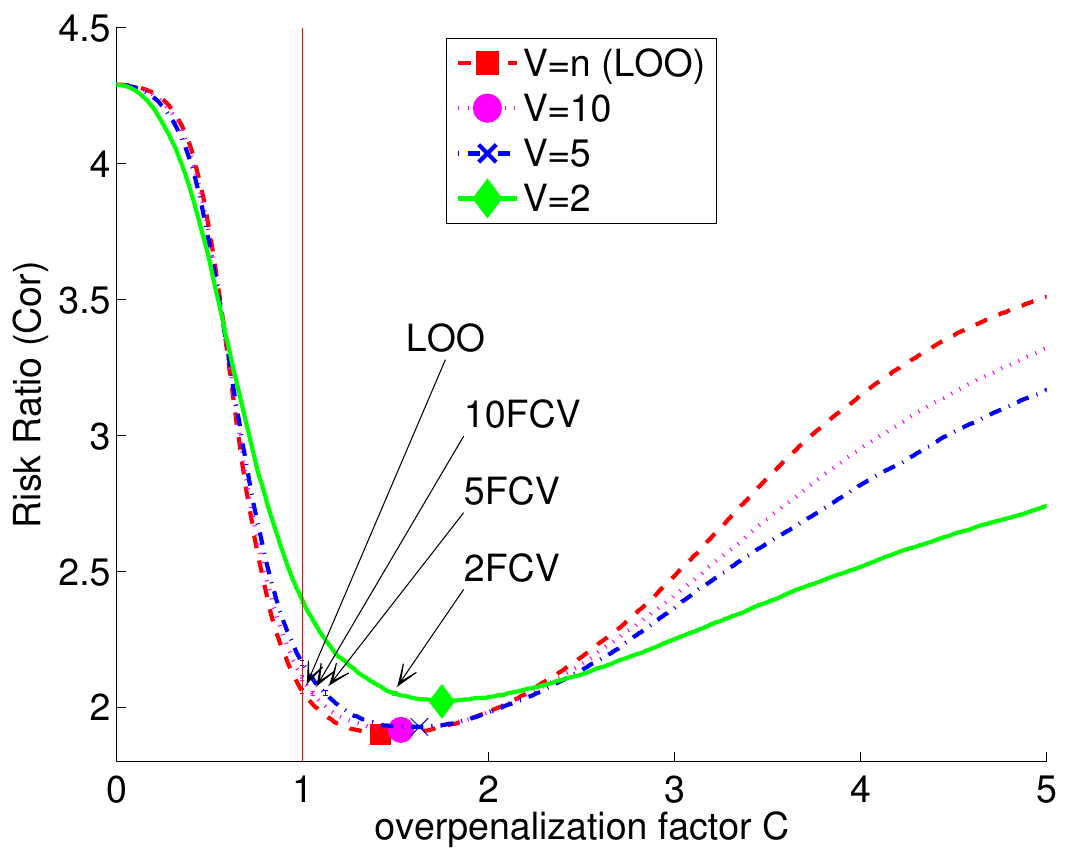}
\end{minipage}
\hspace{.025\linewidth}
\begin{minipage}[b]{.48\linewidth}
\includegraphics[width=\textwidth]{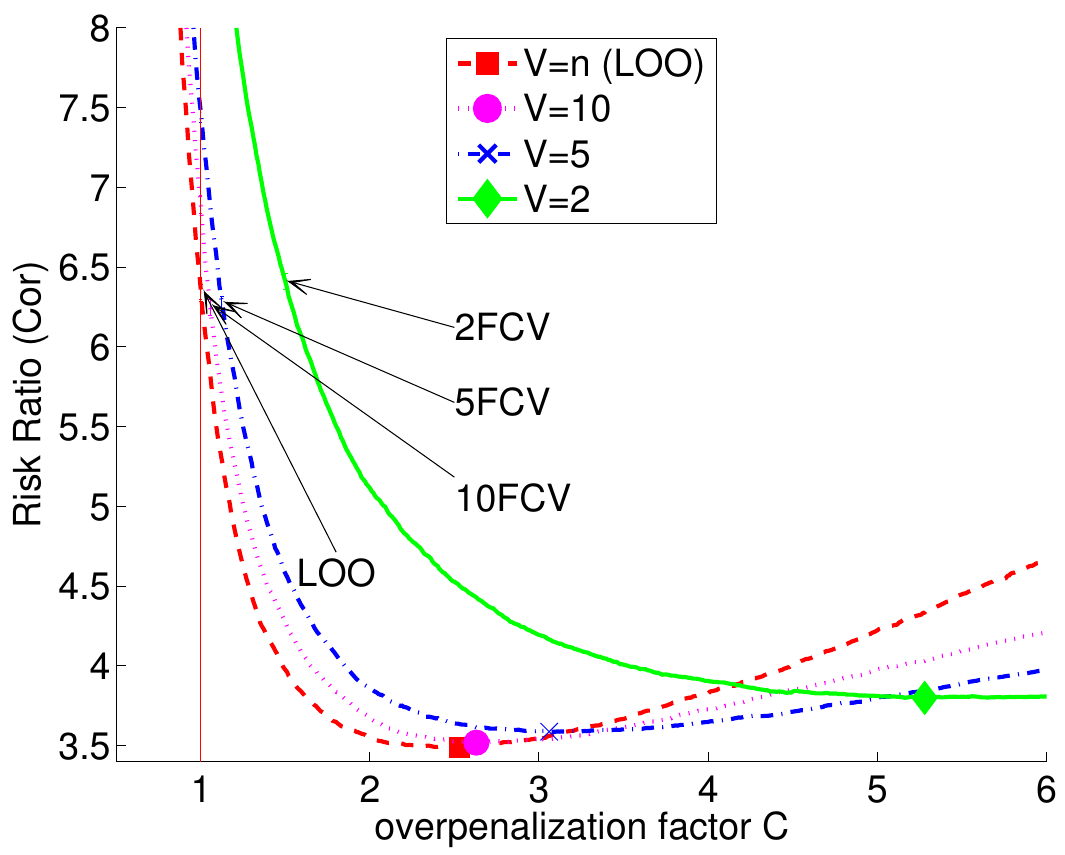}
\end{minipage}
\end{center}
\caption{%
Overpenalization in settings S-Dya2 (left) and L-Dya2 (right), with $n=500$ in both cases. 
Each plot represents the estimated model selection performance $\Cor(\CV_{(C,\B)})$ of several penalization procedures, as a function of the overpenalization constant $C$; unbiased risk estimation ($C=1$) is materialized by a vertical red line. 
For each value of $V$, the estimated optimal value of $C$ is shown on the graph; 
some arrows also show the performance of $V$-fold cross-validation, that is, $C=1+1/[2(V-1)]$. 
Error bars are not shown for clarity; Table~\ref{tab.simus-results} shows their order of magnitude, which is smaller than visible differences in the above graph. 
The performance obtained with the penalty $\E[\penid(m)]$ (not shown on the graph) is almost the same as with the leave-one-out penalty. 
\label{fig.surpen.LS-Dya2.n500}
}
\end{figure}

The results reported in Section~\ref{sec.supmat.simus} lead to similar conclusions in several other settings, as well as unshown results in a truly parametric setting, with a true model of dimension~2. 
Although a wider simulation study would be necessary to get general conclusions, this suggests at least that the heuristic of Section~\ref{sec.oracle.key-quant} and the theoretical results of Section~\ref{sec.variance} can be applied to both parametric and nonparametric settings.

Figure~\ref{fig.surpen.LS-Dya2.n500} also helps understanding how the performance of $V$-fold cross-validation depends on $V$ in Table~\ref{tab.simus-results}. 
Indeed, the performance of $V$-fold cross-validation for each value of $V$ can be visualized on Figure~\ref{fig.surpen.LS-Dya2.n500} by taking the point of abscissa $C=\CteVFCV(V)$ on the curve associated with $V$-fold penalization. 
Two phenomena are coupled when $C\le C^{\star}_n$, which always holds in our simulations for $V$-fold cross-validation since 
$\max_{V}\CteVFCV(V) = 1.5$ and the estimated value of $C^{\star}_n$ is always larger. 
(i) The performance improves when $V$ is fixed and $C$ gets closer to $C^{\star}_n$. 
(ii) The performance improves when $C$ is fixed and $V$ increases. 
Even if both phenomena (i) and (ii) seem quite universal, their coupling can result in various behaviours for $V$-fold cross-validation as a function of $V$, as shown by Table~\ref{tab.complet.Li01.SiG5} in Section~\ref{sec.supmat.simus} for instance.

\fauxparagraph{Other comments}
\begin{itemize}
\item $\pendim$ performs much worse than $V$-fold penalization (except $V=2$ in setting L) with the collection Dya2. On the contrary, $\pendim$ does well with Regu (see Table~\ref{tab.complet.Li01.SiG5} in Section~\ref{sec.supmat.simus}), but $V$-fold penalization then performs as well. 
\item In other settings considered in a preliminary phase of our experiments, for $V$-fold penalization, differences between $V=2$ and $V=5$ were sometimes smaller or not significant, but always with the same ordering (that is, the worse performance for $V=2$ when $C$ is fixed). 
In a few settings, for which the ``change-point'' in the smoothness of $\bayes$ was close to the median of $\bayes \mathrm{d}\mu$, we found $\pendim$ among the best procedures with collection Dya2; then, $V$-fold penalization and cross-validation always had a performance very close to $\pendim$. 
Both phenomena lead us to discard all settings for which there were no significant difference to comment. 
\end{itemize}

\subsection{Variance as a Function of $V$} \label{sec.simus.variance}
We now illustrate the results of Section~\ref{sec.variance} about the variance of $V$-fold penalization and the heuristic of Section~\ref{sec.oracle.key-quant} about its influence on model selection. 
We focus on the unbiased case, that is, criteria $\CV_{(1,\B)}$ with partitions $\B$ satisfying \eqref{hyp.part-reg.exact}. 
Since the distribution of $(\CV_{(1,\B)}(m))_{\mM_n}$ then only depends on $V = |\B|$, we write $\CV_V$ instead of $\CV_{(1,\B)}$ by abuse of notation. 
All results presented in this subsection have been obtained from $N=10\,000$ independent samples in setting S with a sample size $n=100$ and the collection Regu---for which models are naturally indexed by their dimension.

First, Figure~\ref{fig.variance.std.SiG5Regu.n100} shows the variance of $\Delta_{\CV_V}(m,\mo)=\CV_V(m)-\CV_V(\mo)$ as a function of the dimension $m$ of $S_m$, illustrating the conclusions of Theorem~\ref{theo.variance.penVF}: the variance decreases when $V$ increases. 
More precisely, the variance decrease is significant between $V=2$ and $V=5$, an order of magnitude smaller between $V=5$ and $V=10$ and between $V=10$ and $V=n$, while the leave-one-out $\CV_n$ is hard to distinguish from the ideal penalized criterion $\critEpenid$.  
On Figure~\ref{fig.variance.std.SiG5Regu.n100}, we can remark that for $m > \mo$  
\begin{equation*}
\var\parenj{ \Delta_{\CV_V}(m,m^{\star}) } \approx \frac{1}{n^2} \crochj{ K_1 \parenj{ 1 + \frac{K_2}{V-1}}  + K_3 \parenj{ 1 + \frac{K_4}{V-1}} (m - \mo) }
\end{equation*}
with $K_1 \approx 29$, $K_2 \approx 0.81$, $K_3 \approx 3.7$ and $K_4 \approx 3.8$. 
The shape of the dependence on $V$ already appears in Theorem~\ref{theo.variance.penVF}, the above formula clarifies the relative importance of the terms called $\ctevarA$ and $\ctevarB$ in Section~\ref{sec.variance}, and their dependence on the dimension $m$ of $S_m$. 
Remark that the same behaviour holds when $n=500$ with very close values for $K_3$ and $K_4$ (see Figure~\ref{fig.variance.var.SiG5Regu.n500} in Section~\ref{sec.supmat.simus}), 
as well as in setting L with $n=100$ or $n=500$ with $K_3 \approx 2.1$ and $K_4 \approx 4.2$ (see Figures~\ref{fig.variance.var.Li01Regu.n100} and~\ref{fig.variance.var.Li01Regu.n500} in Section~\ref{sec.supmat.simus}). 
The fact that $K_4$ is close to $4$ in both settings supports that the term $1+4/(V-1)$ appearing  Theorem~\ref{theo.variance.penVF} indeed drives how $\var\parens{ \Delta_{\CV_V}(m,m^{\star}) }$ depends on $V$. 
\begin{figure}
\begin{center}
\begin{minipage}[b]{.65\linewidth}
\includegraphics[width=\textwidth]{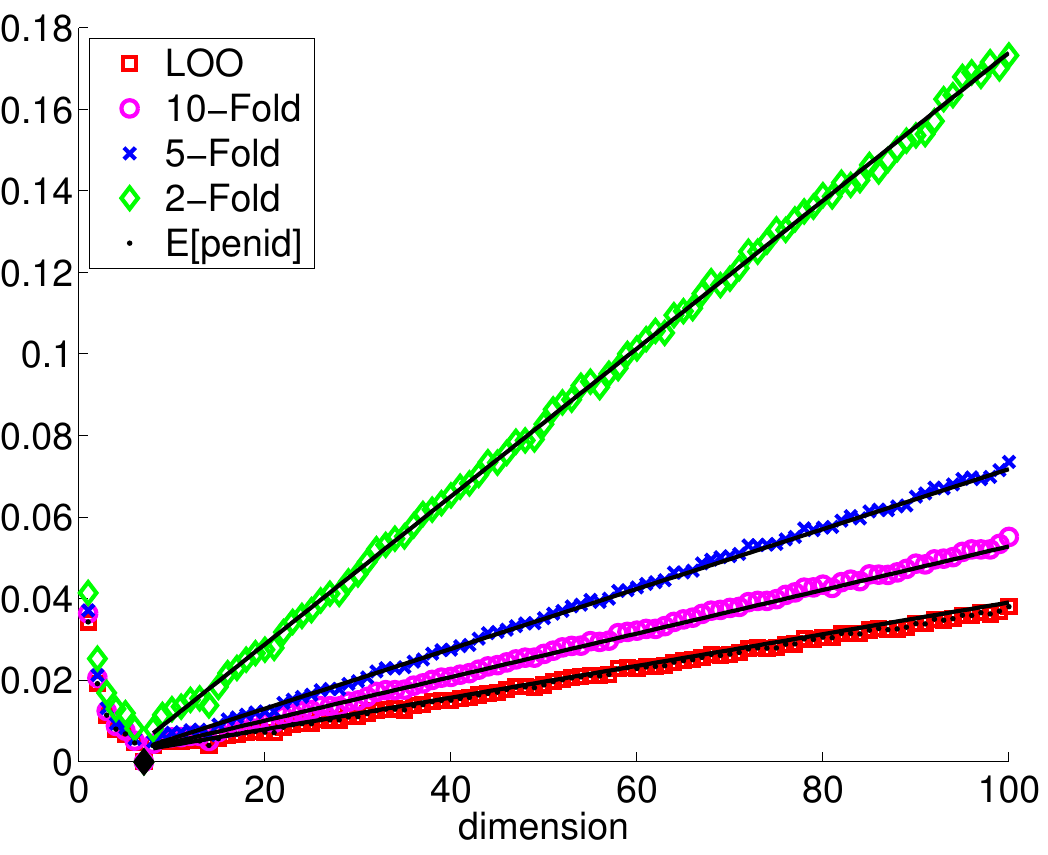}
\end{minipage}
\end{center}
\caption{%
Illustration of the variance heuristic: 
$\var(\Delta_{\CV}(m,\mo))$ as a function of $m$ for five different $\CV$. 
Setting S-Regu, $n=100$. 
The black diamond shows $\mo=7$. 
The black lines show the linear approximation $n^{-2} [ 29 (1 + \frac{0.81}{V-1} ) + 3.7 ( 1 + \frac{3.8}{V-1} ) (m-\mo)]$ for $m > \mo$. 
\label{fig.variance.std.SiG5Regu.n100}
}
\end{figure}

Figures~\ref{fig.variance.Pmh.SiG5Regu.n100} and~\ref{fig.variance.PhiSR.SiG5Regu.n100} respectively show $\Prob(\mh(\CV)=m)$ and its proxy $\overline{\Phi}(\SR_{\, \CV}(m))$ as a function of $m$ for $\CV=\CV_V$ with $V \in \setj{2, 5, 10, n}$ and for $\CV=\critEpenid$. 
First, we remark that both quantities behave similarly as a function of $m$ and $\CV$---see also Figure~\ref{fig.variance.SR-vs-Pmh.SiG5Regu.n100} in Section~\ref{sec.supmat.simus}---supporting empirically the heuristic of Section~\ref{sec.oracle.key-quant}. 
The decrease of the variance observed on Figure~\ref{fig.variance.std.SiG5Regu.n100} when $V$ increases here translates into a better concentration of the distribution of $\mh(\CV_V)$ around $\mo$, which can explain the performance improvement observed in Section~\ref{sec.simus.modsel}. 
Figures~\ref{fig.variance.Pmh.SiG5Regu.n100}--\ref{fig.variance.PhiSR.SiG5Regu.n100} actually show how the decrease of the variance quantitatively influences the distribution of $\mh(\CV_V)$: 
$\mh(\CV_5)$ is significantly more concentrated than $\mh(\CV_2)$, while the difference between $V=10$ and $V=5$ is much smaller and comparable to the difference between $V=n$ and $V=10$; 
$\CV_n$ is hard to distinguish from $\critEpenid$. 
Similar experiments with $n=500$ and in setting L are reported in Section~\ref{sec.supmat.simus}, leading to similar conclusions. 
\begin{figure}
\begin{center}
\begin{minipage}[b]{.65\linewidth}
\includegraphics[width=\textwidth]{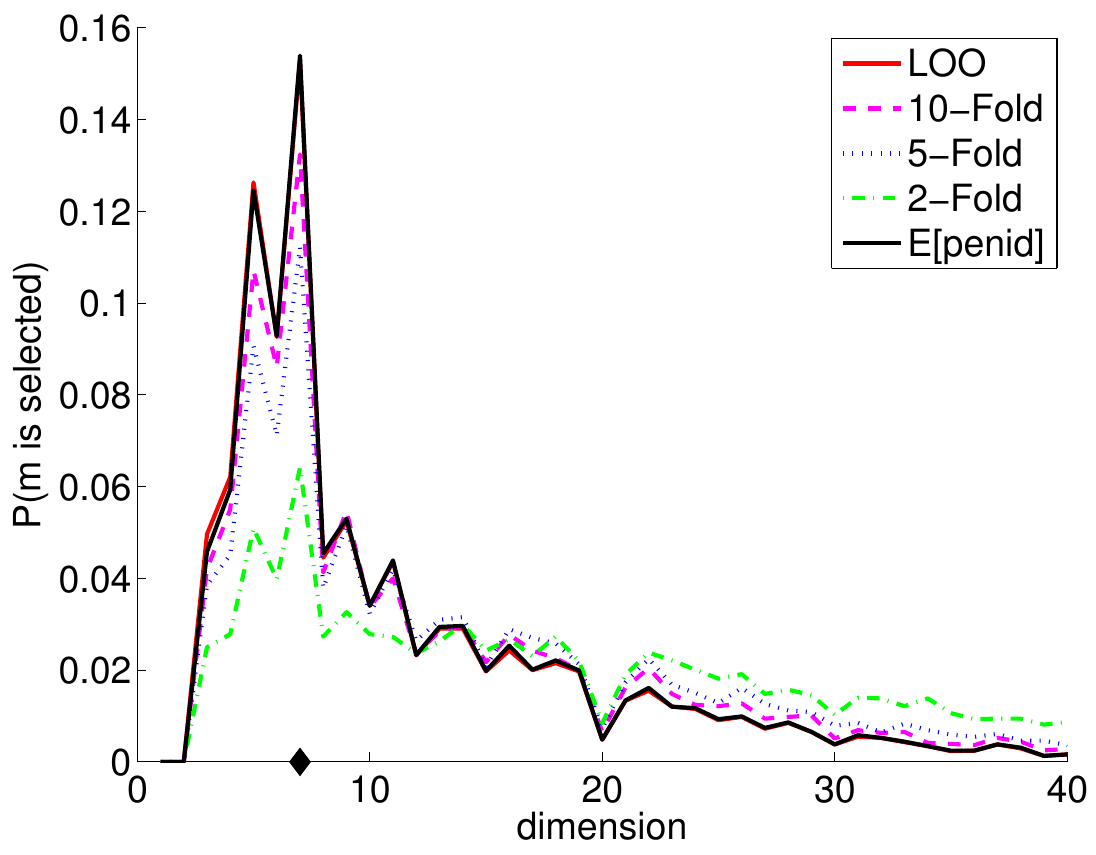}
\end{minipage}
\end{center}
\caption{%
$\Prob(\mh(\CV)=m)$ as a function of $m$ for five different $\CV$. 
Setting S-Regu, $n=100$.  
The black diamond shows $\mo=7$. 
\label{fig.variance.Pmh.SiG5Regu.n100}
}
\end{figure}
\begin{figure}
\begin{center}
\begin{minipage}[b]{.65\linewidth}
\includegraphics[width=\textwidth]{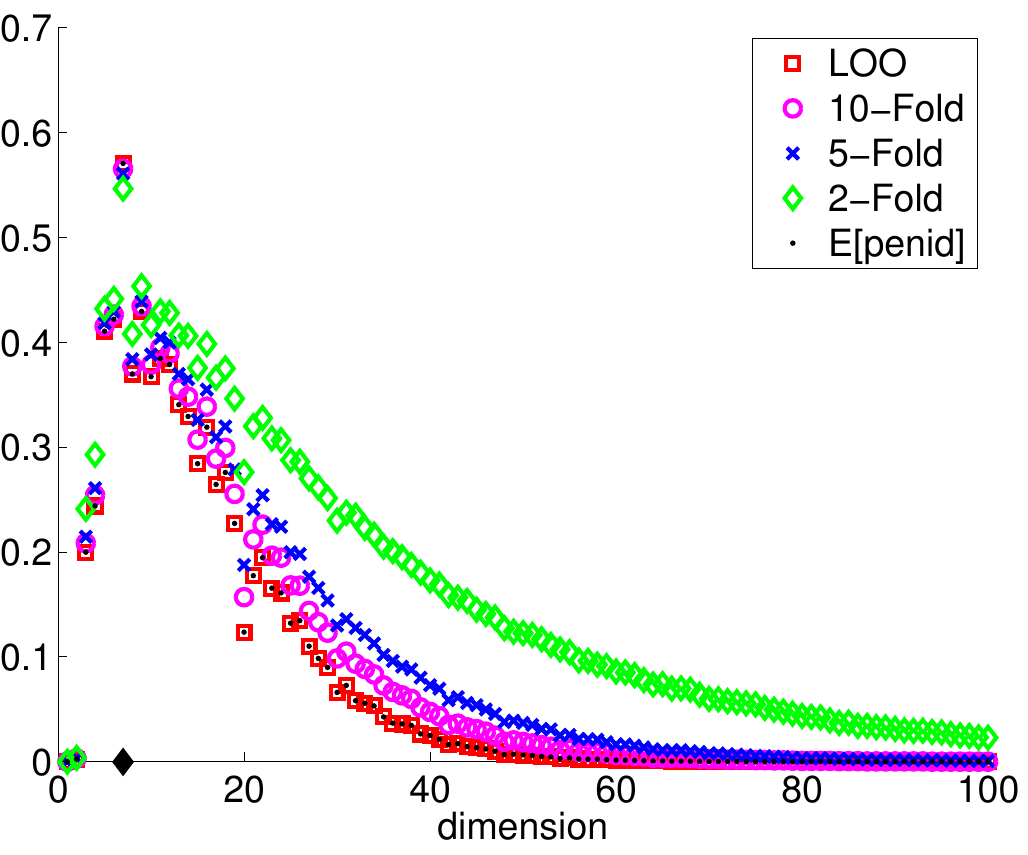}
\end{minipage}
\end{center}
\caption{%
Illustration of the variance heuristic: $\overline{\Phi}(\SR_{\, \CV}(m))$ as a function of $m$ for five different $\CV$. 
Setting S-Regu, $n=100$.  
The black diamond shows $\mo=7$. 
\label{fig.variance.PhiSR.SiG5Regu.n100}
}
\end{figure}

\section[Fast Algorithm for Computing V-Fold Penalties]{Fast Algorithm for Computing V-Fold Penalties for Least-Squares Density Estimation} \label{sec.algo}
Since the use of $V$-fold algorithms is motivated by computational reasons, it is important to discuss the actual computational cost of $V$-fold penalization and cross-validation as a function of $V$. 
In the least-squares density estimation framework, two approaches are possible: a naive one---valid for all other frameworks---, and a faster one---specific to least-squares density estimation. 
For clarifying the exposition, we assume in this section that \eqref{hyp.part-reg.exact} holds true---so, $V$ divides $n$.
The general algorithm for computing the $V$-fold penalized criterion and/or the $V$-fold cross-validation criterion consists in training the estimator with data sets $(\xi_i)_{i \notin \B_j}$ for $j=1, \ldots, V$ and then testing each trained estimator on the data sets $(\xi_i)_{i \in \B_j}$ and/or $(\xi_i)_{i \notin \B_j}$. 
In the least-squares density estimation framework, for any model $S_m$ given through an orthonormal family $\parenj{\psil}_{\lamm}$ of elements of $L^2(\mu)$, we get the ``naive'' algorithm described and analysed more precisely in Section \ref{sec.algo.naive}, whose complexity is of order $n V \card(\Lambda_m)$. 

Several simplifications occur in the least-squares density estimation framework, that allow to avoid a significant part of the computations made in  the naive algorithm. 
\begin{proc} \label{proc.penVF.fast}\hfill \\
\vspace{-0.5cm}
\begin{enumerate}
\item[] \textbf{Input\textup{:}} $\B$ some partition of $\setj{1, \ldots ,n}$ satisfying \eqref{hyp.part-reg.exact}, 
 $\xi_1, \ldots, \xi_n \in \X$ 
 and $\parenj{\psil}_{\lamm}$ a finite orthonormal family of $L^2(\mu)$. 
\item For $i \in \sets{1, \ldots, V}$ and $\lamm$, compute $ A_{i,\lambda} \egaldef \frac{V}{n} \sum_{j \in B_i} \psil(\xi_j) $. 
\item For $i,j \in \sets{1, \ldots, V}$, compute $  C_{i,j} \egaldef \sum_{\lamm} A_{i,\lambda} A_{j,\lambda} $.
\item Compute  $ \mathcal{S} \egaldef \sum_{1 \leq i,j \leq V} C_{i,j}$ and $ \mathcal{T} \egaldef \tr(C)$.
\item[] \textbf{Output\textup{:}} \\
Empirical risk\textup{:} $\displaystyle P_n \gamma\parenj{\ERM_m} = \frac{- \mathcal{S} }{ V^2 }$; 
\\
$V$-fold cross-validation criterion\textup{:} $\displaystyle \critVF(m) = \frac{\mathcal{T}}{V(V-1)} - \frac{\mathcal{S} - \mathcal{T}}{(V-1)^2}$; 
\\
$V$-fold penalty\textup{:}  $\displaystyle \penVF(m) = \parenb{ \critVF(m) - P_n \gamma\parenj{\ERM_m} } \frac{V-1/2}{V-1}$.
\end{enumerate}
\end{proc}
To the best of our knowledge, Algorithm~\ref{proc.penVF.fast} is new, even for computing the $V$-fold cross-validation criterion. 
Its correctness and complexity are analyzed with the following proposition. 
\begin{proposition} \label{pro.proc.penVF-fast.general-density}
Algorithm~\ref{proc.penVF.fast} is correct and has a computational complexity of order 
\[ 
\parenb{ n + V^2 } \card(\Lambda_m) 
\enspace . 
\]
In the histogram case, that is, when $\Lambda_m$ is a partition of $\X$ and $\forall \lamm$, $\psil = \mu(\lambda)^{-1/2} \un_{\lambda}$, the computational complexity of Algorithm~\ref{proc.penVF.fast} can be reduced to the order of $n + V^2 \card(\Lambda_m)$. 
\end{proposition}
Proposition~\ref{pro.proc.penVF-fast.general-density} is proved in Section~\ref{sec.app.proof.pro.algo}.
It shows that Algorithm~\ref{proc.penVF.fast} is significantly faster than the ``naive'' Algorithm~\ref{proc.penVF.naive} described in Section~\ref{sec.algo.naive}, by a factor of order 
\[ \frac{nV}{n+V^2} = \mathopen{} \left( \frac{1}{V} + \frac{V}{n} \right)^{-1} \mathclose{}
\ll 1 \qquad \mathrm{if} \qquad 1 \ll V \ll n \, . \]
Note that closed-form formulas are available for the leave-$p$-out criterion in least-squares density estimation \citep{Cel:2008}, allowing to compute it with a complexity of order $n \card(\Lambda_m)$ in general, and smaller in some particular cases---for instance, $n$ for histograms.

\section{Discussion} \label{sec.discussion}
Before discussing how to choose $V$ when using $V$-fold methods for model selection---or more generally for choosing among a given family of estimators---, we state some additional results and we discuss the model selection literature in least-squares density estimation. 

\subsection{Monte-Carlo Cross-Validation} \label{sec.discussion.MCCV}
Our analysis of $V$-fold procedures for model selection can be extended to some other 
cross-validation procedures. 
We here present results for Monte-Carlo cross-validation \citep[MCCV,][]{Pic_Coo:1984}, 
also known as repeated cross-validation, 
where $B$ training samples of the same size $n-p$ are chosen independently and uniformly 
\citep[see also][Section~4.3.2]{Arl_Cel:2010:surveyCV}. 
Formally, we consider the criterion 
\begin{equation}
\label{def.CVgal}
\critCV \parenb{ m , (T_K)_{1 \leq K \leq B} } 
\egaldef \frac{1}{B} \sum_{K=1}^B \critHO(m, T_K)
\enspace , 
\end{equation}
where 
$T_1, \ldots, T_B$ are subsets of $\inter{n}$ and 
we recall that the hold-out criterion is defined by Eq.~\eqref{def:crit.HO}. 
We make the following three assumptions throughout this subsection 
\begin{gather}
\label{hyp.CV.same-size} 
\tag{\ensuremath{\mathbf{SameSize}}}
\exists p \in \inter{n-1}, \quad 
\forall j \in \inter{B}, \qquad 
| T_j | = n - p=n\fractrain_n
\enspace , 
\\
\label{hyp.CV.ind}
\tag{\ensuremath{\mathbf{Ind}}}
(T_K)_{1 \leq K \leq B} 
\quad \text{is independent from} \quad 
D_n 
\enspace , 
\\
\label{hyp.MCCV}
\tag{\ensuremath{\mathbf{MCCV}}}
T_1, \ldots, T_B  
\quad \text{are independent with uniform distribution over} \quad 
\mathcal{E}_{n-p} 
\enspace , 
\end{gather}
where we recall that 
$ \mathcal{E}_{n-p} = \setj{ A \subset \inter{n} \telque |A| = n-p }$. 
Under these assumptions, we write $\CVMCCV(m)$ as a shortcut for 
$\critCV \parens{ m , (T_K)_{1 \leq K \leq B} } $. 

\medbreak

Similarly to Theorem~\ref{thm.oracle-penVF.cas_reel}, 
we prove in Section~\ref{sec.supmat.MCCV.oracle} the following oracle inequality for MCCV. 
\begin{theorem} \label{thm.oracle-MCCV}
Let $\xi_{\inter{n}}$ be i.i.d. real-valued random variables with common density $\bayes \in L^{\infty}(\mu)$, 
$(T_K)_{1 \leq K \leq B}$ some sequence of subsets of $\inter{n}$ satisfying \eqref{hyp.CV.same-size}, \eqref{hyp.CV.ind} and \eqref{hyp.MCCV}
and $(S_m)_{m\in\M_n}$ be a collection of separable linear spaces satisfying \eqref{hyp.NormSupNorm2}. 
Assume that either \eqref{hyp.UBbayes} or \eqref{hyp.Nested} holds true.
For every $\mM_n$, let $\ERM_m$ be the estimator defined by Eq.~\eqref{def.ERM}, 
and $\widetilde{s} = \ERM_{\mh}$ where 
\[ 
\mh \in \argmin_{\mM_n} \setB{ \critCV \parenb{ m , (T_K)_{1 \leq K \leq B} } }
\]
and $\critCV$ is defined by Eq.~\eqref{def.CVgal}. 
Let us define, for any $x,y,\epsilon>0$, 
$x_n = x + \log |\M_n| $ and 
\begin{equation*}
\resteMCCV \parenj{\epsilon,x,y,n,\fractrain_n,B,A}
\egaldef 
\frac{1}{n \fractrain_n^2 } \mathopen{} \left( { 1+ \frac{B\wedge(\log n+y)}{B (1-\fractrain_n)} } \right)^{\alpha} \mathclose{} \parenj{\frac{ A x}{ \fractrain_n \epsilon} + \frac{(A \vee 1) x^2}{\epsilon^3}} 
\end{equation*}
with $\alpha = 1$ under assumption \eqref{hyp.UBbayes}
and $\alpha = 2$ under assumption \eqref{hyp.Nested}. 
Then, an absolute constant $\kappa>0$ exists such that, for any $x,y \geq 0$, with probability at least $1- \e^{-x} - \e^{-y}$, 
for any $\epsilon \in (0,\kappa^{-1})$, 
\begin{equation}\label{eq:oracle_MCCV}
\parenj{1-\frac{\epsilon}{\fractrain_n}} \perte{ \widetilde{s} }
\leq
\frac{1+\epsilon}{\fractrain_n} \inf_{\mM_n} \setj{ \perte{\ERM_{m}} }
+ \kappa \resteMCCV \parenj{\epsilon,x_n,y,n,\fractrain_n,B,A}
\enspace . 
\end{equation}
\end{theorem}
Theorem~\ref{thm.oracle-MCCV} actually is a corollary of a more general result
(Theorem~\ref{thm.oracle-CV-gal} in Section~\ref{sec.supmat.MCCV.oracle}), which is valid without assumption~\eqref{hyp.MCCV} and extends therefore our previous results on $V$-fold cross-validation). 

Very few results exist in the literature about the model selection performance of MCCV with an estimation goal. 
Some asymptotic optimality result has been obtained by \citet{Bur:1990} for spline regression, 
and some oracle inequalities comparing the risk of the selected estimator with the risk of an oracle trained with $\fractrain_n n<n$ data 
have been proved by \citet{vdL_Dud:2003} in a general framework and by \citet{vdL_Dud_Kel:2004} for density estimation with the Kullback-Leibler loss. 
In comparison, Theorem~\ref{thm.oracle-MCCV} provides a precise non-asymptotic comparison to an oracle trained with $n$ data. 

As in Theorem~\ref{thm.oracle-penVF.cas_reel}, the leading constant of the oracle inequality \eqref{eq:oracle_MCCV} 
is directly related to the bias, which is here quantified by $\fractrain_n^{-1} - 1 \geq 0$ instead of $\delta$. 
The remainder term $\resteMCCV$ is also comparable to $\resteB$ in Theorem~\ref{thm.oracle-penVF.cas_reel}: 
they differ by a factor between  $\fractrain_n^{-2}$ (when $B$ is large enough) 
and $\fractrain_n^{-2} (1-\fractrain_n)^{-\alpha}$ (when $B$ is small). 
In particular, let $V \geq 2$ and assume that $p=n/V$ in Theorem~\ref{thm.oracle-MCCV}, hence $\fractrain_n = 1 - V^{-1} \in [1/2,1)$. 
Then, for the hold-out ($B=1$), $\resteMCCV$ is larger than $\resteB$ by a factor $V^{\alpha}$ with $\alpha\in\sets{1,2}$. 
For $B=V$, MCCV with $\fractrain_n = 1 - V^{-1}$ can be called ``Monte-Carlo $V$-fold'' (MCVF); 
then, with $y \approx \log n$, we loose a factor at most $\log n$ for MCVF compared to $V$-fold cross-validation. 
Finally, when $B$ is large enough, that is, larger than $V \log n$, $\resteMCCV$ and $\resteB$ are of the same order. 

\medbreak

The above comparison of remainder terms suggests a hierarchy between several cross-validation methods with a common 
training sample size $n-p=n\fractrain_n$: 
from the (presumably) worse to the (presumably) best procedure, 
the hold-out, Monte-Carlo CV with $B=V$, $V$-fold CV, 
Monte-Carlo CV with $B$ large and the leave-$p$-out. 
Nevertheless, upper bounds comparison can be misleading, so, following the heuristics 
\eqref{eq.heuristic.variance-comp} presented in Section~\ref{sec.oracle.key-quant}, we compute below the variance of 
$\Delta_{\CV}(m,m')$ when $\CV$ is a Monte-Carlo CV criterion. 
\begin{theorem} \label{thm.var-MCCV}
We consider the setting and notation of Theorem~\ref{theo.variance.penVF}, 
and we assume that \eqref{hyp.CV.same-size}, \eqref{hyp.MCCV} and \eqref{hyp.CV.ind} hold true. 
We recall that 
$\CVMCCV(m)$ is defined above at the beginning of Section~\ref{sec.discussion.MCCV}.
Then, for regular histogram models $m_1,m_2$ 
\textup{(}Example~\ref{ex:regular} in Section~\ref{sec.ex.histos}\textup{)}, we have 
\begin{align}
\label{eq.var.MCCV.incr.histos}
\var\parenj{ \CVMCCV(m_1) - \CVMCCV(m_2) } 
&= 
\cteMCCVa (B,n,\fractrain_n) 
\frac{2}{n^2} \termeBvaracr{m_1,m_2}
\\
\notag 
&\qquad + 
\cteMCCVb (B,n,\fractrain_n) \frac{4}{n} \var\parenb{\bayes_{m_1} (\xi_1) - \bayes_{m_2}(\xi_1)} 
\end{align}
where 
\begin{align*}
\cteMCCVa  (B,n,\fractrain_n) 
&= 
\frac{1}{B}  \parenj{ \frac{1}{\fractrain_n^2} + \frac{2}{\fractrain_n (1-\fractrain_n)} - \frac{1}{n \fractrain_n^3} } 
+ \parenj{ 1 - \frac{1}{B} } \crochj{1+\frac{1}{n-1} \parensq{ \frac{1}{\fractrain_n} + 1 } -\frac{1}{n \fractrain_n^2}}   
\\
\cteMCCVb (B,n,\fractrain_n) 
&= \frac{1}{B} \parenj{ \frac{1}{n^2 \fractrain_n^3}  + \frac{1}{1-\fractrain_n} } 
+ \parenj{ 1 - \frac{1}{B} } 
\parensq{1+\frac{1}{n \fractrain_n}}
\end{align*}
and we recall that $\fractrain_n = |T_K|/n = 1 - (p/n)$. 
\end{theorem}
Theorem~\ref{thm.var-MCCV} is proved in Section~\ref{sec.supmat.MCCV.var}, 
as a corollary of a more general result, Theorem~\ref{thm.var-MCCV-gal}, 
which holds for all models $m_1,m_2$---not only regular histograms---and 
provides a formula for the variance of the criterion itself---not its increments. 
Let us make a few comments. 

Eq.~\eqref{eq.var.MCCV.incr.histos} is similar to the formula obtained for 
bias-corrected $V$-fold and $V$-fold penalization, see 
Eq.~\eqref{eq.pro.variance.penVF-penid-incrReg} in Theorem~\ref{theo.variance.penVF}. 
In the particular case of regular histogram models, 
Eq.~\eqref{eq.pro.variance.penVF-penid-incrReg} even fits the general form of Eq.~\eqref{eq.var.MCCV.incr.histos}, 
with constants $\ctepenVFi  (V,n,C)$ 
instead of $\cteMCCVi  (B,n,\fractrain_n)$. 

Assuming the heuristics of Section~\ref{sec.oracle.key-quant} is valid, 
for $m_1,m_2$ which matter for model selection, 
the two terms 
$ 2 n^{-2} \termeBvaracr{m_1,m_2}$ and 
$ 4 n^{-1} \var\parens{\bayes_{m_1} (\xi_1) - \bayes_{m_2}(\xi_1)}$ 
are of the same order of magnitude (see Section~\ref{sec.variance}). 
Then, we can compare model selection performance of several cross-validation 
methods by comparing the values of the constants $\cteGENi$ only. 

In order to get a variance of the same order of magnitude as the one of 
bias-corrected $V$-fold CV---that is, constants $\cteGENi$ of order~$1$---, 
MCCV requires to take $\fractrain_n$ far enough from $0$ and~$1$, 
hence training and sample sets of comparable sizes, 
unless $B$ is large enough. 

Eq.~\eqref{eq.var.MCCV.incr.histos} allows to compare the hold-out ($B=1$) with the 
leave-$p$-out ($B\to+\infty$), for a given value $n\fractrain_n = n-p$ of the training sample size. 
Let us assume for simplicity that $n\to+\infty$ 
and $\fractrain_n \gg n^{-1/2}$.  
Then, 
\begin{align*} 
&\hspace*{-0.75cm}\cteMCCVa(1,n,\fractrain_n) \sim \frac{1}{\fractrain_n^2} + \frac{2}{\fractrain_n (1-\fractrain_n)} 
> 11 
\qquad &\text{and} \qquad 
\cteMCCVb(1,n,\fractrain_n) &\sim \frac{1}{1-\fractrain_n} 
\geq 1
\\
\text{whereas} \qquad 
&\cteMCCVa(\infty,n,\fractrain_n) \to 1
\qquad &\text{and} \qquad 
\cteMCCVb(\infty,n,\fractrain_n) &\to 1
\end{align*}
which shows an improvement at least by a constant factor in general. 
When $\fractrain_n$ tends to zero---leave-most-out---or~$1$---such as for the leave-one-out---, 
the improvement is by an order of magnitude. 
The fact that the leave-$p$-out has a smaller variance than the hold-out is not surprising at all---it holds in full generality, as a consequence of Jensen's inequality---, 
but the exact quantification of the improvement given by Theorem~\ref{thm.var-MCCV} 
is new and can be useful in practice for choosing the number of splits $B$ 
when using Monte-Carlo cross-validation. 

Eq.~\eqref{eq.var.MCCV.incr.histos} also allows to compare $V$-fold cross-validation, 
given by Theorem~\ref{theo.variance.penVF} with \[ C=1+ \frac{1}{2(V-1)} 
\enspace , \]
with MCCV with $B=V$ and $\fractrain_n = (V-1)/V$, 
which can be named ``Monte-Carlo $V$-fold'' cross-validation. 
The only difference between the two methods is that the $V$ splits are chosen independently 
for ``Monte-Carlo $V$-fold'', 
whereas the usual $V$-fold makes a balanced use of each observation---putting it exactly $(V-1)$ times in the training set. 
Let us assume for simplicity that $n\to+\infty$ while $V=V_n$ can vary with $n$.  
Then, we have 
\begin{align*} 
\cteMCVFa(V_n , n) &\egaldef 
\cteMCCVa \parenj{ V_n , n , \frac{V_n - 1}{V_n} } \sim 
3 + \frac{2 V_n + 1}{V_n (V_n-1)} + \frac{1}{(V_n - 1)^2}
%%\qquad %%\text{and} \qquad 
\\
\cteVFCVa(V_n , n) &\egaldef 
\ctepenVFa \parenj{ V_n , n, 1 + \frac{1}{2(V_n-1)} } \sim 
1 + \frac{4}{V_n-1} + \frac{4}{(V_n-1)^2} + \frac{1}{(V_n-1)^3}
\end{align*}
%%\\
\begin{align*}
&\text{hence} \qquad \qquad 
\frac{\cteMCVFa(V_n , \infty)}{\cteVFCVa(V_n , \infty)} > 1 \text{ if } V_n \geq 3 
\enspace , \qquad 
%%\qquad \text{and} \qquad 
\frac{\cteMCVFa(V_n , n)}{\cteVFCVa(V_n , n)} \xrightarrow[n,V_n \to +\infty]{}  3 
\enspace , 
\end{align*}
%%\\
\begin{align*}
%&\phantom{\text{hence}} \qquad \qquad 
\cteMCVFb(V_n , n) 
&\egaldef 
\cteMCCVb \parenj{ V_n , n , \frac{V_n - 1}{V_n} } \sim 
2 - \frac{1}{V_n} \in \crochj{ \frac{3}{2}, 2 }
\\
\text{and} \qquad \qquad \hspace{1.91cm}
\cteVFCVb(V_n , n) 
&\egaldef 
\ctepenVFb \parenj{ V_n , n, 1 + \frac{1}{2 (V_n - 1)} } \to 
1
\enspace . 
\end{align*}
Overall, we get that $V$-fold cross-validation has a smaller variance than 
``Monte-Carlo $V$-fold'' for $V \geq 3$, at least for $n$ large enough, 
and that the improvement is by a constant factor between $3/2$ and $3$. 
Since increasing $V$ cannot decrease the variance of (bias-corrected) VFCV 
by more than a small constant factor, 
the above difference between two methods with the same computational complexity 
is quite important. 
This supports strongly the use of $V$-fold CV methods instead of ``Monte-Carlo $V$-fold''. 
Such an improvement was previously noticed in the asymptotic computations of 
\citet{Bur:1989}; 
here we show that it holds in a non-asymptotic framework, 
where the models $m_1,m_2$ can depend on $n$. 

\subsection{Hold-Out Criteria} \label{sec.discussion.hold-out}
Our analysis of cross-validation procedures for model selection can also be extended to hold-out criteria. 
First, let us emphasize that the hold-out criterion defined by Eq.~\eqref{def:crit.HO} 
corresponds to taking $B=1$ in the results of Section~\ref{sec.discussion.MCCV}, 
since choosing $T$ uniformly over $\mathcal{E}_{n-p}$, independently from $D_n$, 
is equivalent to choosing some arbitrary $T$ of size $n-p$ before seeing the data $D_n$. 

Second, similarly to the definition of the hold-out criterion in Eq.~\eqref{def:crit.HO}, 
we can define the hold-out penalty by 
\begin{equation} \label{def.penHO} 
\forall x \geq 0 , \quad 
\penHO(m,\ItHO,x) \egaldef 2 x \parenj{ P^{(\ItHO)}_n-P_{n}} \parenj{ \ERM^{(\ItHO)}_m-\ERM_{m}}
\enspace , 
\end{equation}
that is, the hold-out estimator of $\E\crochs{2 (P_n - P) (\ERM_m - \bayes_m)}$ which is equal to the expectation of the ideal penalty, see Eq.~\eqref{def.penid}. 
We do not define $\penHO$ by Eq.~\eqref{eq.def.penVF}  with $V=1$ and $\ItHO=\B_1^c$---that is, the hold-out estimator of $\E\crochs{(P-P_n)\gamma(\ERM_m)}$, which amounts to removing the centering term $-\ERM_{m}$ in Eq.~\eqref{def.penHO}---because this would dramatically increase its variability. 
Note that adding such a term  $-\ERM_{m}$ in Eq.~\eqref{eq.def.penVF} does not change the value of the $V$-fold penalty under \eqref{hyp.part-reg.exact} since $\sum_{K=1}^V (P_n^{(\B_K^c)} - P_n) = 0$. 

Denoting by $\fractrain_n=\absj{T}/n$ as in Section~\ref{sec.discussion.MCCV}, it comes from Lemma~\ref{lem:DecPenHO} in Section~\ref{sec.supmat.penHO.oracle} that
\[
\E\crochb{\penHO(m,\ItHO,x)}
= x \frac{1-\fractrain_n}{\fractrain_n} \E\crochb{\penid(m)}
\enspace.\]
In the following, we choose $x=C\fractrain_n/(1-\fractrain_n)$ so that $C=1$ corresponds to the unbiased case, as in the previous sections for the $V$-fold penalty.

\begin{remark} \label{rk.ho=2-fold}
Since $P_n=\fractrain_n P_n^{(\ItHO)}+(1-\fractrain_n)P_n^{(\ItHO^c)}$, by linearity of the estimator $\ERM_m$, 
\begin{equation} \notag %\label{def.penHO.alt} 
\penHO(m,\ItHO,x) \egaldef 2 x (1-\fractrain_n)^2 \parenj{ P^{(\ItHO)}_n-P^{(\ItHO^c)}_{n}} \parenj{ \ERM^{(\ItHO)}_m-\ERM^{(\ItHO^c)}_{m}}
\end{equation}
which is symmetric in $\ItHO$ and $\ItHO^c$, hence 
$\penHO(m,\ItHO^c,x)=\penHO(m,\ItHO,x)$. 
In particular, if $|T|=n/2$, the $2$-fold penalty computed on the partition $\B=\sset{\ItHO,\ItHO^c}$ and the hold-out penalty coincide 
\begin{equation}
\notag 
\forall x >0, \quad 
\penVF \parenb{ m,\sets{\ItHO,\ItHO^c},x } = \penHO(m,\ItHO,x)
\enspace . 
\end{equation}
\end{remark}

\begin{theorem}\label{thm:PenHo}
Let $\xi_{\inter{n}}$ be i.i.d. real-valued random variables, 
$\bayes \in L^{\infty}(\mu)$ their common density, 
$T \subset \inter{n}$ with $\fractrain_n = |T|/n \in (0,1)$ 
and $(S_m)_{m\in\M_n}$ be a collection of separable linear spaces satisfying \eqref{hyp.NormSupNorm2}. 
Assume that either \eqref{hyp.UBbayes} or \eqref{hyp.Nested} holds true.
Let $C \in (1/2,2]$ and \( \delta\egaldef 2(C-1)\). 
For every $\mM_n$, let $\ERM_m$ be the projection estimator onto $S_m$ defined by Eq.~\eqref{def.ERM}, and 
$\widetilde{s}_{\HO}= \ERM_{\mHO}$ where 
\begin{equation*}%\label{def.est.HO}
%\quad \text{where} \quad 
\mHO 
%= \mHO(\ItHO,x) 
= \argmin_{m\in\M_{n}} \setj{ P_n\gamma\parenj{\ERM_{m}} +  \penHO\parenj{m,\ItHO,\frac{C\fractrain_n}{1-\fractrain_n}}} 
%= \argmin_{m\in\M_{n}} \setj{ \CVho_{C,\ItHO}(m) } 
%\quad \text{where} \quad 
%\CVho_{C,\ItHO}(m) \egaldef P_n\gamma\parenj{\ERM_{m}} +  \penHO\parenj{m,\ItHO,\frac{C\fractrain_n}{1-\fractrain_n}}
\enspace .
\end{equation*}
Then, an absolute constant $\kappa$ exists such that, for any $x>0$, defining $x_n = x + \log |\M_n|$, with probability at least $1-\e^{-x}$, for any $\epsilon \in (0,1]$, 
\begin{align}
\label{eq.ineq-oracle.penHO}
\frac{1-\delta_{-}-\epsilon}{1+\delta_{+}+\epsilon} 
\perte{ \widetilde{s}_{\HO} }\leq \inf_{\mM_n} \perte{\ERM_m}+\kappa\parenj{\frac{Ax_n}{ \epsilon n}+\frac{\fractrain_n^2+(1-\fractrain_n)^2}{\fractrain_n(1-\fractrain_n)}\frac{x_n^2}{ \epsilon^3 n}}\enspace .
\end{align}
\end{theorem}
Theorem~\ref{thm:PenHo} is proved in Section~\ref{sec.supmat.penHO.oracle}. 

Theorem~\ref{thm:PenHo} extends Theorem~\ref{thm.oracle-penVF.cas_reel} to hold-out penalties, under similar assumptions. 
As in Theorem~\ref{thm.oracle-penVF.cas_reel}, $\delta$ quantifies the bias of the hold-out penalized criterion, and plays the same role in the leading constant of the oracle inequality \eqref{eq.ineq-oracle.penHO}. 

We can compare the results obtained for hold-out and $V$-fold penalization in Theorems~\ref{thm.oracle-penVF.cas_reel} and~\ref{thm:PenHo}. 
For this comparison, let $V$ be some divisor of $n$, $T \subset \inter{n}$ such that $|T|=n-n/V$ and choose the same $C$ so that both criteria have the same bias $\delta$. 
Then, the only difference lies in the remainder term, the one in Eq.~\eqref{eq.ineq-oracle.penHO} is larger than the one of Eq.~\eqref{eq.thm.oracle-penVF.cas_reel} in Theorem~\ref{thm.oracle-penVF.cas_reel} by a factor of order $V$ 
when $V$ is large. 
These only are upper bounds, but at least they are consistent with the common intuition about the stabilizing effect of averaging over $V$ folds. \\
We can also compare the results obtained for hold-out penalization in Theorem~\ref{thm:PenHo} 
and for the hold-out criterion in Theorem~\ref{thm.oracle-MCCV}. 
First, hold-out penalization gives a flexibility to choose an unbiased criterion and therefore to obtain asymptotically optimal oracle inequalities while hold-out criteria are always biased for fixed $\fractrain_n$, 
hence a leading constant $\fractrain_n^{-1}>1$ in the oracle inequality. 
The loss in the remainder term is also smaller in Eq.~\eqref{eq.ineq-oracle.penHO} than in Eq.~\eqref{eq:oracle_MCCV} 
by a factor of order $\fractrain_n^{-1} (1-\fractrain_n)^{-1}$ under assumption \eqref{hyp.Nested}.

Similarly to Theorems~\ref{theo.variance.penVF} and~\ref{thm.var-MCCV}, the variance terms can be computed for the hold-out penalty in order to understand separately the roles of the training sample size and of averaging over the $V$ splits, in the $V$-fold criteria. 
Detailed results are given by Proposition~\ref{pro.variance.penHO} in Section~\ref{sec.supmat.penHO.variance}. 

\subsection{Other Oracle Inequalities for Least-Squares Density Estimation} \label{sec.discussion.other-proc}
Although the primary topic of the paper is the study of $V$-fold procedures, let us compare briefly our results to other oracle inequalities that have been proved in the least-squares density estimation setting. 
For projection estimators, \citet[Section~7.2]{Mas:2003:St-Flour} proves an oracle inequality for some penalization procedures, which are suboptimal since the leading constant $C_n$ does not tend to~1 as $n$ goes to $+\infty$. 
Oracle inequalities have also been proved for other estimators: blockwise Stein estimators \citep{Ri06}, 
linear estimators \citep{Gol_Lep:2011:AoS} and some $T$-estimators \citep{Bi08}. 
The models considered by \citet{Bi08} are more general than ours, but the corresponding estimators are not computable in practice, and the oracle inequality by \citet{Bi08} also has a suboptimal constant $C_n$. 
Some aggregation procedures also satisfy oracle inequalities %, as proved for instance in 
\citep{RT07,Bun_Tsy_Weg_Bar:2010}. 
Overall, under our assumptions, none of these results imply strictly better bounds than ours. 

Let us finally mention that \citet{BR03} propose a precise evaluation of the penalty term in the case of regular histogram models and the log-likelihood contrast. 
Their final penalty is a function of the dimension, only slightly modified compared to $\pendim$, performing very well on regular histograms. 
These performances are likely to become much worse on the collection Dya2 presented in Section~\ref{sec.simus}. 
This can be seen, for example, in Table~\ref{tab.complet.Li01.SiG5} in Section~\ref{sec.supmat.simus}, where we present the performances of $\pendim$ with different over-penalizing constants.

\subsection{Conclusion on the Choice of $V$} \label{sec.discussion.choice-V}
This section summarizes the results of the paper in order to address the main question we would like to answer: 
How to choose a $V$-fold procedure for model selection? 

\fauxparagraph{Generality of the results} 
The results of the paper only hold for projection estimators in least-squares density estimation, but we conjecture that most of the statements below are valid much more generally. 
At least, they have been observed experimentally for projection estimators in least-squares regression \citep{Arl:2008a} and they are supported by theoretical results for kernel density estimators \citep[][Chapters~3--4]{Mag:2015}. 
Nevertheless, it is reported in the literature that $V$-fold cross-validation can behave differently in other settings \citep{Arl_Cel:2010:surveyCV}, so we must keep in mind that the statements below may not be universal. 

Let us also recall that we focus here on model selection with an estimation goal, that is, minimizing the risk of the final estimator; see \citet{Yan:2006,Yan:2007b} and \citet{Cel:2008} for results when the goal is identification. 

\medbreak

\fauxparagraph{Choice of a model selection procedure} 
Choosing among procedures of the form $\mh(\CV)$, as defined by Eq.~\eqref{eq.def.mhCVgal}, 
requires to take into account three quantities:
\begin{itemize}
\item \emph{the bias} of $\CV(m)$ as an estimator of the risk of $\ERM_m$ for every $\mM_n$, 
or equivalently, the \emph{overpenalization factor} $C$, 
which usually drives the performance at first order when $n \to +\infty$, as in Theorem~\ref{thm.oracle-penVF.cas_reel}. 
The simulation experiments of Section~\ref{sec.simus} also show that varying $C$ can strongly change the performance of the procedure. 
In all settings considered in the paper, some $C^{\star}_n$ exists (the optimal overpenalization constant) such that the performance decreases for $C \in [0,C^{\star}_n]$ and increases for $C>C^{\star}_n$ (Figure~\ref{fig.surpen.LS-Dya2.n500}). 

Note that $C^{\star}_n$ strongly depends on the setting, and can also vary with $V$ when using $V$-fold penalization (in particular from $V=2$ to $V\geq 5$). 
In the nonparametric case, when $n \to +\infty$, Theorem~\ref{thm.oracle-penVF.cas_reel} shows that 
$C^{\star}_n \sim 1$. 
On the contrary, in the parametric case, when $n \to +\infty$, it is known that a BIC-type penalty performs better, hence $C^{\star}_n \to +\infty$. 
For a finite sample size, Section~\ref{sec.simus} and \citet{Liu_Yan:2011} show that some 
nonparametric settings can be ``practically parametric'', that is, $C^{\star}_n$ can be much larger than~$1$. 

\item \emph{the variance} of increments $\CV(m) - \CV(m')$ drives the performance $\mh(\CV)$ at second order, according to the heuristic of Section~\ref{sec.oracle.key-quant}, which suggests that this variance should be minimized, at least for a given ``good enough'' value of the overpenalization factor $C$. 

\item \emph{the computational complexity} of the procedure $\mh(\CV)$, that we want to minimize---for a given statistical performance---, or on which some upper bound is given---fixed budget. 
\end{itemize}

\fauxparagraph{$V$-fold cross-validation} 
The paper analyzes how the aboves three terms depend on $V$ when $\CV=\CV^{\mathrm{VFCV}}_V$ is a $V$-fold cross-validation procedure, under assumption \eqref{hyp.part-reg.exact}. 
First, by Lemma~\ref{le.penVF-VFCV}, its overpenalization factor is $\CteVFCV(V) = 1 + 1/[2(V-1)] \in [1,3/2]$, which decreases to 1 as $V$ increases to $+\infty$. 
Second, by Theorem~\ref{theo.variance.penVF}, its variance decreases as $V$ increases. 
Theoretical and empirical arguments in Sections~\ref{sec.variance} and~\ref{sec.simus} show that the variance almost reaches its minimal value by taking, say, $V=5$ or $V=10$. 
Third, by Section~\ref{sec.algo}, its computational complexity is proportional to $V$ in general; 
in the least-squares density estimation setting, it can be reduced to $(n+V^2) \card(\Lambda_m)$ . 

These three results can explain why the most common advices for choosing $V$ in the literature \citep[for instance][Section~7.10.1]{Bre_Spe:1992,Has_Tib_Fri:2001v2009} are between $V=5$ and $V=10$. 
Indeed, taking $V$ larger does not reduce the variance significantly---with almost no impact on the risk of the final estimator---, and it reduces the overpenalization factor although $C^{\star}_n$ is often larger than $\CteVFCV(10)=19/18$ or $\CteVFCV(5)=9/8$. 
So, if $C^{\star}_n$ is not much larger than $1+1/8$, which is likely to occur in many nonparametric settings, taking $V=5$ or $10$ can be close to be optimal. 

Nevertheless, other situations can occur, for instance in (practically) parametric settings where $C^{\star}_n$ is much larger, possibly leading to the failure of the heuristic ``$5 \leq V \leq 10$ is almost optimal''. 
More generally, understanding precisely how $\CV^{\mathrm{VFCV}}_V$ performs as a function of $V$ seems to be a difficult question: 
$V$ influences the performance in two opposite directions simultaneously, through the bias and the variance, so that various behaviours can result from this coupling of bias and variance, as shown in the simulation experiments. 

\medbreak

\fauxparagraph{$V$-fold penalization} 
Lemma~\ref{le.penVF-VFCV} shows that a natural way to solve this difficulty is to consider 
instead a $V$-fold penalization procedure $\CV^{\penVF}_{(C,V)}$, with overpenalization factor $C>0$. 
The value $C=\CteVFCV(V)$ corresponds to $V$-fold cross-validation, but any other value of $C$ can also be considered, making it easier to understand. 
Indeed, the overpenalization factor is directly given by $C$, while the variance and computational complexity of $\CV^{\penVF}_{(C,V)}$ vary with $V$---independently from $C$---exactly as for $V$-fold cross-validation. 
So, $V$ should be taken as large as possible---depending on the maximal computational budget available---, while $C$ should be taken as close as possible to $C^{\star}_n$. 

Compared to $V$-fold cross-validation, another interest of $V$-fold penalization is the improvement of the performance for a given computational cost, that is, a given value of $V$, because it is then possible to take $C$ closer to $C^{\star}_n$ than $\CteVFCV(V)$. 
This is especially true in (practically) parametric settings for which $C^{\star}_n>3/2 \geq \CteVFCV(V)$ for all $V \geq 2$. 

\fauxparagraph{Data-driven overpenalization factor $C$}
Although the paper shows that choosing well $C$ is a key practical problem, 
making an optimal data-driven choice of $C$ remains an open question 
which deserves to be studied, even independently from the analysis of cross-validation procedures. 
We postpone such a study to future works, but we can already make two suggestions. 
First, an external cross-validation loop can be used for choosing $C$, if the computational power is not a limitation. 
Second, a procedure built for choosing between AIC and BIC can be used in order to detect whether $C$ should be close to~$1$ or significantly larger (see, for instance, \citealp{Liu_Yan:2011} and references therein).

\section*{Acknowledgments} 
The authors thank the two referees for their comments that allowed us to improve the paper. 
The authors thank gratefully Yannick Baraud and Guillaume Obozinski for precious comments on an earlier version of the paper, and Nelo Magalh\~{a}es for his careful reading and helpful remarks on this earlier version. 
Let us also emphasize that this paper has changed a lot since its first version \citep{Arl_Ler:2012:penVF.v1}, 
and that the most recent results (Sections~\ref{sec.discussion.MCCV} and~\ref{sec.supmat.MCCV}) 
and some aspects of the proofs (for instance, the systematic use of $U_m(x,y)$ and $K_m(x,y)$) 
have been strongly influenced by our ongoing collaboration with Nelo Magalh\~{a}es \citep[Chapters~3--4][]{Mag:2015}. 
\\
This work was done while the first author was financed by CNRS 
and member of the Sierra team in the Departement d'Informatique de l'Ecole normale superieure (CNRS/ENS/INRIA UMR 8548), 
45 rue d'Ulm, F-75230 Paris Cedex 05, France. 
The authors acknowledge the support of the French Agence Nationale de la Recherche (ANR) under reference ANR-09-JCJC-0027-01 ({\sc Detect} project)  and ANR 2011 BS01 010 01 (projet Calibration). 
The first author also acknowledges the support of the GARGANTUA project funded by the Mastodons program of CNRS, 
and the support of Institut des Hautes \'Etudes Scientifiques (IHES, Le Bois-Marie, 35, route de Chartres,
91440 Bures-Sur-Yvette, France) during the last days of writing of this paper.

\appendix

\section{Proofs} \label{sec.proofs}
Before proving the main results stated in the paper, let us recall two simple results that we use repeatedly in the paper. 
First, if $(b_{\lambda})_{\lamm}$ is a family of real numbers such that $\sum_{\lamm} b_{\lambda}^2 < \infty$, then 
\begin{equation}
\label{eq.sup-CS} 
\sup_{\sum_{\lamm} a_{\lambda}^2 \leq 1} \parensqBb{ \sum_{\,\lamm} a_{\lambda} b_{\lambda} } = \sum_{\lamm} b_{\lambda}^2
\enspace . 
\end{equation}
The left-hand side is smaller than the right-hand side by Cauchy-Schwarz inequality, and considering $a_{\lambda} = b_{\lambda} / \parens{\sum_{\lambda' \in \Lambda_m} b_{\lambda'}^2}^{1/2}$ shows that the converse inequality holds true. 
Second, for any probability distribution $Q$ on $\X$, 
\begin{equation} \label{eq.proj-Q-formule}
\sum_{\lamm} (Q \psil) \psil \in \argmin_{t \in S_m} \setb{ Q \gamma(t) }
\enspace , 
\end{equation}
a result which provides in particular a formula for $\ERM_m$ and for $\bayes_m$, 
by taking $Q=P_n$ and $Q=P$, respectively.

\subsection{Proof of Lemma~\ref{le.penVF-VFCV}}
\label{sec.app.proof.lemmeLPO}
Let us first recall here the proof of Eq.~\eqref{eq.le.penVF-corrVFCV}---coming from \citet{Arl:2008a}---for the sake of completeness.
By \eqref{hyp.part-reg.exact}, 
\begin{equation*}
P_n - P_n^{(\B_K^{c})} = \frac{1}{V} \parenB{ P_n^{(\B_K)} - P_n^{(\B_K^{c})} }
\qquad \text{and} \qquad 
P_n^{(\B_K)} - P_n = \frac{V-1}{V} \parenB{ P_n^{(\B_K)} - P_n^{(\B_K^{c})} }
\enspace , 
\end{equation*}
so that 
\begin{align*}
\CV_{1,\B}(m) 
&\egaldef P_n\gamma\parenj{\ERM_m} + \penVF(m,\B,V-1) 
\\
&= P_n\gamma\parenj{\ERM_m} + \frac{V-1}{V^2} \sum_{K=1}^V \crochj{ \parenj{P_n^{(\B_K)} - P_n^{(\B_K^{c})} } \gamma\parenj{ \ERM_m^{(\B_K^{c})}} } 
\\
&= P_n\gamma\parenj{\ERM_m} + \frac{1}{V} \sum_{K=1}^V \crochj{ \parenj{P_n^{(\B_K)} - P_n } \gamma\parenj{ \ERM_m^{(\B_K^{c})}} } 
\\
&= \critcorrVF(m,\B) \enspace . 
\end{align*}

Eq.~\eqref{eq.le.penVF-VFCV} and \eqref{eq.le.penLPO-LPO} follow simultaneously from Eq.~\eqref{eq.le.penVF-VFCV.gal} below. 
Let $\mathcal{E}$ be a set of subsets of $\inter{n}$ such that 
\begin{equation}
\label{hyp.le.penVF-VFCV.gal}
\forall A \in \mathcal{E} , \quad |A|=p 
\quad \text{and} \quad 
\frac{1}{|\mathcal{E}|} \sum_{A \in  \mathcal{E}} P_n^{(A^{c})} = P_n \enspace .
\end{equation}
Let us consider the associated penalty 
\[ \pen_{\mathcal{E}}(m,C) = \frac{C}{|\mathcal{E}|} \sum_{A \in \mathcal{E}} \parenB{ P_n - P_n^{(A^{c})} } \gamma\parenj{\ERM_m^{(A^{c})}} = \frac{2 C}{|\mathcal{E}|} \sum_{A \in \mathcal{E}} \parenB{ P_n^{(A^{c})} - P_n } \parenj{\ERM_m^{(A^{c})}} \]
and the associated cross-validation criterion
\[ \crit_{\mathcal{E}}(m) = \frac{1}{\absj{\mathcal{E}}} \sum_{A \in \mathcal{E}} P_n^{(A)} \gamma\parenj{\ERM_m^{(A^{c})}} \enspace . \]
When $\mathcal{E} = \B$, we get the $V$-fold penalty $\penVF=\pen_{\mathcal{E}}$ and the $V$-fold cross-validation criterion $\critVF=\crit_{\mathcal{E}}$, and Eq.~\eqref{hyp.le.penVF-VFCV.gal} holds true with $p = n/V$ under assumption \eqref{hyp.part-reg.exact}. 
When $\mathcal{E} = \mathcal{E}_p \egaldef \setj{A \subset \inter{n} \telque |A|=p}$, Eq.~\eqref{hyp.le.penVF-VFCV.gal} always holds true and we get the leave-$p$-out penalty $\penLPO=\pen_{\mathcal{E}}$ and the leave-$p$-out cross-validation criterion $\critLPO=\crit_{\mathcal{E}}$.

Let $(\psil)_{\lamm}$ be some orthonormal basis of $S_m$ in $L^2(\mu)$. 
On the one hand, using Eq.~\eqref{hyp.le.penVF-VFCV.gal}, we get 
\begin{align}
\pen_{\mathcal{E}}(m,C)
&= \frac{2 C}{\absj{\mathcal{E}}} \sum_{A \in \mathcal{E}} \parenB{ P_n^{(A^{c})} - P_n } \parenj{\ERM_m^{(A^{c})}}
\notag \\
&= \frac{2 C}{\absj{\mathcal{E}}}  \sum_{A \in \mathcal{E}} \sum_{\lamm} \crochj{ \parenj{P_n^{(A^{c})}(\psil) - P_n(\psil) } P_n^{(A^{c})} (\psil)}
\notag \\
&= \frac{2 C}{\absj{\mathcal{E}}} \sum_{\lamm} \crochj{ \sum_{A \in \mathcal{E}}  \parensq{P_n^{(A^{c})}(\psil)} - P_n(\psil) \sum_{A \in \mathcal{E}}  P_n^{(A^{c})} (\psil) }
\notag \\
&= \frac{2 C}{\absj{\mathcal{E}}} \sum_{\lamm} \sum_{A \in \mathcal{E}}  \crochj{  \parensq{P_n^{(A^{c})}(\psil)} - \parensq{ P_n(\psil) } }
\label{eq.pr.le.penVF-VFCV.gal.penVF}
\enspace .
\end{align}
On the other hand, using that $P_n^{(A)} = \frac{n}{p} P_n - \frac{n-p}{p} P_n^{(A^{c})}$ by Eq.~\eqref{hyp.le.penVF-VFCV.gal}, 
\begin{align}
&\qquad \crit_{\mathcal{E}}(m) - P_n \gamma\parenj{\ERM_m} 
\notag \\
&= \frac{1}{\absj{\mathcal{E}}} \sum_{A \in \mathcal{E}} \crochj{ P_n^{(A)} \gamma\parenj{ \ERM_m^{(A^{c})}} - P_n \gamma\parenj{\ERM_m} }
\notag \\
&= \frac{1}{\absj{\mathcal{E}}} \sum_{A \in \mathcal{E}} \crochj{ \normb{ \ERM_m^{(A^{c})} }^2 - 2 P_n^{(A)} \parenj{ \ERM_m^{(A^{c})}} - \norms{\ERM_m}^2 + 2 P_n \parenj{\ERM_m} }
\notag \\
&= \frac{1}{\absj{\mathcal{E}}} \sum_{A \in \mathcal{E}} \sum_{\lamm} \crochj{ \parenB{ P_n^{(A^{c})} (\psil) }^2 - 2 P_n^{(A)} (\psil) P_n^{(A^{c})} (\psil) + \parenb{ P_n(\psil) }^2  }
\notag \\
&= \frac{1}{\absj{\mathcal{E}}} \sum_{\lamm} \sum_{A \in \mathcal{E}} \crochj{ \parenj{ \frac{2 n}{p} -1 } \parenB{ P_n^{(A^{c})} (\psil) }^2 -  \frac{2 n}{p} P_n (\psil) P_n^{(A^{c})} (\psil)  + \parenb{ P_n(\psil) }^2  }
\notag \\
&= \parenj{ \frac{2 n}{p} -1 } \frac{1}{\absj{\mathcal{E}}} \sum_{\lamm} \sum_{A \in \mathcal{E}}  \crochj{   \parenB{ P_n^{(A^{c})} (\psil) }^2 - \parenb{ P_n(\psil) }^2  }
\enspace ,
\label{eq.pr.le.penVF-VFCV.gal.VFCV}
\end{align}
where we used again Eq.~\eqref{hyp.le.penVF-VFCV.gal}. 
Comparing Eq.~\eqref{eq.pr.le.penVF-VFCV.gal.penVF} and~\eqref{eq.pr.le.penVF-VFCV.gal.VFCV} gives 
\begin{equation} 
\label{eq.le.penVF-VFCV.gal}
\crit_{\mathcal{E}}(m) = P_n \gamma\parenj{\ERM_m} +  \pen_{\mathcal{E}} \parenj{ m, \frac{n}{p} - \frac{1}{2} }
\end{equation}
which implies Eq.~\eqref{eq.le.penVF-VFCV} and~\eqref{eq.le.penLPO-LPO}. 
Eq.~\eqref{eq.le.penLOO-LPO} follows by Lemma~A.11 of \citet{Le09}.
\qed

We now prove the statements made in Remarks~\ref{rk.le.penVF-VFCV.Alain}--\ref{rk.le.penVF-VFCV.Pascal} below Lemma~\ref{le.penVF-VFCV}. 
\medbreak
\begin{proofof}{Remark~\ref{rk.le.penVF-VFCV.Alain}} 
We first note that 
Eq.~\eqref{eq.le.penLOO-LPO} can also be deduced from \citet[Proposition~2.1]{Cel:2008}, which proves  
\begin{align*}
\critLPO(m,p)= \frac{1}{n(n-p)} \sum_{\lL_{m}}\parenj{\sum_{i=1}^{n}\psil(\xi_{i})^{2} - \frac{n -p +1}{n-1} \sum_{1 \leq i\neq j \leq n}\psil(\xi_{i})\psil(\xi_{j})}\enspace.
\end{align*}
Elementary algebraic computations then show that
\begin{multline}\label{eq:penLPO}
\critLPO(m,p)-P_{n}\gamma(\ERM_{m})\\
=\frac{2n-p}{n^{2}(n-p)}\sum_{\lL_{m}}\parenj{\sum_{i=1}^{n}\psil(\xi_{i})^{2}-\frac{1}{n-1}\sum_{1 \leq i\neq j \leq n}\psil(\xi_{i})\psil(\xi_{j})}
\end{multline}
hence for any $p,p' \in \inter{n}$, 
\begin{align*}
\frac{n/p-1}{n/p-1/2}\parenb{\critLPO\parenj{ m,p}-P_{n}\gamma(\ERM_{m})}=\frac{n/p'-1}{n/p'-1/2}\parenb{\critLPO\parenj{ m,p'}-P_{n}\gamma(\ERM_{m})}\enspace. 
\end{align*}
In particular, when $p'=1$, from Eq.~\eqref{eq.le.penLPO-LPO}, since $\penLPO(m,1,C)=\penLOO(m,C)$,
\begin{align*}
 \penLPO\parenj{ m,p,\frac{n}{p}-\frac{1}{2} }&= \frac{n/p-1/2}{n/p-1}\frac{n-1}{n-1/2}\penLPO\parenj{ m,1,n-\frac{1}{2}}\\
 &=\penLOO\parenj{ m,(n-1)\frac{n/p-1/2}{n/p-1} }\enspace.
\end{align*}
\end{proofof}
\begin{proofof}{Remark~\ref{rk.le.penVF-VFCV.Pascal}}
Note first that the CV estimator of \citet[Sec. 7.2.1, p. 204--205]{Mas:2003:St-Flour} is defined as the minimizer of 
\begin{multline}\label{eq:CVEst}
\norms{\ERM_m}^2-\frac{2}{n(n-1)} \sum_{1 \leq i\neq j \leq n} \sum_{\lL_m}\psil(\xi_i)\psil(\xi_j)\\
= P_n\gamma\parenj{\ERM_m} + \frac{2}{n^2} \sum_{\lL_m} \parenj{\sum_{i=1}^n\psil(\xi_i)^2 - \frac{1}{n-1}\sum_{1 \leq i\neq j \leq n}\psil(\xi_i)\psil(\xi_j)}\enspace. 
\end{multline}
On the other hand, from Eq.~\eqref{eq:penLPO} and \eqref{eq.le.penLPO-LPO} with $p=1$, we have
\begin{align*}
\penLOO\parenj{m,n-1} 
&=\frac{2}{n^2}\sum_{\lL_{m}}\parenj{\sum_{i=1}^{n}\psil(\xi_{i})^{2}-\frac{1}{n-1}\sum_{1 \leq i\neq j \leq n}\psil(\xi_{i})\psil(\xi_{j})}
\enspace. 
\end{align*}
Hence, from Eq.~\eqref{eq:CVEst}, the CV estimator is the minimizer of $\critcorrVF\parens{m,\B_{\mathrm{LOO}}}$. 
\citet[Theorem 7.6]{Mas:2003:St-Flour} studies the minimizers of the criterion
\begin{equation} 
\label{eq.crit.MasThm7.6}
P_n\gamma(\ERM_m)+\frac{C}{n^2}\sum_{i=1}^n \sum_{\lL_m} \psil(\xi_i)^2 
\enspace,\end{equation}
where $C=(1+\epsilon)^6$ for any $\epsilon>0$. 
Let $\alpha = C/n$, so that $\alpha=(C-\alpha)/(n-1)$.  
Then, the criterion \eqref{eq.crit.MasThm7.6} is equal to
\begin{align*}
& \quad 
(1-\alpha) 
P_n\gamma(\ERM_m) + \frac{C-\alpha}{n^2}\sum_{\lL_{m}}\sum_{i=1}^n\psil(\xi_i)^2-\frac{\alpha}{n^2}\sum_{\lL_{m}}\sum_{1 \leq i\neq j \leq n} \psil(\xi_i)\psil(\xi_j)
\\
&= 
(1-\alpha) P_n\gamma(\ERM_m) 
+ \frac{C-\alpha}{n^2} \sum_{\lL_{m}} \parenj{\sum_{i=1}^n\psil(\xi_i)^2-\frac{1}{n-1}\sum_{\lL_{m}}\sum_{1 \leq i\neq j \leq n} \psil(\xi_i)\psil(\xi_j)}
\\
&= 
(1-\alpha) \crochj{ P_n\gamma(\ERM_m)+\frac{C-\alpha}{2 (1-\alpha)}\penLOO\parenj{m,n-1} }
\\
&= 
(1-\alpha) \crochj{ P_n\gamma(\ERM_m)+\penLOO\parenj{m,\frac{C (n-1)^2}{2(n-C)}} }
\enspace.
\end{align*}
\end{proofof}

\subsection{Proof of Proposition~\ref{prop:concpenvf}}\label{sect:proofthmconcpenvf}
Note that the two formulas given for $\Psi_m$ in the statement of Proposition~\ref{prop:concpenvf} coincide by Eq.~\eqref{eq.sup-CS}. 
The proof is decomposed into 3 lemmas.
\begin{lemma}\label{lem:exact.formula.penvf}
Let $\xi_{\inter{n}}$ denote i.i.d. random variables taking value in a Polish space $\X$, $\B_{\inter{V}}$ some partition of $\inter{n}$ satisfying \eqref{hyp.part-reg.exact}, $S_m$ some separable linear subspace of $L^2(\mu)$ with orthonormal basis $(\psil)_{\lL_m}$ and 
\begin{align} 
\label{def:Um}
U(m) &\egaldef \frac{1}{n^{2}}\sum_{1 \leq k \neq k' \leq V} \sum_{i\in \B_{k}} \sum_{j\in \B_{k'}} \sum_{\lL_m} \parenb{ \psil(\xi_{i})-P\psil } \parenb{ \psil(\xi_{j})-P\psil }
\enspace . 
\end{align}
Then, the $V$-fold penalty is equal to 
\begin{align} 
\label{eq.lem:exact.formula.penvf}
\penVF(m,\B,C) &= 
\frac{2C}{V-1} \norms{\bayes_m-\ERM_m}^2 - \frac{2VC}{(V-1)^{2}} U(m)
\\
\hspace{-2cm}
\text{and} \quad 
\E\crochj{\penVF\parenj{m,\B,\frac{V-1}{2}}}
&=\E\crochj{ \norms{\bayes_m-\ERM_m}^2 } 
= \frac{\Dcal_m}{2 n}
\label{eq.EpenVF.2}
\enspace. 
\end{align}
\end{lemma}
\begin{proof}
Let $W_{i}=\frac{V}{V-1}\un_{i\notin \B_{J}}$ and use the formulation \eqref{def.pen.Res} of the $V$-fold penalty as a resampling penalty. 
Then,
\begin{align}
\notag\hspace{-0.75cm} 
\penVF(m,\B,C) &= C \, \E_{W}\crochB{ \parenb{ P_{n}-P_{n}^{W} } \parenB{ \gamma \parenb{ \ERMW_m }} }\\
\notag &= 2C\, \E_{W}\crochB{ \parenb{ P_{n}^{W}-P_{n} } \parenb{ \ERMW_m }}\\
\notag &= 2C\, \E_{W}\crochB{ \parenb{ P_{n}^{W}-P_{n} } \parenb{ \ERMW_m-\ERM_{m}} }
\qquad \text{by \eqref{hyp.part-reg.exact}} \\
\notag &= 2C \sum_{\lL_m}\E_{W}\crochj{\parensqB{ \parenb{ P_{n}^{W}-P_{n} }(\psil)}}\\
\notag &= 2C \sum_{\lL_m}\E_{W}\crochj{\parensqB{ \parenb{ P_{n}^{W}-P_{n} } (\psil-P\psil)}}\\
\label{eq:penvf.poids}
&= \frac{2C}{n^{2}} \sum_{\lL_m}\sum_{1 \leq i,j \leq n}
e_{i,j}^{(\VF)} \parenb{ \psil(\xi_{i})-P\psil } \parenb{ \psil(\xi_{j})-P\psil } 
\end{align}
where $e_{i,j}^{(\VF)} \egaldef \E\crochj{(W_{i}-1)(W_{j}-1)}$. 
Since $\E[W_i]=1$ by \eqref{hyp.part-reg.exact} and 
\[ 
W_{i} W_{j} = \parensq{ \frac{V}{V-1} } \un_{J\notin\setj{J_{0},J_{1}}}
\qquad 
\text{if} 
\quad i \in \B_{J_0} \quad \text{and} \quad j \in \B_{J_1}
\enspace , 
\]
we get that $e_{i,j}^{(\VF)} = (V-1)^{-1}$ if $i$ and $j$ belong to the same block and $e_{i,j}^{(\VF)} = - (V-1)^{-2}$ otherwise. 
So, 
\begin{align}
\notag
& \qquad \penVF(m,\B,C) 
\\
\notag
&=\frac{2C}{n^{2}(V-1)}
\sum_{\lL_m}\sum_{k=1}^{V}\sum_{(i,j)\in \B_{k}} \parenb{ \psil(\xi_{i})-P\psil } \parenb{ \psil(\xi_{j})-P\psil }
-\frac{2C}{(V-1)^{2}} U(m)
\\
\notag
&=\frac{2C}{V-1}\sum_{\lL_m} \parenb{(P_n-P)\psil}^2
- \frac{2C V}{(V-1)^{2}} U(m)
\end{align}
and Eq.~\eqref{eq.lem:exact.formula.penvf} follows by Eq.~\eqref{eq.p1.4formules}. 
Eq.~\eqref{eq.EpenVF.2} directly follows from Eq.~\eqref{eq.lem:exact.formula.penvf}. 
\end{proof}

\begin{lemma}\label{lem:conc.supPn-P}
Let $\xi_{\inter{n}}$ be i.i.d. random variables taking values in a Polish space $\X$ with common density $\bayes\in L^{\infty}(\mu)$, 
$S_m$ a separable linear subspace of $L^2(\mu)$ and denote by $(\psil)_{\lL_m}$ an orthonormal basis of $S_m$. 
Let $\Boule_m=\setj{t\in S_m \telque \norms{t}\leq 1}$, 
$\Dcal_{m}=\sum_{\lL_m} P\parenj{\psil^{2}} - \norms{\bayes_{m}}^{2}$ 
and assume that $b_{m}=\sup_{t\in\Boule_m} \norms{t}_{\infty}<\infty$. 
An absolute constant $\kappa$ exists such that, for any $x>0$, with probability larger than $1-2\e^{-x}$, we have for every $\epsilon >0$, 
\[
\absj{  \norms{\bayes_m-\ERM_m}^2 - \frac{\Dcal_{m}}{n} }
\leq \epsilon\frac{\Dcal_{m}}{n} + \kappa\parenj{\frac{\norms{\bayes}_{\infty} x}{(\epsilon \wedge 1) n} + \frac{b_{m}^2 x^2}{(\epsilon \wedge 1)^3 n^2}}
\enspace.\]
\end{lemma}
\begin{proof}
By Eq.~\eqref{eq.p1.4formules}, 
$ \norms{\bayes_m-\ERM_m}^2 = \sup_{t\in \Boule_{m}} \crochsq{ (P_{n}-P)(t)} $ has expectation 
$\Dcal_m/n$. 
In addition, for any $t\in \Boule_{m}$, 
\begin{equation} \label{eq:cont.vm}
\var \parenb{ t(\xi_1)}\leq \int_{\R}t^{2}\bayes \dd\mu\leq \norms{\bayes}_{\infty}\norms{t}^{2}
\leq \norms{\bayes}_{\infty} 
\enspace ,
\end{equation}
which gives the conclusion thanks to Proposition~\ref{prop:concLe09} in Section~\ref{sec.supmat.proba-tools}. 
\end{proof}

\begin{lemma}\label{lem:concUm}
Assume that $\xi_{\inter{n}}$ is a sequence of i.i.d. real-valued random variables with common density $\bayes\in  L^{\infty}(\mu)$ and $\B_{\inter{V}}$ is some partition of $\inter{n}$ satisfying \eqref{hyp.part-reg.exact}. 
Let $S_m$ denote a separable subspace of $L^2(\mu)$ with orthonormal basis $(\psil)_{\lL_m}$ such that 
\[ 
b_{m} \egaldef \sup_{t\in S_m, \norms{t}\leq 1} \norms{t}_{\infty} < +\infty
\enspace . 
\]
Let $U(m)$ be the $U$-statistics defined by Eq.~\eqref{def:Um}. 
Using the notations of Lemma~\ref{lem:conc.supPn-P}, an absolute constant $\kappa$ exists such that, 
with probability larger than $1-6\e^{-x}$, 
\begin{align*}
\absb{U(m)}
& \leq \frac{3 \sqrt{(V-1)\norms{\bayes}_\infty\Dcal_{m}x}}{\sqrt{V}n} 
+ \kappa \parenj{ 
\frac{\norms{\bayes}_{\infty} x}{n} 
+ \frac{ \parenb{ b_m^2+\norms{\bayes}^2 } x^2}{n^2} }
\enspace .
\end{align*}
Hence, an absolute constant $\kappa'$ exists such that, for any $x>0$, with probability larger than $1-6\e^{-x}$, for any $\theta \in (0,1]$,
\begin{equation*}
 \absb{U(m)} \leq \theta \frac{\Dcal_{m}}{n} 
 + \kappa' \parenj{\frac{\norms{\bayes}_{\infty} x}{\theta n}
 + \frac{\parenb{ b_m^2+\norms{\bayes}^2 } x^2}{n^2}}
 \enspace.
\end{equation*}
\end{lemma}
\begin{proof}
For any $x,y \in \R$ and $i,j \in \inter{n}$, let us define 
\begin{align*}
U_m(x,y) &= \sum_{\lambda\in\Lambda_m} \parenb{ \psil(x)-P\psil } \parenb{ \psil(y)-P\psil }
\\
\text{and} \qquad 
g_{i,j}(x,y) &= U_m(x,y) \un_{\sets{ \exists k,k' \in \inter{V} \telque k\neq k', \, i \in \B_k, \, j \in \B_{k'} }}
\\
\text{so that} \qquad  
U(m) &=\frac{2}{n^{2}} \sum_{i=2}^{n} \sum_{j=1}^{i-1} g_{i,j}(\xi_{i},\xi_{j})
= \frac{2}{n^{2}} \sum_{k=2}^{V}\sum_{k'=1}^{k-1}\sum_{i\in \B_{k},j\in \B_{k'}} U_m(\xi_i,\xi_j)
\enspace .
\end{align*}
From \citet[Theorem~3.4]{Hou_Rey:2003}, an absolute constant $\kappa$ exists such that, for any $x>0$ and $\epsilon \in (0,1]$,
\begin{equation} \label{eq:Conc.U.Stat}
\P\parenj{\absj{U(m)}\geq \frac{1}{n^{2}} 
\crochj{(4+\epsilon)\oA\sqrt{x} + \kappa\parenj{\frac{\oB x}{\epsilon} 
+ \frac{\oC x^{3/2}}{\epsilon^{3}} + \frac{\oD x^{2}}{\epsilon^{3}} } } } 
\leq 6\e^{-x}
\enspace . 
\end{equation}
\begin{align*}
 \oA^{2}&=\sum_{i=2}^{n}\sum_{j=1}^{i-1}\E\crochj{g_{i,j}(\xi_{i},\xi_{j})^{2}}
 \enspace,\\
 \oB &=\sup \mathopen{}\left\{ \E\crochj{\sum_{i=2}^{n}\sum_{j=1}^{i-1}a_{i}(\xi_{i})b_{j}(\xi_{j})g_{i,j}(\xi_{i},\xi_{j})}
\right. \mathclose{} \\
&\hspace{2cm} \mathopen{}\left. 
\text{such that} \qquad 
\E\crochj{\sum_{i=1}^{n}a_{i}^{2}(\xi_{i})}\leq 1
\quad \text{and} \quad 
\E\crochj{\sum_{i=1}^{n}b_{i}^{2}(\xi_{i})}\leq 1
\right\} \mathclose{}
\enspace,\\
 \oC^{2} &= 
\sup_{x \in \R} \setj{ \sum_{i=2}^{n} \E\crochj{g_{i,1}(\xi_{i},x)^{2}} } 
\quad 
\text{and} \quad 
\oD =\sup_{x,y}\absb{g_{i,j}(x,y)}
\enspace.
\end{align*}
It remains to upper bound these different terms for proving the first inequality, and the second inequality follows. 
First, 
\begin{align}
\notag 
\E\crochB{ U_m(\xi_1,\xi_2)^2 }
&= 
\sum_{\lambda\in\Lambda_m,\lambda^{\prime}\in \Lambda_m}\E\crochsqB{ \parenb{ \psil(\xi_{1})-P\psil } \parenb{ \psilp(\xi_{1})-P\psilp } }
\\
\notag 
&= 
\sum_{\lambda\in\Lambda_m} \mathopen{}\left( \sup_{\sum_{\lambda^{\prime}\in \Lambda_m}a_{\lambda^{\prime}}^2\leq 1}\E\crochBb{\parenb{ \psil(\xi_{1})-P\psil }\sum_{\lambda^{\prime}\in \Lambda_m}a_{\lambda^{\prime}} \parenb{ \psilp(\xi_{1})-P\psilp } } \right)^2 \mathclose{}
\\
\notag 
&= 
\sum_{\lambda\in\Lambda_m} \mathopen{}\left( {\sup_{t\in\Boule_m}\E\crochB{\parenb{ \psil(\xi_{1})-P\psil } \parenb{ t(\xi_{1})-P(t) }}} \right)^2 \mathclose{}
\\
\notag 
&\leq  
\Dcal_{m} \sup_{t \in \Boule_m} \E\crochj{ \parenb{t(\xi_1) - P(t)}^2 } 
\\
\label{eq.maj-E[Um^2]}
&\leq 
\norms{\bayes}_\infty\Dcal_{m} \qquad \qquad \text{by Eq.~\eqref{eq:cont.vm} } 
\end{align}
so that  
\begin{align*}
\oA^{2} 
&= \sum_{k=2}^{V}\sum_{k'=1}^{k-1}\sum_{i\in \B_{k},j\in \B_{k'}}
\E\crochB{ U_m(\xi_i,\xi_j)^2 }
\leq \frac{n^{2}(V-1)}{2V}
\times \norms{\bayes}_\infty\Dcal_{m}
\enspace.
\end{align*}
Second, let $a_1, \ldots, a_n, b_1, \ldots, b_n$ be functions in $L^2(\mu)$ such that
\[
\E\crochj{\sum_{i=1}^{n}a_{i}^{2}(\xi_{i})}\leq 1 
\qquad \text{and} \qquad 
\E\crochj{\sum_{i=1}^{n}b_{i}^{2}(\xi_{i})}\leq 1
\enspace.\]
Using successively the independence of the $\xi_i$ and that $\alpha \beta \leq (\alpha^2 + \beta^2)/2$ for every $\alpha,\beta\in \R$, 
for every $i \neq j$, 
\begin{align}
\notag 
&\qquad \absB{\E\crochb{ a_i(\xi_i) b_j(\xi_j) U_m(\xi_i,\xi_j) }}
\\ %%% changement ici pour mise en page
\notag 
&= 
\absBb{\sum_{\lambda\in\Lambda_m} 
\E \crochB{ a_{i}(\xi_{i}) \parenb{ \psil(\xi_{i})-P\psil } } \E \crochB{ b_{j}(\xi_{j})\parenb{ \psil(\xi_{j})-P\psil } }}
\\ 
&\le \frac{1}{2} 
\sum_{\lambda\in\Lambda_m} 
\parenbb{ \E \crochB{ a_{i}(\xi_{i}) \parenb{ \psil(\xi_{i})-P\psil } }^2 
+ \E \crochB{ b_{j}(\xi_{j})\parenb{ \psil(\xi_{j})-P\psil } }^2 }
\enspace . 
\label{eq:BorneB}
\end{align}
Now, we have, for every $i \in \inter{n}$, using Eq.~\eqref{eq.sup-CS}, Cauchy-Schwarz inequality and the fact that for every $t \in L^2(\mu)$, $\var \parenj{ t(\xi_1)} \leq \norms{\bayes}_{\infty} \norms{t}^2$, 
\begin{align*}
\sum_{\lambda\in\Lambda_m} \E\crochsqB{a_{i}(\xi_{i}) \parenb{ \psil(\xi_{i})-P\psil }}
&= \sup_{\sum_{\lL_m} t_\lambda^2 \le 1} \parensq{ \E\crochj{ a_{i} ( \xi_{i} ) \sum_{\lambda\in\Lambda_m} t_\lambda \psil(\xi_{i}) - P(t_\lambda\psil) } } \\
&=\sup_{t\in\Boule_m} \parensqbb{ \E\crochB{a_{i}(\xi_{i}) \parenb{ t(\xi_{i})-P(t) } } }  \\
&\le \E\crochb{ a_{i}(\xi_{i})^2 } \sup_{t\in\Boule_m}\var \parenb{ t(\xi_1)} 
\le \E\crochb{ a_{i}(\xi_{i})^2 } \norms{\bayes}_\infty
\enspace.
\end{align*}
Plugging this bound in \eqref{eq:BorneB} yields
\begin{equation}
\absB{ \E\crochb{ a_i(\xi_i) b_j(\xi_j) U_m(\xi_i,\xi_j) }}
\leq \frac{\norms{\bayes}_{\infty}}{2} \parenB{ \E\crochb{ a_i(\xi_i)^2 } + \E\crochb{ b_j(\xi_j)^2 } }
\label{eq.maj-terme-B-Um}
\end{equation}
hence 
\begin{equation} \notag %\label{eq:contB}
 \oB \leq n\norms{\bayes}_\infty
 \enspace.
\end{equation}
Third, for every $x,y \in \R$, let $g_{x}(y)=\sum_{\lambda\in\Lambda_m} \parens{ \psil(x)-P\psil } \psil(y)$ so that 
\begin{align*}
\norms{g_{x}}^{2}
=\sum_{\lambda\in\Lambda_m}\parenb{ \psil(x)-P\psil }^{2}
&\leq 2 \sum_{\lambda\in\Lambda_m} \parenb{ \psil(x) }^2 + 2 \sum_{\lambda\in\Lambda_m} (P\psil)^{2}
\\
&= 2 \Psi_m(x)^2 + 2 \norms{\bayes_m}^2 
\leq 2  \parenB{ b_{m}^{2}+\norms{\bayes_{m}}^{2} }
\enspace.
\end{align*}
Then, 
\begin{align}
\E \crochj{ U_m(\xi_i,x)^{2} }
= \var \parenb{ g_{x}(\xi_{1})} 
\leq \norms{ g_x }^2 \norms{\bayes}_{\infty}
&\leq 2  \parenB{ b_{m}^{2}+\norms{\bayes_{m}}^{2} } \norms{\bayes}_{\infty}
\label{eq.maj-terme-C-Um}
\end{align}
and, using \eqref{hyp.part-reg.exact}, we get that 
\begin{equation} \notag %\label{eq:contC}
 \oC^{2} \leq \frac{2 n (V-1)}{V} \parenB{b_{m}^2+\norms{\bayes_{m}}^{2}} \norms{\bayes}_{\infty} 
 \enspace.
\end{equation}
Fourth, from Cauchy-Schwarz inequality, 
for every $x,y\in \X$, 
\begin{equation}
\label{eq.maj-sup-Um}
U_m(x,y) 
\leq \sup_{x\in \R}\sum_{\lambda\in\Lambda_m} \parenb{\psil(x)-P\psil}^{2}
\leq 2 \parenB{ b_{m}^2+\norms{\bayes_{m}}^2 }
\enspace.
\end{equation}
Hence,
\begin{equation} \notag %\label{eq:contD}
 \oD \leq 2 \parenB{ b_{m}^2+\norms{\bayes_{m}}^{2} }
\end{equation}
and we get the desired result. 
\end{proof}

Let us conclude the proof of Proposition~\ref{prop:concpenvf}. 
From Lemmas \ref{lem:exact.formula.penvf} and~\ref{lem:concUm}, an absolute constant $\kappa$ exists such that, with probability larger than $1-6\e^{-x}$, for every $\epsilon \in (0,1]$, 
\begin{align} 
\notag 
&\quad \absB{ \penVF(m,V,V-1) - 2 \norms{\bayes_m-\ERM_m}^2}
\\
\label{eq:conc.penvf.exp:pr}
&= \frac{2V}{V-1} \absb{U(m)} 
\leq \epsilon\frac{\Dcal_{m}}{n} 
+ \kappa \parenj{\frac{\norms{\bayes}_{\infty} x}{\epsilon n}
+\frac{\parenb{ b_m^2+\norms{\bayes}^2 } x^2}{n^2}}
\enspace. 
\end{align}
Using in addition Lemma~\ref{lem:conc.supPn-P}, we get that an absolute constant $\kappa^{\prime}$ exists such that with probability larger than $1-8\e^{-x}$, for every $\epsilon \in (0,1]$, Eq.~\eqref{eq:conc.penvf.exp:pr} holds true and 
\begin{equation*} 
\absj{ \penVF(m,V,V-1) - \frac{2 \Dcal_{m}}{n} }
\leq \epsilon\frac{\Dcal_{m}}{n} 
+ \kappa \parenj{\frac{\norms{\bayes}_{\infty} x}{\epsilon n}
+\frac{ \parenb{ b_m^2 \epsilon^{-3} + \norms{\bayes}^2 } x^2}{n^2}}
\enspace, 
\end{equation*}
which implies Eq.~\eqref{eq:conc.penvf.exp} and~\eqref{eq:concpenvf.penid}. 
\qed

\subsection{Proof of Theorem~\ref{thm.oracle-penVF.cas_reel}} \label{sec.app.proof.thm-oracle.histo}
By construction, the penalized estimator satisfies, for any $m\in \M_n$,
\begin{equation*} %%\label{eq:rel.chosen.oracle}
\begin{split}
\perte{\ERM_{\mh}}-\parenB{\penid(\mh)-\penVF \parenb{ \mh,V,C(V-1) } }
\\
\leq \perte{\ERM_{m}}+\parenB{ \penVF \parenb{ m,V,C(V-1) } - \penid(m) }
\enspace. 
\end{split}
\end{equation*}
Now, by Eq.~\eqref{def.penid} and~\eqref{eq.p1.4formules}, 
$\penid(m) = 2 \norms{\ERM_m - \bayes_m}^2 + 2(P_{n}-P)(\bayes_{m})$, hence 
\begin{align}
\notag \perte{ \ERM_{\mh} }
&\leq \perte{ \ERM_{m} } 
+ \crochB{ \penVF \parenb{ m,V,C(V-1) } - 2 \norms{\bayes_m-\ERM_m}^2 }
\\
 \notag& \qquad -\crochB{ \penVF \parenb{ \mh,V,C(V-1) } 
 - 2 \norms{\bayes_{\mh}-\ERM_{\mh}}^2 }
 + 2(P_{n}-P)(\bayes_{m}-\bayes_{\mh}) 
 \\
\notag &= \perte{ \ERM_{m} } + \crochB{ \penVF \parenb{ m,V,C(V-1) } 
- 2C \norms{\bayes_{m}-\ERM_{m} }^2 } 
\\
 \notag& \qquad -\crochB{ \penVF \parenb{ \mh,V,C(V-1) } 
 - 2C \norms{\bayes_{\mh}-\ERM_{\mh}}^2  } 
 + 2(P_{n}-P)(\bayes_{m}-\bayes_{\mh}) 
 \\
 \label{eq:int.oracle1}
 &\qquad + 2 \parens{ C-1 } \parenB{ \norms{\ERM_m-\bayes_m}^{2} - \norms{\ERM_{\mh}-\bayes_{\mh} }^{2} }
 \enspace.
\end{align}
Let $x>0$ and $x_{n}=\log(\absj{\M_n})+x$. 
A union bound in Proposition~\ref{prop:concpenvf} gives 
\begin{equation}
\label{eq:cont.unif.penvf}
\begin{split}
&\P 
\Bigl( 
\exists m\in \M_n, \, \epsilon \in (0,1] \, \telque 
\absj{\penVF(m,V,V-1)-2 \norms{\bayes_m-\ERM_m}^2 } 
\\
&\qquad \qquad > 
\epsilon\frac{\Dcal_m}{n}
+\kappa \resteA(m,\epsilon,\bayes,x_n,n)
\biggr) 
\leq 8\sum_{m\in \M_n}\e^{-x_{n}}
=8\e^{-x}\sum_{m\in\M_n}\frac{1}{\absj{\M_n}}
=8\e^{-x}
\end{split}
\end{equation}
and a union bound in Lemma~\ref{lem:conc.supPn-P} gives
\begin{align}\label{eq:cont.unif.Pn-P}
 \notag&\P\parenj{
 \exists m\in \M_n, \, \epsilon \in (0,1] \, \telque 
\absj{\norms{\ERM_m-\bayes_m}^{2}-\frac{\Dcal_{m}}{n}} 
> \epsilon\frac{\Dcal_m}n+\kappa \resteA(m,\epsilon,\bayes,x_n,n)}\\
 &\leq 2 \sum_{m\in\M_n}\e^{-x_{n}}
 = 2 \e^{-x}
 \enspace.
\end{align}

It remains to bound $2(P_n-P)(\bayes_{m}-\bayes_{m^{\prime}})$ uniformly over $m$ and $m^{\prime}$ in $\M_n$. 
In order to apply Bernstein's inequality, we first bound the variance and the sup norm of $\bayes_{m}-\bayes_{m^{\prime}}$ for some $m, m^{\prime} \in \M_n$. 
Since $\bayes \in L^{\infty}(\mu)$, 
\[
\var \parenb{ (\bayes_{m}-\bayes_{m^{\prime}})(\xi_1)}\leq \norms{\bayes}_{\infty}\norms{\bayes_{m}-\bayes_{m^{\prime}}}^2
\enspace.\]
Under assumption~\eqref{hyp.UBbayes}
\[
\norms{\bayes_{m}-\bayes_{m^{\prime}}}_{\infty}\leq \norms{\bayes_{m}}_{\infty}+\norms{\bayes_{m^{\prime}}}_{\infty}\leq 2a
\enspace.\]
Under assumption~\eqref{hyp.Nested}, $\bayes_{m}-\bayes_{m^{\prime}} \in S_{m^{\prime\prime}}$ for some $m^{\prime\prime} \in \{m,m'\}$, hence by \eqref{hyp.NormSupNorm2} we have
\[
\norms{\bayes_{m}-\bayes_{m^{\prime}}}_{\infty}
\leq b_{m^{\prime\prime}} \norms{\bayes_{m}-\bayes_{m^{\prime}}} 
\leq \sqrt{n}\norms{\bayes_{m}-\bayes_{m^{\prime}}}
\enspace.\]
Therefore, by Bernstein's inequality, for any $x>0$, for any $m,m^{\prime}$, with probability larger than $1-\e^{-x}$, for any $\epsilon \in (0,1]$, 
\begin{align*}
(P_n-P)(\bayes_{m}-\bayes_{m^{\prime}})
&\leq \sqrt{ \frac{2 x \var \parenb{ (\bayes_{m}-\bayes_{m^{\prime}})(\xi_1)} }{n} } 
+ \frac{\norms{\bayes_{m}-\bayes_{m^{\prime}}}_{\infty}x}{3n}
\\
&\leq \epsilon\norms{\bayes_{m}-\bayes_{m^{\prime}}}^2 + \frac{\kappa \parenb{ Ax + x^2 }}{\epsilon n}
\enspace. 
\end{align*}
for some absolute constant $\kappa$, where the last inequality is obtained by considering separately the cases \eqref{hyp.UBbayes} and \eqref{hyp.Nested}, and by using that for every $\alpha,\beta,\epsilon>0$, $\alpha \beta \leq \epsilon \alpha^2 + (\beta^2)/(4\epsilon)$. 
A union bound gives that for any $x>0$, with probability at least $1-|\M_n|^2 \e^{-x}$, 
for every $m,m' \in \M_n$ and every $\epsilon \in (0,1]$, 
\begin{equation}\label{eq:termes.croises}
(P_n-P)(\bayes_{m}-\bayes_{m^{\prime}}) \leq \epsilon\norms{\bayes_{m}-\bayes_{m^{\prime}}}^2+\frac{\kappa  \parenb{ A x + x^2 } }{\epsilon n}
\end{equation}
for some absolute constant $\kappa$. 
Plugging Eq.~\eqref{eq:cont.unif.penvf}, \eqref{eq:cont.unif.Pn-P} and~\eqref{eq:termes.croises} into Eq.~\eqref{eq:int.oracle1} and using that $C \in (1/2,2]$ yields that, 
with probability $1-(|\M_n|^2  + 10) \e^{-x}$, 
for any $\epsilon\in (0,1/2]$, 
\begin{align*}
\notag
(1-4\epsilon)\perte{ \ERM_{\mh} }
&\leq (1+4\epsilon)\perte{ \ERM_{m} }
+\parenj{\delta_{+} + 4 \epsilon}\frac{\Dcal_{m}}{n} 
+\parenj{\delta_{-}+ 3\epsilon} \frac{\Dcal_{\mh}}{n} 
\\
\notag&
\quad + \kappa\parenj{\resteA(m,\epsilon,\bayes,x,n) 
+\resteA(\mh,\epsilon,\bayes,x,n) 
+ \frac{A x + x^2}{\epsilon n}} 
 \\
\notag
&\leq (1+\delta_{+}+16\epsilon) \perte{ \ERM_{m} }
+\parenj{\delta_{-}+8\epsilon} \norms{\ERM_{\mh}-\bayes_{m}}^{2} 
\\
&\quad
+\kappa^{\prime} \parenj{
\resteA(m,\epsilon,\bayes,x,n) 
+ \resteA(\mh,\epsilon,\bayes,x,n)
+ \frac{A x + x^2}{\epsilon n} } 
\end{align*}
for some absolute constants $\kappa, \kappa^{\prime}>0$. 
Since $b_m \leq \sqrt{n}$ for all $\mM_n$, we get 
\[ 
2 \sup_{\mM_n} \resteA(m,\epsilon,\bayes,x,n) 
+ \frac{A x + x^2}{\epsilon n}
\leq 
\frac{ \parenb{ 2\norms{\bayes}_{\infty} + A } x}{\epsilon n} + \parenj{ 3 + \frac{2 \norms{\bayes}^2}{n} } \frac{x^2}{\epsilon^3 n} 
\]
for every $\epsilon \in (0,1]$. 
Hence, 
with probability larger than $1-(|\M_n|^2 + 10) \e^{-x}$, 
for any $\epsilon\in (0,1]$, 
\begin{equation} \label{pr.eq.thm.oracle-penVF.cas_reel.fin}
\frac{1-\delta_{-}-\epsilon}{1+\delta_{+}+\epsilon}\perte{ \ERM_{\mh} }
\leq \perte{ \ERM_{m} } + 
\kappa 
\crochj{ \frac{ \parenb{ \norms{\bayes}_{\infty} + A } x}{\epsilon n} + \parenj{ 1 + \frac{ \norms{\bayes}^2}{n} } \frac{x^2}{\epsilon^3 n} }
\end{equation}
for some absolute constant $\kappa>0$. 
To conclude, we remark that Eq.~\eqref{eq.thm.oracle-penVF.cas_reel} clearly holds true when $|\M_n|=1$, so we can assume that $|\M_n| \geq 2$. 
Therefore, for every $x>0$, Eq.~\eqref{pr.eq.thm.oracle-penVF.cas_reel.fin} holds true 
with probability at least 
\[ 
1-\parenB{ |\M_n|^2 + 10 } \e^{-x} \geq 
1 - |\M_n|^4 \e^{-x} \geq 1 - \e^{ -  x + 4 \log|\M_n|  }
\enspace . 
\]
So, if we replace $x$ by $4 x_n \geq x + 4 \log|\M_n| $ in Eq.~\eqref{pr.eq.thm.oracle-penVF.cas_reel.fin}, 
we get that 
Eq.~\eqref{eq.thm.oracle-penVF.cas_reel} holds true with probability at least $1 - \e^{-x}$ 
for some absolute constant $\kappa>0$, 
slightly larger than the one appearing in Eq.~\eqref{pr.eq.thm.oracle-penVF.cas_reel.fin}.
\qed

\subsection{Proof of Theorem~\ref{theo.variance.penVF}} \label{sec.proof.variance.main}
For every $x,y \in \X$ and $m \in \{m_1,m_2\}$, let $K_m(x,y)\egaldef \sum_{\lL_m}\psil(x)\psil(y)$ and remark that 
\begin{align}
\notag 
 U_m(x,y)&= \sum_{\lL_m} \parenb{ \psil(x)-P\psil } \parenb{ \psil(y)-P\psil } \\
 &=K_m(x,y)-\bayes_m(x)-\bayes_m(y)+\norms{\bayes_m}^2
\label{eq.Um}
 \enspace.
\end{align}
For every $x \in \X$, $K_m(x,x)=\Psi_m(x)$ by Eq.~\eqref{eq.sup-CS}, $U_m(x,x)=\Psi_m(x)-2\bayes_m(x)+\norms{\bayes_m}^2$ and, by independence, for every $m,m' \in \{m_1,m_2\}$
\begin{align*}
\cov&\parenb{U_m(\xi_1,\xi_2),U_{m'}(\xi_1,\xi_2)}\\
&=\sum_{\lL_m,\lpL_{m'}}\E\crochB{\parenb{\psil(\xi_1)-P\psil}\parenb{\psil(\xi_2)-P\psil}\parenb{\psilp(\xi_1)-P\psilp}\parenb{\psilp(\xi_2)-P\psilp}}\\
&=\sum_{\lL_m,\lpL_{m'}}\E\crochsqB{\parenb{\psil(\xi_1)-P\psil}\parenb{\psilp(\xi_1)-P\psilp}}= \termeBvar{m , m^{\prime}}\enspace, 
\end{align*}
hence, $\var \parenj{ U_{m_1}(\xi_1,\xi_2)-U_{m_2}(\xi_1,\xi_2)}=\termeBvaracr{m_1,m_2}$.  
For every $m \in \{m_1,m_2\}$, by Eq.~\eqref{eq.Um}, 
\begin{align}\label{eq:DecRiskEmp}
 P_n\gamma(\ERM_m)
 &=-\sum_{\lL_m}(P_n\psil)^2=-\frac{1}{n^2}\sum_{1\leq i,j\leq n}K_m(\xi_i,\xi_j)\\
 \notag 
 &=-\frac{1}{n^2}\sum_{1\leq i,j\leq n}U_m(\xi_i,\xi_j)-\frac{2}n\sum_{i=1}^n\bayes_m(\xi_i)+\norms{\bayes_m}^2
 \enspace.
\end{align}
Moreover, by Eq.~\eqref{eq:penvf.poids} in the proof of Lemma~\ref{lem:exact.formula.penvf}, 
\begin{gather*}
\penVF \parenb{ m,\B,C(V-1) } = \frac{2C}{n^2} \sum_{1\leq i,j\leq n} E_{i,j}^{(\VF)} U_m(\xi_i,\xi_j)
\\
\text{where} \quad 
\forall I,J \in \setj{1, \ldots, V}, \, \forall i \in B_I , \, \forall j \in B_J, \quad 
E_{i,j}^{(\VF)} = 1 - \frac{V \un_{I \neq J}}{V-1} = (V-1) e_{i,j}^{(\VF)} 
\enspace . 
\end{gather*}
It follows that
\begin{equation} \label{eq.crit-penVF-gal.U-stat.1}
\CV_{C,\B}(m) = 
\sum_{1\leq i,j\leq n} \frac{2CE_{i,j}^{(\VF)}-1}{n^2} U_m(\xi_i,\xi_j)
+ \sum_{i=1}^n \frac{-2\bayes_m(\xi_i)}{n} 
+ \norms{\bayes_m}^2
\enspace.
\end{equation}
Hence, up to the deterministic term $\norms{\bayes_m}^2$, 
$\CV_{C,\B}(m)$ has the form of a function $\CV_m$ defined in Lemma~\ref{lem:CovGen.2} below with  
\[
\olomega_{i,j}=\frac{2CE_{i,j}^{(\VF)}-1}{n^2} 
\enspace ,
\qquad f_m=\frac{-2\bayes_m}{n}
\qquad \text{and} \qquad \olcteformCVb_i = 1 
\enspace .
\]
It remains to evaluate the quantities appearing in Lemma~\ref{lem:CovGen.2} for these weights and function. 
First, 
\[ 
\sum_{i=1}^n E_{i,i}^{(\VF)}  
= n
\qquad \text{and} \qquad 
\sum_{i=1}^n \parenB{ E_{i,i}^{(\VF)} }^2 = n 
\enspace . 
\]
Second, by \eqref{hyp.part-reg.exact}, 
\begin{align}
\sum_{1 \leq i \neq j \leq n} \parenB{ E_{i,j}^{(\VF)} }  
&= n \parenj{ \frac{n}{V} - 1} + \frac{-1}{(V-1)} \times \frac{n^2 (V-1)}{ V} 
= -n 
\notag %%\label{eq.proof.pro.variance.SumEij}
\\
\notag %%\label{eq.proof.pro.variance.SumEijcarre}
\text{and} \quad 
\sum_{1 \leq i \neq j \leq n} \parenB{ E_{i,j}^{(\VF)} }^2 
&= n \crochj{ \parenj{ \frac{n}{V} - 1} + \frac{n}{V(V-1)} }
= \frac{n^2}{V-1} - n 
\enspace . 
\end{align}
It follows that 
\begin{gather*}
\sum_{1\leq i\leq n}\olomega_{i,i}^2 
= \frac{\parens{2C-1}^2}{n^{3}}
\enspace , 
\qquad 
\sum_{i=1}^n \olomega_{i,i} \olcteformCVb_i = \frac{2C-1}{n} 
\\
\text{and} \quad 
\sum_{1\leq i\neq j\leq n} \olomega_{i,j} \olomega_{j,i}
= \sum_{1\leq i\neq j\leq n}\olomega_{i,j}^2
= \frac{1}{n^2} \parenj{1+\frac{4C^2}{V-1} - \frac{\parens{2C-1}^2}{n}}
\enspace. 
\end{gather*}
Hence, from Lemma~\ref{lem:CovGen.2}, for every $m,m' \in \{m_1,m_2\}$, 
\begin{align*}
 \cov&\parenj{\CV_{C,\B}(m),\CV_{C,\B}(m')}
 = \frac{2}{n^2}\parenj{1+\frac{4C^2}{V-1}-\frac{(2C-1)^2}n} \termeBvar{m , m^{\prime}}
 \\
 &\quad+\frac{\parens{2C-1}^2}{n^3} \cov\parenb{U_m(\xi,\xi),U_{m'}(\xi,\xi)} + \frac{4}{n}\cov\parenb{\bayes_m(\xi),\bayes_{m'}(\xi)}
 \\
 &\quad-\frac{2\parenj{2C-1}}{n^2} \crochB{ \cov \parenb{ U_{m}(\xi,\xi),\bayes_{m'}(\xi)}+\cov\parenb{U_{m'}(\xi,\xi),\bayes_{m}(\xi)}}
 \\
 &=\frac{2}{n^2} \parenj{1+\frac{4C^2}{V-1}-\frac{\parens{2C-1}^2}n}\termeBvar{m , m^{\prime}}
 \\
 &\quad+\frac{1}{n} \cov\parenj{\frac{2C-1}nU_m(\xi,\xi)-2\bayes_m(\xi),\frac{2C-1}nU_{m'}(\xi,\xi)-2\bayes_{m'}(\xi)}\enspace.
\end{align*}
Therefore,
\begin{align*}
&\var \parenb{ \CV_{C,\B}(m_1)}
= \frac{2}{n^2}\parenj{1+\frac{4C^2}{V-1}-\frac{\parens{2C-1}^2}n}\termeBvar{m_1 , m_1}
 \\
&\phantom{\var \parenb{ \CV_{C,\B}(m_1)}} \qquad +\frac{1}{n}\var \parenj{ \frac{2C-1}{n} U_{m_1}(\xi,\xi)-2\bayes_{m_1}(\xi)}
\\
\text{and} \qquad 
&\var \parenj{ \CV_{C,\B}(m_1)-\CV_{C,\B}(m_2)}=\frac{2}{n^2}\parenj{1+\frac{4C^2}{V-1}-\frac{\parens{2C-1}^2}n}\termeBvaracr{m_1,m_2}
 \\
 & \qquad +\frac{1}{n}\var \parenj{ 2(\bayes_{m_1}-\bayes_{m_2})(\xi)-\frac{2C-1}n\parenb{U_{m_1}(\xi,\xi)-U_{m_2}(\xi,\xi)}} \\
 &\quad =\frac{2}{n^2}\parenj{1+\frac{4C^2}{V-1}-\frac{\parens{2C-1}^2}n}\var \parenb{ U_{m_1}(\xi,\xi)-U_{m_2}(\xi_1,\xi_2)}
 \\
 &\qquad +\frac{4}{n}\var \parenj{ \parenj{1+\frac{2C-1}n}(\bayes_{m_1}-\bayes_{m_2})(\xi)-\frac{2C-1}{2n}\parenb{\Psi_{m_1}(\xi)-\Psi_{m_2}(\xi)}}\enspace,
\end{align*}
which concludes the proof. \qed
\begin{lemma}\label{lem:CovGen.2}
 Let $\CV_m=\sum_{1\leq i,j\leq n}\olomega_{i,j}U_m(\xi_i,\xi_j)+\sum_{i=1}^n \olcteformCVb_i f_m(\xi_i)$, where $U_m$ is defined by Eq.~\eqref{eq.Um} and $f_m \in L^2(\mu)$. 
For every $m,m'$, we have 
\begin{align*}
 \cov&\parenj{\CV_m,\CV_{m'}}
= \parenj{\sum_{1\leq i\neq j\leq n} \olomega_{i,j}^2 + \olomega_{i,j}\olomega_{j,i} }\cov\parenb{U_m(\xi_1,\xi_2),U_{m'}(\xi_1,\xi_2)}\\
 &+\parenj{\sum_{i=1}^n\olomega_{i,i}^2}\cov\parenb{U_m(\xi_1,\xi_1),U_{m'}(\xi_1,\xi_1)}\\
 &+\parenj{\sum_{i=1}^n\olomega_{i,i} \olcteformCVb_i }\crochB{\cov\parenb{U_{m}(\xi_1,\xi_1),f_{m'}(\xi_1)} + \cov\parenb{U_{m'}(\xi_1,\xi_1),f_{m}(\xi_1)}}\\
 &+\parenj{ \sum_{i=1}^n \olcteformCVb_i^2} \cov\parenb{f_m(\xi_1),f_{m'}(\xi_1)}
\enspace.
\end{align*}
\end{lemma}
\begin{proof} 
We develop the covariance to get
 \begin{align*}
 \cov\parenj{\CV_m,\CV_{m'}}
 &=\sum_{1\leq i,j,k,\ell\leq n}\olomega_{i,j}\olomega_{k,\ell}\cov\parenb{U_m(\xi_i,\xi_j),U_{m'}(\xi_k,\xi_{\ell})}\\
 &\quad +\sum_{1\leq i,j,k\leq n}\olomega_{i,j}\olcteformCVb_k \cov\parenb{U_m(\xi_i,\xi_j),f_{m'}(\xi_k)}\\
 &\quad +\sum_{1\leq i,j,k\leq n}\olomega_{i,j}\olcteformCVb_k \cov\parenb{U_{m'}(\xi_i,\xi_j),f_{m}(\xi_k)}\\
 &\quad +\sum_{1\leq i,j\leq n}\olcteformCVb_i \olcteformCVb_j \cov\parenb{f_m(\xi_i),f_{m'}(\xi_j)}\enspace.
\end{align*}
The proof is then concluded with the following remarks, which rely on the fact that the random variables $\xi_{\inter{n}}$ 
are independent and identically distributed. 
\begin{enumerate}
\item $\cov\parenb{f_m(\xi_i),f_{m'}(\xi_j)}= 0$ unless $i \neq j$, therefore
\begin{align*}
\sum_{1\leq i,j\leq n} \olcteformCVb_i \olcteformCVb_j \cov\parenb{f_m(\xi_i),f_{m'}(\xi_j)}
%&=\sum_{i=1}^n \olcteformCVb_i^2 \cov\parenj{f_m(\xi_i),f_{m'}(\xi_i)}
%\\
&=\parenj{ \sum_{i=1}^n \olcteformCVb_i^2} \cov\parenb{f_m(\xi_1),f_{m'}(\xi_1)}
\enspace. 
\end{align*}
\item By definition \eqref{eq.Um} of $U_m$, 
$\cov\parenb{U_{m}(\xi_i,\xi_j),f_{m'}(\xi_k)} = 0$ unless $i=j=k$, hence
\begin{equation*}
\sum_{1\leq i,j,k\leq n}\olomega_{i,j} \olcteformCVb_k \cov\parenb{U_{m}(\xi_i,\xi_j),f_{m'}(\xi_k)}
=\parenj{\sum_{i=1}^n\olomega_{i,i} \olcteformCVb_i } \cov\parenb{U_{m}(\xi_1,\xi_1),f_{m'}(\xi_1)}
\enspace. 
\end{equation*}
\item By definition \eqref{eq.Um} of $U_m$, 
$\cov\parenb{U_m(\xi_i,\xi_j),U_m(\xi_k,\xi_l)} = 0$ unless 
%% $\{ i, j \} = \{ k , \ell \}$, that is,  
$i=j=k=\ell$ or $i=k \neq j=\ell$ or $i=\ell \neq j=k$. 
It follows that
\begin{align*}
& \quad 
\sum_{1\leq i,j,k,\ell\leq n}\olomega_{i,j}\olomega_{k,\ell}
\cov\parenb{U_m(\xi_i,\xi_j),U_{m'}(\xi_k,\xi_{\ell})}
\\
&= \parenj{\sum_{1\leq i\neq j\leq n}\olomega_{i,j}^2 + \olomega_{i,j}\olomega_{j,i}} 
\cov\parenb{U_m(\xi_1,\xi_2),U_{m'}(\xi_1,\xi_2)}
\\
&\quad +
\parenj{\sum_{i=1}^n\olomega_{i,i}^2} 
\cov\parenb{U_m(\xi_1,\xi_1),U_{m'}(\xi_1,\xi_1)}
\enspace.
\end{align*}
\end{enumerate}
\end{proof}

\bibliography{penvfreech}

\begin{thebibliography}{41}
\providecommand{\natexlab}[1]{#1}
\providecommand{\url}[1]{\texttt{#1}}
\expandafter\ifx\csname urlstyle\endcsname\relax
  \providecommand{\doi}[1]{doi: #1}\else
  \providecommand{\doi}{doi: \begingroup \urlstyle{rm}\Url}\fi

\bibitem[Allen(1974)]{All:1974}
David~M. Allen.
\newblock {The relationship between variable selection and data augmentation
  and a method for prediction}.
\newblock \emph{Technometrics}, 16:\penalty0 125--127, 1974.

\bibitem[Arlot(2008)]{Arl:2008a}
Sylvain Arlot.
\newblock {$V$}-fold cross-validation improved: {$V$}-fold penalization,
  February 2008.
\newblock \url{http://arxiv.org/pdf/0802.0566v2.pdf}.

\bibitem[Arlot(2009)]{Arl:2009:RP}
Sylvain Arlot.
\newblock Model selection by resampling penalization.
\newblock \emph{Electronic Journal of Statistics}, 3:\penalty0 557--624
  (electronic), 2009.

\bibitem[Arlot and Celisse(2010)]{Arl_Cel:2010:surveyCV}
Sylvain Arlot and Alain Celisse.
\newblock A survey of cross-validation procedures for model selection.
\newblock \emph{Statistics Surveys}, 4:\penalty0 40--79, 2010.

\bibitem[Arlot and Lerasle(2012)]{Arl_Ler:2012:penVF.v1}
Sylvain Arlot and Matthieu Lerasle.
\newblock {$V$}-fold cross-validation and {$V$}-fold penalization in
  least-squares density estimation, October 2012.
\newblock \url{http://arxiv.org/pdf/1210.5830v1.pdf}.

\bibitem[Audibert(2004)]{Aud:2004:tech}
Jean-Yves Audibert.
\newblock A better variance control for pac-bayesian classification.
\newblock Technical Report 905b, Laboratoire de Probabilit\'es et Mod\`eles
  Al\'eatoires, 2004.
\newblock Available electronically at
  \url{http://imagine.enpc.fr/publications/papers/04PMA-905Bis.pdf}.

\bibitem[Barron et~al.(1999)Barron, Birg{\'e}, and Massart]{Bar_Bir_Mas:1999}
Andrew Barron, Lucien Birg{\'e}, and Pascal Massart.
\newblock Risk bounds for model selection via penalization.
\newblock \emph{Probability Theory and Related Fields}, 113\penalty0
  (3):\penalty0 301--413, 1999.

\bibitem[Bengio and Grandvalet(2005)]{Ben_Gra:2005}
Yoshua Bengio and Yves Grandvalet.
\newblock Bias in estimating the variance of {$K$}-fold cross-validation.
\newblock In \emph{Statistical Modeling and Analysis for Complex Data
  Problems}, volume~1 of \emph{GERAD 25th Anniversary Series}, pages 75--95.
  Springer, New York, 2005.

\bibitem[Birg\'e(2013)]{Bi08}
Lucien Birg\'e.
\newblock Model selection for density estimation with $\mathbb{L}_2$-loss.
\newblock \emph{Probability Theory and Related Fields}, pages 1--42, 2013.

\bibitem[Birg\'e and Rozenholc(2006)]{BR03}
Lucien Birg\'e and Yves Rozenholc.
\newblock How many bins should be put in a regular histogram.
\newblock \emph{ESAIM: Probability and Statistics}, 10, 2006.

\bibitem[Boucheron et~al.(2013)Boucheron, Lugosi, and
  Massart]{Bou_Lug_Mas:2011:livre}
St\'ephane Boucheron, G{\'a}bor Lugosi, and Pascal Massart.
\newblock \emph{Concentration Inequalities: A Nonasymptotic Theory of
  Independence}.
\newblock Oxford University Press, Oxford, 2013.

\bibitem[Breiman and Spector(1992)]{Bre_Spe:1992}
Leo Breiman and Philip Spector.
\newblock {Submodel Selection and Evaluation in Regression. The X-Random Case}.
\newblock \emph{International Statistical Review}, 60\penalty0 (3):\penalty0
  291--319, 1992.

\bibitem[Bunea et~al.(2010)Bunea, Tsybakov, Wegkamp, and
  Barbu]{Bun_Tsy_Weg_Bar:2010}
Florentina Bunea, Alexandre~B. Tsybakov, Marten~H. Wegkamp, and Adrian Barbu.
\newblock Spades and mixture models.
\newblock \emph{The Annals of Statistics}, 38\penalty0 (4):\penalty0
  2525--2558, 2010.

\bibitem[Burman(1989)]{Bur:1989}
Prabir Burman.
\newblock A comparative study of ordinary cross-validation, {$v$}-fold
  cross-validation and the repeated learning-testing methods.
\newblock \emph{Biometrika}, 76\penalty0 (3):\penalty0 503--514, 1989.

\bibitem[Burman(1990)]{Bur:1990}
Prabir Burman.
\newblock {Estimation of optimal transformations using {$v$}-fold cross
  validation and repeated learning-testing methods}.
\newblock \emph{Sankhy\=a (Statistics). Series A}, 52\penalty0 (3):\penalty0
  314--345, 1990.

\bibitem[Catoni(2007)]{Cat:2007}
Olivier Catoni.
\newblock \emph{{Pac-Bayesian Supervised Classification: The Thermodynamics of
  Statistical Learning}}, volume~56 of \emph{{IMS Lecture Notes Monograph
  Series}}.
\newblock Institute of Mathematical Statistics, 2007.

\bibitem[Celisse(2008)]{Cel:2008:phd}
Alain Celisse.
\newblock \emph{{{M}odel Selection Via Cross-Validation in Density Estimation,
  Regression and Change-Points Detection}}.
\newblock PhD thesis, University Paris-Sud 11, December 2008.
\newblock Available electronically at
  \url{http://tel.archives-ouvertes.fr/tel-00346320/}.

\bibitem[Celisse(2014)]{Cel:2008}
Alain Celisse.
\newblock Optimal cross-validation in density estimation with the
  {${L}^{2}$}-loss.
\newblock \emph{The Annals of Statistics}, 42\penalty0 (5):\penalty0
  1879--1910, 10 2014.

\bibitem[Celisse and Robin(2008)]{Cel_Rob:2006}
Alain Celisse and St{\'e}phane Robin.
\newblock {Nonparametric density estimation by exact leave-{$p$}-out
  cross-validation}.
\newblock \emph{Computational Statistics \& Data Analysis}, 52\penalty0
  (5):\penalty0 2350--2368, 2008.

\bibitem[DeVore and Lorentz(1993)]{Dev_Lor:1993}
Ronald~A. DeVore and George~G. Lorentz.
\newblock \emph{Constructive Approximation}, volume 303 of \emph{Grundlehren
  der Mathematischen Wissenschaften \textup{[}Fundamental Principles of
  Mathematical Sciences\textup{]}}.
\newblock Springer-Verlag, Berlin, 1993.

\bibitem[Efron(1983)]{Efr:1983}
Bradley Efron.
\newblock Estimating the error rate of a prediction rule: improvement on
  cross-validation.
\newblock \emph{Journal of the American Statistical Association}, 78\penalty0
  (382):\penalty0 316--331, 1983.

\bibitem[Geisser(1975)]{Gei:1975}
Seymour Geisser.
\newblock {The predictive sample reuse method with applications}.
\newblock \emph{Journal of the American Statistical Association}, 70:\penalty0
  320--328, 1975.

\bibitem[Goldenshluger and Lepski(2011)]{Gol_Lep:2011:AoS}
Alexander Goldenshluger and Oleg Lepski.
\newblock Bandwidth selection in kernel density estimation: oracle inequalities
  and adaptive minimax optimality.
\newblock \emph{The Annals of Statistics}, 39\penalty0 (3):\penalty0
  1608--1632, 2011.

\bibitem[Hastie et~al.(2009)Hastie, Tibshirani, and
  Friedman]{Has_Tib_Fri:2001v2009}
Trevor Hastie, Robert Tibshirani, and Jerome Friedman.
\newblock \emph{The Elements of Statistical Learning}.
\newblock Springer Series in Statistics. Springer, New York, second edition,
  2009.
\newblock Data Mining, Inference, and Prediction.

\bibitem[Houdr{\'e} and Reynaud-Bouret(2003)]{Hou_Rey:2003}
Christian Houdr{\'e} and Patricia Reynaud-Bouret.
\newblock Exponential inequalities, with constants, for {U}-statistics of order
  two.
\newblock In \emph{Stochastic Inequalities and Applications}, volume~56 of
  \emph{Progress in Probability}, pages 55--69. Birkh\"auser, Basel, 2003.

\bibitem[Lerasle(2011)]{Ler:2010:mixing}
Matthieu Lerasle.
\newblock Optimal model selection for stationary data under various mixing
  conditions.
\newblock \emph{The Annals of Statistics}, 39\penalty0 (4):\penalty0
  1852--1877, 2011.

\bibitem[Lerasle(2012)]{Le09}
Matthieu Lerasle.
\newblock Optimal model selection in density estimation.
\newblock \emph{Annales de l'Institut Henri Poincar{\'e}. Probabilit{\'e}s et
  Statistiques}, 48\penalty0 (3):\penalty0 884--908, 2012.

\bibitem[Liu and Yang(2011)]{Liu_Yan:2011}
Wei Liu and Yuhong Yang.
\newblock Parametric or nonparametric? {A} parametricness index for model
  selection.
\newblock \emph{The Annals of Statistics}, 39\penalty0 (4):\penalty0
  2074--2102, 2011.

\bibitem[Magalh\~{a}es(2015)]{Mag:2015}
Nelo Magalh\~{a}es.
\newblock \emph{Cross-Validation and Penalization for Density Estimation}.
\newblock PhD thesis, University Paris-Sud 11, May 2015.
\newblock Available electronically at
  \url{http://tel.archives-ouvertes.fr/tel-01164581/}.

\bibitem[Massart(2007)]{Mas:2003:St-Flour}
Pascal Massart.
\newblock \emph{Concentration Inequalities and Model Selection}, volume 1896 of
  \emph{Lecture Notes in Mathematics}.
\newblock Springer, Berlin, 2007.
\newblock Lectures from the 33rd Summer School on Probability Theory held in
  Saint-Flour, July 6--23, 2003.

\bibitem[Picard and Cook(1984)]{Pic_Coo:1984}
Richard~R. Picard and R.~Dennis Cook.
\newblock {Cross-validation of regression models}.
\newblock \emph{Journal of the American Statistical Association}, 79\penalty0
  (387):\penalty0 575--583, 1984.

\bibitem[Rigollet(2006)]{Ri06}
Philippe Rigollet.
\newblock Adaptive density estimation using the blockwise {S}tein method.
\newblock \emph{Bernoulli}, 12\penalty0 (2):\penalty0 351--370, 2006.

\bibitem[Rigollet and Tsybakov(2007)]{RT07}
Philippe Rigollet and Alexander~B. Tsybakov.
\newblock Linear and convex aggregation of density estimators.
\newblock \emph{Mathematical Methods of Statistics}, 16\penalty0 (3):\penalty0
  260--280, 2007.

\bibitem[Rudemo(1982)]{Rud:1982}
Mats Rudemo.
\newblock Empirical choice of histograms and kernel density estimators.
\newblock \emph{Scandinavian Journal of Statistics. Theory and Applications},
  9\penalty0 (2):\penalty0 65--78, 1982.

\bibitem[Shao(1997)]{Sha:1997}
Jun Shao.
\newblock {An asymptotic theory for linear model selection}.
\newblock \emph{Statistica Sinica}, 7\penalty0 (2):\penalty0 221--264, 1997.
\newblock With comments and a rejoinder by the author.

\bibitem[Stone(1974)]{Sto:1974}
Mervyn Stone.
\newblock {Cross-validatory choice and assessment of statistical predictions}.
\newblock \emph{Journal of the Royal Statistical Society. Series B.
  Methodological}, 36:\penalty0 111--147, 1974.

\bibitem[{van der Laan} and Dudoit(2003)]{vdL_Dud:2003}
Mark~J. {van der Laan} and Sandrine Dudoit.
\newblock Unified cross-validation methodology for selection among estimators
  and a general cross-validated adaptive epsilon-net estimator: Finite sample
  oracle inequalities and examples.
\newblock Working Paper 130, U.C. Berkeley Division of Biostatistics, November
  2003.
\newblock Available electronically at
  \url{http://www.bepress.com/ucbbiostat/paper130}.

\bibitem[van~der Laan et~al.(2004)van~der Laan, Dudoit, and
  Keles]{vdL_Dud_Kel:2004}
Mark~J. van~der Laan, Sandrine Dudoit, and Sunduz Keles.
\newblock Asymptotic optimality of likelihood-based cross-validation.
\newblock \emph{Statistical Applications in Genetics and Molecular Biology},
  3:\penalty0 Art. 4, 27 pp. (electronic), 2004.

\bibitem[Yang(2005)]{Yan:2005a}
Yuhong Yang.
\newblock {Can the strengths of {AIC} and {BIC} be shared? {A} conflict between
  model indentification and regression estimation}.
\newblock \emph{Biometrika}, 92\penalty0 (4):\penalty0 937--950, 2005.

\bibitem[Yang(2006)]{Yan:2006}
Yuhong Yang.
\newblock {Comparing learning methods for classification}.
\newblock \emph{Statistica Sinica}, 16\penalty0 (2):\penalty0 635--657, 2006.

\bibitem[Yang(2007)]{Yan:2007b}
Yuhong Yang.
\newblock {Consistency of cross validation for comparing regression
  procedures}.
\newblock \emph{The Annals of Statistics}, 35\penalty0 (6):\penalty0
  2450--2473, 2007.

\end{thebibliography}

\clearpage

\section{Supplementary Material} 
\label{sec.supmat}
The supplementary material is organized as follows. 
Section~\ref{sec.app.proof.thm-variance} gives complementary computations of variances.
Then, results concerning hold-out penalization are detailed in Section~\ref{sec.supmat.penHO}, with the proof of the oracle inequality stated in Section~\ref{sec.discussion.hold-out} (Theorem~\ref{thm:PenHo}) and an exact computation of the variance. 
Section~\ref{sect.comp.comp} provides complements on the computational aspects stated in Section~\ref{sec.algo}. In particular, we state and analyse the basic algorithm for computing the $V$-fold criteria and we give the proof of Proposition~\ref{pro.proc.penVF-fast.general-density}.
A useful concentration inequality is recalled in Section~\ref{sec.supmat.proba-tools}. 
Finally, some simulation results are detailed in Section~\ref{sec.supmat.simus}, as a supplement to the ones of Section~\ref{sec.simus}.

\subsection{Additional Variance Computations}
\label{sec.app.proof.thm-variance}

\begin{proposition} \label{pro.variance}
Let $(\psil)_{\lL_{m_1}}$ and $(\psil)_{\lL_{m_2}}$ be two finite orthonormal families of vectors of $L^{4}(\mu)$.
Assume that $\B$ satisfies \eqref{hyp.part-reg.exact} and, for any $m\in\setj{m_1,m_2}$, let 
\[ \CV_{id}(m)=P_n\gamma(\ERM_m)+\E\crochb{\penid(m)}\enspace.\]
Then, with the notation of Theorem~\ref{theo.variance.penVF}, 

\begin{align*}
\var \parenb{ \CV_{id}(m_1)}&
=  \frac{2(n-1)}{n^3}\termeBvar{m_1 , m_1} +\frac{2}{n}\var \parenj{ \parenj{1-\frac{1}{n}}\bayes_{m_1}(\xi)+\frac{1}{2n}\Psi_{m_1}(\xi)}\enspace .
\end{align*}
We also have
\begin{align*}
 \var \parenb{ \CV_{id}(m_1)-\CV_{id}(m_2)}
&=  \frac{2(n-1)}{n^3}\termeBvaracr{m_1,m_2}
\\
& +\frac{2}{n}\var \parenj{ \parenj{1-\frac{1}{n}} \parenb{ \bayes_{m_1}(\xi)-\bayes_{m_2}(\xi) } + \frac{1}{2n} \parenb{ \Psi_{m_1}(\xi)-\Psi_{m_2}(\xi) } }
\enspace .
\end{align*}
\end{proposition}
\begin{proof} %of Proposition~\ref{pro.variance}.
Simply notice that 
\begin{align} \notag 
\var \parenb{ \CV_{id}(m_1)}=\var \parenb{ P_n\gamma(\ERM_{m_1})}
\label{eq.penid.formule-lambda} \enspace.
\end{align}
Therefore, from \eqref{eq:DecRiskEmp}, the variance of $\CV_{id}(m_1)$ is the one of
\[-\frac{1}{n^2}\sum_{1\leq i,j\leq n}U_{m_1}(\xi_i,\xi_j)-\sum_{i=1}^n\frac{2\bayes_{m_1}(\xi_i)}n\enspace.\]
so that, by Lemma~\ref{lem:CovGen.2},
\begin{align*} %%\label{eq.proof.pro.variance.var-penid-gal}
\var \parenb{ \CV_{id}(m_1)} 
&= \frac{2(n-1)}{n^3}\termeBvar{m_1 , m_1} +\frac{1}{n^3}\var \parenb{ \Psi_{m_1}(\xi)-2\bayes_{m_1}(\xi)}\\
&\qquad +\frac{4}{n^2}\sum_{i=1}^n\cov\parenb{\Psi_{m_1}(\xi)-2\bayes_{m_1}(\xi),\bayes_{m_1}(\xi)}+\frac{4}{n}\var \parenb{ \bayes_{m_1}(\xi)}\\
&= \frac{2(n-1)}{n^3}\termeBvar{m_1 , m_1} +\frac{2}{n}\var \parenj{ \parenj{1-\frac{1}{n}}\bayes_{m_1}(\xi)+\frac{1}{n}\Psi_{m_1}(\xi)}
\enspace .
\end{align*}
The variance of the increments follows from the same computations.
\end{proof}

\subsubsection{Evaluation of the Terms in the Variance Formula} 
The following proposition gives a formula for the terms appearing in Theorem~\ref{theo.variance.penVF} and Proposition~\ref{pro.variance} which does not depend on the basis $(\psil)_{\lL_{m}}$. 
\begin{proposition}\label{pro.calcul.termes.variance}
For any $m_1 , m_2 \in \M_n$, we have
\begin{align}
\notag\termeBvar{m_1 , m_2} &=n\cov\parenb{\ERM_{m_1} (\xi),\ERM_{m_2}(\xi)} - (n+1)\cov\parenb{\bayes_{m_1}(\xi),\bayes_{m_2}(\xi)}\\
\label{eq.pro.calcul.termes.variance.B}
\termeBvaracr{ m_1 , m_2 } &=n\var \parenb{ (\ERM_{m_1}-\ERM_{m_2})(\xi)}-(n+1)\var \parenb{ (\bayes_{m_1}-\bayes_{m_2})(\xi)}
\enspace ,
\end{align}
where $\xi$ denotes a copy of $\xi_1$, independent of $\xi_{\inter{n}}$. 
\end{proposition}
\begin{proof} %%of Proposition~\ref{pro.calcul.termes.variance}.
By definition, we have
\begin{align*}
\termeBvar{m_1 , m_2} 
&= 
\sum_{\lL_{m_1} } \sum_{\lpL_{m_2}} \cov\parens{\psil(\xi_1), \psi_{\lp}(\xi_1)}^2 \\
&= \sum_{\lL_{m_1} } \sum_{\lpL_{m_2}}  \parenb{P\parens{\psil \psi_{\lp}}-P\psil P\psilp}^2
\\
&= \sum_{\lL_{m_1} } \sum_{\lpL_{m_2}}  \parenb{P\parens{\psil \psi_{\lp}}}^{2} 
- 2 \sum_{\lL_{m_1} } \sum_{\lpL_{m_2}} P\psil P\psilp P\parenj{\psil \psi_{\lp}} 
\\
&\qquad 
+ \sum_{\lL_{m_1} } \sum_{\lpL_{m_2}} \parens{P\psil P\psilp}^2
\\
&=\sum_{\lL_{m_1} } \sum_{\lpL_{m_2}}  \parenb{P\parens{\psil \psi_{\lp}}}^{2} 
- 2 P\parens{\bayes_{m_1}\bayes_{m_2}} 
+ \norms{\bayes_{m_1}}^{2} \norms{\bayes_{m_2}}^{2}.
\end{align*}
Now, by Eq.~\eqref{eq.proj-Q-formule}, we have
\begin{align*}
&\qquad 
\cov\parenb{\ERM_{m_1}(\xi),\ERM_{m_2}(\xi)}
\\
&=\frac{1}{n^2}\sum_{1\leq i,j\leq n} \sum_{\lL_{m_1} } \sum_{\lpL_{m_2}}  
\cov\parenb{\psil(\xi_i)\psil(\xi),\psilp(\xi_j)\psilp(\xi)} \\
&=\frac{1}{n} \sum_{\lL_{m_1} } \sum_{\lpL_{m_2}}  \parenb{P\parens{\psil\psilp}}^2-\parens{P\psil P\psilp}^2
\\
&\qquad 
+\frac{n-1}{n} \sum_{\lL_{m_1} } \sum_{\lpL_{m_2}}  
\parenb{P\parens{\psil\psilp}-P\psil P\psilp}P\psil P\psilp
\\
&=\frac{1}{n} \sum_{\lL_{m_1} } \sum_{\lpL_{m_2}}   \parenb{P\parens{\psil\psilp}}^2
%%\\
%%&\qquad 
-\frac{1}{n} \norms{\bayes_{m_1}}^2\norms{\bayes_{m_2}}^2 
+ \frac{n-1}{n} \cov\parenb{\bayes_{m_1}(\xi),\bayes_{m_2}(\xi)}
\enspace . 
\end{align*}
It follows that
\begin{equation*}
\begin{split}
\sum_{\lL_{m_1} } \sum_{\lpL_{m_2}}  
\parenb{P\parens{\psil\psilp}}^2
&= n\cov \parenb{\ERM_{m_1}(\xi),\ERM_{m_2}(\xi)} + \norms{\bayes_{m_1}}^2 \norms{\bayes_{m_2}}^2
\\
&\qquad -(n-1)\cov\parenb{\bayes_{m_1}(\xi),\bayes_{m_2}(\xi)}
\enspace . 
\end{split}
\end{equation*}
Thus,
\begin{align*}
\termeBvar{m_1 , m_2} =n\cov\parenb{\ERM_{m_1}(\xi),\ERM_{m_2}(\xi)}
- (n+1)\cov\parenb{\bayes_{m_1}(\xi),\bayes_{m_2}(\xi)} 
\enspace .
\end{align*}
Eq.~\eqref{eq.pro.calcul.termes.variance.B} follows. 
\end{proof}

\subsubsection{Evaluation of the Variance in the Regular Histogram Case}
\label{sec.app.variance.eval-terms-histos}
The following lemma gives the value of the terms appearing in Theorem~\ref{theo.variance.penVF} 
for two nested regular histogram models. 
\begin{lemma}\label{lem:var2}
Let $m_1 = \Lambda_{m_1}$ and $m_2=\Lambda_{m_2}$ be two regular partitions of $\R$, 
as defined by Example~\ref{ex:regular} in Section~\ref{sec.ex.histos}, 
so that for $i \in \sets{1,2}$, for any $\lambda\in m_i$, $\mu(\lambda)=d_{m_i}^{-1}$. 
We assume that $m_2$ is a subpartition of $m_1$, 
that is, any element of $m_2$ is a subset of an element of $m_1$. 
For any $m^{\star} \in \sets{m_1,m_2}$, 
we define 
\[ 
T_{m^{\star}}(x)=\sum_{\lambda\in m^{\star}} \parens{ \psil(x)-P\psil }^{2}=\sup_{t\in\Boule_{m^{\star}}} 
\parenb{ t(x)-Pt }^{2}
\]
where we recall that 
$\Boule_{m^{\star}} = \sets{ t \in S_{m^{\star}} \, / \, \norms{t} \leq 1 }$ 
and 
for any $\lambda \in m_1 \cup m_2$, 
$\psil = \parens{ \mu(\lambda) }^{-1/2} \un_{\lambda}$. 
Then, we have 
\begin{align}
\label{eq.termeBvar.histos-reg-nested}
\termeBvar{m_1 , m_2}
&
= d_{m_1}\norms{\bayes_{m_2}}^{2} - 2 P \parens{\bayes_{m_1} \bayes_{m_2}} + \norms{\bayes_{m_1}}^{2}\norms{\bayes_{m_2}}^{2}
= P\parenj{T_{m_1}\bayes_{m_2}}
\\
\notag 
\text{and} \qquad 
\termeBvaracr{ m_1 , m_2 }
&= P\parenb{ T_{m_1}(\bayes_{m_1}-\bayes_{m_2})+(T_{m_2}-T_{m_1})\bayes_{m_2}}
\\ \notag 
&=
(d_{m_2}-d_{m_1})\norms{\bayes_{m_2}}^{2} + d_{m_1}\norms{\bayes_{m_1}-\bayes_{m_2}}^{2} 
\\ \notag 
&\qquad \qquad - 2\var_{P}  (\bayes_{m_1}-\bayes_{m_2}) - \norms{\bayes_{m_1}-\bayes_{m_2}}^{4}
%%d_{m_2}\norms{\bayes_{m_2}}^{2}-d_{m_1}\norms{\bayes_{m_1}}^{2}-2\var_{P}(\bayes_{m_1}-\bayes_{m_2}) - \norms{\bayes_{m_1}-\bayes_{m_2}}^{4}
\enspace.
\end{align}
\end{lemma}
\begin{proof}
On the one hand, 
by definition, 
\begin{align*}
&\qquad 
\termeBvar{m_1 , m_2}
\\
&= 
%%\mathop{\sum_{\lambda\in m_1}}_{\lambda^{\prime}\in m_2} 
\sum_{\lambda\in m_1} \sum_{\lambda^{\prime}\in m_2}
  \parenbb{\E\crochB{\parenb{\psil(\xi_1)-P\psil}\parenb{\psilp(\xi_1)-P\psilp}}}^2
\\
&= 
\sum_{\lambda\in m_1} \sum_{\lambda^{\prime}\in m_2}
 \parenj{ \crochb{  P\parenj{\psi_{\lambda}\psi_{\lambda^{\prime}}} }^2
 -2P\parenj{ \psi_{\lambda} \psi_{\lambda^{\prime}} } P\psi_{\lambda} P\psi_{\lambda^{\prime}}  
 + \parens{P\psi_{\lambda}}^{2} \parens{P\psi_{\lambda^{\prime}}}^{2} }
\\
&= 
\sum_{\lambda\in m_1} \sum_{\lambda^{\prime}\in m_2}
\crochb{  P\parenj{\psi_{\lambda}\psi_{\lambda^{\prime}}} }^2 
- 2 P \parenbb{ \underbrace{\parenB{ \sum_{\lambda\in m_1} (P\psil) \psil }}_{= \bayes_{m_1}} \underbrace{\parenB{ \sum_{\lambda\in m_2} (P\psil) \psil }}_{= \bayes_{m_2}} }
\\
&\qquad \qquad 
+ \underbrace{\sum_{\lambda\in m_1} \parens{P\psi_{\lambda}}^{2}}_{= \norms{ \bayes_{m_1} }^2}
\underbrace{\sum_{\lambda\in m_2} \parens{P\psi_{\lambda}}^{2}}_{= \norms{ \bayes_{m_2} }^2}
\enspace.
\end{align*}
For computing the first term, we use that $\psil \psilp = 0$ if $\lambda \cap \lambda^{\prime} = \emptyset$ and  $m_2$ is a subpartition of $m_1$, so that 
\begin{align*}
\sum_{\lambda\in m_1} \sum_{\lambda^{\prime}\in m_2}
\crochsqb{P\parenj{\psi_{\lambda}\psi_{\lambda^{\prime}}}}
&= \sum_{\lambda\in m_1}  \mathop{\sum_{\lambda^{\prime}\in m_2}}_{\lambda^{\prime}\subset \lambda} 
\crochsqb{P\parenj{\psi_{\lambda}\psi_{\lambda^{\prime}}}} 
\\
&=\sum_{\lambda\in m_1} \frac{1}{\mu(\lambda)} 
\mathop{\sum_{\lambda^{\prime}\in m_2}}_{\lambda^{\prime}\subset \lambda} 
\parens{P\psi_{\lambda^{\prime}}}^{2}
= d_{m_1}\sum_{\lambda^{\prime}\in m_2}(P\psi_{\lambda^{\prime}})^{2} 
= d_{m_1} \norms{ \bayes_{m_2} }^2
\end{align*}
hence 
\begin{align*}
\termeBvar{m_1 , m_2}
= d_{m_1}\norms{\bayes_{m_2}}^{2}-2P(\bayes_{m_1}\bayes_{m_2})
+ \norms{\bayes_{m_1}}^{2}\norms{\bayes_{m_2}}^{2}
\enspace .
\end{align*}
On the other hand, by definition of $T_m$, 
\begin{align*}
P(T_{m_1}\bayes_{m_2})
&= 
\sum_{\lambda \in m_1} \sum_{\lp \in m_2} P \parenb{ \parens{\psil - P \psil}^2 \psilp P(\psilp) }
\\
&= 
\sum_{\lambda \in m_1} \sum_{\lp \in m_2} 
\parenj{ P(\psil^2 \psilp) (P \psilp) - 2 P(\psil \psilp) (P\psil) (P\psilp) + (P\psil)^2 (P\psilp)^2 }
\\
\\
&= 
P \parenBb{ \underbrace{\sum_{\lambda \in m_1} \psil^2}_{=d_{m_1}}  
\underbrace{\sum_{\lp \in m_2} (P\psilp) \psilp }_{= \bayes_{m_2}}  }
- 2 P (\bayes_{m_1} \bayes_{m_2}) + \norms{\bayes_{m_1}}^2 \norms{\bayes_{m_2}}^2 
\end{align*}
which proves Eq.~\eqref{eq.termeBvar.histos-reg-nested} 
since $P(\bayes_{m_2}) = \norms{\bayes_{m_2}}^2$. 

Now, we remark that Eq.~\eqref{eq.termeBvar.histos-reg-nested} also gives formulas for 
$\termeBvar{m_i,m_i}$, $i\in \sets{1,2}$, since $m_i$ is a subpartition of itself. 
So, the second formula for $\termeBvar{m_i,m_j}$ in Eq.~\eqref{eq.termeBvar.histos-reg-nested} 
yields 
\begin{align*}
\termeBvaracr{ m_1 , m_2 } 
&= P\parenb{ T_{m_1} \bayes_{m_1} + T_{m_2} \bayes_{m_2} - 2 T_{m_1} \bayes_{m_2} }
\\
&= P\parenb{ T_{m_1} (\bayes_{m_1}-\bayes_{m_2})+(T_{m_2}-T_{m_1})\bayes_{m_2}}
\enspace . 
\end{align*}
Similarly, the first formula for $\termeBvar{m_i,m_j}$ in Eq.~\eqref{eq.termeBvar.histos-reg-nested} 
gives 
\begin{align*}
&\qquad 
\termeBvaracr{ m_1 , m_2 } 
\\
&=d_{m_1} \parenb{ \norms{\bayes_{m_1}}^{2}-\norms{\bayes_{m_2}}^{2} } 
+ (d_{m_2}-d_{m_1}) \norms{\bayes_{m_2}}^{2} 
-2P \parenb{ (\bayes_{m_1}-\bayes_{m_2})^{2} }
+ \parenb{ \norms{\bayes_{m_1}}^{2}-\norms{\bayes_{m_2}}^{2} }^{2}
\\
&=(d_{m_2}-d_{m_1})\norms{\bayes_{m_2}}^{2} + d_{m_1}\norms{\bayes_{m_1}-\bayes_{m_2}}^{2} 
- 2\var_{P}  (\bayes_{m_1}-\bayes_{m_2}) - \norms{\bayes_{m_1}-\bayes_{m_2}}^{4}
\enspace ,
\end{align*}
where we used that $P(\bayes_m)=\norms{\bayes_m}^2$ and 
$\norms{\bayes_{m_1}-\bayes_{m_2}}^{2} = 
\norms{\bayes_{m_1}}^{2}-\norms{\bayes_{m_2}}^{2}$. 
\end{proof}

%%%%%%%%%%%%%%%%%%%%%%%%%%%%%%%%%%%%%%%%%%%%%%%%%%%%%%%%%%%%%%%%%%%%%%%%%%%%%%%%%%%%%%%%%

\subsection{Results on MCCV and Some Other Cross-Validation Criteria} \label{sec.supmat.MCCV}
We prove here the results stated in Section~\ref{sec.discussion.MCCV}. 
Note that we here prove slightly more general results 
(Theorems \ref{thm.oracle-CV-gal} and~\ref{thm.var-MCCV-gal}), 
from which Theorems \ref{thm.oracle-MCCV} and~\ref{thm.var-MCCV} are corollaries. 
In particular, we do not always restrict to MCCV criteria: 
we always assume \eqref{hyp.CV.same-size} and \eqref{hyp.CV.ind} hold true, 
but we sometimes do not need to have \eqref{hyp.MCCV} satisfied. 

\subsubsection{Preliminary Computations} \label{sec.supmat.MCCV.calculs}

Our proofs rely on a simple closed-form formula for cross-validation criteria. 
Let us start by the hold-out criterion. 
Let $T \subset \inter{n}$ with $|T|=n-p$, independent from $D_n$. 
Then, 
\begin{align}
\notag 
\critHO(m, T) 
&= P_n^{(T^c)} \gamma\parenj{ \ERM_m^{(T)} }
\\
\notag 
&= \normb{ \ERM_m^{(T)} }^2 - 2 P_n^{(T^c)} \parenj{ \ERM_m^{(T)} }
\\
\notag 
&= \normb{ \ERM_m^{(T)} - \bayes_m }^2 + \norms{\bayes_m}^2 + 2 \prodscal{\ERM_m^{(T)} - \bayes_m }{\bayes_m} 
\\
\notag 
&\qquad 
- 2 \parenj{ P_n^{(T^c)} - P } \parenj{ \ERM_m^{(T)} - \bayes_m } - 2 P \parenj{ \ERM_m^{(T)} - \bayes_m } - 2 P_n^{(T^c)} (\bayes_m)
\\
\label{eq.critHO.calcul}
&= \normb{ \ERM_m^{(T)} - \bayes_m }^2 - 2 \parenj{ P_n^{(T^c)} - P } \parenj{ \ERM_m^{(T)} - \bayes_m }
- 2 P_n^{(T^c)} (\bayes_m) + \norms{\bayes_m}^2
\end{align}
where the last equality uses that 
\[ 
P \parenj{ \ERM_m^{(T)} - \bayes_m } = \prodscal{ \ERM_m^{(T)} - \bayes_m }{ \bayes }
= \prodscal{ \ERM_m^{(T)} - \bayes_m }{ \bayes_m }
\]
since $\bayes_m$ is the orthogonal projection in $L^2(\mu)$ of $\bayes_m$ onto $S_m$ and $\ERM_m^{(T)} - \bayes_m \in S_m$. 

The last two terms in the right-hand side of Eq.~\eqref{eq.critHO.calcul} 
can be rewritten as 
\begin{align*}
 - 2 P_n^{(T^c)} (\bayes_m) + \norms{\bayes_m}^2
&= - 2 \parenj{ P_n^{(T^c)} - P} (\bayes_m) - 2 P(\bayes_m) + \norms{\bayes_m}^2
\\
&= - 2 \parenj{ P_n^{(T^c)} - P} (\bayes_m) - \norms{\bayes_m}^2 
\end{align*}
since $\norms{\bayes_m}^2 = P(\bayes_m)$. 
For the first two terms, 
we write that 
\begin{align}
\notag 
&\qquad 
\norms{ \ERM_m^{(T)} - \bayes_m }^2 - 2 \parenj{ P_n^{(T^c)} - P } \parenj{ \ERM_m^{(T)} - \bayes_m }
\\
\notag 
&= 
\sum_{\lambda \in \Lambda_m} \crochj{ \parenb{ \parens{ P_n^{(T)} - P }(\psil) }^2 
- 2 \parenj{ P_n^{(T^c)} - P }(\psil) \parenj{ P_n^{(T)} - P }(\psil) }
\\
\notag 
&= 
\sum_{\lambda \in \Lambda_m} \biggl[  
\frac{1}{(n-p)^2} \sum_{1 \leq i,j \leq n} \un_{i \in T, \, j \in T} \parenb{ \psil(\xi_i) - P\psil } \parenb{ \psil(\xi_j) - P\psil } 
\\
\notag 
&\qquad 
- \frac{2}{p (n-p)} \sum_{1 \leq i,j \leq n} \un_{i \in T^c, \, j \in T} \parenb{ \psil(\xi_i) - P\psil } \parenb{ \psil(\xi_j) - P\psil }
\biggr]
\\
&= 
\sum_{1 \leq i,j \leq n} \crochj{ \frac{\un_{j \in T}}{n-p} \parenj{  \frac{\un_{i \in T}}{n-p}  - \frac{2 \un_{i \in T^c}}{p}  } U_m(\xi_i , \xi_j) }
\notag 
\end{align}
where we recall that for any $x,y \in \X$, 
\[ 
U_m(x,y) = \sum_{\lL_m}\parenb{ \psil(x)-P\psil }\parenb{ \psil(y)-P\psil } 
= \sum_{\lL_m}\psil(x)\psil(y) - \bayes_m(x) - \bayes_m(y) + \norms{\bayes_m}^2
\]
is defined by Eq.~\eqref{eq.Um}, 
and that 
$U_m(x,x)=\Psi_m(x) - 2 \bayes_m(x) + \norms{\bayes_m}^2$. 

Therefore, Eq.~\eqref{eq.critHO.calcul} can be rewritten as 
\begin{align}
\critHO (m, T) 
&= 
\sum_{1 \leq i,j \leq n} \crochj{ \frac{\un_{j \in T}}{n-p} \parenj{  \frac{\un_{i \in T}}{n-p}  - \frac{2 \un_{i \in T^c}}{p}  } U_m(\xi_i , \xi_j) }
- 2 \parenb{ P_n^{(T^c)} - P} (\bayes_m) - \norms{\bayes_m}^2
\notag 
\\
&= 
\sum_{1 \leq i,j \leq n} \omegaHO_{i,j}(T) U_m(\xi_i , \xi_j) 
+ \sum_{i=1}^n \cteformHOb_i(T) \parenb{ \bayes_m(\xi_i) - P(\bayes_m) } 
- \norms{\bayes_m}^2
\label{eq.critHO.calcul.2}
\end{align}
with 
\begin{align*}
\omegaHO_{i,j} (T)
&= 
\frac{\un_{j \in T}}{n-p} \parenj{  \frac{\un_{i \in T}}{n-p}  - \frac{2 \un_{i \in T^c}}{p}  } 
\\
\cteformHOb_i (T)
&= 
\frac{-2}{p} \un_{i \in T^c}
\enspace .  
\end{align*}
As a consequence, under assumption \eqref{hyp.CV.same-size}, 
\begin{align}
\critCV \parenb{ m, (T_j)_{1 \leq j \leq K} } 
&= 
\sum_{1 \leq i,j \leq n} \omega_{i,j} U_m(\xi_i , \xi_j) 
+ \sum_{i=1}^n \cteformCVb_i \parenb{\bayes_m(\xi_i) - P(\bayes_m) } 
- \norms{\bayes_m}^2
\label{eq.critCV.calcul}
\end{align}
with 
\begin{align*}
\omega_{i,j} 
&= 
\frac{1}{B} \sum_{K=1}^B \crochj{ \frac{\un_{j \in T_K}}{n-p} \parenj{  \frac{\un_{i \in T_K}}{n-p}  - \frac{2 \un_{i \in T_K^c}}{p}  } }
\\
\cteformCVb_i 
&= 
\frac{-2}{p B} \sum_{K=1}^B \un_{i \in T_K^c}
\enspace .  
\end{align*}

\medbreak

Note that Eq.~\eqref{eq.critCV.calcul} is consistent with previously obtained formulas. 
For $V$-fold cross-validation, under assumption \eqref{hyp.part-reg.exact}, 
Eq.~\eqref{eq.critCV.calcul} holds with 
\begin{align*}
\omega_{i,j} 
&= 
\omegaVFCV_{i,j} 
\egaldef 
\frac{1}{n^2} 
\begin{cases}
\frac{V}{V-1} \qquad  &\text{if $i$ and $j$ belong to the same block} \\
- \parensq{ \frac{V}{V-1} }  &\text{otherwise}
\end{cases}
\\
\cteformCVb_i 
&= 
\cteformVFCVb_{i} 
\egaldef 
\frac{-2}{n}
\enspace , 
\end{align*}
which can also be obtained from the combination of 
Eq.~\eqref{eq.le.penVF-VFCV} in Lemma~\ref{le.penVF-VFCV} 
and 
Eq.~\eqref{eq.crit-penVF-gal.U-stat.1}. 
For the leave-$p$-out, 
Eq.~\eqref{eq.critCV.calcul} holds with 
\begin{align*}
\omega_{i,j} 
&= 
\omegaLPO_{i,j} 
\egaldef 
\begin{cases}
\frac{1}{n(n-p)} \qquad &\text{if } i \neq j \\
\frac{-(n-p+1)}{n(n-1)(n-p)} \qquad &\text{otherwise} 
\end{cases}
\\
\cteformCVb_i 
&= 
\cteformLPOb_{i} 
\egaldef 
\frac{-2}{n}
\enspace , 
\end{align*}
which 
can also be obtained from 
Eq.~\eqref{eq.le.penLOO-LPO} in Lemma~\ref{le.penVF-VFCV} 
and 
Eq.~\eqref{eq.crit-penVF-gal.U-stat.1}. 

\medbreak

Using that $U_m(x,x)=\Psi_m(x) - 2 \bayes_m(x) + \norms{\bayes_m}^2$, 
Eq.~\eqref{eq.critCV.calcul} can be rewritten as  
\begin{align*}
\crit_{\CV}\parenb{ m, (T_j)_{1 \leq j \leq K} } 
&= 
\parenj{\sum_{i=1}^n \omega_{i,i}} \parenB{\Dcal_m-\norms{\bayes_m}^2} 
- \norms{\bayes_m}^2
+ \sum_{i=1}^n\omega_{i,i} \parenb{ \Psi_m(\xi_i)-\Dcal_m }
\\ 
& + \sum_{i=1}^n (-2\omega_{i,i}+\cteformCVb_i) \parenb{ \bayes_m(\xi_i)-P\bayes_m } 
+ \sum_{1 \leq i\neq j \leq n} \omega_{i,j} U_m(\xi_i , \xi_j) 
\enspace . 
\end{align*}
Using \eqref{hyp.CV.same-size} 
we have 
\[
\sum_{i=1}^n\omega_{i,i}
= \frac{1}{B(n-p)^2} \sum_{K=1}^B\sum_{i=1}^n\un_{i\in T_K}
= \frac{1}{n-p}
\enspace ,
\]
and we get 
\begin{equation}
\label{eq.critCV.calcul.bis}
\begin{split}
\crit_{\CV}\parenb{ m, (T_j)_{1 \leq j \leq K} } 
&= 
\frac{\Dcal_m - \norms{\bayes_m}^2}{n-p} - \norms{\bayes_m}^2
+ \sum_{i=1}^n \omega_{i,i} \parenb{ \Psi_m(\xi_i)-\Dcal_m }
\\
& + \sum_{i=1}^n (-2\omega_{i,i}+\cteformCVb_i) \parenb{ \bayes_m(\xi_i)-P\bayes_m } 
+ \sum_{1 \leq i\neq j \leq n} \omega_{i,j} U_m(\xi_i , \xi_j) 
\enspace . 
\end{split}
\end{equation}

\subsubsection{Concentration Inequalities} \label{sec.supmat.MCCV.conc}
In the proof of Theorem~\ref{thm.oracle-MCCV} in Section~\ref{sec.supmat.MCCV.oracle}, 
given formula \eqref{eq.critCV.calcul.bis} for the cross-validation criterion, 
we need concentration inequalities for the three random sums appearing in 
Eq.~\eqref{eq.critCV.calcul.bis}. 
These are stated and proved in three lemmas below. 

\fauxparagraph{Concentration of $\sum_{i=1}^n\omega_{i,i}(\Psi_m(\xi_i)-\Dcal_m)$} 
\begin{lemma}\label{le.eq:TermePsiCritVFGen}
Assume that \eqref{hyp.CV.same-size}, \eqref{hyp.CV.ind} and~\eqref{hyp.NormSupNorm2} hold true.
Then, for any $x>0$, 
an event of probability at least $1-2 \e^{-x}$ exists on which the following holds true\textup{:} 
for any $\epsilon \in (0,1]$, 
\begin{equation*} %%\label{eq:TermePsiCritVFGen}
\absj{ \sum_{i=1}^n\omega_{i,i} \parenb{ \Psi_m(\xi_i)-\Dcal_m } }
\leq 
\epsilon \frac{\Dcal_m}{n-p} 
+ \frac{5 x \parens{ n + A } }{3 \epsilon (n-p)^2} 
\enspace . 
\end{equation*}
\end{lemma}
\begin{proof} 
By \eqref{hyp.CV.ind}, 
conditionally to $(\omega_{i,i})_{1 \leq i \leq n}$, 
$\sum_{i=1}^n\omega_{i,i} \parenb{ \Psi_m(\xi_i)-\Dcal_m }$ 
is a sum of independent real-valued random variables. 
So, we can apply Bernstein's inequality. 

First, 
for any $i \in \inter{n}$, 
using \eqref{hyp.CV.same-size}, 
\[ 
\omega_{i,i} = \frac{1}{B} \sum_{K=1}^B \frac{\un_{i \in T_K}}{(n-p)^2} 
\leq \frac{1}{(n-p)^2}
\]
and using Eq.~\eqref{eq.maj-sup-Um}, 
\[
\norms{\Psi_m}_{\infty} \leq \norms{U_m}_{\infty} \leq 2 \parenB{ b_m^2 + \norms{\bayes_m}^2 }
\enspace , 
\]
so that 
\[ 
\omega_{i,i} \Psi_m(\xi_i) \leq 
\max_{1 \leq i \leq n} \omega_{i,i} \times \norms{\Psi_m}_{\infty} 
\leq 
\frac{2 \parenB{ b_m^2 + \norms{\bayes_m}^2 }}{(n-p)^2} 
\]
almost surely. 

Second, 
using \eqref{hyp.CV.same-size}, we have 
\[
\sum_{i=1}^n \omega_{i,i}^2 \leq 
\max_{1 \leq i \leq n} \omega_{i,i} \times \sum_{i=1}^n \omega_{i,i}
\leq \frac{1}{(n-p)^3 }
\]
and using Eq.~\eqref{eq.maj-sup-Um} again, 
\[ 
\E\crochj{ \Psi_m(\xi_i)^2 } \leq \norms{\Psi_m}_{\infty} \times P(\Psi_m) 
=  \norms{\Psi_m}_{\infty} \times \Dcal_m
\leq 2 \parenB{ b_m^2 + \norms{\bayes_m}^2 } \Dcal_m
\enspace , 
\]
so that 
\[
\sum_{i=1}^n \omega_{i,i}^2 \E\crochj{ \Psi_m(\xi_i)^2 } 
\leq 
\frac{2 \parenB{ b_m^2 + \norms{\bayes_m}^2 } \Dcal_m}{(n-p)^3}
\enspace. 
\]

Then, by Bernstein's inequality \citep[Theorem~2.10]{Bou_Lug_Mas:2011:livre}, 
conditionally to $(\omega_{i,i})_{1 \leq i \leq n}$, 
an event of probability at least $1-2 \e^{-x}$ exists on which 
\begin{align*}
\absj{ \sum_{i=1}^n\omega_{i,i} \parenb{ \Psi_m(\xi_i)-\Dcal_m } }
&\leq 
2 \sqrt{ \frac{x \parenB{ b_m^2 + \norms{\bayes_m}^2 } \Dcal_m}{(n-p)^3} } 
 +  \frac{2 \parenB{ b_m^2 + \norms{\bayes_m}^2 }}{(n-p)^2} \frac{x }{3}
\\
&\leq 
\epsilon \frac{\Dcal_m}{n-p} 
+ \parenj{ \frac{2}{3} + \frac{1}{\epsilon} } \frac{ x \parenB{ b_m^2 + \norms{\bayes_m}^2 } }{(n-p)^2} 
\\
&\leq 
\epsilon \frac{\Dcal_m}{n-p} 
+ \frac{5}{3 \epsilon} \frac{ x \parens{ n + A } }{(n-p)^2} 
\end{align*}
for any $\epsilon \in (0,1]$, 
where we used that $b_m^2 \leq n$ by \eqref{hyp.NormSupNorm2}, 
and that $\norms{\bayes_m}^2 \leq \norms{\bayes}^2 \leq \norms{\bayes}_{\infty} \leq A$. 
The result follows by integrating this conditional concentration inequality 
with respect to $(\omega_{i,i})_{1 \leq i \leq n}$. 
\end{proof}

\fauxparagraph{Concentration of $\sum_{i=1}^n(-2\omega_{i,i}+\cteformCVb_i)(\bayes_m(\xi_i)-P\bayes_m)$}
\begin{lemma} \label{le.conc.TermeBiaisCritVFGen}
Assume that \eqref{hyp.CV.same-size} and \eqref{hyp.CV.ind} hold true.
Then, for any $x>0$, 
an event of probability at least $1- \e^{-x}$ exists on which the following holds true\textup{:} 
for any $\epsilon \in (0,1]$, 
\begin{equation}\label{eq:TermeBiaisCritVFGen}
\begin{split}
\sum_{i=1}^n (-2\omega_{i,i}+\cteformCVb_i) \parenb{ \bayes_m(\xi_i)-P\bayes_m-\bayes_{m'}(\xi_i)-P\bayes_{m'} } 
\\
\leq \epsilon \norms{\bayes_m - \bayes_{m'}}^2 + \RbVFgal_n(x,\epsilon,\pi^*,A)
\end{split}
\end{equation}
where the remainder term depends on the additional assumption that we make.
If \eqref{hyp.UBbayes} holds true, then 
\[ 
\RbVFgal_n(x,\epsilon,\pi^*,A) \egaldef 
\frac{ 16 A x}{ 3 \epsilon} \parenj{ \frac{1}{(n-p)^2} + \frac{\pi^{\star}}{p} } 
\enspace . 
\]
If \eqref{hyp.NormSupNorm2} and \eqref{hyp.Nested} hold true, then, 
some numerical constant $\kappa>0$ exists such that 
\[
\RbVFgal_n(x,\epsilon,\pi^*,A) \egaldef 
\frac{ \kappa  }{ \epsilon} \crochj{ 
Ax \parenj{ \frac{1}{(n-p)^2} + \frac{\pi^{\star}}{p} } 
+ x^2 n  \parensq{ \frac{1}{(n-p)^2} + \frac{\pi^{\star}}{p} } 
}
\enspace . 
\]
\end{lemma}
Before proving Lemma~\ref{le.conc.TermeBiaisCritVFGen}, 
let us introduce some useful notation: 
given a sequence $T_1, \ldots, T_B$ of subsets of $\inter{n}$, 
for every $i,j \in \inter{n}$, we define 
\[
\pi_i = \frac{1}{B} \sum_{K=1}^B \un_{i\in T_K^c}
\qquad 
\pi_{i,j} = \frac{1}{B} \sum_{K=1}^B \un_{i\in T_K^c}\un_{j\in T_K^c}
\qquad \text{and} \qquad 
\pi^* = \max_{i=1,\ldots,n}\pi_i
\enspace.
\]
Note that, assuming \eqref{hyp.CV.same-size}, we have 
\begin{equation}
\label{eq.ineg-simple-piij-pi}
\begin{split}
0 \leq \pi_{i,j} \leq \min\parens{ \pi_i, \pi_j } \leq \pi^* \leq 1
\qquad 
\sum_{i=1}^n \pi_i = p
\\ %\qquad 
\sum_{i=1}^n\pi_{i,j} = p \pi_j \leq p \pi^* 
\qquad \text{and} \qquad 
\sum_{1\le i,j\le n}\pi_{i,j}=p^2 
\enspace . 
\end{split}
\end{equation}

\begin{proofof}{Lemma~\ref{le.conc.TermeBiaisCritVFGen}} 
By \eqref{hyp.CV.ind}, 
conditionally to $(-2\omega_{i,i}+\cteformCVb_i)_{1 \leq i \leq n}$, 
\[ 
\sum_{i=1}^n\omega_{i,i} \parenb{ \Psi_m(\xi_i)-\Dcal_m }
\] 
is a sum of independent real-valued random variables. 
So, we can apply Bernstein's inequality. 

First, we notice that for every $i \in \inter{n}$, 
\[ 
-2\omega_{i,i} + \cteformCVb_i
= 
\frac{1}{B} \sum_{K=1}^B \parenj{ \frac{-2}{(n-p)^2} \un_{i \in T_K} - \frac{2}{p} \un_{i \notin T_K} }
= 
-2 \parenj{ 
\frac{1}{(n-p)^2} (1-\pi_i) + \frac{\pi_i}{p} 
}
\]
hence
\[ 
\absj{-2\omega_{i,i} + \cteformCVb_i}
= 
2 \parenj{ \frac{1}{(n-p)^2} (1-\pi_i) + \frac{\pi_i}{p} }
\leq 
2 \parenj{ \frac{1}{(n-p)^2} + \frac{\pi^*}{p} }
\]
since $0 \leq \pi_i \leq \pi^* \leq 1$. 
So, for every $i \in \inter{n}$, 
\[
\parens{ -2\omega_{i,i} + \cteformCVb_i } \parenb{ \bayes_m(\xi_i) - \bayes_{m'}(\xi_i) }
\leq 
2 \parenj{ \frac{1}{(n-p)^2} + \frac{\pi^*}{p} } \norms{\bayes_m - \bayes_{m'}}_{\infty}
\]
almost surely. 
Second, 
\begin{align*}
\sum_{i=1}^n \parens{-2\omega_{i,i} + \cteformCVb_i}^2 
&\leq 
2 \parenj{ \frac{1}{(n-p)^2} + \frac{\pi^*}{p} }
\sum_{i=1}^n \abss{-2\omega_{i,i} + \cteformCVb_i}
\\
&= 2 \parenj{ \frac{1}{(n-p)^2} + \frac{\pi^*}{p} }
2 \parenj{ \frac{1}{n-p} + 1 }
\\
&\leq 8 \parenj{ \frac{1}{(n-p)^2} + \frac{\pi^*}{p} }
\end{align*}
and 
\[
\E\crochj{ \parenb{ \bayes_m(\xi) - \bayes_{m'}(\xi) }^2 }
\leq \norms{\bayes}_{\infty} \norms{ \bayes_m - \bayes_{m'} }^2
\leq A \norms{ \bayes_m - \bayes_{m'} }^2
\]
so that 
\[ 
\sum_{i=1}^n \E\crochbb{ \parenB{ (-2\omega_{i,i} + \cteformCVb_i) \parenb{ \bayes_m(\xi) - \bayes_{m'}(\xi) } }^2 } 
\leq 
8 A \parenj{ \frac{1}{(n-p)^2} + \frac{\pi^*}{p} } 
\norms{ \bayes_m - \bayes_{m'} }^2
\enspace . 
\]

Then, by Bernstein's inequality \citep[Theorem~2.10]{Bou_Lug_Mas:2011:livre}, 
conditionally to $(-2\omega_{i,i}+\cteformCVb_i)_{1 \leq i \leq n}$, 
an event of probability at least $1-\e^{-x}$ exists on which 
\begin{gather}
\sum_{i=1}^n (-2\omega_{i,i}+\cteformCVb_i) \parenb{ \bayes_m(\xi_i)-P\bayes_m-\bayes_{m'}(\xi_i)-P\bayes_{m'} } 
\leq 
R^{0} (m,m')
\notag %\label{eq:TermeBiaisCritVFGen.pr.1}
\\
\notag 
%\text{with} \qquad 
R^{0} (m,m') \egaldef 
 \sqrt{16 x  A \parenj{ \frac{1}{(n-p)^2} + \frac{\pi^*}{p} } \norms{ \bayes_m - \bayes_{m'} }^2 } 
+ \frac{ 2  x \norms{\bayes_m - \bayes_{m'}}_{\infty} }{ 3 } \parenj{ \frac{1}{(n-p)^2} + \frac{\pi^*}{p} }  
\enspace . 
\end{gather}
Since $1-2 \e^{-x}$ is deterministic, the same inequality holds unconditionally on an event of probability at least $1-2\e^{-x}$. 

We now upperbound $R^{0} (m,m')$, 
differently depending on the assumption we make.
On the one hand, if \eqref{hyp.UBbayes} holds true, 
\[ 
\norms{\bayes_m - \bayes_{m'}}_{\infty}
\leq 
\norms{\bayes_m}_{\infty} + \norms{\bayes_{m'}}_{\infty}
\leq 
2 A 
\]
and we get 
\begin{align*}
R^{0} (m,m') 
&\leq  
 \sqrt{16 A x  \parenj{ \frac{1}{(n-p)^2} + \frac{\pi^*}{p} } \norms{ \bayes_m - \bayes_{m'} }^2 } 
+ \frac{ 4  A  x  }{ 3 } \parenj{ \frac{1}{(n-p)^2} + \frac{\pi^*}{p} }  
\\
&\leq \epsilon \norms{\bayes_m - \bayes_{m'}}^2 
+ \frac{ 16 A x}{ 3 \epsilon} \parenj{ \frac{1}{(n-p)^2} + \frac{\pi^{\star}}{p} } 
\end{align*}
for any $\epsilon \in (0,1]$, 
which proves Eq.~\eqref{eq:TermeBiaisCritVFGen}. 
On the other hand, if  \eqref{hyp.NormSupNorm2} and \eqref{hyp.Nested} hold true, 
$\bayes_m - \bayes_{m'} \in S_{m''}$ with $m'' \in \sets{ m , m'}$, so that 
\[
\norms{\bayes_m - \bayes_{m'}}_{\infty}
\leq 
b_{m''} \norms{\bayes_m - \bayes_{m'}}
\leq 
\sqrt{n} \norms{\bayes_m - \bayes_{m'}}
\]
and we get 
\begin{align*}
R^{0} (m,m') 
&\leq  
 \sqrt{16 x  A \parenj{ \frac{1}{(n-p)^2} + \frac{\pi^*}{p} } \norms{ \bayes_m - \bayes_{m'} }^2 } 
+ \frac{ 2  x \sqrt{n} \norms{\bayes_m - \bayes_{m'}} }{ 3 } \parenj{ \frac{1}{(n-p)^2} + \frac{\pi^*}{p} }  
\\
&\leq \epsilon \norms{\bayes_m - \bayes_{m'}}^2 
+ \frac{ 1  }{ \epsilon} \crochj{ 
8 Ax \parenj{ \frac{1}{(n-p)^2} + \frac{\pi^{\star}}{p} } 
+ \frac{2}{9} x^2 n  \parensq{ \frac{1}{(n-p)^2} + \frac{\pi^{\star}}{p} } 
} 
\end{align*}
for any $\epsilon \in (0,1]$, 
which proves Eq.~\eqref{eq:TermeBiaisCritVFGen} with $\kappa=8$.
\end{proofof}

\fauxparagraph{Concentration of $\sum_{1 \leq i\neq j \leq n} \omega_{i,j} U_m(\xi_i , \xi_j)$}
\begin{lemma} \label{le.eq:Conc.UStat.Gen}
Suppose that assumptions~\eqref{hyp.CV.same-size}, \eqref{hyp.CV.ind} and \eqref{hyp.NormSupNorm2} hold true.
Then, an absolute constant $\kappa>0$ exists such that, for any $x>1$, with probability larger than $1-6\e^{-x}$, for any $\epsilon\in(0,1]$,
\begin{align}\label{eq:Conc.UStat.Gen}
\hspace*{-0.22cm}
\absj{\sum_{1 \leq i\neq j \leq n} \omega_{i,j} U_m(\xi_i , \xi_j)}
\leq \frac{\epsilon \Dcal_m}{n-p} +  \frac{\kappa n}{(n-p)^2} \parenj{1+\frac{n \pi^*}{p}} \crochj{\frac{nA x}{(n-p)\epsilon} + \parenj{1+\frac{A}{n}}  x^2 } 
\enspace .  
\end{align}
\end{lemma}

\begin{proof}
We start with the following symmetrization trick
\begin{align*}
 \sum_{1 \leq i\neq j \leq n} \omega_{i,j} U_m(\xi_i , \xi_j)&=\sum_{1\le i<j\le n}\omega_{i,j} U_m(\xi_i , \xi_j)+\omega_{j,i} U_m(\xi_j , \xi_i)\\
 &=\sum_{1\le i<j\le n}(\omega_{i,j}+\omega_{j,i})  U_m(\xi_i , \xi_j)\\
 &=\sum_{1\le i\ne j\le n}\omega'_{i,j}  U_m(\xi_i , \xi_j)\enspace,
\end{align*}
where
\begin{align*}
\omega'_{i,j}
= \frac{\omega_{i,j}+\omega_{j,i}}{2}
&=\frac{1}{(n-p)^2}\crochj{1-(\pi_i+\pi_j)\frac{n}p+\parenj{\frac{2n}{p} - 1}\pi_{i,j}}
\\
&= \frac{1}{(n-p)^2}\crochj{ (1-\pi_{i,j}) + \frac{n}{p} (\pi_{i,j} - \pi_i) + \frac{n}{p} (\pi_{i,j} - \pi_j) }
\enspace.
\end{align*}
From the last formula for $\omega'_{i,j}$, 
using Eq.~\eqref{eq.ineg-simple-piij-pi}, 
we get that 
\begin{align} 
\label{eq.maj-carre-omegaprimeij}
( \omega'_{i,j}  )^2 
&\leq \frac{1}{(n-p)^4} \crochj{ 1 + \frac{n^2}{p^2} \parens{ \pi_i + \pi_j }^2 }
\\
\label{eq.maj-max-omegaprimeij}
\text{and} \qquad 
\max_{i,j \in \inter{n}} \abss{ \omega'_{i,j} } &\leq \frac{1}{(n-p)^2} \parenj{ 1 + \frac{2n}{p} \pi^*}
\enspace . 
\end{align}
The concentration of the $U$-statistics follows from \citet[Theorem~3.4]{Hou_Rey:2003}, that is Eq.~\eqref{eq:Conc.U.Stat} 
with $g_{i,j}(\xi_i,\xi_j) = \omega'_{i,j} U_m(\xi_i,\xi_j)$. 
To apply this result, it remains to compute the terms $\oA$, $\oB$, $\oC$, $\oD$. 
First, 
\begin{align*}
 2\oA^2=\sum_{1\le i\ne j\le n}(\omega_{i,j}')^2\E\crochj{U_m(\xi_i,\xi_j)^2}
\leq \norms{\bayes}_\infty\Dcal_m \sum_{1\le i\ne j\le n}(\omega_{i,j}')^2
\end{align*}
by Eq~\eqref{eq.maj-E[Um^2]}. 
Algebraic computations and Eq.~\eqref{eq.maj-carre-omegaprimeij} and~\eqref{eq.ineg-simple-piij-pi} show that
\begin{align*}
\sum_{1\le i\ne j\le n}(\omega_{i,j}')^2
&\leq \frac{1}{(n-p)^4} \sum_{1\le i\ne j\le n} \crochj{ 1 + \frac{n^2}{p^2} (\pi_i + \pi_j)^2 }
\\
&\leq \frac{1}{(n-p)^4} \sum_{1\le i, j\le n} \crochj{ 1 + \frac{n^2}{p^2} (\pi^* \pi_i + \pi^* \pi_j + 2 \pi_i \pi_j) }
\\
&= \frac{n^2}{(n-p)^4} \parenj{ 3 + \frac{2 \pi^* n}{p}  }
\end{align*}
Hence,
\begin{equation*} %%\label{eq:Control.oA.Gal}
 \oA\le \frac{n}{(n-p)^2}\sqrt{\parenj{\frac{3}{2}+\frac{\pi^* n}{p}}\norms{\bayes}_\infty\Dcal_m}
 \enspace.
\end{equation*}
Second, let $a_i$ and $b_j$ be functions such that $\sum_{i=1}^n\E\crochj{a_i(\xi)^2}\le 1$ and $\sum_{i=1}^n\E\crochj{b_i(\xi)^2}\le 1$. 
Eq~\eqref{eq.maj-terme-B-Um} shows that 
\[ \absB{ \E\crochj{a_i(\xi)b_j(\xi')U_m(\xi,\xi')} } \le \frac{\norms{\bayes}_\infty}2\parenB{\E\crochj{a_i(\xi)^2}+\E\crochj{b_j(\xi)^2}}\enspace,\]
hence, using Eq.~\eqref{eq.maj-max-omegaprimeij}, 
\begin{align}
\notag 
\oB &= 
\sum_{1\le i\ne j\le n}\omega_{i,j}'\E\crochj{a_i(\xi)b_j(\xi')U_m(\xi,\xi')}
\\
\notag 
&\le \max_{1\le i\ne j\le n} \absj{ \omega'_{i,j} } \frac{\norms{\bayes}_\infty}{2} \sum_{1\le i\ne j\le n} \parenB{\E\crochj{a_i(\xi)^2}+\E\crochj{b_j(\xi)^2}} 
\\
&\le \frac{ n \norms{\bayes}_\infty  }{(n-p)^2} \parenj{ 1 + \frac{2n}{p} \pi^*}  
\enspace . 
\notag %%\label{eq:Control.oB.gal}
\end{align}
Third, Eq~\eqref{eq.maj-terme-C-Um} shows that, for any $x>0$,
\[
\E\crochj{U_m(\xi,x)^2}\le 2 \parenB{ b_m^2+\norms{\bayes_m}^2 } \norms{\bayes}_\infty
\]
and by Eq.~\eqref{eq.maj-carre-omegaprimeij} we have 
\begin{align*}
 \sum_{i=2}^n (\omega'_{i,1})^2
\leq 
\frac{1}{(n-p)^4} \sum_{i=2}^n \parenj{ 1 + (\pi_i+\pi_1)^2 \frac{n^2}{p^2} }
\leq \frac{n}{(n-p)^4} \parensq{ 1 + 2\pi^*\frac{n}{p} }
\enspace.
\end{align*}
So, for any $x>0$, 
\begin{align*}
 \sum_{i=2}^n (\omega'_{i,1})^2\E\crochj{U_m(\xi,x)^2}
&\le 
2 \parenB{ b_m^2+\norms{\bayes_m}^2 } \norms{\bayes}_\infty
\times \frac{1}{(n-p)^4} \parensq{ 1 + 2\pi^*\frac{n}{p} }
\end{align*}
hence
\begin{equation*} %%\label{eq:Controle.oC.gal}
 \oC \le\parenj{1+2\pi^*\frac np}\frac{ n}{(n-p)^2} 
 \sqrt{\frac{2 \parenb{ b_m^2+\norms{\bayes_m}^2 } \norms{\bayes}_\infty}{n}} 
 \enspace .
\end{equation*}
Fourth, using Eq~\eqref{eq.maj-sup-Um} and~\eqref{eq.maj-max-omegaprimeij}, 
\begin{equation*} %%\label{eq:Controle.oD.gal} 
\oD \leq
\max_{i,j \in \inter{n}} \absj{ \omega'_{i,j} } \sup_{x,y} \absb{ U_m(x,y) } 
\leq 
\parenj{1 + \frac{2n}{p} \pi^*} \frac{n}{(n-p)^2} \frac{2 \parenb{ b_m^2+\norms{\bayes_m}^2 }}{n}
\enspace.
\end{equation*}

%%% 8. Application du Theoreme U stats
%
Now, we remark that $b_m^2 \leq n$ by \eqref{hyp.NormSupNorm2}, 
and $\norms{\bayes_m}^2 \leq \norms{\bayes}^2 \leq \norms{\bayes}_{\infty} \leq A$, 
and we can plug this two inequalities in the upper bounds above. 
By \eqref{hyp.CV.ind}, we can apply \citet[Theorem~3.4]{Hou_Rey:2003}, conditionally on the weights $\omega_{i,j}$. 
We obtain that an absolute constant $\kappa>0$ exists such that, 
for any $x>1$, with probability larger than $1-6\e^{-x}$, for any $\epsilon\in(0,1]$, 
Eq.~\eqref{eq:Conc.UStat.Gen} holds true. 
\end{proof}

\subsubsection{Oracle Inequality (Proof of Theorem~\ref{thm.oracle-MCCV})} \label{sec.supmat.MCCV.oracle}

Theorem~\ref{thm.oracle-MCCV} actually is a corollary of the following general result. 
\begin{theorem} \label{thm.oracle-CV-gal}
Let $\xi_{\inter{n}}$ be i.i.d. real-valued random variables with common density $\bayes \in L^{\infty}(\mu)$, 
$(T_K)_{1 \leq K \leq B}$ some sequence of subsets of $\inter{n}$ satisfying \eqref{hyp.CV.same-size} and \eqref{hyp.CV.ind}, 
and $(S_m)_{m\in\M_n}$ be a collection of separable linear spaces satisfying \eqref{hyp.NormSupNorm2}. 
Assume that either \eqref{hyp.UBbayes} or \eqref{hyp.Nested} holds true.
For every $\mM_n$, let $\ERM_m$ be the estimator defined by Eq.~\eqref{def.ERM}, 
and $\widetilde{s} = \ERM_{\mh}$ where 
\[ 
\mh \in \argmin_{\mM_n} \setB{ \critCV \parenb{ m , (T_K)_{1 \leq K \leq B} } }
\]
and $\critCV$ is defined by Eq.~\eqref{def.CVgal}. 
Define $\pi^*=\max_{i=1,\ldots,n}\frac{1}{B} \sum_{K=1}^B\un_{i\in T_K^c}$ 
and 
for any $x,\epsilon,\kappa>0$,
\begin{equation*}
\resteCVgal \parenj{\epsilon,x,\kappa,n,\fractrain_n,\pi^*,A}
\egaldef 
%%R_{\pi^*,p}(x,\epsilon)
%%+ \frac{\kappa }{n \fractrain_n^2 } \parenj{1+\frac{\pi^*}{1-\fractrain_n}} \crochj{\frac{ A x}{ \fractrain_n \epsilon} + \frac{(A \vee 1) x^2}{\epsilon^3}} 
%%\leq 
\frac{\kappa }{n \fractrain_n^2 } \mathopen{} \left( {1+\frac{\pi^*}{1-\fractrain_n}} \right)^{\alpha} \mathclose{} \crochj{\frac{ A x}{ \fractrain_n \epsilon} + \frac{(A \vee 1) x^2}{\epsilon^3}} 
\end{equation*}
with $\alpha = 1$ under assumption \eqref{hyp.UBbayes}
and $\alpha = 2$ under assumption \eqref{hyp.Nested}. 
Then, an absolute constant $\kappa>0$ exists such that, for any $x\geq 0$, with probability at least $1- 12|\M_n|^2 \e^{-x}$, for any $\epsilon \in (0,\kappa^{-1})$, 
\begin{equation*}
\parenj{1-\frac{\epsilon}{\fractrain_n}}\perte{\widetilde{s}}
\leq
\frac{1+\epsilon}{\fractrain_n} \inf_{\mM_n} \setj{ \perte{\ERM_{m}} }
+ \resteCVgal \parenj{\epsilon,x,\kappa,n,\fractrain_n,\pi^*,A}
\enspace . 
\end{equation*}
\end{theorem}

The oracle inequality of Theorem~\ref{thm.oracle-CV-gal} 
is similar to the one of Theorem~\ref{thm.oracle-penVF.cas_reel}, 
with $\delta$ replaced by $1/\fractrain_n - 1$ (both quantities correspond to the bias 
of the criterion as an estimator of the risk) 
and a slightly different remainder term. 
In addition to the remarks already made about Theorem~\ref{thm.oracle-penVF.cas_reel}, 
we can make the following comments. 
\begin{itemize}
\item 
The remainder term $\resteCVgal$ is of order $x^2 / n$,  
as in Theorem~\ref{thm.oracle-penVF.cas_reel}  
under the following sufficient conditions: 
(i) $\fractrain_n$ stays away from $0$, 
(ii) $\pi^* / (1-\fractrain_n)$ is bounded. 

\item 
For $V$-fold criteria, $\fractrain_n=(V-1)/V\ge 1/2$ and $\pi^*/(1-\fractrain_n) = 1$, 
so conditions (i) and (ii) are satisfied and we recover an oracle inequality 
for $V$-fold cross-validation similar to Theorem~\ref{thm.oracle-penVF.cas_reel}. 

\item 
The leading constant in front of the oracle inequality of Theorem~\ref{thm.oracle-CV-gal} 
is of order $1/\fractrain_n$, 
so we can get asymptotic optimality only if $\fractrain_n\to 1$, 
that is, $p\ll n$.
This is consistent with the fact that the bias of the cross-validation criterion is 
negligible at first order if and only if $\fractrain_n \to 1$. 

\item 
For hold-out criteria, $\pi^* = 1$ so the remainder term is 
of order $x^2 / (n(1-\fractrain_n)^{\alpha}) \geq x^2 / p$
which is large when $\fractrain_n$ is close to~$1$, that is, when $p$ is small. 
Hence, for such criteria, we cannot get a leading constant close to~$1$ and a ``small'' remainder term. 
\end{itemize}

Let us now explain why Theorem~\ref{thm.oracle-MCCV} is also a corollary of Theorem~\ref{thm.oracle-CV-gal}. 

\medbreak

\begin{proofof}{Theorem~\ref{thm.oracle-MCCV}} 
We only have to prove some upper bound on $\pi^*$ under assumption \eqref{hyp.MCCV}, 
thanks to which Theorem~\ref{thm.oracle-MCCV} is a straightforward corollary 
of  Theorem~\ref{thm.oracle-CV-gal}. 

By \eqref{hyp.CV.same-size} and~\eqref{hyp.MCCV}, 
for any $i\in \inter{n}$, $\pi_i$ is the empirical mean of $K$ independent Bernoulli random variables with common parameter $\P\parenj{i\in T_K^c} = p/n$. 
Then, by Bernstein's inequality \citep[Theorem~2.10]{Bou_Lug_Mas:2011:livre} 
\[
\forall y>0,\forall i\in\inter{n},\qquad 
\P\parenj{ \pi_i-\frac{p}{n} > \sqrt{\frac{2 p (n-p) y}{n^2 B}}+\frac{x}{3B}}\le \e^{-y}
\enspace.\]
A union bound over $i \in \inter{n}$ yields that 
for any $x>0$,
\begin{equation} \notag %\label{eq:upper_bound_pi*}
\P\parenj{\pi^*\le1\wedge\parenj{ \frac{2p}{n}+\frac{\log n+x}{B} }}\ge 1-\e^{-x}\enspace ,
\end{equation}
where we used also that $\pi^* \leq 1$ almost surely. 
Theorem~\ref{thm.oracle-MCCV} follows. 
\end{proofof}

We finally prove Theorem~\ref{thm.oracle-CV-gal}.

\medbreak

\begin{proofof}{Theorem~\ref{thm.oracle-CV-gal}}
Throughout the proof, $L$ denotes some positive numerical constant, whose value may change from line to line. 
Given Eq.~\eqref{eq.critCV.calcul.bis}, 
the proof relies on concentration inequalities that are detailed in Section~\ref{sec.supmat.MCCV.conc}. 
Let us fix $x \geq 0$ and define for every $\kappa\geq 1$ the event $ \Omega_{good}(\kappa,x)$ where all the following inequalities hold 
for any $m,m' \in \M_n$ and any $\epsilon\in(0,1]$
\begin{align*}
&\absj{\sum_{i=1}^n\omega_{i,i} \parenb{ \Psi_m(\xi_i)-\Dcal_m }}
\le \epsilon\frac{\Dcal_m}{n-p}+\kappa\frac{(n+A)x}{\epsilon(n-p)^2}
\\
& \sum_{i=1}^n (-2\omega_{i,i}+\cteformCVb_i) \parenb{ \bayes_m(\xi_i)-P\bayes_m-\bayes_{m'}(\xi_i)-P\bayes_{m'} }
\le \epsilon\norms{\bayes_m-\bayes_{m'}}^2 + \RbVFgal_n(x,\epsilon,\pi^*,A) 
\\
&\absj{\sum_{1 \leq i\neq j \leq n} \omega_{i,j} U_m(\xi_i , \xi_j)}
\le \epsilon \frac{\Dcal_m}{n-p} + \kappa\frac{n}{(n-p)^2} \parenj{1+\pi^*\frac{n}{p}} \crochj{\frac{n A x}{(n-p)\epsilon} + \parenj{n+A}\frac{x^2}{n}} \\
& \absj{\norms{\ERM_m-\bayes_m}^{2}-\frac{\Dcal_{m}}{n}} 
\le \epsilon\frac{\Dcal_m}{n} + \kappa \frac{A x^2}{\epsilon^3 n}
\enspace . 
\end{align*}
It follows from Lemmas~\ref{lem:conc.supPn-P}, \ref{le.eq:TermePsiCritVFGen}, \ref{le.conc.TermeBiaisCritVFGen} and~\ref{le.eq:Conc.UStat.Gen} 
that an absolute constant $\kappa>0$ exists such that $\P\parens{\Omega_{good}(\kappa,x)}\ge 1 - |\M_n|^2 \e^{-x} - 10 |\M_n| \e^{-x}$. 
Let us remark that we can assume $x \geq \log(11) \geq 1$ in the following, since otherwise the above probability bound is negative. 
On $\Omega_{good}(\kappa,x)$, for every $\mM_n$ and $\epsilon \in (0,1)$, 
\begin{equation}
\label{eq.Omgood.controle-err-estim}
\frac{\Dcal_m}{n} \leq \frac{1}{1 - \epsilon} \norms{\ERM_m-\bayes_m}^{2} + \frac{L A x^2}{\epsilon^3 (1-\epsilon) n}
\enspace . 
\end{equation}
By definition of $\mh$, for every $\mM_n$, 
\begin{equation} \label{eq.pr.Oracle.Gal.debut}
\perte{\ERM_{\mh}} 
\leq \perte{\ERM_m} 
+ \parenB{ \critCV(m) - \perte{\ERM_m} }
- \parenB{ \critCV(\mh) - \perte{\ERM_{\mh}} }
\enspace . 
\end{equation}
In addition, by Eq.~\eqref{eq.critCV.calcul.bis}, 
\begin{align}
\critCV(m) - \perte{\ERM_m}
\notag %%\label{eq.pr.Oracle.Gal.1}
&= 
\frac{\Dcal_m - \norms{\bayes_m}^2}{n-p} 
- \underbrace{ \parenj{ \norms{\bayes_m}^2 + \pertes{\bayes_m} } }_{= \norms{\bayes}^2}
- \frac{\Dcal_m}{n}
\\
\notag 
&+ \sum_{i=1}^n \omega_{i,i} \parenb{ \Psi_m(\xi_i)-\Dcal_m }
+ \sum_{i=1}^n (-2\omega_{i,i}+\cteformCVb_i) \parenb{ \bayes_m(\xi_i)-P\bayes_m } 
\\
&+ \sum_{1 \leq i\neq j \leq n} \omega_{i,j} U_m(\xi_i , \xi_j) 
- \parenj{ \perte{\ERM_m} - \frac{\Dcal_m}{n} }
\enspace . 
\notag 
\end{align}
So, on $\Omega_{good}(\kappa,x)$, for every $m,m' \in \M_n$ and $\epsilon \in (0,1/5)$, 
\begin{align*}
&\qquad 
\critCV(m) - \perte{\ERM_m} 
- \parenB{ \critCV(m') - \perte{\ERM_{m'}} }
\\ &
\leq 
\Dcal_m \mathopen{} \left( { \frac{1 + 2 \epsilon}{n-p}  - \frac{ 1 - \epsilon }{n} } \right)_{\!\!+} \mathclose{} 
+ \Dcal_{m'} \mathopen{} \left( { \frac{1 + \epsilon}{n}  - \frac{ 1 - 2\epsilon }{n-p} } \right)_{\!\!+} \mathclose{} 
+ \epsilon\norms{\bayes_m-\bayes_{m'}}^2
\\ & \qquad 
+ \RbVFgal_n(x,\epsilon,\pi^*,A)
+ \frac{L n}{(n-p)^2} \parenj{1+\pi^*\frac{n}{p}} \parenj{\frac{n A x}{(n-p)\epsilon} + \frac{(A \vee 1) x^2}{\epsilon^3}} 
+ \frac{\norms{\bayes_{m'}}^2}{n-p}
\\ &
\leq 
\frac{n}{1-\epsilon} \mathopen{} \left( { \frac{1 + 2 \epsilon}{n-p}  - \frac{ 1 - \epsilon }{n} } \right)_{\!\!+} \mathclose{} \norms{\ERM_m - \bayes_m}^2 
+ \frac{n}{1-\epsilon}  \mathopen{} \left( { \frac{1 + \epsilon}{n}  - \frac{ 1 - 2\epsilon }{n-p} } \right)_{\!\!+} \mathclose{} \norms{\ERM_{m'} - \bayes_{m'}}^2
\\ & \qquad 
+ 2 \epsilon \perte{\bayes_m}
+ 2 \epsilon \perte{\bayes_{m'}}
\\ & \qquad 
+ \RbVFgal_n(x,\epsilon,\pi^*,A)
+ \frac{L n}{(n-p)^2} \parenj{1+\pi^*\frac{n}{p}} \parenj{\frac{n A x}{(n-p)\epsilon} + \frac{(A \vee 1) x^2}{\epsilon^3}} 
\\ &
\leq 
\max\setj{ 
\frac{1}{1-\epsilon} \mathopen{} \left( { \frac{1 }{\fractrain_n}  - 1 + \epsilon + \frac{ 2\epsilon }{\fractrain_n} } \right)_{\!\!+} \mathclose{} 
, 2 \epsilon}
\pertes{\ERM_m}
\\ & \qquad 
+ 
\max\setj{ 
\frac{1}{1-\epsilon}  \mathopen{} \left( { 1  - \frac{ 1 }{\fractrain_n} + \epsilon + \frac{ 2\epsilon }{\fractrain_n} } \right)_{\!\!+} \mathclose{} 
, 2 \epsilon}
\perte{ \ERM_{m'} }
\\ & \qquad 
+ \RbVFgal_n(x,\epsilon,\pi^*,A)
+ \frac{L }{n \fractrain_n^2 } \parenj{1+\frac{\pi^*}{1-\fractrain_n}} \parenj{\frac{ A x}{ \fractrain_n \epsilon} + \frac{(A \vee 1) x^2}{\epsilon^3}} 
\\ &
\leq 
\parenj{ \frac{1 }{\fractrain_n}  - 1 + \frac{L \epsilon}{\fractrain_n} } \perte{\ERM_m }
+ \frac{4 \epsilon}{\fractrain_n} \perte{\ERM_{m'} }
\\ & \qquad 
+ \RbVFgal_n(x,\epsilon,\pi^*,A)
+ \frac{L }{n \fractrain_n^2 } \parenj{1+\frac{\pi^*}{1-\fractrain_n}} \parenj{\frac{ A x}{ \fractrain_n \epsilon} + \frac{(A \vee 1) x^2}{\epsilon^3}} 
\end{align*}
where we used Eq.~\eqref{eq.Omgood.controle-err-estim} for the second inequality. 
Note also that by Lemma~\ref{le.conc.TermeBiaisCritVFGen}, 
\begin{align*}
\RbVFgal_n(x,\epsilon,\pi^*,A) 
& 
\leq 
\frac{ \kappa  }{ n \epsilon} \crochj{ 
Ax \parenj{ \frac{1}{n \fractrain_n^2} + \frac{\pi^{\star}}{1-\fractrain_n} } 
+ x^2 \parensq{ \frac{1}{n \fractrain_n^2} + \frac{\pi^{\star}}{1-\fractrain_n} } 
}
\\ &
\leq 
\frac{ \kappa  }{ n \epsilon} \crochj{ 
A x \parenj{ \frac{1}{\fractrain_n} + \frac{\pi^{\star}}{1-\fractrain_n} } 
+ x^2 \parensq{ \frac{1}{\fractrain_n} + \frac{\pi^{\star}}{1-\fractrain_n} } 
}
\\ &
\leq 
\frac{ \kappa  }{ n \epsilon} \crochj{ 
 \frac{Ax}{\fractrain_n} \parenj{ 1 + \frac{\pi^{\star}}{1-\fractrain_n} } 
+  \frac{x^2}{\fractrain_n^2} \parensq{ 1 + \frac{\pi^{\star}}{1-\fractrain_n} } 
}
\end{align*}
where the term 
$\frac{x^2}{\fractrain_n^2} \parens{ 1 + \frac{\pi^{\star}}{1-\fractrain_n} }^2 $
is not present in $\RbVFgal_n(x,\epsilon,\pi^*,A) $ 
under assumption \eqref{hyp.UBbayes}, 
so that 
$\RbVFgal_n(x,\epsilon,\pi^*,A) \leq \resteCVgal \parenj{\epsilon,x,\kappa,n,\fractrain_n,\pi^*,A}$ 
whatever the assumption among \eqref{hyp.UBbayes} and \eqref{hyp.Nested}. 

Therefore, Eq.~\eqref{eq.pr.Oracle.Gal.debut} yields that, 
on $\Omega_{good}(\kappa,x)$, for every $m,m' \in \M_n$ and $\epsilon \in (0,1/5)$, 
\begin{align*}
\parenj{ 1 - \frac{4 \epsilon}{\fractrain_n} } \perte{\ERM_{\mh}} 
&\leq 
\frac{ 1  + L \epsilon}{\fractrain_n} \perte{\ERM_m}
+ \RbVFgal_n(x,\epsilon,\pi^*,A)
\\ & \qquad 
+ \frac{L }{n \fractrain_n^2 } \parenj{1+\frac{\pi^*}{1-\fractrain_n}} \parenj{\frac{ A x}{ \fractrain_n \epsilon} + \frac{(A \vee 1) x^2}{\epsilon^3}} 
\end{align*}
hence the result by changing $\epsilon$ into $\epsilon /L$. 
\end{proofof}

\subsubsection{Variance (Proof of Theorem~\ref{thm.var-MCCV})} \label{sec.supmat.MCCV.var}

We prove in this section the variance computation of Theorem~\ref{thm.var-MCCV}, which is a 
straightforward corollary of the following result, 
since $\Psi_{m_1}$ and $\Psi_{m_2}$ are constant for regular histogram models. 
\begin{theorem} \label{thm.var-MCCV-gal}
We consider the setting and notation of Theorem~\ref{theo.variance.penVF}. 
We recall that 
\[ 
\CVMCCV(m) = \critCV \parens{ m , (T_K)_{1 \leq K \leq B} }
\]
 for some sequence 
$T_1, \ldots, T_B$ of subsets of $\inter{n}$ satisfying 
\eqref{hyp.CV.same-size}, \eqref{hyp.MCCV} and \eqref{hyp.CV.ind}, 
where $\critCV$ is defined by Eq.~\eqref{def.CVgal}
Then, we have 
\begin{align}
\notag 
&\quad \var\parenB{ \CVMCCV(m_1) - \CVMCCV(m_2) } 
\\
\label{eq.var.MCCV.incr}
&= 
\cteMCCVa (B,n,\fractrain_n) 
\frac{2}{n^2} \termeBvaracr{m_1,m_2}
\\
\notag 
&\quad 
+ 
\frac{4}{B n} \frac{1}{n^2 \fractrain_n^3} \var\parenj{ ( \bayes_{m_1} - \bayes_{m_2} )(\xi_1) - \frac{1}{2}(\Psi_{m_1} - \Psi_{m_2}) (\xi_1)}
\\
\notag 
&\quad 
 + \frac{4}{B n} \frac{1}{1-\fractrain_n} \var\parenb{\bayes_{m_1} (\xi_1) - \bayes_{m_2}(\xi_1)}
\\
\notag 
&\quad 
+ \parenj{ 1 - \frac{1}{B} } 
\frac{4}{n} \var \parenj{ \parenj{1+\frac{1}{n \fractrain_n}} (\bayes_{m_1}-\bayes_{m_2}) (\xi_1)-\frac{1}{2 n \fractrain_n}(\Psi_{m_1}-\Psi_{m_2})(\xi_1)}
\end{align}
and 
\begin{align}
\var\parenB{ \CVMCCV(m_1)} 
\label{eq.var.MCCV}
&= 
\cteMCCVa (B,n,\fractrain_n) 
\frac{2}{n^2} \termeBvar{m_1 , m_1}
\\
\notag 
&+ \frac{1}{B} \frac{4}{n} \crochj{ 
\frac{1}{n^2 \fractrain_n^3} \var\parenj{ \bayes_{m_1} (\xi_1) - \frac{1}{2} \Psi_{m_1} (\xi_1)}
 + \frac{1}{1-\fractrain_n} \var\parenb{\bayes_{m_1} (\xi_1)}
}
\\
\notag 
&\quad 
+ \parenj{ 1 - \frac{1}{B} } 
\frac{4}{n} \var \parenj{ \parenj{1+\frac{1}{n \fractrain_n}} \bayes_{m_1} (\xi_1)-\frac{1}{2 n \fractrain_n} \Psi_{m_1}(\xi_1)}
\end{align}
where 
\begin{align*}
\cteMCCVa  (B,n,\fractrain_n) 
&= 
\frac{1}{B}  \parenj{ \frac{1}{\fractrain_n^2} + \frac{2}{\fractrain_n (1-\fractrain_n)} - \frac{1}{n \fractrain_n^3} } 
+ \parenj{ 1 - \frac{1}{B} } \crochj{1+\frac{1}{n-1} \parensq{ \frac{1}{\fractrain_n} + 1 } -\frac{1}{n \fractrain_n^2}}   
\end{align*}
and we recall that $\fractrain_n = |T_K|/n = 1 - (p/n)$. 
\end{theorem}
Theorem~\ref{thm.var-MCCV-gal} is proved below. 
Note that a similar argument can be used for computing the variance of Monte-Carlo penalized criteria, 
where the Monte-Carlo penalty is defined from hold-out penalties similarly to MCCV. 
Indeed, given Lemma~\ref{le.var-gal}, we only need to compute the variance of 
hold-out penalized criteria (as done by Proposition~\ref{pro.variance.penHO}) 
and the variance of leave-$p$-out criteria (as done by combining Lemma~\ref{le.penVF-VFCV}
and Theorem~\ref{theo.variance.penVF}). 

\medbreak

Before proving Theorem~\ref{thm.var-MCCV-gal}, we state and prove a general result that 
relates the variance of (increments of) Monte-Carlo CV criteria 
to the variance of (increments of) hold-out and leave-$p$-out criteria. 

\begin{lemma}\label{le.var-gal}
Let $n \geq p \geq 1$ and $F: \X^n \times \mathfrak{P}(\inter{n}) \to \R$ be some mesurable function. 
Let $D_n$ denote some sample of $n$ independent variables with common distribution $P$. 
Assume that \eqref{hyp.CV.same-size}, \eqref{hyp.CV.ind} and \eqref{hyp.MCCV} hold true, 
as well as 
\begin{equation}
\label{hyp.F-mom}
\forall T \in \mathcal{E}_{n-p}, \qquad \E\crochj{ F(D_n,T)^2 } < +\infty
\enspace . 
\end{equation}

Let $B \geq 1$ and $T_1, \ldots, T_B$ 
be some random sequence of subsets of $\inter{n}$. 
Let us define 
\[ 
Z_B \egaldef \frac{1}{B} \sum_{K=1}^B F(D_n, T_K)
\qquad \text{and} \qquad 
F^{\mathrm{lpo}}(D_n,p) \egaldef \frac{1}{\binom{n}{p}} \sum_{T \in \mathcal{E}_{n-p}} F(D_n,T)
\enspace . 
\]
Then, we have 
\begin{align}\label{eq.le.var-gal.MCCV}
\var(Z_B) 
&= \var\parenb{ F^{\mathrm{lpo}}(D_n,p) } + \frac{1}{B} \E\crochB{ \var \parenb{ F(D_n,T_1) \sachant D_n } }
\\
\label{eq.le.var-gal.MCCV.2}
&= \var\parenb{ F^{\mathrm{lpo}}(D_n,p) } + \frac{1}{B} \crochB{ \var\parenb{ F(D_n,T_1) } - \var\parenb{ F^{\mathrm{lpo}}(D_n,p) } }
\\
\notag 
&= \parenj{ 1 - \frac{1}{B}  } \var\parenb{ F^{\mathrm{lpo}}(D_n,p) } + \frac{1}{B} \var\parenb{ F(D_n,T_1) } 
\enspace . 
\end{align}
\end{lemma}

\begin{proofof}{Lemma~\ref{le.var-gal}} 
By \eqref{hyp.F-mom}, $Z_B$ admits a finite variance. 
Then, we can write that 
\begin{align*}
\var\parens{ Z_B } 
&= \var\parenb{ \E\crochs{Z_B \sachant D_n} } + \E\crochb{ \var\parens{Z_B \sachant D_n}}
\enspace . 
\end{align*}
By \eqref{hyp.CV.same-size}, \eqref{hyp.MCCV} and \eqref{hyp.CV.ind}, 
\[ 
\E\crochj{Z_B \sachant D_n} = F^{\mathrm{lpo}} (D_n,p) 
\qquad \text{and} \qquad 
\var\parenj{Z_B \sachant D_n} = \frac{1}{B} \var\parenb{ F(D_n, T_1) \sachant D_n }
\]
which proves Eq.~\eqref{eq.le.var-gal.MCCV}. 
Eq.~\eqref{eq.le.var-gal.MCCV.2} follows by remarking that $Z_B = F(D_n,T_1)$ when $B=1$. 
\end{proofof}

\medbreak

We can now prove Theorem~\ref{thm.var-MCCV-gal}. 
The idea is to apply Lemma~\ref{le.var-gal} when $F(D_n,T)$ is the hold-out estimator of the risk of $\ERM_m$, 
so that $F^{\mathrm{lpo}}(D_n,p)$ corresponds to some leave-$p$-out estimator of the risk. 
Similarly, Lemma~\ref{le.var-gal} applies when $F(D_n,T)$ is the difference between the hold-out estimators of 
the risks of $\ERM_{m_1}$ and $\ERM_{m_2}$. 
\medbreak
\begin{proofof}{Theorem~\ref{thm.var-MCCV-gal}} 
First note that the variance of the CV criterion at model $m_1$ 
can be deduced from the variance of the increment between 
the CV criterion at model $m_1$ and CV criterion at model $m_0$ with $S_{m_0} = \{0\}$ 
the null model. 
So, Eq.~\eqref{eq.var.MCCV} directly follows from Eq.~\eqref{eq.var.MCCV.incr}. 

For proving Eq.~\eqref{eq.var.MCCV.incr}, we apply Lemma~\ref{le.var-gal} with 
\[ 
F(D_n,T) = 
F_{\HO,m_1,m_2}(D_n,T) \egaldef \critHO(m_1, T) - \critHO(m_2, T)
\]
so that 
\[ 
F_{\HO,m_1,m_2}^{\mathrm{lpo}}(D_n,p) 
= \CV_{ \parenj{ \frac{n/p - 1/2}{ n/p - 1},\B_{{\rm LOO}} } } (m_1) 
- \CV_{ \parenj{ \frac{n/p - 1/2}{ n/p - 1},\B_{{\rm LOO}} } } (m_2)
\]
by Eq.~\eqref{eq.le.penLOO-LPO} in Lemma~\ref{le.penVF-VFCV}. 

The variance of $F_{\HO,m_1,m_2}^{\mathrm{lpo}}(D_n,p)$ 
is given by Theorem~\ref{theo.variance.penVF}:  
by Eq.~\eqref{eq.pro.variance.penVF-penid-incrReg} with 
$V=n$ and 
\[ 
C = \frac{n/p - 1/2}{ n/p - 1} = 1 + \frac{p}{2(n-p)}
\enspace , 
\]
we get  
\begin{align}
\notag 
&\qquad \var \parenj{  F_{\HO,m_1,m_2}^{\mathrm{lpo}}(D_n,p)  }
\\
&= \frac{2}{n^2} \crochj{1+\frac{4}{n-1} \parensq{ 1 + \frac{p}{2(n-p)} } -\frac{n}{(n-p)^2}} \termeBvaracr{m_1 , m_2 }
\\
\notag 
&\qquad +\frac{4}{n} \var \parenj{ \parenj{1+\frac{1}{n-p}} (\bayes_{m_1}-\bayes_{m_2}) (\xi_1)-\frac{1}{2(n-p)}(\Psi_{m_1}-\Psi_{m_2})(\xi_1)}
\\
\label{eq.var.critLPO.incr}
&= \crochj{1+\frac{1}{n-1} \parensq{ \frac{1}{\fractrain_n} + 1 } -\frac{1}{n \fractrain_n^2}}
\frac{2}{n^2} \termeBvaracr{m_1,m_2}
\\
\notag 
&\qquad + \frac{4}{n} \var \parenj{ \parenj{1+\frac{1}{n \fractrain_n}} (\bayes_{m_1}-\bayes_{m_2}) (\xi_1)-\frac{1}{2 n \fractrain_n}(\Psi_{m_1}-\Psi_{m_2})(\xi_1)}
\end{align} 
where we recall that $\fractrain_n = |T|/n = 1 - (p/n)$. 

\medbreak

It now remains to compute the variance of 
\[ 
F_{\HO,m_1,m_2}(D_n,T) \egaldef \critHO(m_1, T) - \critHO(m_2, T)
\enspace . 
\] 
By Eq.~\eqref{eq.critHO.calcul.2}, 
$F_{\HO,m_1,m_2}(D_n,T)$ has the same variance as $\CV_{m_1} - \CV_{m_2}$ where 
$\CV_m$ is defined as in Lemma~\ref{lem:CovGen.2} with 
\[
\olomega_{i,j} = \omegaHO_{i,j}(T) 
\qquad 
\olcteformCVb_{i} = \cteformHOb_{i}(T) 
\qquad \text{and} \qquad 
f_m = \bayes_m
\enspace . 
\]
Since $|T|=n-p$, we have 
\begin{align*}
\sum_{1\leq i\neq j\leq n} \parenb{ \olomega_{i,j}^2 + \olomega_{i,j}\olomega_{j,i}  }
&= 
\sum_{i,j \in T, \, i \neq j} \parenb{ \olomega_{i,j}^2 + \olomega_{i,j}\olomega_{j,i}  }
+ \sum_{1 \leq i \leq n, j \in T^c, i\neq j} \parenb{ \underbrace{\olomega_{i,j}^2 + \olomega_{i,j}\olomega_{j,i}}_{=0}  }
\\&\quad + \sum_{i \in T^c , \, j \in T} \parenb{ \olomega_{i,j}^2 + \olomega_{i,j}\underbrace{\olomega_{j,i}}_{=0}  }
\\
&= \frac{2 (n-p-1)}{(n-p)^3} 
+ \frac{4}{p(n-p)}
\end{align*}
and
\[ 
\sum_{i=1}^n\olomega_{i,i}^2
= \frac{1}{(n-p)^3}
\qquad 
\sum_{i=1}^n\olomega_{i,i} \olcteformCVb_i 
= 0
\qquad 
\sum_{i=1}^n \olcteformCVb_i^2
= \frac{4}{p}
\enspace . 
\]
Therefore, by Lemma~\ref{lem:CovGen.2}, 
\begin{align}
\notag 
&\qquad 
\var\parenb{ \crit_{\HO}(m_1,T) - \crit_{\HO}(m_2, T)}
\\
\notag 
&= 
\parenj{ \frac{2 (n-p-1)}{(n-p)^3} + \frac{4}{p(n-p)}  }
\var\parenb{U_{m_1}(\xi_1,\xi_2)-U_{m_2}(\xi_1,\xi_2)}
\\
\notag 
 &\qquad + \frac{1}{(n-p)^3} 
\var\parenb{U_{m_1}(\xi_1,\xi_1)-U_{m_2}(\xi_1,\xi_1)}
 + \frac{4}{p}
\var\parenb{\bayes_{m_1} (\xi_1) - \bayes_{m_2}(\xi_1)}
\\
\notag 
&= 
\parenj{ \frac{2 (n-p-1)}{(n-p)^3} + \frac{4}{p(n-p)}  }
\termeBvaracr{m_1,m_2}
\\
\notag 
 &\qquad + \frac{1}{(n-p)^3} 
\var\parenb{ (\Psi_{m_1} - \Psi_{m_2}) (\xi_1) - 2 ( \bayes_{m_1} - \bayes_{m_2} )(\xi_1) }
 + \frac{4}{p}
\var\parenb{\bayes_{m_1} (\xi_1) - \bayes_{m_2}(\xi_1)}
\\
\label{eq.var.critHO.incr}
&= 
\parenj{ \frac{1}{\fractrain_n^2} + \frac{2}{\fractrain_n (1-\fractrain_n)} - \frac{1}{n \fractrain_n^3} } 
\frac{2}{n^2} \termeBvaracr{m_1,m_2}
\\
\notag 
 &\qquad + 
 \frac{4}{n} 
\frac{1}{n^2 \fractrain_n^3} \var\parenb{ ( \bayes_{m_1} - \bayes_{m_2} )(\xi_1) - \frac{1}{2}(\Psi_{m_1} - \Psi_{m_2}) (\xi_1)}
\\
\notag 
 &\qquad + 
 \frac{4}{n} 
\frac{1}{1-\fractrain_n} \var\parenb{\bayes_{m_1} (\xi_1) - \bayes_{m_2}(\xi_1)}
\end{align}
where we used that $\var \parenj{ U_{m_1}(\xi_1,\xi_2)-U_{m_2}(\xi_1,\xi_2)}=\termeBvaracr{m_1,m_2}$ as proved at the beginning of Section~\ref{sec.proof.variance.main}, 
and that $U_{m}(\xi_1,\xi_1) = \Psi_m(\xi_1) - 2 \bayes_m(\xi_1) + \norms{\bayes_m}^2$. 

\medbreak

Combining Eq.~\eqref{eq.var.critHO.incr} and~\eqref{eq.var.critLPO.incr} with Lemma~\ref{le.var-gal}, 
we get Eq.~\eqref{eq.var.MCCV.incr}. 
\end{proofof}

%%%%%%%%%%%%%%%%%%%%%%%%%%%%%%%%%%%%%%%%%%%%%%%%%%%%%%%%%%%%%%%%%%%%%%%%%%%%%%%%%%%%%%%%%

\subsection{Results on Hold-Out Penalization} \label{sec.supmat.penHO}
This section gathers the proof of Theorem~\ref{thm:PenHo} (oracle inequality for hold-out penalization) and the variance computations we can make for hold-penalization. 

\subsubsection{Proof of Theorem~\ref{thm:PenHo}}
\label{sec.supmat.penHO.oracle}
The hold-out penalty is equal to  
\begin{align*} %%\label{eq:2fold=HO}
\penHO(m,T,x)
&= 2x(1-\fractrain_n)^2 \parenB{ P^{(\ItHO)}_{n}-P_{n}^{(\ItHO^c)} } \parenB{ \ERM^{(\ItHO)}_{m}-\ERM^{(\ItHO^c)}_{m} } \\
&= 2x(1-\fractrain_n)^2\sum_{\lamm}\crochsq{\parenb{P^{(\ItHO)}_{n}-P_{n}^{(\ItHO^c)}}(\psil)}
\enspace ,
\end{align*} 
where we recall that $\fractrain_n=\absj{T}/n$. 
As for Theorem~\ref{thm.oracle-penVF.cas_reel}, the oracle inequality is based on a concentration result for $\penHO(m,T,x)$. 
Let us start with an exact formula for the hold-out penalty (Lemma~\ref{lem:DecPenHO}, analogous to Lemma~\ref{lem:exact.formula.penvf}). 

\begin{lemma}\label{lem:DecPenHO}
For all $\mM_n$, we have
\begin{align*}
\penHO(m,T,x)
&=2x(1-\fractrain_n)^2\crochj{\normsqb{\ERM_m^{(\ItHO)}-\bayes_m} + \normsqb{\ERM_m^{(\ItHO^c)}-\bayes_m}-2 \parenB{ P_n^{(\ItHO)}-P } \parenB{ \ERM_m^{(\ItHO^c)}-\bayes_m} }
.
\end{align*}
In particular, we have 
\[\E\crochb{\penHO(m,T,x)}=2x\frac{1-\fractrain_n}{\fractrain_n}\frac{\Dcal_m}n\enspace.\]
\end{lemma}
\begin{proof} %of Lemma~\ref{lem:DecPenHO}.
By definition 
\begin{align*}
\penHO(m,T,x)
&=2x(1-\fractrain_n)^2 \sum_{\lamm} \setj{ \parensqB{  \parenb{ P_n^{(\ItHO^c)}-P } (\psi_{\lambda}) } 
+ \parensqB{ \parenb{ P_n^{(\ItHO)}-P } (\psi_{\lambda}) } }
\\
& \qquad -2x(1-\fractrain_n)^2 \sum_{\lamm}\setj{ 2 \parenB{ \parenb{ P_n^{(\ItHO^c)}-P } (\psi_{\lambda}) } \parenB{ \parenb{ P_n^{(\ItHO)}-P } (\psi_{\lambda}) } }
\\
&=2x(1-\fractrain_n)^2 \Biggl[ \normb{\ERM_m^{(\ItHO^c)}-\bayes_m}^2 + \normb{\ERM_m^{(\ItHO)}-\bayes_m}^2
\\
&
\qquad - 2 \parenB{ P_n^{(\ItHO)} -P }  \parenBb{ \sum_{\lamm} \parenB{ \parenb{ P_n^{(\ItHO^c)} - P } \psi_{\lambda}} 
\psi_{\lambda} } \Biggr] 
\enspace .
\end{align*}
\end{proof}
\begin{lemma}\label{lem:ConcIneqPenHO}
For all $\mM_n$ and $x>0$, with probability larger than $1-2\e^{-x}$, 
for all $\eta>0$, we have 
\begin{align*} %%\label{eq:termescroises}
\absj{  \parenb{ P_n^{(\ItHO)}-P } \parenb{\ERM_m^{(\ItHO^c)}-\bayes_m} } 
\leq  \frac{\eta}{2} \normb{\ERM_m^{(\ItHO^c)}-\bayes_m}^2 +\frac{2\norms{\bayes}_{\infty} x}{\eta \fractrain_n n} + \frac{b_m^2 x^2}{9\eta (\fractrain_n n)^2}\enspace.
\end{align*}
\end{lemma}
\begin{proof} %of Lemma~\ref{lem:ConcIneqPenHO}.
Let us apply Bernstein's inequality to the function $\parens{\ERM_m^{(\ItHO^c)}-\bayes_m}$, 
conditionally to $(\xi_i)_{i\notin \ItHO}$. Recall that $v_m^2\leq \norms{\bayes}_{\infty}$, hence
\begin{align*}
\normb{\ERM_m^{(\ItHO^c)}-\bayes_m}_{\infty}
&\leq \normb{\ERM_m^{(\ItHO^c)}-\bayes_m}b_{m}
\\
\text{and} \qquad 
\var\parenj{\ERM_m^{(\ItHO^c)}(\xi)-\bayes_m(\xi) \sachantb (\xi_{i})_{i\notin T}}&\leq \normb{\ERM_m^{(\ItHO^c)}-\bayes_m}^{2}v_{m}^{2}
\leq  \normb{\ERM_m^{(\ItHO^c)}-\bayes_m}^{2}  \norms{\bayes}_{\infty}
\enspace.
\end{align*}
Therefore, for all $x>0$, with probability larger than $1-2\e^{-x}$, conditionally to $(\xi_{i})_{i\notin T}$,
\begin{align*}
\absj{ \parenb{ P_n^{(\ItHO)}-P } \parenb{\ERM_m^{(\ItHO^c)}-\bayes_m}}
&\leq \normb{\ERM_m^{(\ItHO^c)}-\bayes_m} \parenj{\sqrt{\frac{2\norms{\bayes}_{\infty} x}{\fractrain_n n}} + \frac{b_m x}{3\fractrain_n n}}
\\
&\leq \frac{\eta}{2} \normb{\ERM_m^{(\ItHO^c)}-\bayes_m}^{2}+\frac{1}{\eta}\parenj{\frac{2\norms{\bayes}_{\infty} x}{\fractrain_n n} + \frac{b_m^2 x^2}{9(\fractrain_n n)^2}}\enspace.
\end{align*}
As the bound on the probability does not depend on $(\xi_i)_{i \notin \ItHO}$, the same inequality holds unconditionally.
\end{proof}
\begin{proofof}{Theorem~\ref{thm:PenHo}} 
From \citet[Theorem~4.1]{Ler:2010:mixing}---a result recalled with Proposition~\ref{prop:concLe09} in Section~\ref{sec.supmat.proba-tools}---, 
Lemma~\ref{lem:DecPenHO} and Lemma~\ref{lem:ConcIneqPenHO}, 
an absolute constant $\kappa$ exists such that, 
for all $x>0$, with probability larger than $1-8 \e^{-x}$, 
for all $\epsilon \in (0,1]$, we have 
\begin{multline*}
\forall m\in\M_n,\;\left|\penHO\parenj{m,T,\frac{\fractrain_n}{1-\fractrain_n}}-\norms{\ERM_{m}-\bayes_{m}}^{2}\right|\\
\leq \epsilon\norms{\ERM_{m}-\bayes_{m}}^{2}+\kappa\parenj{\frac{\norms{\bayes}_{\infty}x_n}{\epsilon n}+\frac{b_m^2x_n^2}{\epsilon^3 n^2}\frac{\fractrain_n^2+(1-\fractrain_n)^2}{\fractrain_n(1-\fractrain_n)}}\enspace .
\end{multline*}
We can then conclude the proof as in Theorem~\ref{thm.oracle-penVF.cas_reel}.
\end{proofof}

\subsubsection{Variance} \label{sec.supmat.penHO.variance}
\begin{proposition} \label{pro.variance.penHO}
Let $(\psil)_{\lL_{m_1}}$ and $(\psil)_{\lL_{m_2}}$ denote two orthonormal families in $L^4(\mu)$.
Assume that $|\ItHO| \in \inter{n-1}$ and denote for any $m\in\setj{m_1,m_2}$, 
\[\CVho_{(C,\ItHO)}(m)=P_n\gamma(\ERM_{m})+\penHO\parenb{m,T,C\fractrain_n/(1-\fractrain_n)} 
\enspace  .\] 
Then, with the notations introduced in Theorem~\ref{theo.variance.penVF}, 
we have 
\begin{align} 
\notag\var \parenj{ \CVho_{(C,\ItHO)}(m_1)}
&=\frac{4}{n}\var \parenj{ \parenj{1+\frac{2C-1}n}\bayes_{m_1}(\xi)-\frac{2C-1}{2n}\Psi_{m_1}(\xi)}
\\
\notag 
&\qquad +\frac{2}{n^2} \crochj{1+4C^2-\frac{(2C-1)^2}n}\beta(\Lambda_{m_1},\Lambda_{m_1})
\\
\notag&\qquad +\frac{4C^2}{n^3}\frac{(1-2\fractrain_n)^2}{\fractrain_n(1-\fractrain_n)}\parenB{\var \parenb{ \Psi_{m_1}(\xi)-2\bayes_{m_1}(\xi)}-2\termeBvar{m , m}}
\end{align}
and
\begin{align}
\notag&\var \parenj{ \CVho_{(C,\ItHO)}(m_1)-\CVho_{(C,\ItHO)}(m_2)}
\\
\notag&=\frac{4}{n}\var \parenj{ \parenj{1+\frac{2C-1}{n}} 
\parenb{ \bayes_{m_1}(\xi)-\bayes_{m_2}(\xi) } 
- \frac{2C-1}{2n} \parenb{ \Psi_{m_1}(\xi)-\Psi_{m_2}(\xi) }}
\\
\label{eq.pro.variance.penHO-penid}
&\qquad +\frac{2}{n^2} \parenj{1+4C^2-\frac{(2C-1)^2}n}\termeBvaracr{m_1,m_2}
\\
\notag &\qquad + \frac{4C^2}{n^3}\frac{(1-2\fractrain_n)^2}{\fractrain_n(1-\fractrain_n)} \parenbb{\var \parenB{ \parenb{ \Psi_{m_1}(\xi)-\Psi_{m_2}(\xi) } 
- 2 \parenb{ \bayes_{m_1}(\xi)-\bayes_{m_2}(\xi) } } - 2\termeBvaracr{m_1,m_2} }
\enspace .
\end{align}
\end{proposition}

\begin{proof} %of Proposition~\ref{pro.variance.penHO}.
By definition
\begin{align}
\notag\penHO(m_1,T,x)&=2x
\sum_{\lL_{m_1}}\crochsqB{\parenb{P^{(\ItHO)}_{n}-P_{n}}\psil}\\
\notag&=\frac{2x}{n^{2}}\sum_{\lL_{m_1}}\parensq{\sum_{i=1}^{n}\parenj{\frac{1}{\fractrain_n}\un_{i\in\ItHO}-1}\psil(\xi_{i})} \\
\label{eq.penHO.formule-lambda}
&= \frac{2x}{n^{2}} \sum_{i,j=1}^{n}E_{i,j}^{(\HO)}U_{m_1}(\xi_i,\xi_j)
\enspace ,
\end{align} 
where, for all $i,j \in \setj{1 \ldots, n}$, 
we recall that 
\begin{align*}
U_{m_1}(\xi_i,\xi_j)
&= \sum_{\lL_{m_1}} \parenb{ \psil(\xi_i)-P\psil } \parenb{ \psil(\xi_j)-P\psil } 
\\
\text{and} \qquad 
E_{i,j}^{(\HO)} &= \parenj{\frac{1}{\fractrain_n}\un_{i\in\ItHO}-1}\parenj{\frac{1}{\fractrain_n}\un_{j\in\ItHO}-1} \enspace . 
\end{align*}
Therefore, from Eq.~\eqref{eq:DecRiskEmp}, if $x=C\fractrain_n/(1-\fractrain_n)$, we have 
\begin{align*}
\CVho_{(C,\ItHO)}(m_1)
&\egaldef P_n\gamma(\ERM_{m_1})+\penHO(m_1,\ItHO,x)
\\
&=\sum_{1\leq i,j\leq n}\frac{2xE_{i,j}^{(\HO)} -1}{n^2}U_{m_1}(\xi_i,\xi_j)-\sum_{i=1}^n\frac{2}{n}\bayes_{m_1}(\xi_i)+\norms{\bayes_{m_1}}^2
\enspace.
\end{align*}
By definition
\begin{align*}
E_{i,j}^{(\HO)} &= \parensq{ \frac{1-\fractrain_n}{\fractrain_n}} \un_{i,j \in \ItHO} 
- \frac{1-\fractrain_n}{\fractrain_n} \un_{ i \in \ItHO , \, j \notin \ItHO } 
- \frac{1-\fractrain_n}{\fractrain_n} \un_{ i \notin \ItHO , \, j \in \ItHO } 
+ \un_{i , j \notin \ItHO} \enspace . 
\end{align*} 
Therefore, we can compute 
\begin{align} 
\label{eq.proof.pro.variance.HO.SumEii}
\sum_{i=1}^n E_{i,i}^{(\HO)}  
&=n\crochj{ \fractrain_n \parensq{ \frac{1-\fractrain_n}{\fractrain_n}} +1-\fractrain_n}
 = n\frac{1-\fractrain_n}{\fractrain_n}\enspace,
\\  
\label{eq.proof.pro.variance.HO.SumEiicarre}
 \sum_{i=1}^n \parenB{ E_{i,i}^{(\HO)} }^2 
&= n\crochj{ \fractrain_n \mathopen{} \left( { \frac{1-\fractrain_n}{\fractrain_n}} \right)^4 \mathclose{} +1-\fractrain_n}
 = n(1-\fractrain_n) \frac{(1-\fractrain_n)^3+\fractrain_n^3}{\fractrain_n^3}
\enspace . 
\end{align}
Moreover, the $E_{i,j}^{(\HO)}$ satisfy %also satisfy Eq.~\eqref{eq.proof.pro.variance.SumEijW=0} :
\begin{equation*} 
\sum_{1 \leq i,j \leq n} E_{i,j}^{(\HO)} = \E\crochj{ \parensq{ \sum_{i=1}^n  \parenj{\frac 1{\fractrain_n}\un_{i\in\ItHO}-1} } } = 0 
\enspace , 
\end{equation*}
so Eq.~\eqref{eq.proof.pro.variance.HO.SumEii} implies that 
\begin{equation} 
\label{eq.proof.pro.variance.HO.SumEij}
 \sum_{1 \leq i \neq j \leq n} E_{i,j}^{(\HO)} 
= - \sum_{i=1}^n E_{i,i}^{(\HO)}   
= -n\frac{1-\fractrain_n}{\fractrain_n}
\enspace . 
\end{equation}
In addition, we compute 
\begin{align}
 \sum_{1 \leq i \neq j \leq n} \parenB{ E_{i,j}^{(\HO)} }^2 
&= 2 n^2\fractrain_n(1-\fractrain_n) \parensq{ \frac{1-\fractrain_n}{\fractrain_n}} + n\fractrain_n (n\fractrain_n - 1) \mathopen{} \left( {\frac{1-\fractrain_n}{\fractrain_n}} \right)^4 \mathclose{} \notag
\\
&\qquad  + n(1 - \fractrain_n) \crochb{ n(1-\fractrain_n) - 1 }
\notag 
\\
&= n^2(1-\fractrain_n)^2\crochj{2\frac{1-\fractrain_n}{\fractrain_n}  + \parensq{\frac{1-\fractrain_n}{\fractrain_n}}+1} \notag
\\
&\qquad -n(1-\fractrain_n) \crochj{ \mathopen{} \left( {\frac{1-\fractrain_n}{\fractrain_n}} \right)^3 \mathclose{} +1} 
\notag 
\\
&= n^2\parensq{\frac{1-\fractrain_n}{\fractrain_n}} - n(1-\fractrain_n)\frac{(1-\fractrain_n)^3+\fractrain_n^3}{\fractrain_n^3} 
\label{eq.proof.pro.variance.HO.SumEijcarre}
\enspace . 
\end{align}

\medbreak

According to Eq.~\eqref{eq.penHO.formule-lambda} and \eqref{eq:DecRiskEmp}, $\var \parens{ \CVho_{(C,\ItHO)}(m_1)}$ can be computed using Lemma~\ref{lem:CovGen.2} with 
\[ 
\forall i,j \in \setj{1, \ldots, n}, \quad 
\omega_{i,j} = \frac{1}{n^2} \parenj{ 2x E_{i,j}^{(\HO)} - 1} 
\qquad f_{m_1}= \frac{- 2 \bayes_{m_1}}{n} 
\qquad \text{and} \qquad \cteformCVb_i = 1 
\enspace . \]
So, using Eq.~\eqref{eq.proof.pro.variance.HO.SumEii}, \eqref{eq.proof.pro.variance.HO.SumEiicarre}, \eqref{eq.proof.pro.variance.HO.SumEij} and \eqref{eq.proof.pro.variance.HO.SumEijcarre}, we have 
\begin{align}
\notag\sum_{i=1}^n \omega_{i,i}^2 
&
=\frac{1}{n^4} \crochj{4x^2\sum_{i=1}^n \parenB{ E_{i,i}^{(\HO)} }^2 -4x \sum_{i=1}^n E_{i,i}^{(\HO)} + n } 
\\
\notag %%\label{eq.termeA}
& = \frac{1}{n^3} \crochj{ 4x^2(1-\fractrain_n) \frac{(1-\fractrain_n)^3+\fractrain_n^3}{\fractrain_n^3}- 4x\frac{1-\fractrain_n}{\fractrain_n} + 1}
\\
%%\notag\sum_{1 \leq i \neq j \leq n} \omega_{i,j}^2
\notag \mathop{\sum_{1 \leq i , j \leq n}}_{i \neq j} \omega_{i,j}^2
&=  \frac{1}{n^4} \crochj{ 4x^2 \sum_{1 \leq i \neq j \leq n} \parenB{ E_{i,j}^{(\HO)} }^2 - 4x\sum_{1 \leq i \neq j \leq n} E_{i,j}^{(\HO)} + n(n-1)} 
\\
\notag %%\label{eq.termeB}
&=  \frac{1}{n^4} \crochj{ 4x^2\parenj{n^2 \parensq{\frac{1-\fractrain_n}{\fractrain_n}}-n(1-\fractrain_n) \frac{(1-\fractrain_n)^3+\fractrain_n^3}{\fractrain_n^3}}+4xn\frac{1-\fractrain_n}{\fractrain_n}+n(n-1) } 
\\
\notag %%\label{eq.termeC}
\sum_{i=1}^n \omega_{i,i} \cteformCVb_i 
&= \frac{1}{n} \parenj{2x\frac{1-\fractrain_n}{\fractrain_n}-1 }
\enspace . 
\end{align}
Therefore, by Lemma~\ref{lem:CovGen.2} with $m=m^\prime =m_1$, we deduce 
\begin{align*}
\var \parenj{ \CVho_{(C,\ItHO)}(m_1)}
&=\frac{1}{n^3} \parenj{4C^2\frac{(1-2\fractrain_n)^2}{\fractrain_n(1-\fractrain_n)}+(2C-1)^2}\var \parenb{ \Psi_{m_1}(\xi)-2\bayes_{m_1}(\xi)}
\\
&\qquad +\frac{2}{n^2} \crochj{1+4C^2-\frac{1}{n}\parenj{4C^2\frac{(1-2\fractrain_n)^2}{\fractrain_n(1-\fractrain_n)}+(2C-1)^2} }\termeBvar{m_1 , m_1}
\\
&\qquad - \frac{4}{n^2} \parens{2C-1 }\cov\parenb{\Psi_{m_1}(\xi)-2\bayes_{m_1}(\xi),\bayes_{m_1}(\xi)} + \frac{4}{n}\var \parenb{ \bayes_{m_1}(\xi)}
\\
&=\frac{4}{n}\var \parenj{ \parenj{1+\frac{2C-1}n}\bayes_{m_1}(\xi)-\frac{2C-1}{2n}\Psi_{m_1}(\xi)}\\
&\qquad +\frac{2}{n^2} \parenj{1+4C^2-\frac{(2C-1)^2}n}\beta(\Lambda_{m_1},\Lambda_{m_1})\\
&\qquad +\frac{4C^2}{n^3} \frac{(1-2\fractrain_n)^2}{\fractrain_n(1-\fractrain_n)}\crochB{ \var \parenb{ \Psi_{m_1}(\xi)-2\bayes_{m_1}(\xi)}-2\termeBvar{m_1 , m_1}}\enspace .
\end{align*}
Eq~\eqref{eq.pro.variance.penHO-penid} follows from similar  computations.
\end{proof}

\subsection{Additional Comments on Computational Issues}\label{sect.comp.comp}
This section is an appendix to Section~\ref{sec.algo}. 
We first detail a naive algorithm for computing $V$-fold criteria, Algorithm~\ref{proc.penVF.naive}. 
Then, we prove Proposition~\ref{pro.proc.penVF-fast.general-density} 
which shows that Algorithm~\ref{proc.penVF.fast} also computes 
correctly the $V$-fold criteria, much faster than Algorithm~\ref{proc.penVF.naive}. 

\subsubsection{Naive Implementation} \label{sec.algo.naive}
\begin{proc} \label{proc.penVF.naive} \hfill \\
\vspace{-0.5cm}
\begin{enumerate}
\item[] \textbf{Input\textup{:}} $\B$ some partition of $\setj{1,...,n}$ satisfying \eqref{hyp.part-reg.exact}, 
 $\xi_1, \ldots, \xi_n \in \X$ 
 and $\parenj{\psil}_{\lamm}$ a finite orthonormal family of $L^2(\mu)$, with $\card(m)=d_m $.
\item For $j \in \setj{1, \ldots, V}$, 
\begin{enumerate}
\item[(a)] train $\ERM_m(\cdot)$ with the data set $(\xi_i)_{i \notin \B_j}$, that is, for all $\lambda\in \Lambda_m $, compute 
\[ 
\alpha_{\lambda,j} \egaldef P_n^{(-\B_{j})}(\psil) = \frac{V}{(V-1) n} \sum_{i \notin \B_j} \psil(\xi_i) 
\] 
so that $\ERM_m^{(-\B_j)} = \sum_{\lambda\in\Lambda_m} \alpha_{\lambda,j} \psil $\textup{;} 
\item[(b)] compute the norm of $\ERM_m^{(-\B_j)}$\textup{:} $N_j \egaldef %\norms{ \ERM_m^{(-\B_j)} }^2 = 
\sum_{ \lambda\in \Lambda_m } \alpha_{\lambda,j}^2$\textup{;}
\item[(c)] compute $Q_j \egaldef P_n^{(\B_j)} \parenj{ \ERM_m^{(-\B_j)} } = 
\frac{V}{n} \sum_{\lambda\in\Lambda_m} \sum_{i \in \B_j} \alpha_{\lambda,j} \psil(\xi_i)$\textup{;}
\item[(d)] compute $R_j \egaldef P_n^{(-\B_{j})} \parenj{ \ERM_m^{(-\B_j)} } = 
\frac{V}{n(V-1)} \sum_{\lambda\in\Lambda_m} \sum_{i \notin \B_j} \alpha_{\lambda,j} \psil(\xi_i)$. 
\end{enumerate}
\item Compute the $V$-fold cross-validation criterion\textup{:} $\mathcal{C} = V^{-1} \sum_{j=1}^V (N_j - 2 Q_j)$. 
\item Compute the empirical risk\textup{:}
\begin{enumerate}
\item[(a)] train $\ERM_m(\cdot)$ with the data set $(\xi_i)_{ 1 \leq i \leq n}$, that is, for all $\lambda\in \Lambda_m $, compute 
\[ 
\alpha_{\lambda} \egaldef P_n(\psil) = \frac{1}{n} \sum_{i=1}^n  \psil(\xi_i) 
\] 
so that $\ERM_m = \sum_{\lambda\in\Lambda_m} \alpha_{\lambda} \psil $\textup{;}
\item[(b)] compute the norm of $\ERM_m$\textup{:} $N \egaldef %\norms{ \ERM_m }^2 = 
\sum_{ \lambda\in \Lambda_m } \alpha_{\lambda}^2$\textup{;}
\item[(c)] compute $R \egaldef %P_n \parenj{ \ERM_m } = 
\frac{1}{n} \sum_{\lambda\in\Lambda_m} \sum_{i=1}^n \alpha_{\lambda} \psil(\xi_i)$. 
\end{enumerate}
\item Compute the $V$-fold penalty\textup{:} $\mathcal{D} \egaldef 2 (V-1) V^{-2} \sum_{j=1}^V (Q_j - R_j)$. 
\item[] \textbf{Output\textup{:}} \\
Empirical risk\textup{:} $N - 2 R $ \\
$V$-fold cross-validation estimator of the risk of $\ERM_m$\textup{:} $\critVF(m) = \mathcal{C} $ \\
$V$-fold penalty\textup{:}  $\penVF(m) = \mathcal{D} $. 
\end{enumerate}
\end{proc}
Assuming that the computational cost of evaluating $\psil$ at some point $\xi \in \Xi$ is of order 1, the computational cost of this naive algorithm~\ref{proc.penVF.naive} is as follows: $n (V-1) d_m $ for step 1, $V$ for steps 2 and 4, $n d_m$ for step 3. 
So the overall cost of computing the $V$-fold penalization criterion for $m$ is of order $n V d_m$.

\subsubsection{Proof of Proposition~\ref{pro.proc.penVF-fast.general-density}}
\label{sec.app.proof.pro.algo}
Let us first note that for every $i \in \setj{1, \ldots, V}$ and $\lamm$, $A_{i,\lambda} = P_n^{(\B_i)} (\psil)$. 
So, at step 2, for every $i,j \in \setj{1, \ldots, V}$, 
we have 
\[ C_{i,j} = \sum_{\lamm} P_n^{(\B_i)} (\psil) P_n^{(\B_j)} (\psil) 
= P_n^{(\B_i)} \parenj{ \sum_{\lamm} P_n^{(\B_j)} (\psil) \psil } = P_n^{(\B_i)} \parenB{ \ERM_m^{(\B_j)} } 
\]
and by symmetry $C_{i,j} = C_{j,i} = P_n^{(\B_j)} \parenB{ \ERM_m^{(\B_i)} }$. 

\fauxparagraph{Correctness of Algorithm~\ref{proc.penVF.fast}}
By assumption \eqref{hyp.part-reg.exact}, we have 
\begin{gather*}
P_n = \frac{1}{V} \sum_{j=1}^V P_n^{(\B_j)} 
\enspace , \qquad 
\ERM_m = \frac{1}{V} \sum_{j=1}^V \ERM_m^{(\B_j)} 
\enspace ,
\\ 
P_n^{(-\B_i)} = \frac{1}{V-1} \mathop{\sum_{1 \leq j \leq V}}_{j \neq i} P_n^{(\B_j)} 
\quad \text{and} \quad 
\ERM_m^{(-\B_i)} = \frac{1}{V-1} \mathop{\sum_{1 \leq j \leq V}}_{j \neq i} \ERM_m^{(\B_j)} 
\enspace .
\end{gather*}
Therefore, 
\[ \norms{\ERM_m}^2 = - P_n \gamma(\ERM_m) = P_n(\ERM_m) = \frac{1}{V^2} \sum_{1 \leq i,j \leq V} P_n^{(\B_i)}\parenB{ \ERM_m^{(\B_j)} } = \frac{1}{V^2} \mathcal{S} \]
and 
\begin{align*}
&\qquad 
\critVF(m) 
\\
&= \frac{1}{V} \sum_{j=1}^V P_n^{(\B_j)} \gamma\parenj{\ERM_m^{(-\B_j)}} \\
 &= \frac{1}{V} \sum_{j=1}^V \crochB{ \normb{\ERM_m^{(-\B_j)}}^2 - 2 P_n^{(\B_j)} \parenb{\ERM_m^{(-\B_j)}} } \\
 &= \frac{1}{V} \sum_{j=1}^V \parenj{ \frac{1}{(V-1)^2} \mathop{\sum_{1 \leq i,\ell \leq V}}_{i,\ell \neq j} P_n^{(\B_i)} \parenB{ \ERM_m^{(\B_{\ell})} } - \frac{2}{V-1} \sum_{i \neq j} P_n^{(\B_j)} \parenj{\ERM_m^{(\B_i)}} } \\
 &= \frac{1}{V (V-1)^2} \sum_{1 \leq i,\ell \leq V} \crochj{ P_n^{(\B_i)} \parenb{ \ERM_m^{(\B_{\ell})} } \sum_{j=1}^V \un_{i\neq j , \, \ell \neq j} } 
 - \frac{2}{V(V-1)}  \sum_{1 \leq i \neq j \leq V} P_n^{(\B_j)} \parenj{\ERM_m^{(\B_i)}} \\
  &= \frac{1}{V (V-1)^2} \sum_{1 \leq i,\ell \leq V} \crochj{ P_n^{(\B_i)} \parenb{ \ERM_m^{(\B_{\ell})} } (V-1 -\un_{i \neq \ell}) } 
 - \frac{2}{V(V-1)}  \parenj{ \mathcal{S} -\mathcal{T}} 
 \\
  &= \frac{1}{V (V-1)} \sum_{1 \leq i \leq V} \crochj{ P_n^{(\B_i)} \parenb{ \ERM_m^{(\B_i)} } } 
  + \frac{V-2}{V (V-1)^2} \sum_{1 \leq i \neq \ell \leq V} \crochj{ P_n^{(\B_i)} \parenb{ \ERM_m^{(\B_{\ell})} } } 
  \\
  &\qquad 
  - \frac{2}{V(V-1)}  \parenj{ \mathcal{S} -\mathcal{T}} \\
  &= \frac{1}{V (V-1)} \mathcal{T}
  + \frac{V-2}{V (V-1)^2} \parenj{ \mathcal{S} -\mathcal{T}} 
 - \frac{2}{V(V-1)}  \parenj{ \mathcal{S} -\mathcal{T}} \\
  &= \frac{1}{V (V-1)} \mathcal{T}
  - \frac{1}{(V-1)^2}  \parenj{ \mathcal{S} -\mathcal{T}} \enspace ,
\end{align*}
so the formula for $\critVF$ is correct. 
Lemma~\ref{le.penVF-VFCV} implies the formula for $\penVF$ is also correct.

\fauxparagraph{Computational cost of Algorithm~\ref{proc.penVF.fast}}
Step~1 has a cost of order 
\[ 
V \times \card(\Lambda_m) \times \frac{n}{V} 
= n \card(\Lambda_m) 
\enspace . 
\]
Step~2 has a cost of order $V^2 \card(\Lambda_m)$. 
Step~3 has a cost of order $V^2$. 
Summing the three steps yields the result.

\fauxparagraph{Computational cost for histograms}
In the histogram case, step 1 can be performed with a cost of order $V \card(\Lambda_m) + n$. 
Indeed, one can initialize the $V \times\card(\Lambda_m)$ matrix $A$ with zeros (cost: $V \card(\Lambda_m)$), and then go sequentially through the data set: for $j =1 , \ldots, n$, find the unique $i(j) \in  \setj{1, \ldots, V}$ such that $j \in \B_{i(j)}$, the unique $\lambda(j) \in \Lambda_m$ such that $\xi_j \in \lambda(j)$, and add $(V/n) \psil(\xi_j)$ to $A_{(i(j),\lambda(j))}$. 
Since the partitions $\B$ and $\Lambda_m$ can be coded so that finding $i(j)$ and $\lambda(j)$ has a cost of order 1, the resulting cost of step 1 is $V \card(\Lambda_m) + n$, hence the overall cost is of order $V^2 \card(\Lambda_m) + n$. 
\qed

\subsection{Probabilistic Tool} \label{sec.supmat.proba-tools}
\begin{proposition}[\citealp{Ler:2010:mixing}]\label{prop:concLe09}
Let $\xi_{\inter{N}}$ be iid random variables valued in a measurable space $(\mathbb{X},\mathcal{X})$, with common distribution $P$. Let $S$ be a symmetric class of functions bounded by $b$. For all $t \in S$, let us define 
\begin{gather*}
P_N t = \frac{1}{N} \sum_{i=1}^{N} t(\xi_i)
\qquad 
v^2=\sup_{t\in S} P \crochj{(t-Pt)^2}
%%\qquad 
\\
Z=\sup_{t\in S} \setb{ (P_N-P)(t) }
\qquad \text{and} \qquad 
D=N \E\crochj{ Z^{2} }
\enspace . 
\end{gather*}
There exists an absolute constant $\kappa$ such that, for all $x>0$, 
with probability larger than $1-2\e^{-x}$, 
for all $\epsilon \in (0,1]$, 
\[ 
\absj{ Z^{2}-\frac{D}{N} } 
\leq \epsilon\frac{D}{N} 
+ \kappa\parenj{\frac{v^{2}x}{\epsilon N}+\frac{b^2x^2}{\epsilon^3 N^2}}
\enspace . 
\]
\end{proposition}
For instance, taking $S = \Boule_m$, by Eq.~\eqref{eq.p1.4formules}, this result applies to 
\[ 
Z = \sup_{t \in \Boule_m} \setb{ (P_n - P)(t) } 
= \norms{ \ERM_m - \bayes_m}^2 
\enspace .  
\]

\subsection{Additional Simulation Results} \label{sec.supmat.simus}
This section provides simulation results in addition to the ones of Section~\ref{sec.simus}. 

\medbreak

Figure~\ref{fig.oracles.Regu-Dya2.Li01} is an analogous of Figure~\ref{fig.oracles.Regu-Dya2.SiG5} in setting L, that illustrates the difference between the model collections Regu and Dya2. 
\begin{figure}
\begin{minipage}[b]{.48\linewidth}
\includegraphics[width=\textwidth]{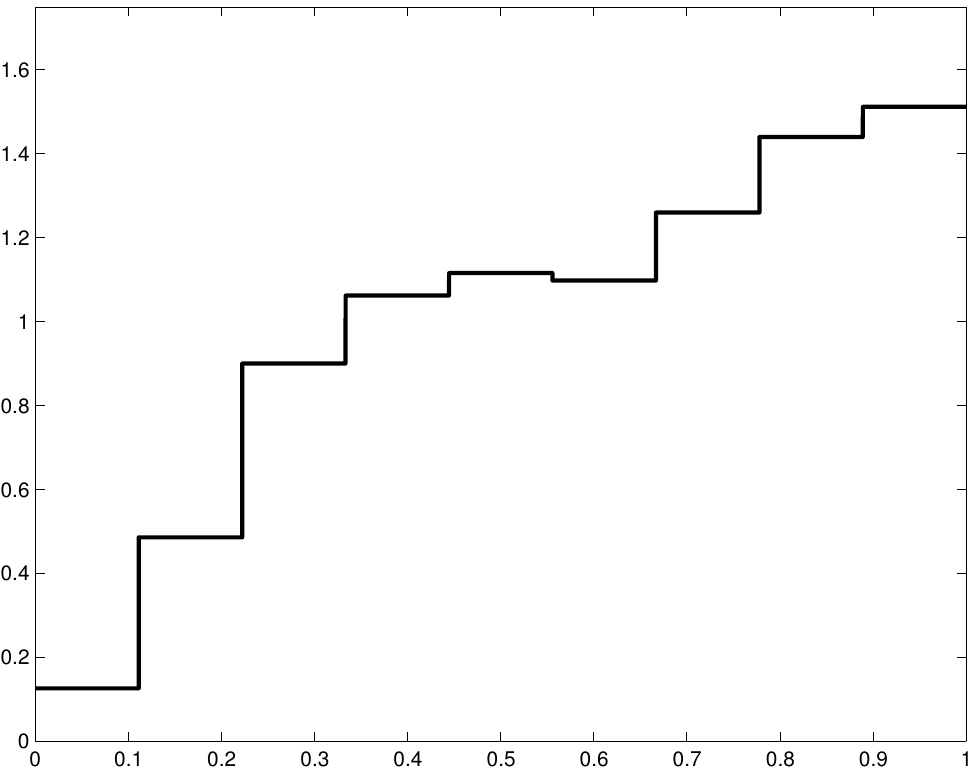}
\end{minipage}
\hspace{.025\linewidth}
\begin{minipage}[b]{.48\linewidth}
\includegraphics[width=\textwidth]{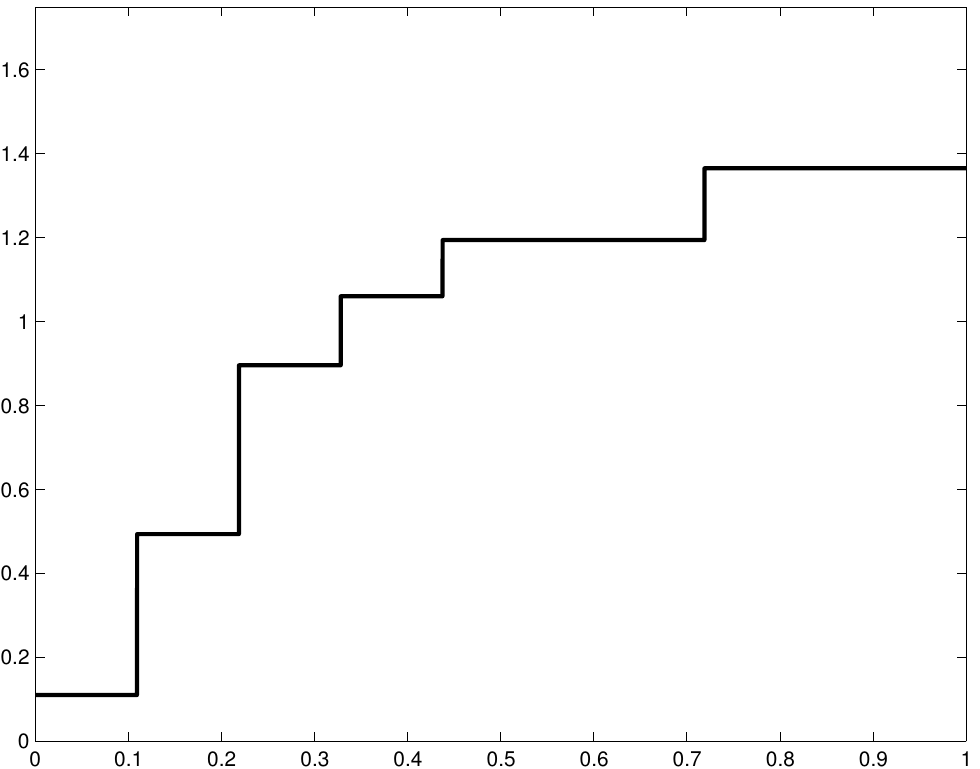}
\end{minipage}
\caption{Oracle model for some sample of size $n=500$, in setting L. Left: Regu. Right: Dya2.
\label{fig.oracles.Regu-Dya2.Li01}
}
\end{figure}

%\clearpage
%
Table~\ref{tab.complet.Li01.SiG5} is an extended version of Table~\ref{tab.simus-results}, with more procedures compared and two additional settings (L-Regu and S-Regu). 
\begin{table}[p] 
\begin{center}
%%%%results_C140704_Li01_SiG5_n500_Nbech10000_seed2
%
\begin{tabular}
{l@{\hspace{0.025\textwidth}}c@{\hspace{0.025\textwidth}}c@{\hspace{0.025\textwidth}}c@{\hspace{0.025\textwidth}}c}
\hline\noalign{\smallskip}
Experiment  & L--Dya2 & L--Regu  & S--Dya2 & S--Regu \\ 
%Experiment  & Li01Dya2 & Li01Regu  & SiG5Dya2 & SiG5Regu \\ 
\noalign{\smallskip} 
\hline 
\noalign{\smallskip} 
$\E[\penid]$   & $  6.52 \pm 0.05  $  & $  2.33 \pm 0.01  $  & $  2.07 \pm 0.01  $  & $  1.75 \pm 0.01  $ \\ 
1.25 $\times$ $\E[\penid]$   & $  4.81 \pm 0.04  $  & $  2.01 \pm 0.01  $  & $  1.94 \pm 0.01  $  & $  { \bf  \phantom{4}1.62 \pm 0.004  }  $ \\ 
1.5 $\times$ $\E[\penid]$   & $  4.12 \pm 0.03  $  & $  { \bf  1.93 \pm 0.01  }  $  & $  { \bf  1.92 \pm 0.01  }  $  & $  \phantom{3}1.65 \pm 0.003  $ \\ 
2 $\times$ $\E[\penid]$   & $  { \bf  3.61 \pm 0.02  }  $  & $  1.96 \pm 0.01  $  & $  2.01 \pm 0.01  $  & $  \phantom{4}1.84 \pm 0.004  $ \\ 
\noalign{\smallskip} 
\hline \hline 
\noalign{\smallskip} 
 $\pendim$    & $  8.27 \pm 0.07  $  & $  2.33 \pm 0.01  $  & $  3.21 \pm 0.01  $  & $  1.75 \pm 0.01  $ \\ 
1.25 $\times$  $\pendim$    & $  5.95 \pm 0.05  $  & $  2.01 \pm 0.01  $  & $  3.01 \pm 0.01  $  & $  { \bf  \phantom{4}1.62 \pm 0.004  }  $ \\ 
1.5 $\times$  $\pendim$    & $  4.99 \pm 0.04  $  & $  { \bf  1.94 \pm 0.01  }  $  & $  3.03 \pm 0.01  $  & $  \phantom{3}1.66 \pm 0.003  $ \\ 
2 $\times$  $\pendim$    & $  4.38 \pm 0.03  $  & $  1.97 \pm 0.01  $  & $  3.24 \pm 0.01  $  & $  \phantom{4}1.85 \pm 0.004  $ \\ 
\noalign{\smallskip} 
\hline 
\noalign{\smallskip} 
$\penLOO$   & $  6.35 \pm 0.05  $  & $  2.33 \pm 0.01  $  & $  2.06 \pm 0.01  $  & $  1.75 \pm 0.01  $ \\ 
1.25 $\times$ $\penLOO$   & $  4.62 \pm 0.04  $  & $  2.01 \pm 0.01  $  & $  1.92 \pm 0.01  $  & $  { \bf  \phantom{4}1.62 \pm 0.004  }  $ \\ 
1.5 $\times$ $\penLOO$   & $  3.97 \pm 0.03  $  & $  { \bf  1.94 \pm 0.01  }  $  & $  { \bf  \phantom{5}1.90 \pm 0.005  }  $  & $  \phantom{3}1.66 \pm 0.003  $ \\ 
2 $\times$ $\penLOO$   & $  { \bf  3.55 \pm 0.02  }  $  & $  1.97 \pm 0.01  $  & $  1.98 \pm 0.01  $  & $  \phantom{4}1.85 \pm 0.004  $ \\ 
\noalign{\smallskip} 
\hline 
\noalign{\smallskip} 
 $\penVF$  ($V$=10)   & $  6.89 \pm 0.06  $  & $  2.42 \pm 0.02  $  & $  2.11 \pm 0.01  $  & $  1.77 \pm 0.01  $ \\ 
1.25 $\times$  $\penVF$  ($V$=10)   & $  5.01 \pm 0.04  $  & $  2.04 \pm 0.01  $  & $  1.95 \pm 0.01  $  & $  { \bf  \phantom{4}1.62 \pm 0.004  }  $ \\ 
1.5 $\times$  $\penVF$  ($V$=10)   & $  4.27 \pm 0.03  $  & $  1.94 \pm 0.01  $  & $  1.92 \pm 0.01  $  & $  \phantom{4}1.63 \pm 0.004  $ \\ 
2 $\times$  $\penVF$  ($V$=10)   & $  3.68 \pm 0.02  $  & $  1.94 \pm 0.01  $  & $  1.98 \pm 0.01  $  & $  \phantom{4}1.78 \pm 0.004  $ \\ 
\noalign{\smallskip} 
\hline 
\noalign{\smallskip} 
 $\penVF$  ($V$=5)   & $  7.47 \pm 0.06  $  & $  2.55 \pm 0.02  $  & $  2.16 \pm 0.01  $  & $  1.80 \pm 0.01  $ \\ 
1.25 $\times$  $\penVF$  ($V$=5)   & $  5.50 \pm 0.04  $  & $  2.10 \pm 0.01  $  & $  1.98 \pm 0.01  $  & $  \phantom{4}1.63 \pm 0.004  $ \\ 
1.5 $\times$  $\penVF$  ($V$=5)   & $  4.58 \pm 0.03  $  & $  1.96 \pm 0.01  $  & $  1.93 \pm 0.01  $  & $  { \bf  \phantom{4}1.62 \pm 0.004  }  $ \\ 
2 $\times$  $\penVF$  ($V$=5)   & $  3.86 \pm 0.02  $  & $  { \bf  1.93 \pm 0.01  }  $  & $  1.98 \pm 0.01  $  & $  \phantom{4}1.73 \pm 0.004  $ \\ 
\noalign{\smallskip} 
\hline 
\noalign{\smallskip} 
 $\penVF$  ($V$=2)   & $  10.21 \pm 0.08  $  & $  3.37 \pm 0.03  $  & $  2.39 \pm 0.01  $  & $  2.01 \pm 0.01  $ \\ 
1.25 $\times$  $\penVF$  ($V$=2)   & $  7.69 \pm 0.06  $  & $  2.49 \pm 0.02  $  & $  2.15 \pm 0.01  $  & $  1.71 \pm 0.01  $ \\ 
1.5 $\times$  $\penVF$  ($V$=2)   & $  6.41 \pm 0.05  $  & $  2.18 \pm 0.01  $  & $  2.05 \pm 0.01  $  & $  \phantom{4}1.63 \pm 0.004  $ \\ 
2 $\times$  $\penVF$  ($V$=2)   & $  5.11 \pm 0.04  $  & $  1.99 \pm 0.01  $  & $  2.04 \pm 0.01  $  & $  \phantom{4}1.64 \pm 0.004  $ \\ 
\noalign{\smallskip} 
\hline 
\noalign{\smallskip} 
LOO   & $  6.34 \pm 0.05  $  & $  2.33 \pm 0.01  $  & $  2.06 \pm 0.01  $  & $  1.75 \pm 0.01  $ \\ 
10-fold CV    & $  6.24 \pm 0.05  $  & $  2.29 \pm 0.01  $  & $  2.05 \pm 0.01  $  & $  1.71 \pm 0.01  $ \\ 
5-fold CV    & $  6.27 \pm 0.05  $  & $  2.26 \pm 0.01  $  & $  2.05 \pm 0.01  $  & $  1.68 \pm 0.01  $ \\ 
2-fold CV    & $  6.41 \pm 0.05  $  & $  2.18 \pm 0.01  $  & $  2.05 \pm 0.01  $  & $  \phantom{4}1.63 \pm 0.004  $ \\ 
\noalign{\smallskip} 
\hline \hline 
\noalign{\smallskip} 
Oracle: $10^{-3} \times $   & $ 5.46 \pm 0.02 $  & $ 13.39 \pm 0.05 $  & $ 43.86 \pm 0.09 $  & $ 62.37 \pm 0.13 $ \\ 
Best: $10^{-3} \times $   & $ 19.38 \pm 0.10 $  & $ 25.77 \pm 0.10 $  & $ 83.39 \pm 0.22 $  & $ 100.86 \pm 0.23 $ \\ 
\hline
\end{tabular}
\end{center}
\caption{Simulation results: settings L and S, $n=500$. 
The best procedures \textup{(}up to standard-deviations\textup{)} are bolded, where the data-driven procedures are considered separately from the procedures using the knowledge of $\E\crochs{\penid}$. 
\label{tab.complet.Li01.SiG5}}
\end{table}
Table~\ref{tab.complet.Li01.SiG5.n100} provides a similar comparison of model selection performances with a reduced sample size $n=100$, again from $N=10\,000$ independent samples. 
\begin{table}[p] 
\begin{center}
%%%%results_C140630_Li01_SiG5_n100_Nbech10000_seed2
%
\begin{tabular}
{l@{\hspace{0.025\textwidth}}c@{\hspace{0.025\textwidth}}c@{\hspace{0.025\textwidth}}c@{\hspace{0.025\textwidth}}c}
\hline\noalign{\smallskip}
Experiment  & L--Dya2 & L--Regu  & S--Dya2 & S--Regu \\ 
%Experiment  & Li01Dya2 & Li01Regu  & SiG5Dya2 & SiG5Regu \\ 
\noalign{\smallskip} 
\hline 
\noalign{\smallskip} 
$\E[\penid]$   & $  8.38 \pm 0.08  $  & $  3.29 \pm 0.03  $  & $  1.97 \pm 0.01  $  & $  2.09 \pm 0.01  $ \\ 
1.25 $\times$ $\E[\penid]$   & $  6.53 \pm 0.07  $  & $  2.61 \pm 0.02  $  & $  { \bf  1.93 \pm 0.01  }  $  & $  1.72 \pm 0.01  $ \\ 
1.5 $\times$ $\E[\penid]$   & $  5.59 \pm 0.06  $  & $  { \bf  2.46 \pm 0.02  }  $  & $  { \bf  1.92 \pm 0.01  }  $  & $  1.61 \pm 0.01  $ \\ 
2 $\times$ $\E[\penid]$   & $  { \bf  4.72 \pm 0.05  }  $  & $  2.57 \pm 0.01  $  & $  \phantom{5}1.94 \pm 0.005  $  & $  { \bf  \phantom{4}1.60 \pm 0.004  }  $ \\ 
\noalign{\smallskip} 
\hline \hline 
\noalign{\smallskip} 
 $\pendim$    & $  9.67 \pm 0.09  $  & $  3.28 \pm 0.03  $  & $  2.17 \pm 0.01  $  & $  2.09 \pm 0.01  $ \\ 
1.25 $\times$  $\pendim$    & $  7.85 \pm 0.08  $  & $  2.62 \pm 0.02  $  & $  2.10 \pm 0.01  $  & $  1.72 \pm 0.01  $ \\ 
1.5 $\times$  $\pendim$    & $  6.74 \pm 0.07  $  & $  { \bf  2.48 \pm 0.02  }  $  & $  2.05 \pm 0.01  $  & $  1.62 \pm 0.01  $ \\ 
2 $\times$  $\pendim$    & $  5.70 \pm 0.06  $  & $  2.60 \pm 0.01  $  & $  2.00 \pm 0.01  $  & $  \phantom{4}1.61 \pm 0.004  $ \\ 
\noalign{\smallskip} 
\hline 
\noalign{\smallskip} 
$\penLOO$   & $  8.10 \pm 0.08  $  & $  3.29 \pm 0.03  $  & $  1.97 \pm 0.01  $  & $  2.09 \pm 0.01  $ \\ 
1.25 $\times$ $\penLOO$   & $  6.20 \pm 0.06  $  & $  2.62 \pm 0.02  $  & $  1.92 \pm 0.01  $  & $  1.72 \pm 0.01  $ \\ 
1.5 $\times$ $\penLOO$   & $  5.18 \pm 0.05  $  & $  { \bf  2.49 \pm 0.02  }  $  & $  1.91 \pm 0.01  $  & $  1.62 \pm 0.01  $ \\ 
2 $\times$ $\penLOO$   & $  { \bf  4.44 \pm 0.04  }  $  & $  2.59 \pm 0.01  $  & $  \phantom{5}1.94 \pm 0.005  $  & $  \phantom{4}1.61 \pm 0.004  $ \\ 
\noalign{\smallskip} 
\hline 
\noalign{\smallskip} 
 $\penVF$  ($V$=10)   & $  8.61 \pm 0.08  $  & $  3.54 \pm 0.04  $  & $  1.97 \pm 0.01  $  & $  2.21 \pm 0.01  $ \\ 
1.25 $\times$  $\penVF$  ($V$=10)   & $  6.76 \pm 0.07  $  & $  2.76 \pm 0.02  $  & $  1.92 \pm 0.01  $  & $  1.78 \pm 0.01  $ \\ 
1.5 $\times$  $\penVF$  ($V$=10)   & $  5.77 \pm 0.06  $  & $  2.52 \pm 0.02  $  & $  { \bf  1.90 \pm 0.01  }  $  & $  1.64 \pm 0.01  $ \\ 
2 $\times$  $\penVF$  ($V$=10)   & $  4.81 \pm 0.05  $  & $  2.57 \pm 0.01  $  & $  1.91 \pm 0.01  $  & $  { \bf  \phantom{4}1.60 \pm 0.004  }  $ \\ 
\noalign{\smallskip} 
\hline 
\noalign{\smallskip} 
 $\penVF$  ($V$=5)   & $  9.14 \pm 0.08  $  & $  3.92 \pm 0.04  $  & $  1.98 \pm 0.01  $  & $  2.34 \pm 0.02  $ \\ 
1.25 $\times$  $\penVF$  ($V$=5)   & $  7.38 \pm 0.07  $  & $  2.90 \pm 0.03  $  & $  1.93 \pm 0.01  $  & $  1.85 \pm 0.01  $ \\ 
1.5 $\times$  $\penVF$  ($V$=5)   & $  6.31 \pm 0.06  $  & $  2.60 \pm 0.02  $  & $  { \bf  1.91 \pm 0.01  }  $  & $  1.68 \pm 0.01  $ \\ 
2 $\times$  $\penVF$  ($V$=5)   & $  5.21 \pm 0.05  $  & $  2.56 \pm 0.02  $  & $  { \bf  1.90 \pm 0.01  }  $  & $  { \bf  \phantom{5}1.60 \pm 0.005  }  $ \\ 
\noalign{\smallskip} 
\hline 
\noalign{\smallskip} 
 $\penVF$  ($V$=2)   & $  11.15 \pm 0.09  $  & $  6.14 \pm 0.08  $  & $  2.01 \pm 0.01  $  & $  2.92 \pm 0.02  $ \\ 
1.25 $\times$  $\penVF$  ($V$=2)   & $  9.61 \pm 0.08  $  & $  4.05 \pm 0.05  $  & $  1.97 \pm 0.01  $  & $  2.24 \pm 0.01  $ \\ 
1.5 $\times$  $\penVF$  ($V$=2)   & $  8.60 \pm 0.07  $  & $  3.30 \pm 0.03  $  & $  1.94 \pm 0.01  $  & $  1.94 \pm 0.01  $ \\ 
2 $\times$  $\penVF$  ($V$=2)   & $  7.30 \pm 0.07  $  & $  2.80 \pm 0.02  $  & $  1.91 \pm 0.01  $  & $  1.70 \pm 0.01  $ \\ 
\noalign{\smallskip} 
\hline 
\noalign{\smallskip} 
LOO   & $  8.04 \pm 0.08  $  & $  3.26 \pm 0.03  $  & $  1.97 \pm 0.01  $  & $  2.07 \pm 0.01  $ \\ 
10-fold CV    & $  8.11 \pm 0.08  $  & $  3.28 \pm 0.03  $  & $  1.95 \pm 0.01  $  & $  2.06 \pm 0.01  $ \\ 
5-fold CV    & $  8.15 \pm 0.08  $  & $  3.28 \pm 0.03  $  & $  1.95 \pm 0.01  $  & $  2.01 \pm 0.01  $ \\ 
2-fold CV    & $  8.60 \pm 0.07  $  & $  3.30 \pm 0.03  $  & $  1.94 \pm 0.01  $  & $  1.94 \pm 0.01  $ \\ 
\noalign{\smallskip} 
\hline \hline 
\noalign{\smallskip} 
Oracle: $10^{-3} \times $   & $ 12.66 \pm 0.05 $  & $ 33.58 \pm 0.16 $  & $ 118.21 \pm 0.25 $  & $ 133.04 \pm 0.28 $ \\ 
Best: $10^{-3} \times $   & $ 56.15 \pm 0.53 $  & $ 83.42 \pm 0.51 $  & $ 224.09 \pm 0.63 $  & $ 212.84 \pm 0.61 $ \\ 
\hline
\end{tabular}
\end{center}
\caption{Simulation results: settings L and S, $n=100$. %%, $N=10\,000$.  
The best procedures \textup{(}up to standard-deviations\textup{)} are bolded, where the data-driven procedures are considered separately from the procedures using the knowledge of $\E\crochs{\penid}$. 
\label{tab.complet.Li01.SiG5.n100}}
\end{table}

\clearpage
The influence of overpenalization is considered in Figures~\ref{fig.supmat.surpen.C140630_Li01Dya2_n100}--\ref{fig.supmat.surpen.C140704_SiG5Regu_n500}. 
As on Figure~\ref{fig.surpen.LS-Dya2.n500}, the top graph represents the estimated model selection performance $\Cor(\CV_{(C,\B)})$ as a function of $C$, for various values of $V=|\B|$. 
Error bars are not shown on these graphs for clarity; all visible differences on the graph correspond to significant differences, as can be seen in Tables~\ref{tab.complet.Li01.SiG5}--\ref{tab.complet.Li01.SiG5.n100} for instance. 
The bottom tables in Figures~\ref{fig.supmat.surpen.C140630_Li01Dya2_n100}--\ref{fig.supmat.surpen.C140704_SiG5Regu_n500} 
show the estimated model selection performance for three key values of $C$: the optimal one $C^{\star}_n$, the unbiased case ($C=1$, which corresponds to an AIC-type penalty) and the value $C=\log(n)/2$ (which corresponds to a BIC-type penalty). 
The estimated value of the optimal overpenalizing constant $C^{\star}_n$ was obtained by minimizing over $C \in [0,10]$ the estimated value of $\Cor(\CV_{(C,\B)})$. 
Error bars on $C^{\star}_n$ show the maximum of $|C^{\star}_n-C|$ over the set of values of $C$ that are ``not significantly worse than $C^{\star}_n$'', where we define by convention ``significantly worse'' as having a $\abss{ \Cor(\CV_{(C^{\star}_n,\B)}) - \Cor(\CV_{(C,\B)}) }$ larger than the sum of the corresponding error bars. 

\begin{figure}
\begin{center}
\begin{minipage}[b]{\largfiguniq}
\includegraphics[width=\textwidth]{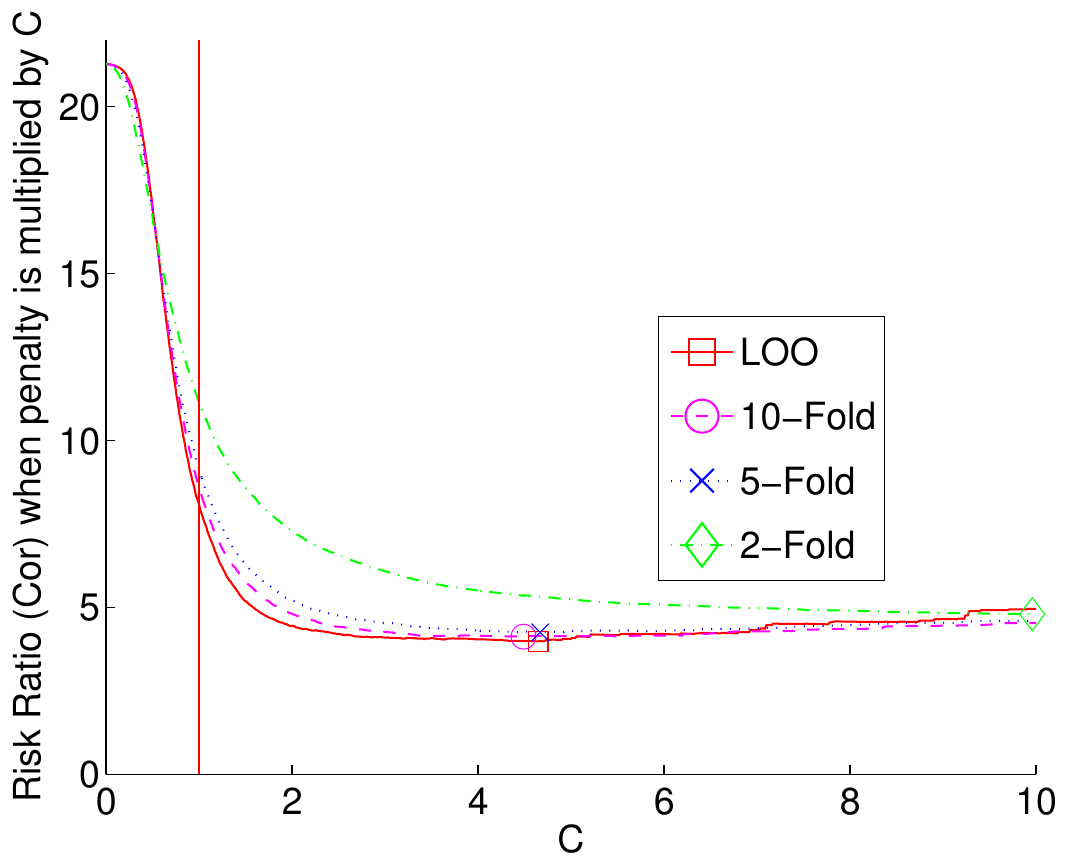}
\end{minipage}

\begin{tabular}{l@{\hspace{0.025\textwidth}}c@{\hspace{0.025\textwidth}}c@{\hspace{0.025\textwidth}}c@{\hspace{0.025\textwidth}}c@{\hspace{0.025\textwidth}}}
\hline\noalign{\smallskip}
 Penalty & $C^{\star}_n$ & \Cor ($C=C^{\star}_n$)  & \Cor ($C=1$)  & \Cor ($C=2.30$) \\ 
\noalign{\smallskip} 
\hline 
\noalign{\smallskip} 
$\E[\penid]$   & $ 6.58 \pm 1.33 $  & $ 3.89 \pm 0.02 $  &  $ 8.38 \pm 0.08$  &  $\mathbf{  4.46 \pm 0.04 }$ \\ 
 $\penLOO$   & $ 4.65 \pm 0.54 $  & $ 3.98 \pm 0.02 $  &  $ 8.10 \pm 0.08$  &  $\mathbf{  4.29 \pm 0.04 }$ \\ 
 $\pen 10$F   & $ 4.49 \pm 1.71 $  & $ 4.12 \pm 0.03 $  &  $ 8.61 \pm 0.08$  &  $\mathbf{  4.53 \pm 0.04 }$ \\ 
 $\pen 5$F   & $ 4.67 \pm 1.88 $  & $ 4.25 \pm 0.03 $  &  $ 9.14 \pm 0.08$  &  $\mathbf{  4.90 \pm 0.05 }$ \\ 
 $\pen 2$F   & $ 9.96 \pm 1.61 $  & $ 4.79 \pm 0.04 $  &  $ 11.15 \pm 0.09$  &  $\mathbf{  6.80 \pm 0.06 }$ \\ 
 \hline
\end{tabular}
\end{center}
\caption{%
Overpenalization in setting L-Dya2, $n=100$. \newline
Top: same as Figure~\ref{fig.surpen.LS-Dya2.n500} (estimated loss ratio as a function of the overpenalization constant $C$). \newline
Bottom: Table showing the estimated optimal overpenalization constant $C^{\star}_n$ as well as the estimated loss ratio for several values of $C$: $C = C^{\star}_n$ (optimal value), $C=1$ (AIC-type penalty) and $C=\log(n)/2$ (BIC-type penalty). 
See text for details. 
\label{fig.supmat.surpen.C140630_Li01Dya2_n100}
}
\end{figure}

\begin{figure}
\begin{center}
\begin{minipage}[b]{\largfiguniq}
\includegraphics[width=\textwidth]{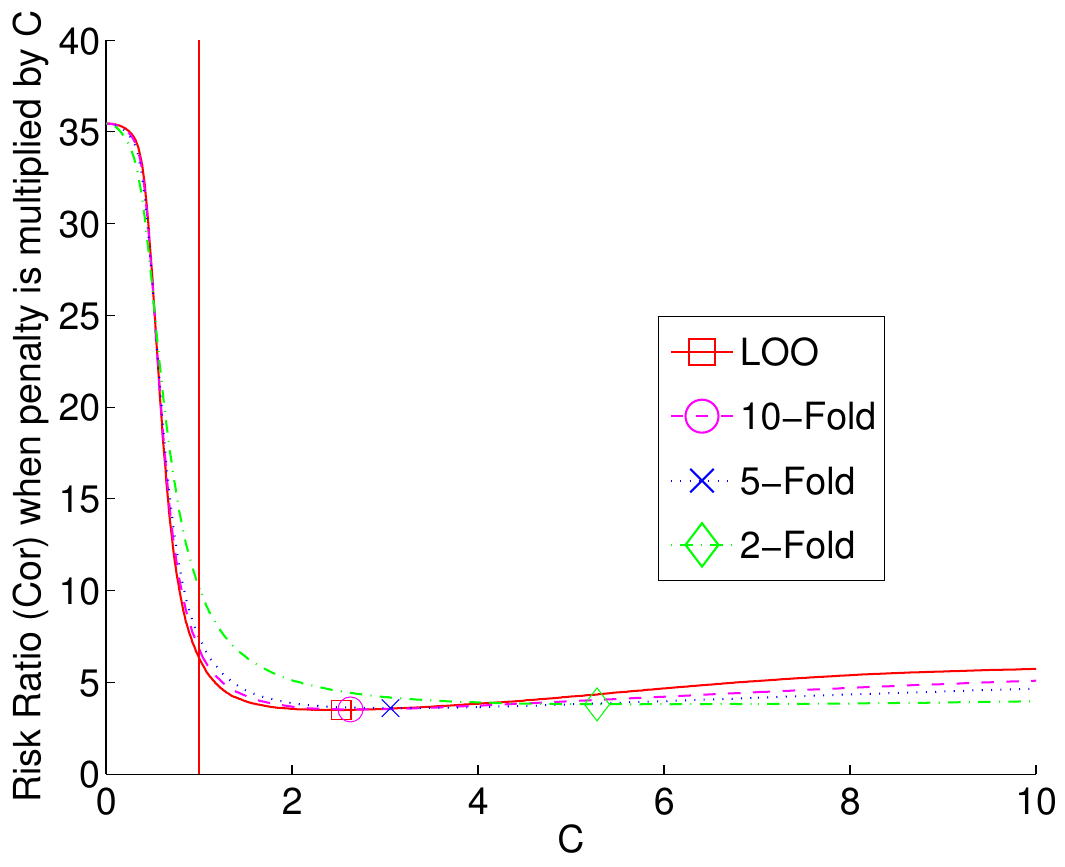}
\end{minipage}

\begin{tabular}{l@{\hspace{0.025\textwidth}}c@{\hspace{0.025\textwidth}}c@{\hspace{0.025\textwidth}}c@{\hspace{0.025\textwidth}}c@{\hspace{0.025\textwidth}}}
\hline\noalign{\smallskip}
 Penalty & $C^{\star}_n$ & \Cor ($C=C^{\star}_n$)  & \Cor ($C=1$)  & \Cor ($C=3.11$) \\ 
\noalign{\smallskip} 
\hline 
\noalign{\smallskip} 
$\E[\penid]$   & $ 2.56 \pm 0.40 $  & $ 3.53 \pm 0.02 $  &  $ 6.52 \pm 0.05$  &  $\mathbf{  3.59 \pm 0.02 }$ \\ 
 $\penLOO$   & $ 2.53 \pm 0.43 $  & $ 3.49 \pm 0.02 $  &  $ 6.35 \pm 0.05$  &  $\mathbf{  3.58 \pm 0.02 }$ \\ 
 $\pen 10$F   & $ 2.63 \pm 0.49 $  & $ 3.52 \pm 0.02 $  &  $ 6.89 \pm 0.06$  &  $\mathbf{  3.55 \pm 0.02 }$ \\ 
 $\pen 5$F   & $ 3.06 \pm 0.67 $  & $ 3.59 \pm 0.02 $  &  $ 7.47 \pm 0.06$  &  $\mathbf{  3.59 \pm 0.02 }$ \\ 
 $\pen 2$F   & $ 5.28 \pm 1.90 $  & $ 3.80 \pm 0.02 $  &  $ 10.21 \pm 0.08$  &  $\mathbf{  4.14 \pm 0.03 }$ \\ 
 \hline
\end{tabular}
\end{center}
\caption{%
Overpenalization in setting L-Dya2, $n=500$. 
\newline
Top: same as Figure~\ref{fig.surpen.LS-Dya2.n500} (estimated loss ratio as a function of the overpenalization constant $C$). \newline
Bottom: Table showing the estimated optimal overpenalization constant $C^{\star}_n$ as well as the estimated loss ratio for several values of $C$: $C = C^{\star}_n$ (optimal value), $C=1$ (AIC-type penalty) and $C=\log(n)/2$ (BIC-type penalty). 
See text for details. 
\label{fig.supmat.surpen.C140704_Li01Dya2_n500}
}
\end{figure}

\clearpage

\begin{figure}
\begin{center}
\begin{minipage}[b]{\largfiguniq}
\includegraphics[width=\textwidth]{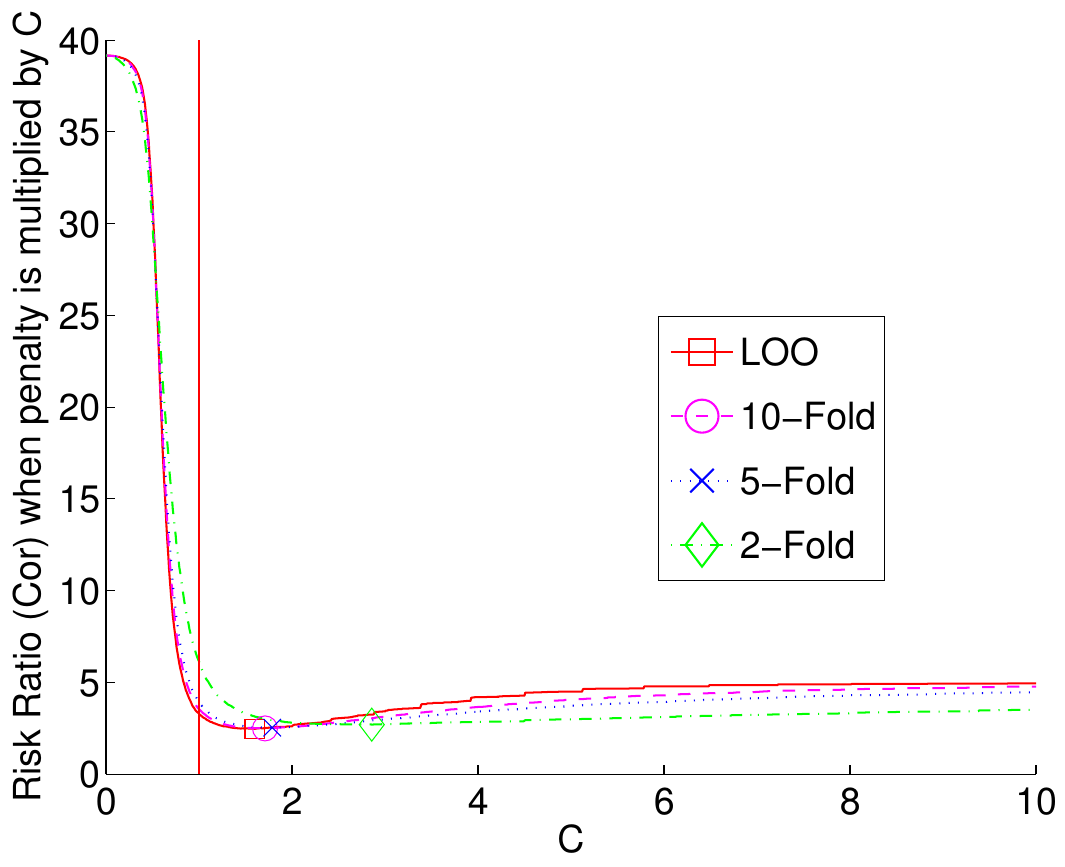}
\end{minipage}

\begin{tabular}{l@{\hspace{0.025\textwidth}}c@{\hspace{0.025\textwidth}}c@{\hspace{0.025\textwidth}}c@{\hspace{0.025\textwidth}}c@{\hspace{0.025\textwidth}}}
\hline\noalign{\smallskip}
 Penalty & $C^{\star}_n$ & \Cor ($C=C^{\star}_n$)  & \Cor ($C=1$)  & \Cor ($C=2.30$) \\ 
\noalign{\smallskip} 
\hline 
\noalign{\smallskip} 
$\E[\penid]$   & $ 1.66 \pm 0.21 $  & $ 2.44 \pm 0.01 $  &  $ 3.29 \pm 0.03$  &  $\mathbf{  2.77 \pm 0.01 }$ \\ 
 $\penLOO$   & $ 1.60 \pm 0.18 $  & $ 2.47 \pm 0.01 $  &  $ 3.29 \pm 0.03$  &  $\mathbf{  2.81 \pm 0.01 }$ \\ 
 $\pen 10$F   & $ 1.71 \pm 0.22 $  & $ 2.49 \pm 0.02 $  &  $ 3.54 \pm 0.04$  &  $\mathbf{  2.71 \pm 0.02 }$ \\ 
 $\pen 5$F   & $ 1.79 \pm 0.31 $  & $ 2.53 \pm 0.02 $  &  $ 3.92 \pm 0.04$  &  $\mathbf{  2.67 \pm 0.02 }$ \\ 
 $\pen 2$F   & $ 2.86 \pm 0.58 $  & $ 2.70 \pm 0.02 $  &  $ 6.14 \pm 0.08$  &  $\mathbf{  2.74 \pm 0.02 }$ \\ 
 \hline
\end{tabular}
\end{center}
\caption{%
Overpenalization in setting L-Regu, $n=100$. 
\newline
Top: same as Figure~\ref{fig.surpen.LS-Dya2.n500} (estimated loss ratio as a function of the overpenalization constant $C$). \newline
Bottom: Table showing the estimated optimal overpenalization constant $C^{\star}_n$ as well as the estimated loss ratio for several values of $C$: $C = C^{\star}_n$ (optimal value), $C=1$ (AIC-type penalty) and $C=\log(n)/2$ (BIC-type penalty). 
See text for details. 
\label{fig.supmat.surpen.C140630_Li01Regu_n100}
}
\end{figure}

\begin{figure}
\begin{center}
\begin{minipage}[b]{\largfiguniq}
\includegraphics[width=\textwidth]{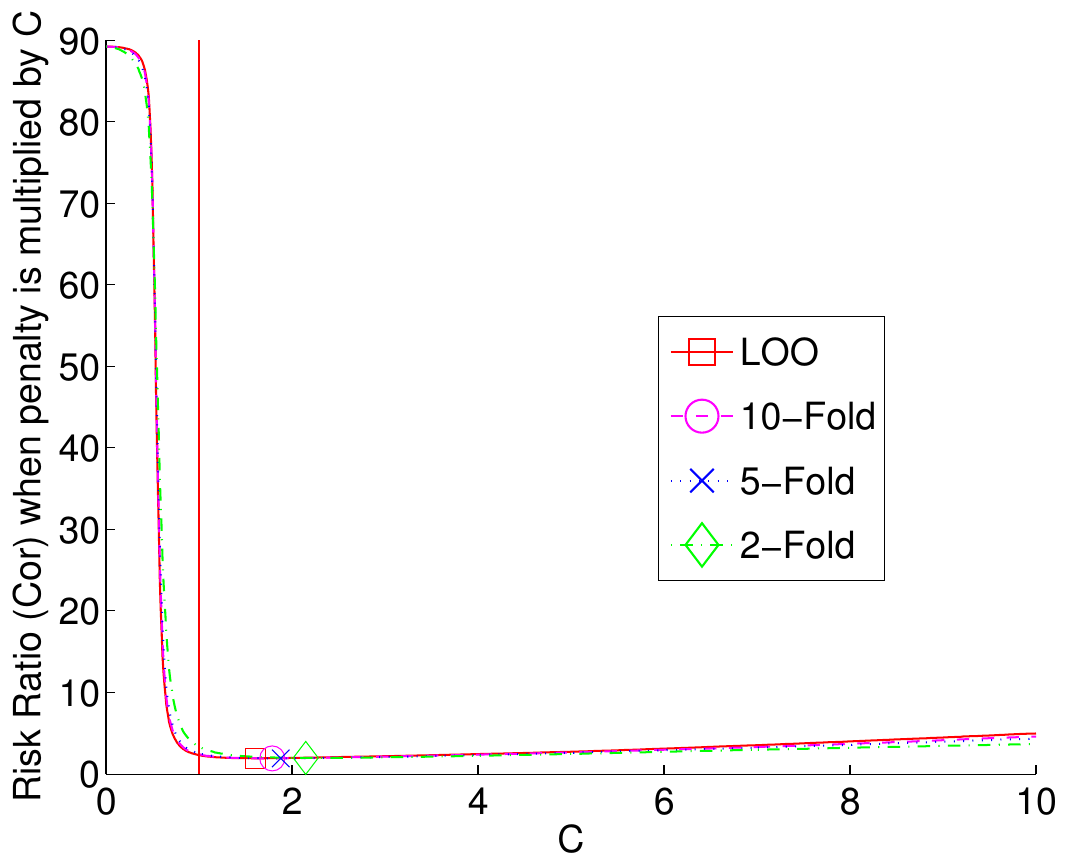}
\end{minipage}

\begin{tabular}{l@{\hspace{0.025\textwidth}}c@{\hspace{0.025\textwidth}}c@{\hspace{0.025\textwidth}}c@{\hspace{0.025\textwidth}}c@{\hspace{0.025\textwidth}}}
\hline\noalign{\smallskip}
 Penalty & $C^{\star}_n$ & \Cor ($C=C^{\star}_n$)  & \Cor ($C=1$)  & \Cor ($C=3.11$) \\ 
\noalign{\smallskip} 
\hline 
\noalign{\smallskip} 
$\E[\penid]$   & $ 1.63 \pm 0.25 $  & $ 1.93 \pm 0.01 $  &  $ 2.33 \pm 0.01$  &  $\mathbf{  2.20 \pm 0.01 }$ \\ 
 $\penLOO$   & $ 1.61 \pm 0.23 $  & $ 1.93 \pm 0.01 $  &  $ 2.33 \pm 0.01$  &  $\mathbf{  2.21 \pm 0.01 }$ \\ 
 $\pen 10$F   & $ 1.79 \pm 0.26 $  & $ 1.92 \pm 0.01 $  &  $ 2.42 \pm 0.02$  &  $\mathbf{  2.16 \pm 0.01 }$ \\ 
 $\pen 5$F   & $ 1.88 \pm 0.29 $  & $ 1.92 \pm 0.01 $  &  $ 2.55 \pm 0.02$  &  $\mathbf{  2.11 \pm 0.01 }$ \\ 
 $\pen 2$F   & $ 2.15 \pm 0.34 $  & $ 1.97 \pm 0.01 $  &  $ 3.37 \pm 0.03$  &  $\mathbf{  2.07 \pm 0.01 }$ \\ 
 \hline
\end{tabular}
\end{center}
\caption{%
Overpenalization in setting L-Regu, $n=500$. 
\newline
Top: same as Figure~\ref{fig.surpen.LS-Dya2.n500} (estimated loss ratio as a function of the overpenalization constant $C$). \newline
Bottom: Table showing the estimated optimal overpenalization constant $C^{\star}_n$ as well as the estimated loss ratio for several values of $C$: $C = C^{\star}_n$ (optimal value), $C=1$ (AIC-type penalty) and $C=\log(n)/2$ (BIC-type penalty). 
See text for details. 
\label{fig.supmat.surpen.C140704_Li01Regu_n500}
}
\end{figure}

\clearpage

\begin{figure}
\begin{center}
\begin{minipage}[b]{\largfiguniq}
\includegraphics[width=\textwidth]{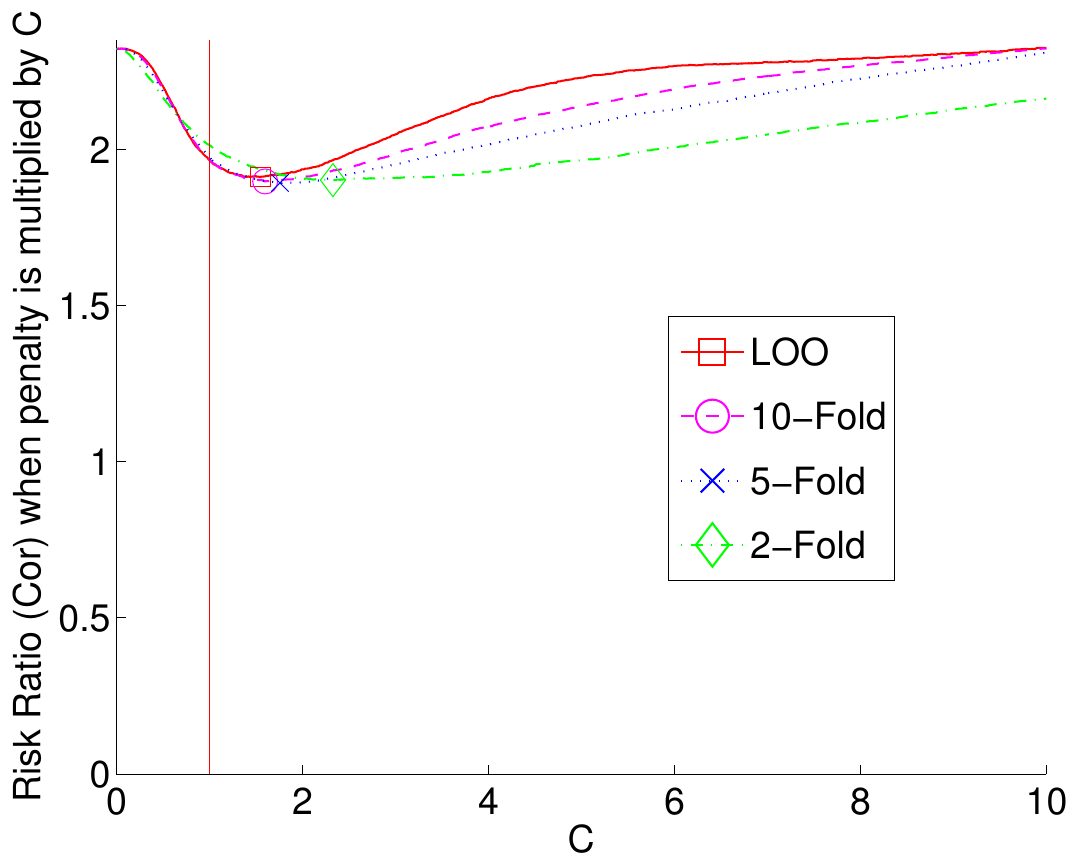}
\end{minipage}

\begin{tabular}{l@{\hspace{0.025\textwidth}}c@{\hspace{0.025\textwidth}}c@{\hspace{0.025\textwidth}}c@{\hspace{0.025\textwidth}}c@{\hspace{0.025\textwidth}}}
\hline\noalign{\smallskip}
 Penalty & $C^{\star}_n$ & \Cor ($C=C^{\star}_n$)  & \Cor ($C=1$)  & \Cor ($C=2.30$) \\ 
\noalign{\smallskip} 
\hline 
\noalign{\smallskip} 
$\E[\penid]$   & $ 1.54 \pm 0.28 $  & $ 1.91 \pm 0.01 $  &  $\mathbf{  1.97 \pm 0.01 }$  &  $\mathbf{  1.97 \pm 0.00 }$ \\ 
 $\penLOO$   & $ 1.55 \pm 0.32 $  & $ 1.91 \pm 0.01 $  &  $\mathbf{  1.97 \pm 0.01 }$  &  $\mathbf{  1.96 \pm 0.00 }$ \\ 
 $\pen 10$F   & $ 1.60 \pm 0.31 $  & $ 1.90 \pm 0.01 $  &  $ 1.97 \pm 0.01$  &  $\mathbf{  1.93 \pm 0.01 }$ \\ 
 $\pen 5$F   & $ 1.76 \pm 0.39 $  & $ 1.89 \pm 0.01 $  &  $ 1.98 \pm 0.01$  &  $\mathbf{  1.91 \pm 0.01 }$ \\ 
 $\pen 2$F   & $ 2.33 \pm 1.09 $  & $ 1.90 \pm 0.01 $  &  $ 2.01 \pm 0.01$  &  $\mathbf{  1.90 \pm 0.01 }$ \\ 
 \hline
\end{tabular}
\end{center}
\caption{%
Overpenalization in setting S-Dya2, $n=100$. 
\newline
Top: same as Figure~\ref{fig.surpen.LS-Dya2.n500} (estimated loss ratio as a function of the overpenalization constant $C$). \newline
Bottom: Table showing the estimated optimal overpenalization constant $C^{\star}_n$ as well as the estimated loss ratio for several values of $C$: $C = C^{\star}_n$ (optimal value), $C=1$ (AIC-type penalty) and $C=\log(n)/2$ (BIC-type penalty). 
See text for details. 
\label{fig.supmat.surpen.C140630_SiG5Dya2_n100}
}
\end{figure}

\begin{figure}
\begin{center}
\begin{minipage}[b]{\largfiguniq}
\includegraphics[width=\textwidth]{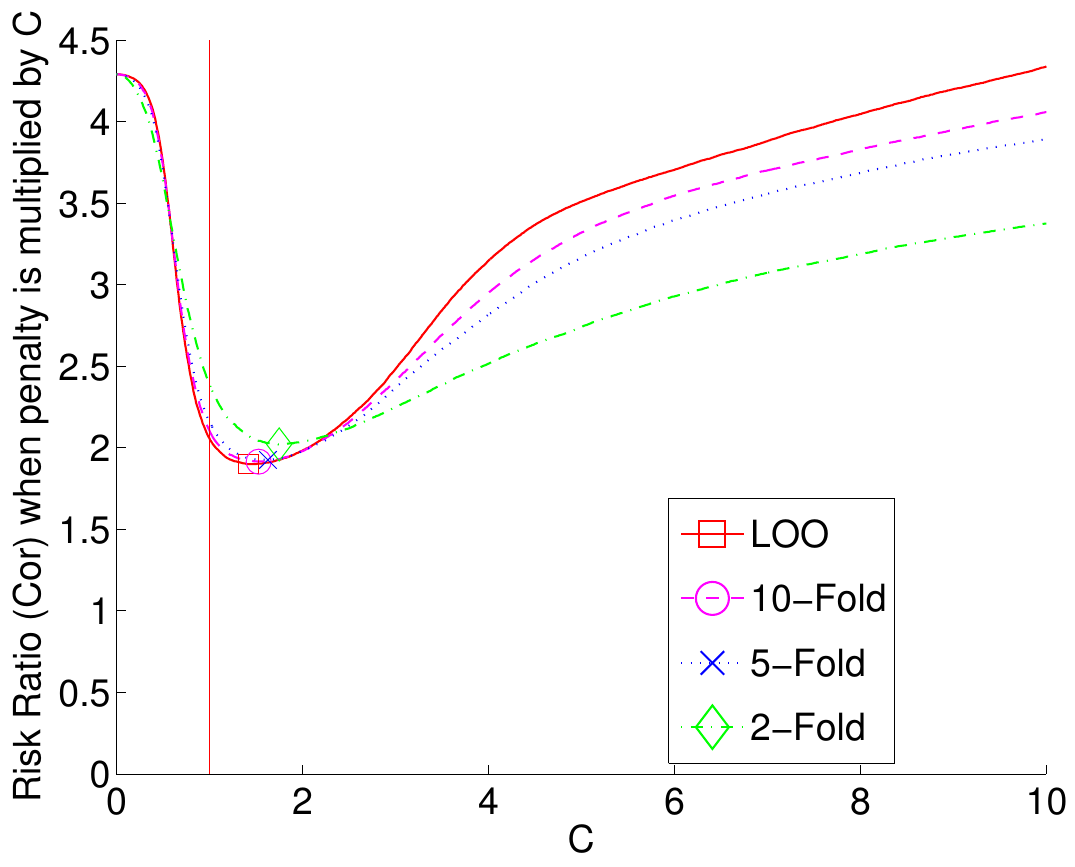}
\end{minipage}

\begin{tabular}{l@{\hspace{0.025\textwidth}}c@{\hspace{0.025\textwidth}}c@{\hspace{0.025\textwidth}}c@{\hspace{0.025\textwidth}}c@{\hspace{0.025\textwidth}}}
\hline\noalign{\smallskip}
 Penalty & $C^{\star}_n$ & \Cor ($C=C^{\star}_n$)  & \Cor ($C=1$)  & \Cor ($C=3.11$) \\ 
\noalign{\smallskip} 
\hline 
\noalign{\smallskip} 
$\E[\penid]$   & $ 1.44 \pm 0.15 $  & $ 1.92 \pm 0.01 $  &  $\mathbf{  2.07 \pm 0.01 }$  &  $ 2.58 \pm 0.01$ \\ 
 $\penLOO$   & $ 1.42 \pm 0.19 $  & $ 1.90 \pm 0.01 $  &  $\mathbf{  2.06 \pm 0.01 }$  &  $ 2.56 \pm 0.01$ \\ 
 $\pen 10$F   & $ 1.53 \pm 0.19 $  & $ 1.92 \pm 0.01 $  &  $\mathbf{  2.11 \pm 0.01 }$  &  $ 2.47 \pm 0.01$ \\ 
 $\pen 5$F   & $ 1.63 \pm 0.21 $  & $ 1.93 \pm 0.01 $  &  $\mathbf{  2.16 \pm 0.01 }$  &  $ 2.42 \pm 0.01$ \\ 
 $\pen 2$F   & $ 1.75 \pm 0.23 $  & $ 2.02 \pm 0.01 $  &  $ 2.39 \pm 0.01$  &  $\mathbf{  2.28 \pm 0.01 }$ \\ 
 \hline
\end{tabular}
\end{center}
\caption{%
Overpenalization in setting S-Dya2, $n=500$. 
\newline
Top: same as Figure~\ref{fig.surpen.LS-Dya2.n500} (estimated loss ratio as a function of the overpenalization constant $C$). \newline
Bottom: Table showing the estimated optimal overpenalization constant $C^{\star}_n$ as well as the estimated loss ratio for several values of $C$: $C = C^{\star}_n$ (optimal value), $C=1$ (AIC-type penalty) and $C=\log(n)/2$ (BIC-type penalty). 
See text for details. 
\label{fig.supmat.surpen.C140704_SiG5Dya2_n500}
}
\end{figure}

\clearpage

\begin{figure}
\begin{center}
\begin{minipage}[b]{\largfiguniq}
\includegraphics[width=\textwidth]{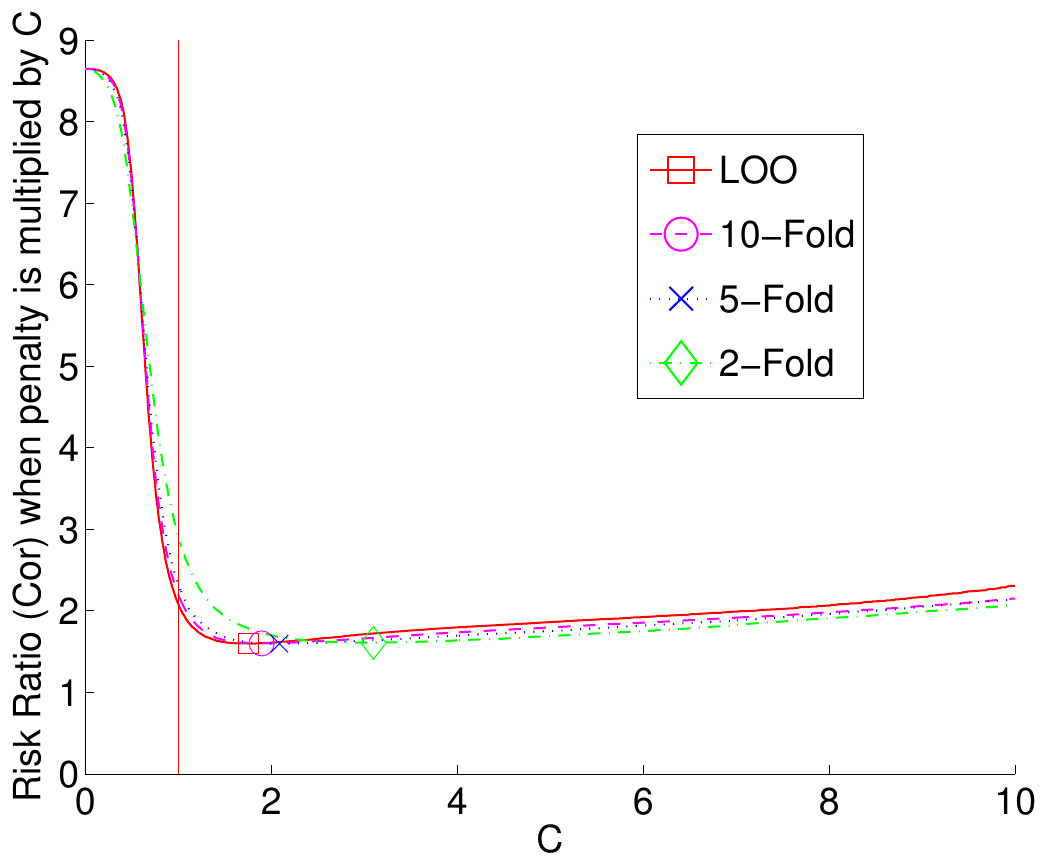}
\end{minipage}

\begin{tabular}{l@{\hspace{0.025\textwidth}}c@{\hspace{0.025\textwidth}}c@{\hspace{0.025\textwidth}}c@{\hspace{0.025\textwidth}}c@{\hspace{0.025\textwidth}}}
\hline\noalign{\smallskip}
 Penalty & $C^{\star}_n$ & \Cor ($C=C^{\star}_n$)  & \Cor ($C=1$)  & \Cor ($C=2.30$) \\ 
\noalign{\smallskip} 
\hline 
\noalign{\smallskip} 
$\E[\penid]$   & $ 1.77 \pm 0.21 $  & $ 1.59 \pm 0.00 $  &  $ 2.09 \pm 0.01$  &  $\mathbf{  1.63 \pm 0.00 }$ \\ 
 $\penLOO$   & $ 1.76 \pm 0.27 $  & $ 1.60 \pm 0.00 $  &  $ 2.09 \pm 0.01$  &  $\mathbf{  1.64 \pm 0.00 }$ \\ 
 $\pen 10$F   & $ 1.90 \pm 0.26 $  & $ 1.60 \pm 0.00 $  &  $ 2.21 \pm 0.01$  &  $\mathbf{  1.62 \pm 0.00 }$ \\ 
 $\pen 5$F   & $ 2.09 \pm 0.34 $  & $ 1.60 \pm 0.00 $  &  $ 2.34 \pm 0.02$  &  $\mathbf{  1.60 \pm 0.00 }$ \\ 
 $\pen 2$F   & $ 3.10 \pm 0.50 $  & $ 1.61 \pm 0.01 $  &  $ 2.92 \pm 0.02$  &  $\mathbf{  1.65 \pm 0.01 }$ \\ 
 \hline
\end{tabular}
\end{center}
\caption{%
Overpenalization in setting S-Regu, $n=100$. 
\newline
Top: same as Figure~\ref{fig.surpen.LS-Dya2.n500} (estimated loss ratio as a function of the overpenalization constant $C$). \newline
Bottom: Table showing the estimated optimal overpenalization constant $C^{\star}_n$ as well as the estimated loss ratio for several values of $C$: $C = C^{\star}_n$ (optimal value), $C=1$ (AIC-type penalty) and $C=\log(n)/2$ (BIC-type penalty). 
See text for details. 
\label{fig.supmat.surpen.C140630_SiG5Regu_n100}
}
\end{figure}

\begin{figure}
\begin{center}
\begin{minipage}[b]{\largfiguniq}
\includegraphics[width=\textwidth]{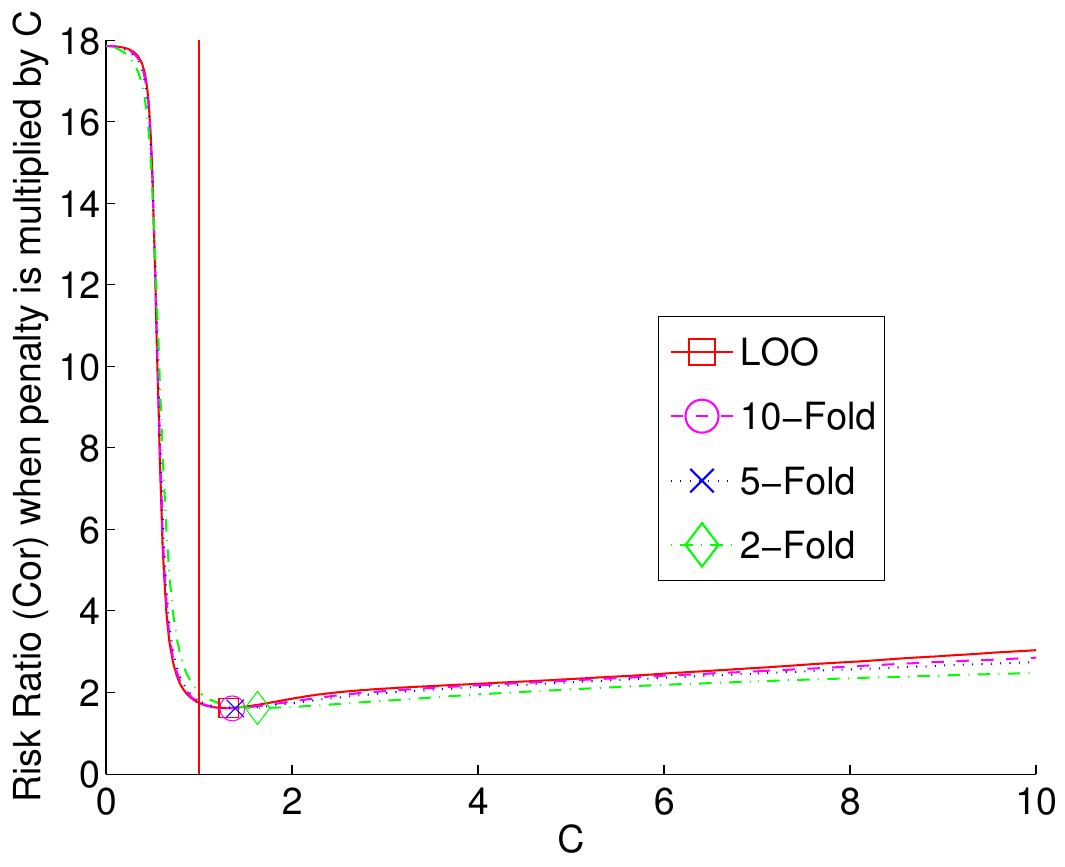}
\end{minipage}

\begin{tabular}{l@{\hspace{0.025\textwidth}}c@{\hspace{0.025\textwidth}}c@{\hspace{0.025\textwidth}}c@{\hspace{0.025\textwidth}}c@{\hspace{0.025\textwidth}}}
\hline\noalign{\smallskip}
 Penalty & $C^{\star}_n$ & \Cor ($C=C^{\star}_n$)  & \Cor ($C=1$)  & \Cor ($C=3.11$) \\ 
\noalign{\smallskip} 
\hline 
\noalign{\smallskip} 
$\E[\penid]$   & $ 1.33 \pm 0.11 $  & $ 1.62 \pm 0.00 $  &  $\mathbf{  1.75 \pm 0.01 }$  &  $ 2.11 \pm 0.00$ \\ 
 $\penLOO$   & $ 1.32 \pm 0.10 $  & $ 1.62 \pm 0.00 $  &  $\mathbf{  1.75 \pm 0.01 }$  &  $ 2.11 \pm 0.00$ \\ 
 $\pen 10$F   & $ 1.36 \pm 0.09 $  & $ 1.61 \pm 0.00 $  &  $\mathbf{  1.77 \pm 0.01 }$  &  $ 2.05 \pm 0.00$ \\ 
 $\pen 5$F   & $ 1.39 \pm 0.08 $  & $ 1.61 \pm 0.00 $  &  $\mathbf{  1.80 \pm 0.01 }$  &  $ 2.01 \pm 0.00$ \\ 
 $\pen 2$F   & $ 1.63 \pm 0.23 $  & $ 1.62 \pm 0.00 $  &  $ 2.01 \pm 0.01$  &  $\mathbf{  1.81 \pm 0.01 }$ \\ 
 \hline
\end{tabular}
\end{center}
\caption{%
Overpenalization in setting S-Regu, $n=500$. 
\newline
Top: same as Figure~\ref{fig.surpen.LS-Dya2.n500} (estimated loss ratio as a function of the overpenalization constant $C$). \newline
Bottom: Table showing the estimated optimal overpenalization constant $C^{\star}_n$ as well as the estimated loss ratio for several values of $C$: $C = C^{\star}_n$ (optimal value), $C=1$ (AIC-type penalty) and $C=\log(n)/2$ (BIC-type penalty). 
See text for details. 
\label{fig.supmat.surpen.C140704_SiG5Regu_n500}
}
\end{figure}

\clearpage

The study of variance of Section~\ref{sec.simus.variance} (setting S with $n=100$) is completed with Figure~\ref{fig.variance.SR-vs-Pmh.SiG5Regu.n100}, which tests the validity of the heuristic of Section~\ref{sec.oracle.key-quant}, Figure~\ref{fig.variance.Pmh-nozoom.SiG5Regu.n100}, which is the equivalent of Figure~\ref{fig.variance.Pmh.SiG5Regu.n100} without zooming on the smallest dimensions, 
and Figure~\ref{fig.variance.SR-vs-Rmo.SiG5-Li01Regu.n100}, which shows that 
\[ \forall m \neq \mo , \quad 
\SR(m) \approx \frac{\E\crochb{ \Delta(m,\mo) }}{\sqrt{\var\parenb{\Delta(m,\mo)}}} \enspace . \]

\medbreak

The next figures present the same results as the ones of 
Section~\ref{sec.simus.variance} about the variance, 
for other experimental settings. 

Figures~\ref{fig.variance.SR-vs-Rmo.SiG5-Li01Regu.n100}--\ref{fig.variance.Pmh-nozoom.Li01Regu.n100} 
show the results for setting L with $n=100$, based upon $N=10\,000$ independent samples.

Figures~\ref{fig.variance.SR-vs-Rmo.SiG5-Li01Regu.n500}--\ref{fig.variance.Pmh-nozoom.Li01Regu.n500} 
show the results for settings S and L with $n=500$, based upon $N=1\,000$ independent samples.

\begin{figure}[p]
\begin{center}
\begin{minipage}[b]{\largfiguniq}
\includegraphics[width=\textwidth]{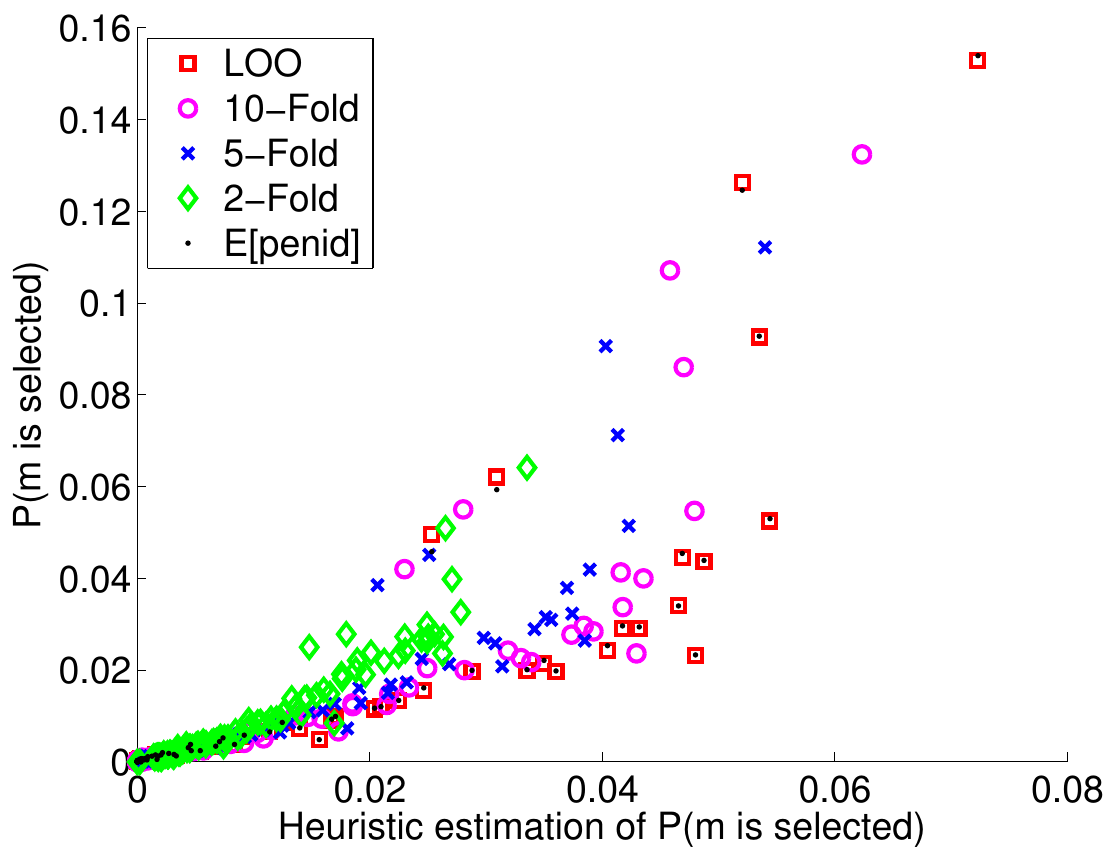}
\end{minipage}
\end{center}
\caption{%
Illustration of the variance heuristic: $\Prob(\mh=m)$ as a function of $\overline{\Phi}(\SR(m))$ (renormalized to have a sum equal to one). 
Setting S-Regu, $n=100$.  
\label{fig.variance.SR-vs-Pmh.SiG5Regu.n100}
}
\end{figure}
\begin{figure}
\begin{center}
\begin{minipage}[b]{\largfiguniq}
\includegraphics[width=\textwidth]{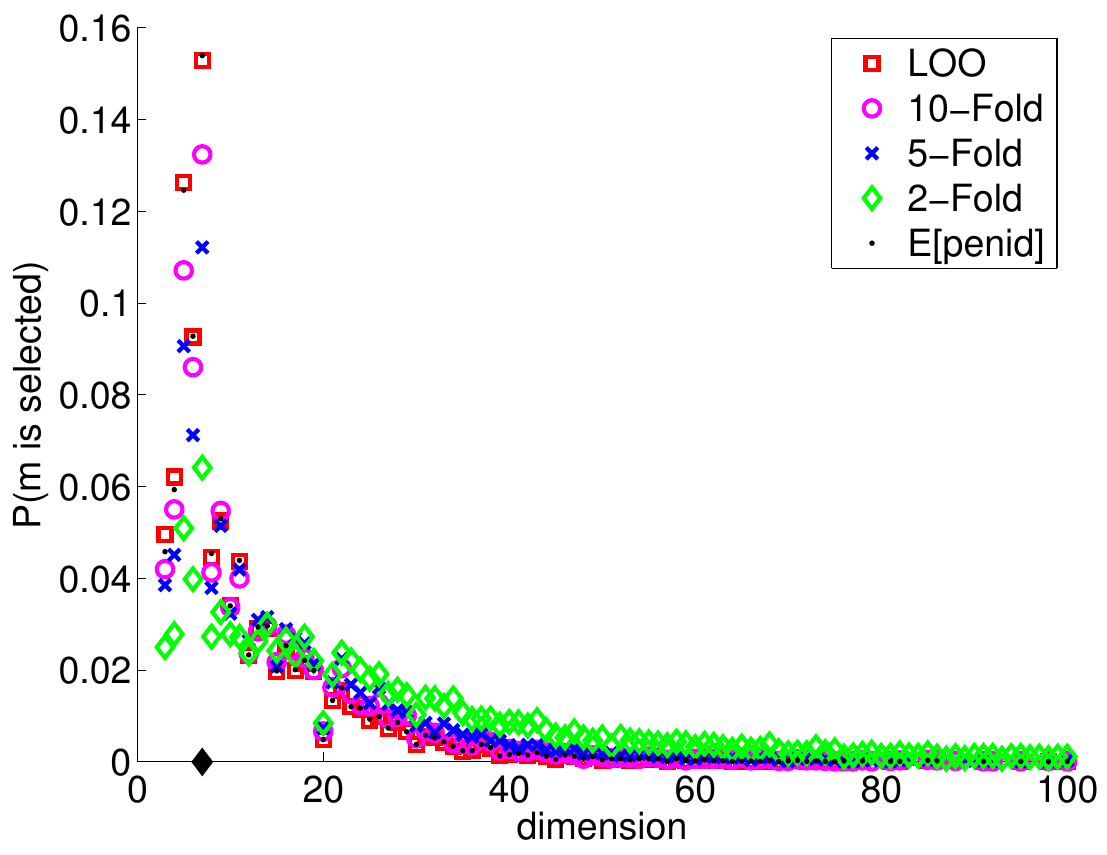}
\end{minipage}
\end{center}
\caption{%
Setting S-Regu, $n=100$.  
$\Prob\parenj{\mh=m}$ as a function of $m$.
The black diamond shows $\mo=7$. 
\label{fig.variance.Pmh-nozoom.SiG5Regu.n100}
}
\end{figure}
\clearpage

\begin{figure}[p]
\begin{center}
\begin{minipage}[b]{.48\linewidth}
\includegraphics[width=\textwidth]{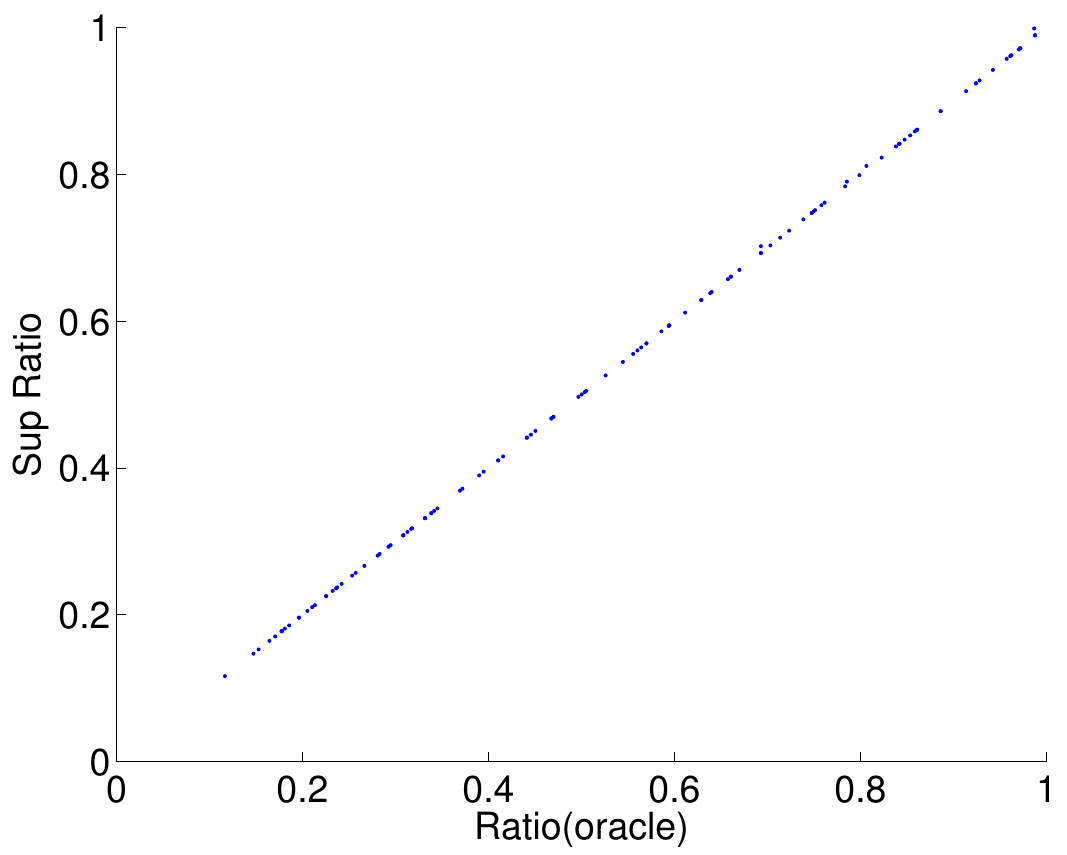}
\end{minipage}
\hspace{.025\linewidth}
\begin{minipage}[b]{.48\linewidth}
\includegraphics[width=\textwidth]{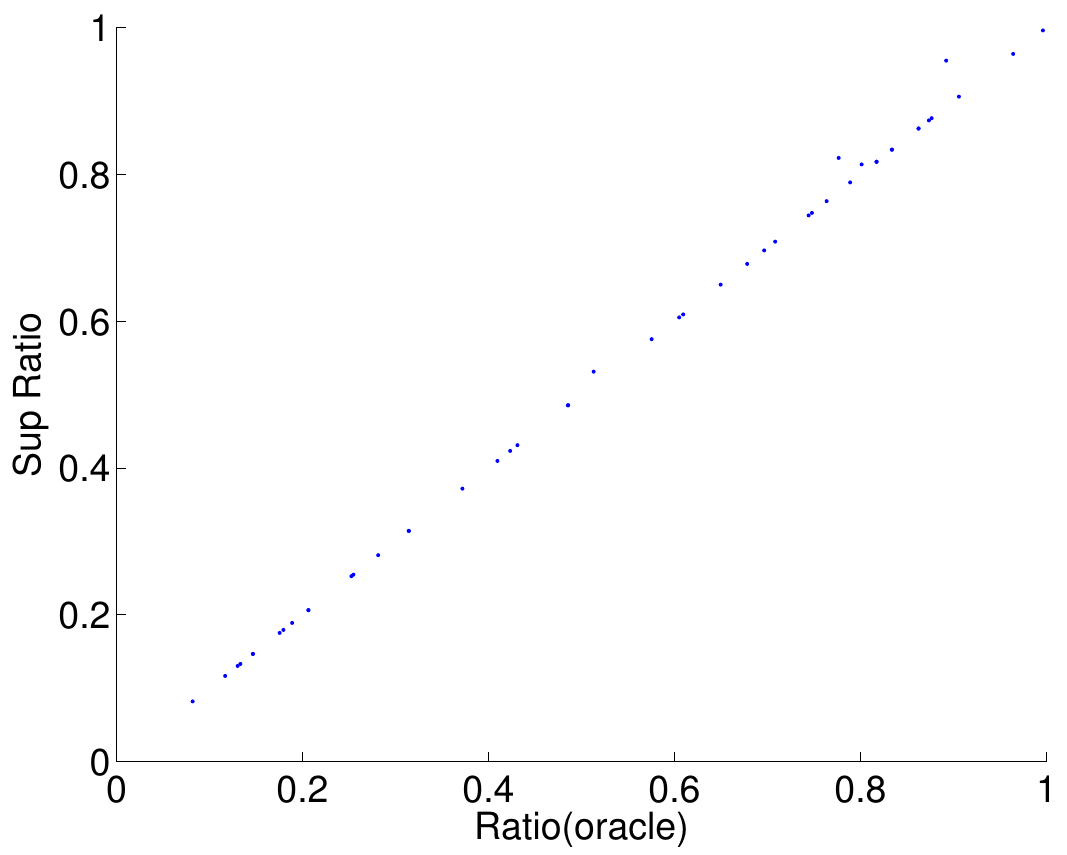}
\end{minipage}
\end{center}
\caption{%
$\SR(m)$ as a function of the ratio at $m^{\prime}=\mo$. 
$n=100$. 
Left: S-Regu. 
Right: L-Regu. 
\label{fig.variance.SR-vs-Rmo.SiG5-Li01Regu.n100}
}
\end{figure}

\begin{figure}
\begin{center}
\begin{minipage}[b]{\largfiguniq}
\includegraphics[width=\textwidth]{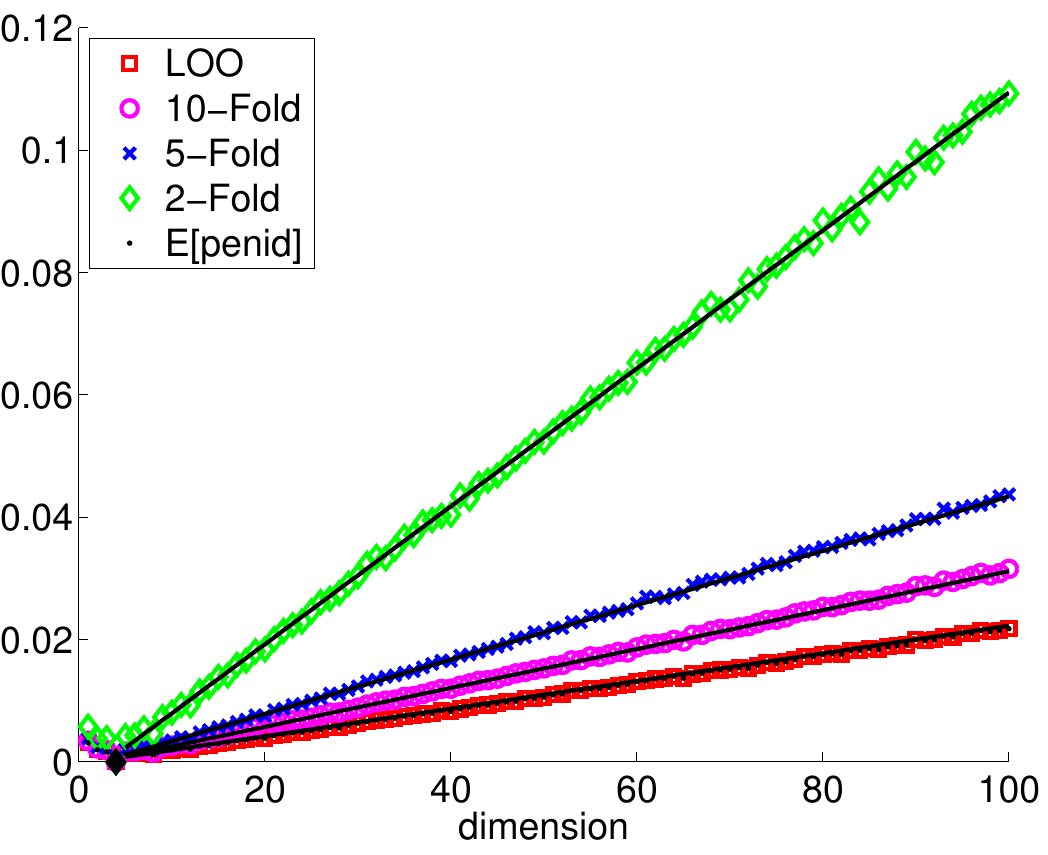}
\end{minipage}
\end{center}
\caption{%
L-Regu, $n=100$. 
$\var(\Delta_{\CV}(m,\mo))$ as a function of $m$. 
The black lines show the linear approximation $n^{-2} [ 5.6 (1 + \frac{1.1}{V-1} ) + 2.2 ( 1 + \frac{4.2 }{V-1} ) (m-\mo) ]$ for $m > \mo = 4$. 
\label{fig.variance.var.Li01Regu.n100}
}
\end{figure}
\begin{figure}
\begin{center}
\begin{minipage}[b]{\largfiguniq}
\includegraphics[width=\textwidth]{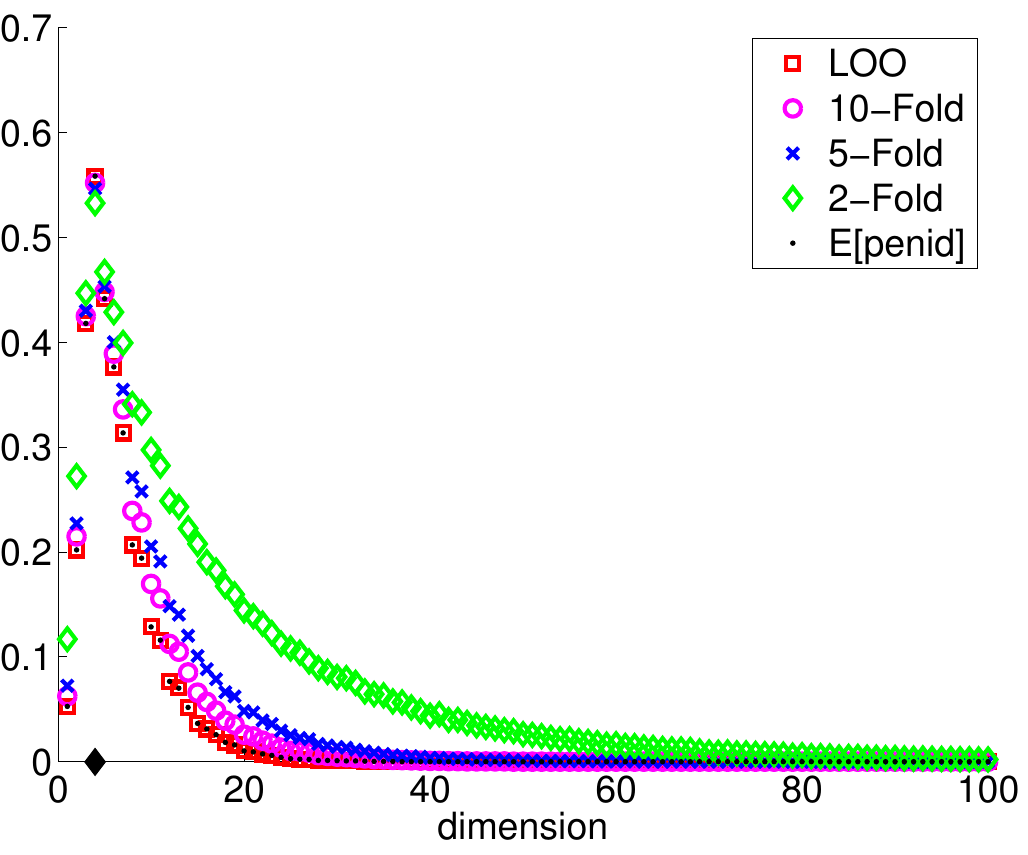}
\end{minipage}
\end{center}
\caption{%
L-Regu, $n=100$. 
$\overline{\Phi}(\SR_{\, \CV}(m))$ as a function of $m$. 
\label{fig.variance.PhiSR.Li01Regu.n100}
}
\end{figure}
\begin{figure}
\begin{center}
\begin{minipage}[b]{\largfiguniq}
\includegraphics[width=\textwidth]{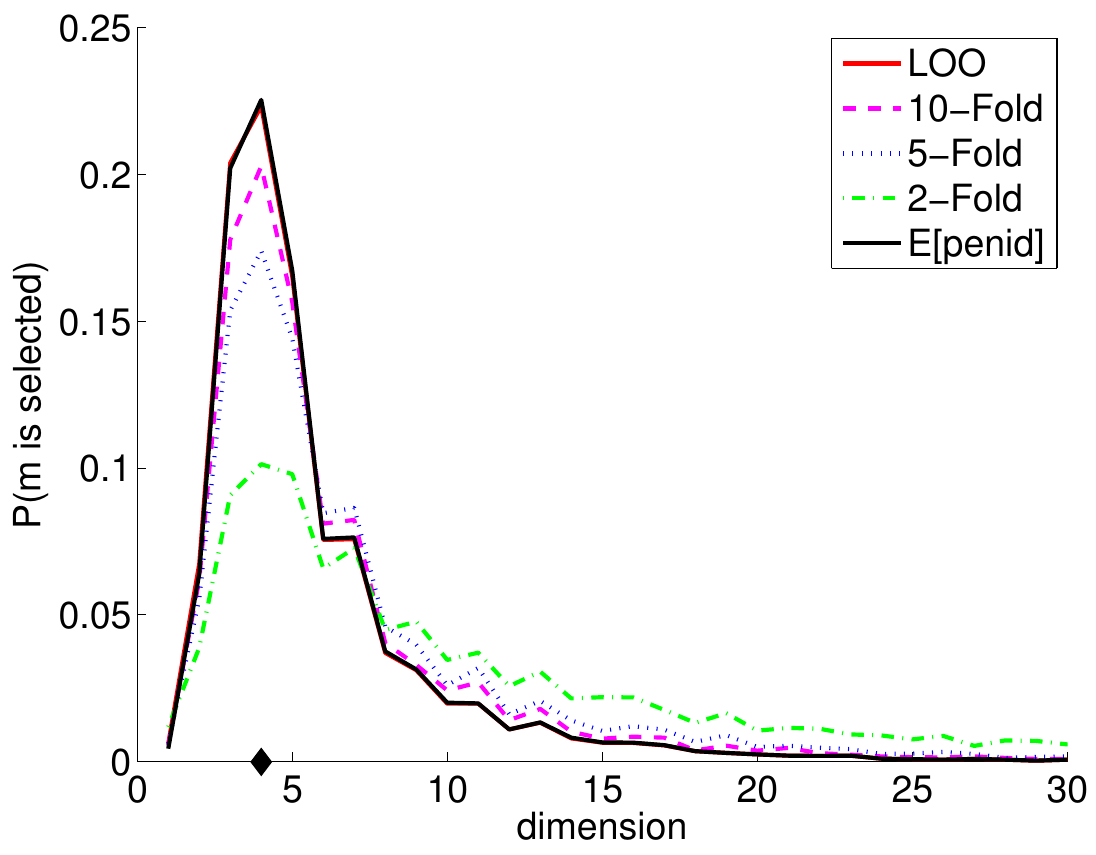}
\end{minipage}
\end{center}
\caption{%
L-Regu, $n=100$. 
$\Prob(\mh(\CV)=m)$ as a function of $m$.
\label{fig.variance.Pmh.Li01Regu.n100}
}
\end{figure}
\begin{figure}
\begin{center}
\begin{minipage}[b]{\largfiguniq}
\includegraphics[width=\textwidth]{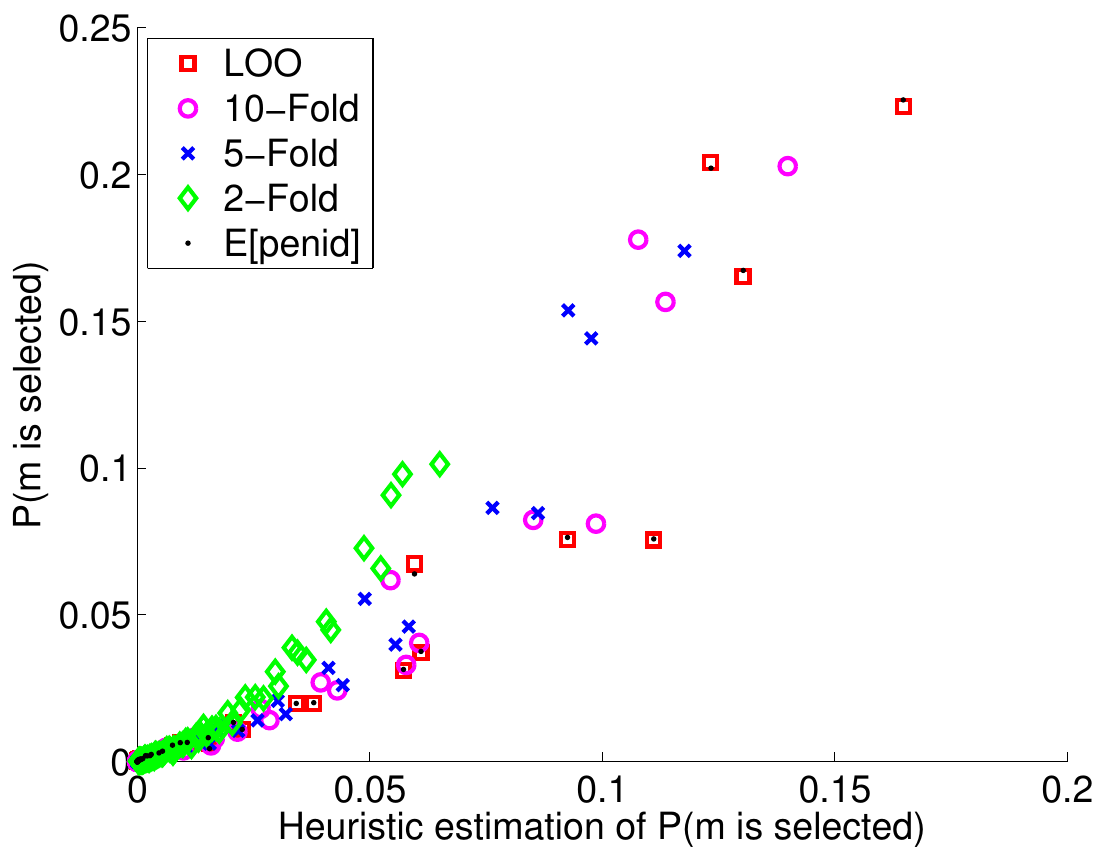}
\end{minipage}
\end{center}
\caption{%
L-Regu, $n=100$. 
$\Prob(\mh(\CV)=m)$ as a function of $\overline{\Phi}(\SR_{\, \CV}(m))$. 
\label{fig.variance.Pmh-vs-PhiSR.Li01Regu.n100}
}
\end{figure}
\begin{figure}
\begin{center}
\begin{minipage}[b]{\largfiguniq}
\includegraphics[width=\textwidth]{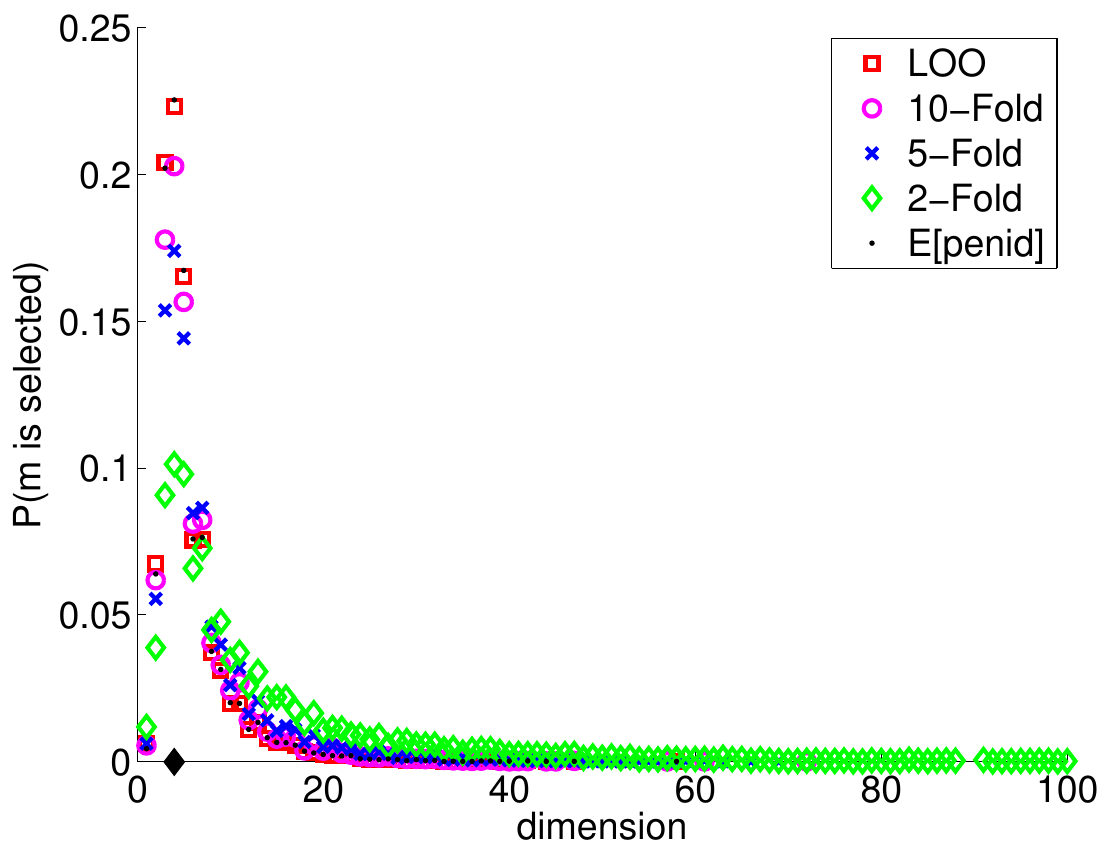}
\end{minipage}
\end{center}
\caption{%
L-Regu, $n=100$. 
$\Prob(\mh(\CV)=m)$ as a function of $m$.
\label{fig.variance.Pmh-nozoom.Li01Regu.n100}
}
\end{figure}
%
%
%
%%%%%%%%%%%% n=500
%
\begin{figure}
\begin{center}
\begin{minipage}[b]{.48\linewidth}
\includegraphics[width=\textwidth]{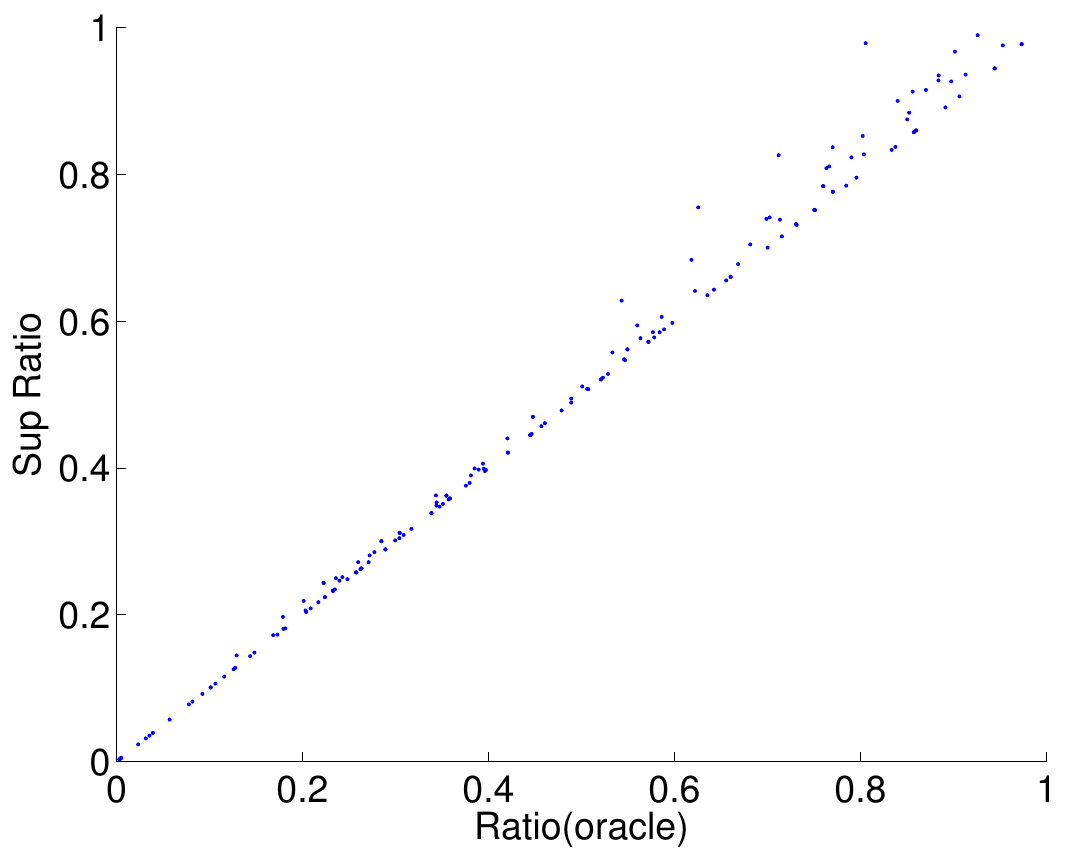}
\end{minipage}
\hspace{.025\linewidth}
\begin{minipage}[b]{.48\linewidth}
\includegraphics[width=\textwidth]{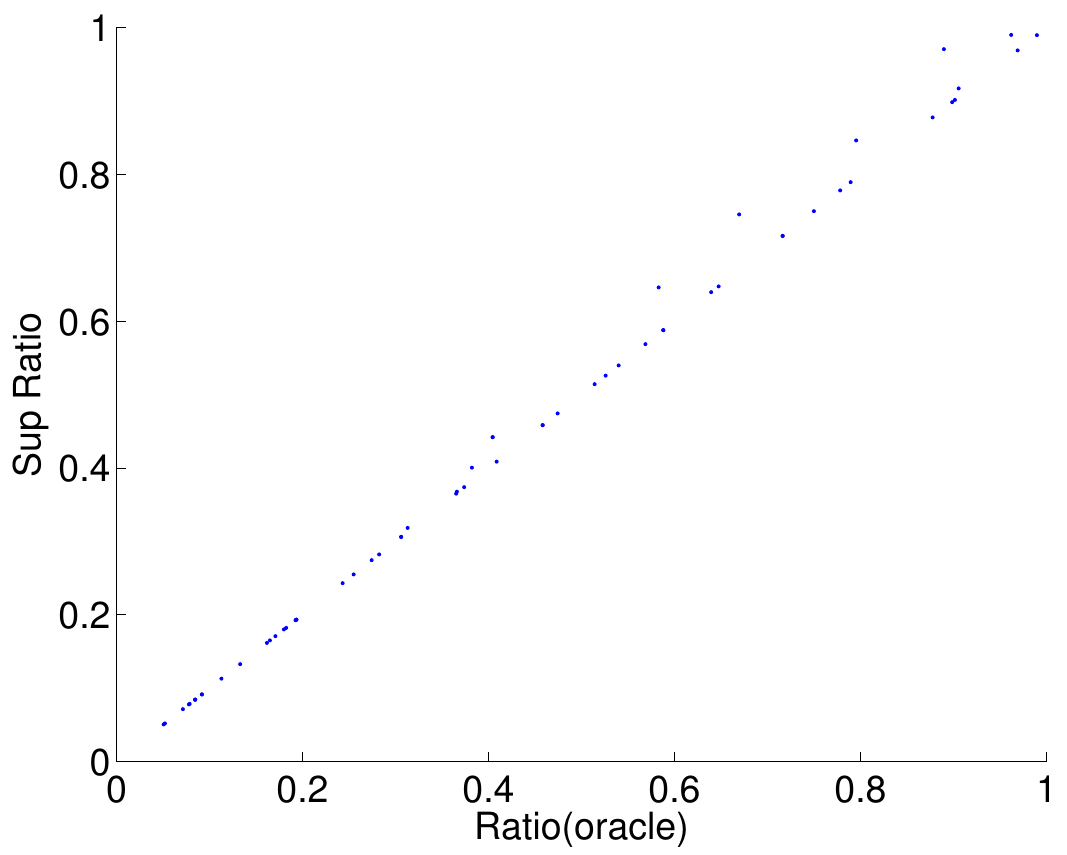}
\end{minipage}
\end{center}
\caption{%
$\SR(m)$ as a function of the ratio at $m^{\prime}=\mo$. 
$n=500$. 
Left: S-Regu. 
Right: L-Regu. 
\label{fig.variance.SR-vs-Rmo.SiG5-Li01Regu.n500}
}
\end{figure}

\begin{figure}
\begin{center}
\begin{minipage}[b]{\largfiguniq}
\includegraphics[width=\textwidth]{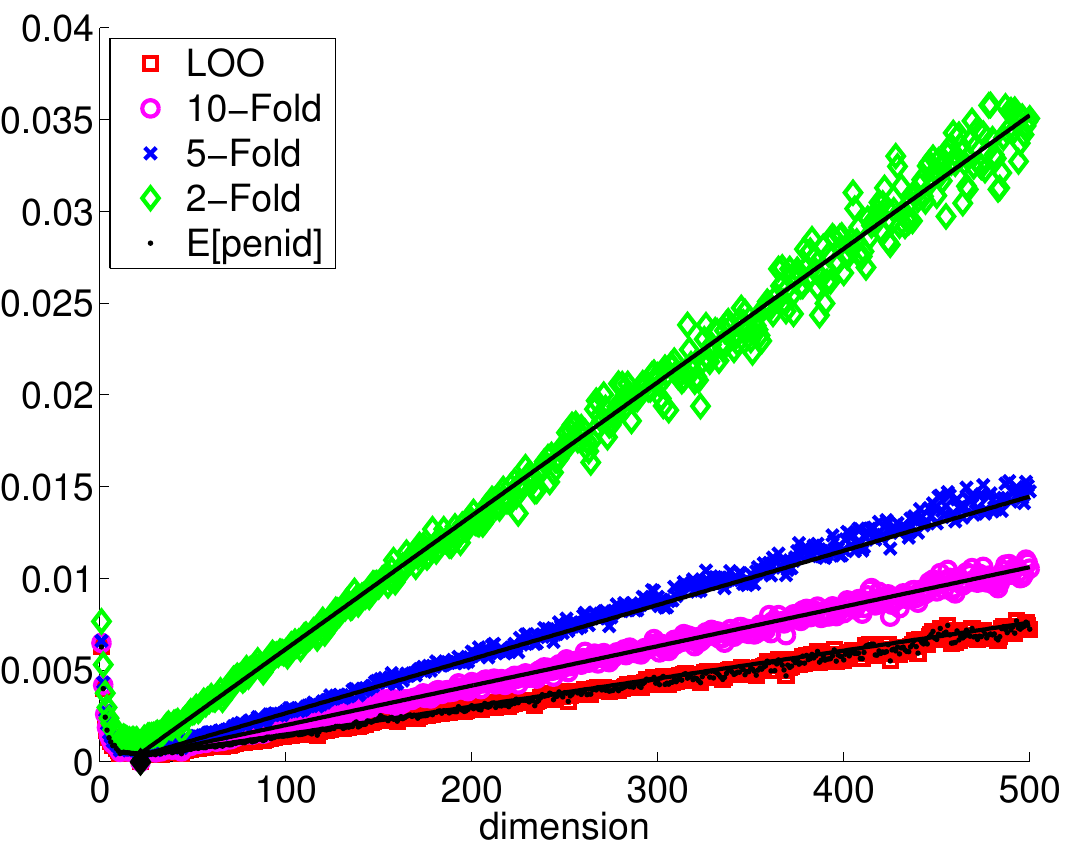}
\end{minipage}
\end{center}
\caption{%
S-Regu, $n=500$. 
$\var(\Delta_{\CV_V}(m,\mo))$ as a function of $m$. 
The black lines show the linear approximation $n^{-2} [ 75 (1 + \frac{0.52}{V-1} ) + 3.8 ( 1 + \frac{3.8}{V-1} ) (m-\mo) ]$ for $m > \mo = 22$. 
\label{fig.variance.var.SiG5Regu.n500}
}
\end{figure}
\begin{figure}
\begin{center}
\begin{minipage}[b]{\largfiguniq}
\includegraphics[width=\textwidth]{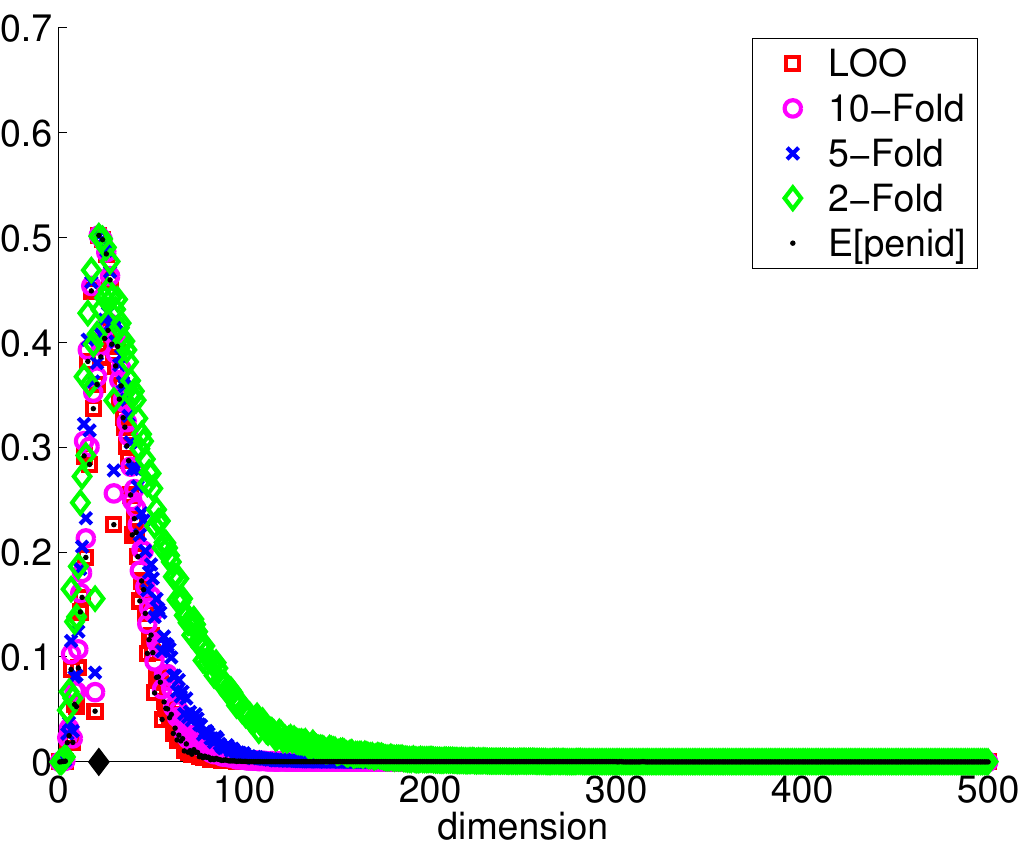}
\end{minipage}
\end{center}
\caption{%
S-Regu, $n=500$. 
$\overline{\Phi}(\SR_{\, \CV_V}(m))$ as a function of $m$. 
\label{fig.variance.PhiSR.SiG5Regu.n500}
}
\end{figure}
\begin{figure}
\begin{center}
\begin{minipage}[b]{\largfiguniq}
\includegraphics[width=\textwidth]{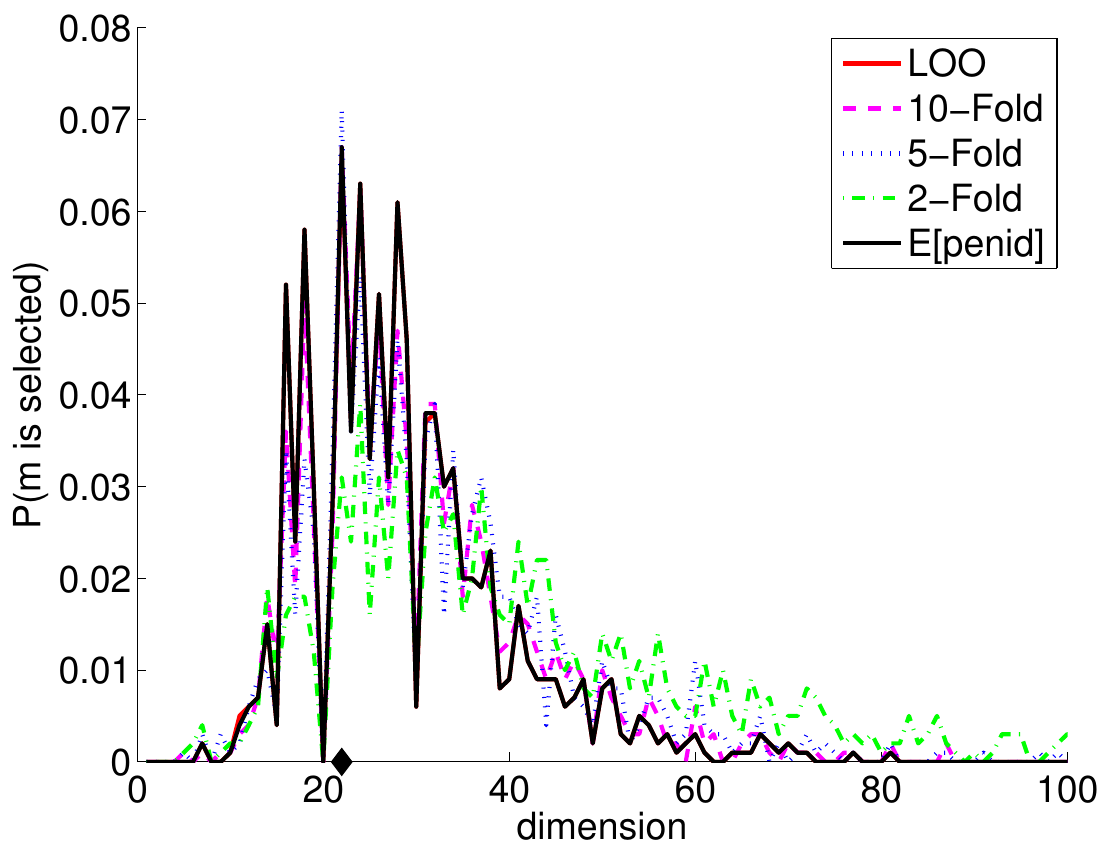}
\end{minipage}
\end{center}
\caption{%
S-Regu, $n=500$. 
$\Prob(\mh=m)$ as a function of $m$.
\label{fig.variance.Pmh.SiG5Regu.n500}
}
\end{figure}
\begin{figure}
\begin{center}
\begin{minipage}[b]{\largfiguniq}
\includegraphics[width=\textwidth]{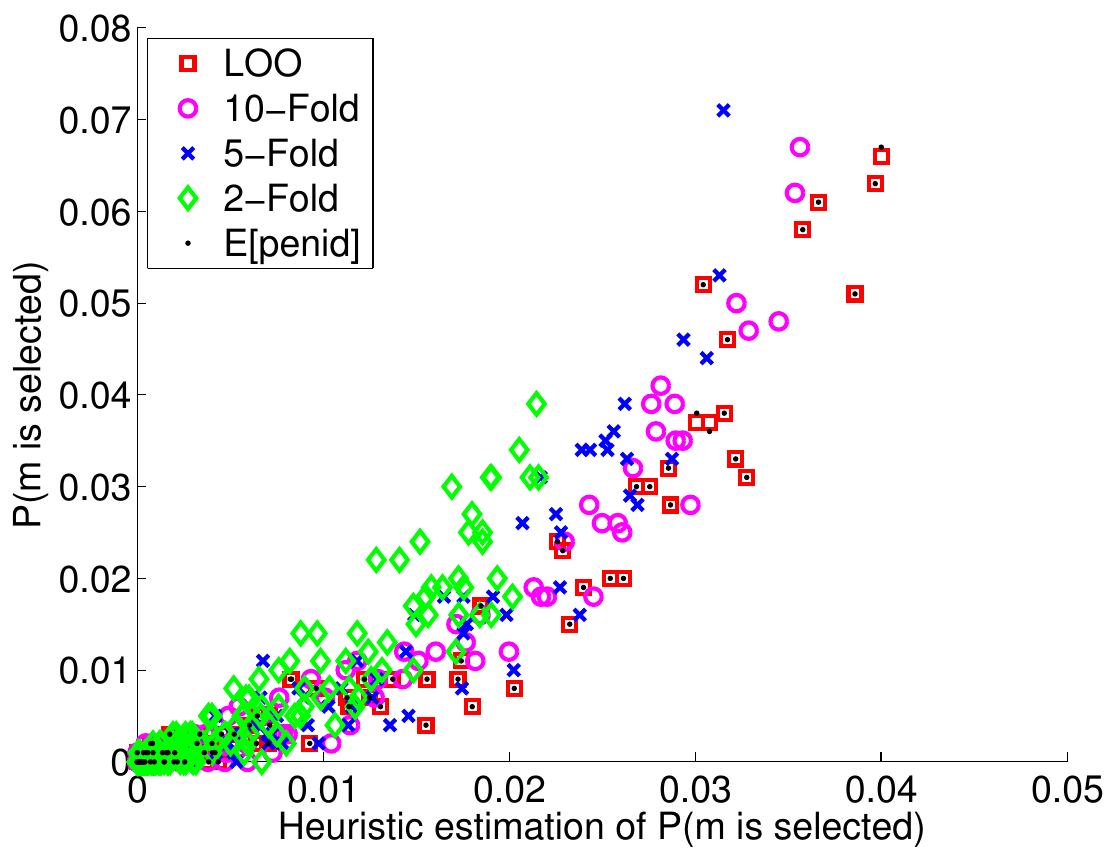}
\end{minipage}
\end{center}
\caption{%
S-Regu, $n=500$. 
$\Prob(\mh(\CV)=m)$ as a function of $\overline{\Phi}(\SR(m))$. 
\label{fig.variance.Pmh-vs-PhiSR.SiG5Regu.n500}
}
\end{figure}
\begin{figure}
\begin{center}
\begin{minipage}[b]{\largfiguniq}
\includegraphics[width=\textwidth]{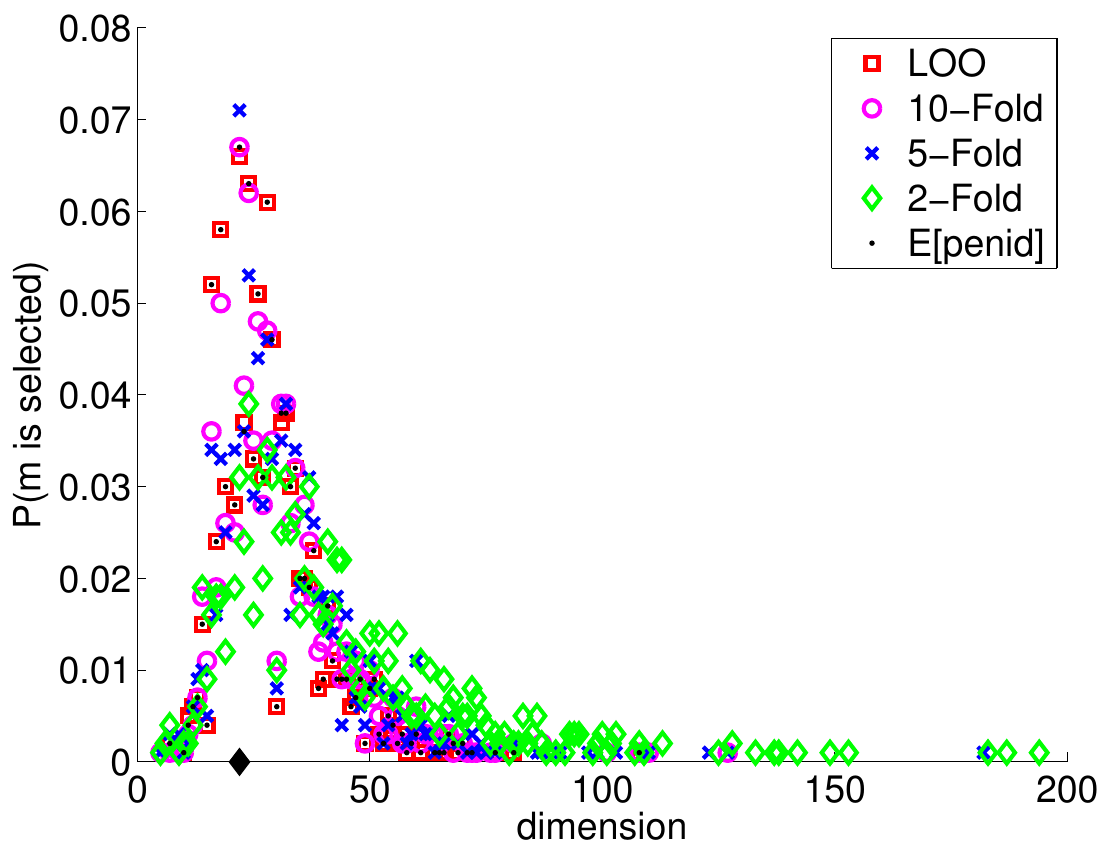}
\end{minipage}
\end{center}
\caption{%
S-Regu, $n=500$. 
$\Prob(\mh=m)$ as a function of $m$.
\label{fig.variance.Pmh-nozoom.SiG5Regu.n500}
}
\end{figure}
\begin{figure}
\begin{center}
\begin{minipage}[b]{\largfiguniq}
\includegraphics[width=\textwidth]{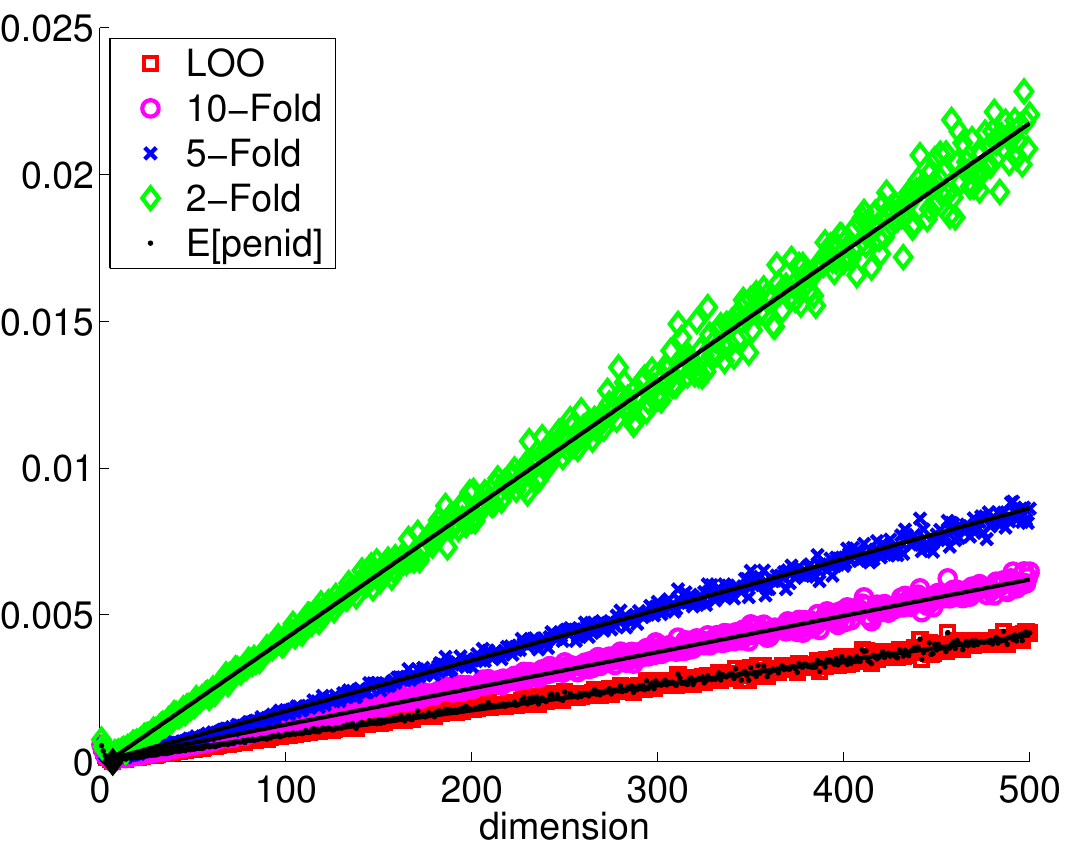}
\end{minipage}
\end{center}
\caption{%
L-Regu, $n=500$. 
$\var(\Delta_{\CV_V}(m,\mo))$ as a function of $m$. 
The black lines show the linear approximation $n^{-2} [ 28 (1 + \frac{0.06}{V-1} ) + 2.1 ( 1 + \frac{4.2 }{V-1} ) (m-\mo) ]$ for $m > \mo = 7$. 
\label{fig.variance.var.Li01Regu.n500}
}
\end{figure}
\begin{figure}
\begin{center}
\begin{minipage}[b]{\largfiguniq}
\includegraphics[width=\textwidth]{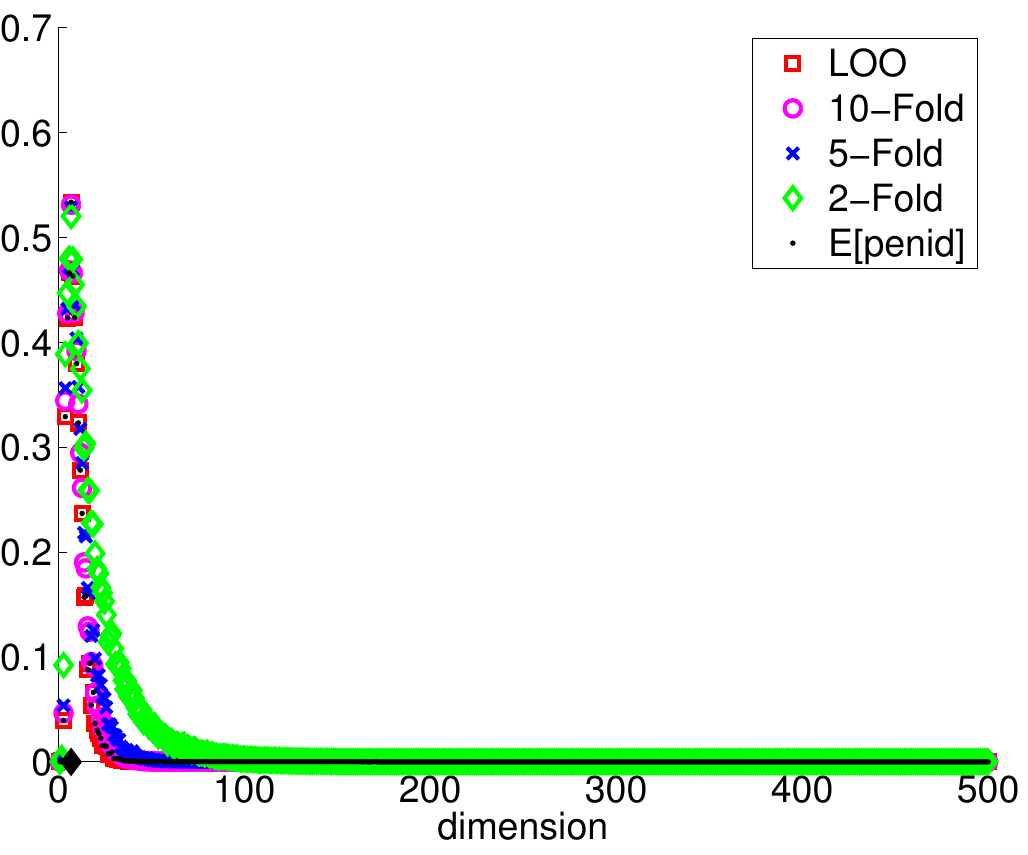}
\end{minipage}
\end{center}
\caption{%
L-Regu, $n=500$. 
$\overline{\Phi}(\SR_{\, \CV_V}(m))$ as a function of $m$.
\label{fig.variance.PhiSR.Li01Regu.n500}
}
\end{figure}
\begin{figure}
\begin{center}
\begin{minipage}[b]{\largfiguniq}
\includegraphics[width=\textwidth]{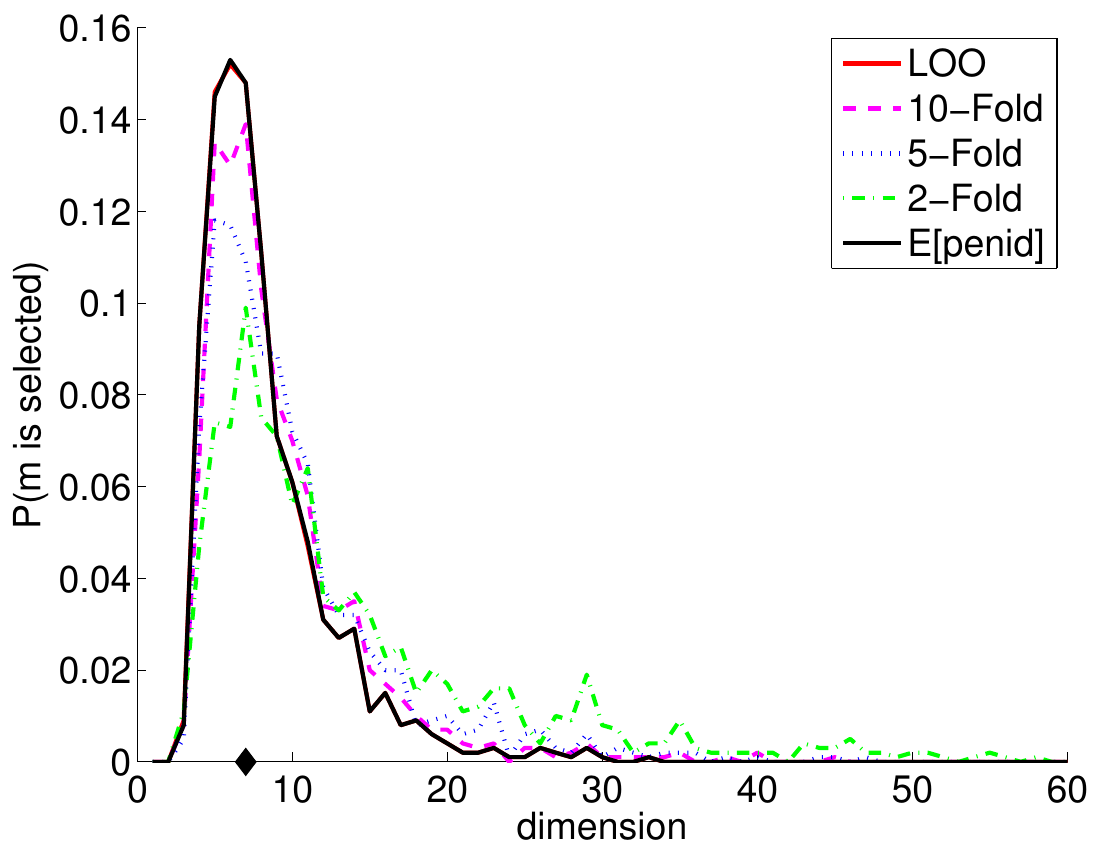}
\end{minipage}
\end{center}
\caption{%
L-Regu, $n=500$. 
$\Prob(\mh=m)$ as a function of $m$.
\label{fig.variance.Pmh.Li01Regu.n500}
}
\end{figure}
\begin{figure}
\begin{center}
\begin{minipage}[b]{\largfiguniq}
\includegraphics[width=\textwidth]{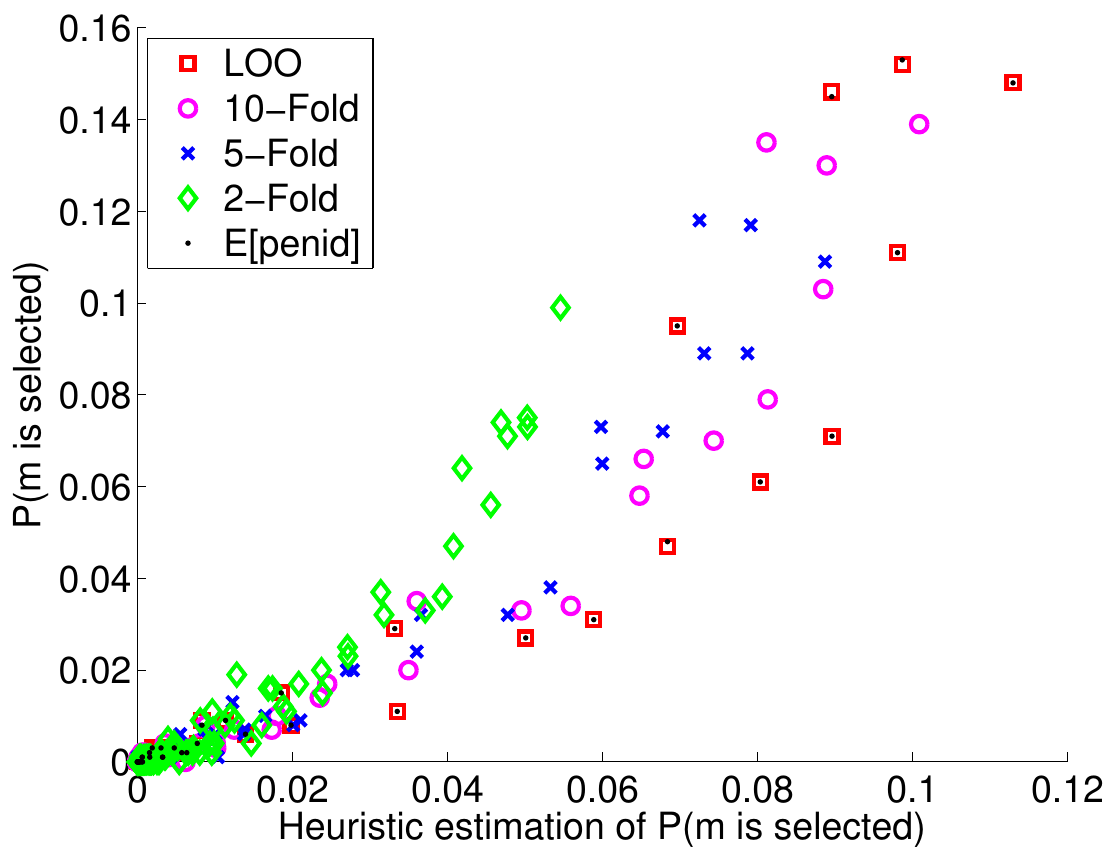}
\end{minipage}
\end{center}
\caption{%
L-Regu, $n=500$. 
$\Prob(\mh(\CV)=m)$ as a function of $\overline{\Phi}(\SR(m))$. 
\label{fig.variance.Pmh-vs-PhiSR.Li01Regu.n500}
}
\end{figure}
\begin{figure}
\begin{center}
\begin{minipage}[b]{\largfiguniq}
\includegraphics[width=\textwidth]{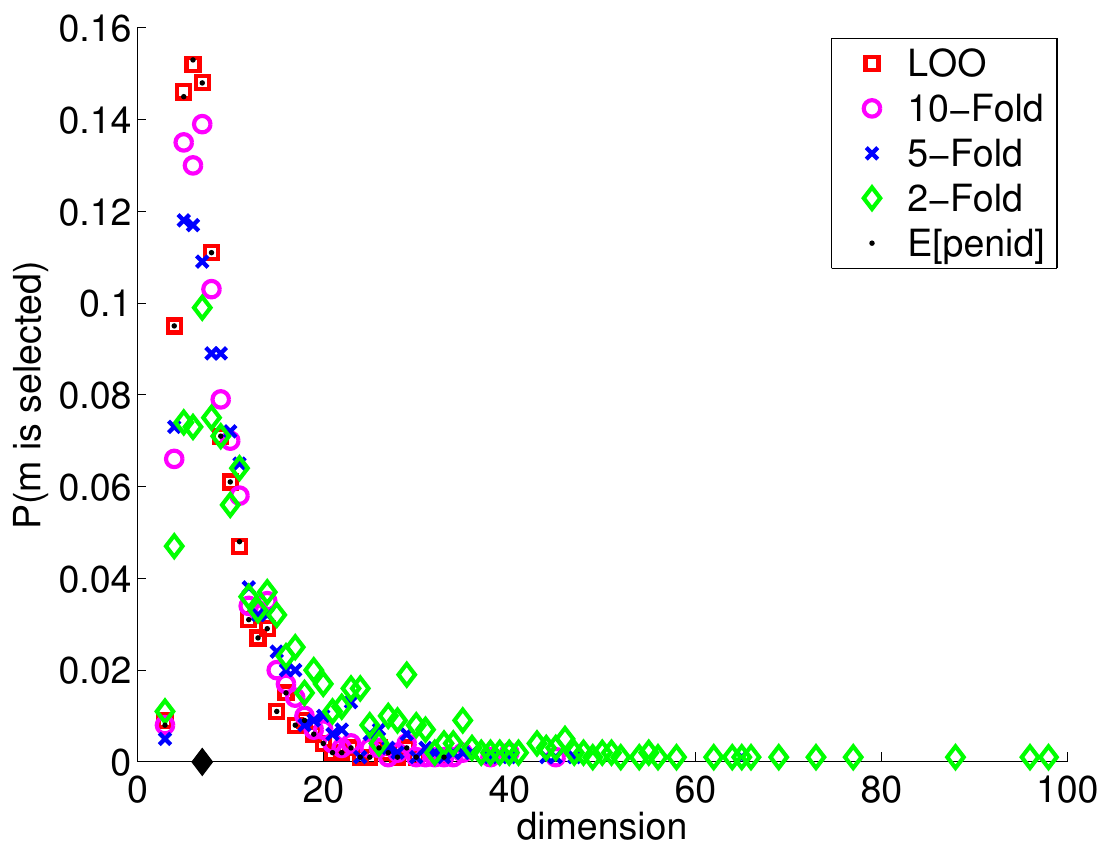}
\end{minipage}
\end{center}
\caption{%
L-Regu, $n=500$. 
$\Prob(\mh=m)$ as a function of $m$.
\label{fig.variance.Pmh-nozoom.Li01Regu.n500}
}
\end{figure}

\end{document}